\documentclass[leqno,11pt]{amsart}
\usepackage[top=1.25in, bottom=1.25in, left=1in, right=1in]{geometry}
\usepackage{amssymb,amsmath,latexsym,amsfonts,amsbsy, amsthm,dsfont}
\usepackage{hyperref}
\usepackage{color}
\usepackage{mathrsfs}
\usepackage{graphicx}
\usepackage{comment}
\usepackage{enumitem}
\usepackage{xcolor,soul}
\definecolor{pku}{RGB}{139,0,18}

\newtheorem{thm}{Theorem}[section]
\newtheorem{cor}[thm]{Corollary} 
\newtheorem{prop}[thm]{Proposition} 

\newtheorem{lem}{Lemma}[section]

\theoremstyle{definition}
\newtheorem{definition}{Definition}[section]

\theoremstyle{remark}

\newtheorem{rmk}{Remark}[section]

\let\pa=\partial
\let\f=\frac

\let\pa=\partial
\let\na=\nabla

\let\al=\alpha
\let\b=\beta
\let\e=\varepsilon
\let\d=\delta
\let\D=\Delta
\let\g=\gamma
\let\G=\Gamma

\let\lam=\lambda
\let\Lam=\Lambda

\let\va=\varphi

\let\r=\rho
\let\s=\sigma
\let\th=\theta

\newcommand{\beq}{\begin{equation}}
\newcommand{\eeq}{\end{equation}}
\newcommand{\beqo}{\begin{equation*}}
\newcommand{\eeqo}{\end{equation*}}

\newcommand{\Id}{\mathrm{Id}}

\newcommand{\pv}{\mathrm{p.v.}}
\newcommand{\di}{\mathrm{div}\,}

\newcommand{\CE}{\mathcal{E}}

\newcommand{\CD}{\mathcal{D}}
\newcommand{\CI}{\mathcal{I}}

\newcommand{\CU}{\mathcal{U}}
\newcommand{\CF}{\mathcal{F}}
\newcommand{\CT}{\mathcal{T}}

\newcommand{\CJ}{\mathcal{J}}

\newcommand{\CX}{\mathcal{X}}

\newcommand{\R}{\mathbb{R}}
\newcommand{\T}{\mathbb{T}}

\newcommand{\BR}{\mathbb{R}}
\newcommand{\BZ}{\mathbb{Z}}
\newcommand{\BT}{\mathbb{T}}
\newcommand{\BP}{\mathbb{P}}

\newcommand{\BN}{\mathbb{N}}

\numberwithin{equation}{section}

\allowdisplaybreaks

\begin{document}
\title[The Immersed Boundary Problem in 2-D]{The Immersed Boundary Problem in 2-D:\\the Navier-Stokes Case}


\author{Jiajun Tong}
\address{Beijing International Center for Mathematical Research, Peking University, Beijing 100871, China}
\email{tongj@bicmr.pku.edu.cn}
\author{Dongyi Wei}
\address{School of Mathematical Sciences, Peking University, Beijing 100871, China}
\email{jnwdyi@pku.edu.cn}

\date{\today}
\maketitle

\begin{abstract}
We study the immersed boundary problem in 2-D.
It models a 1-D elastic closed string immersed and moving in a fluid that fills the entire plane, where the fluid motion is governed by the 2-D incompressible Navier-Stokes equation with a positive Reynolds number subject to a singular forcing exerted by the string.
We introduce the notion of mild solutions to this system, and prove its existence, uniqueness, and optimal regularity estimates when the initial string configuration is $C^1$ and satisfies the well-stretched condition and when the initial flow field $u_0$ lies in $L^p(\mathbb{R}^2)$ with $p\in (2,\infty)$.
A blow-up criterion is also established.
When the Reynolds number is sent to zero, we show convergence in short time of the solution to that of the Stokes case of 2-D immersed boundary problem, with the optimal error estimates derived.
We prove the energy law of the system when $u_0$ additionally belongs to $L^2(\mathbb{R}^2)$.
Lastly, we show that the solution is global when the initial data is sufficiently close to an equilibrium state.
\end{abstract}

\tableofcontents

\section{Introduction}
In this paper, we consider the immersed boundary problem in 2-D with the incompressible Navier-Stokes equation.
It models a 1-D elastic closed string immersed and moving in a fluid that fills the entire $\BR^2$.
The string applies elastic force to the ambient fluid and thus affects the flow field, and meanwhile, the flow moves and deforms the string. This gives an autonomous fluid-structure interaction problem, featuring singular forcing on the flow field along a time-varying 1-D string.
This problem arises from the general mathematical formulation of the immersed boundary problem introduced by Peskin \cite{peskin1972flow,peskin2002immersed}, and it is also closely related to the well-known numerical immersed boundary method he proposed, which has been widely applied as a power computational tool in physics, biology, medical sciences, and engineering over the past half a century \cite{griffith2020reviewIB,mittal2005reviewIB,peskin2002immersed,verzicco2023reviewIB}.
The primary goal of this paper is to present a comprehensive rigorous study of the problem, so as to put its mathematical formulation on the solid ground.

\subsection{Problem formulation}
We model the 1-D elastic string by a planar Jordan curve.
We parameterize it by an $\BR^2$-valued function $X = X(s,t)$, where $s\in \BT : = \BR/(2\pi \BZ)$ denotes the Lagrangian (material) coordinate, and where $t\geq 0$ denotes the time.
Here $X(s,t)$ gives the position of the material point along the string with the label $s$ at time $t$, and we always assume that the parameterization goes counter-clockwise along the curve $X(\BT,t)$.
Note that $s$ is not the arc-length parameter.
We shall use $u=u(x,t)$ and $p=p(x,t)$ to denote the velocity field and the pressure in the fluid, respectively.
Then the problem reads that \cite{peskin2002immersed}
\begin{align}
&\r(\pa_t u + u\cdot \na u) + \na p = \mu\D u +f_X ,\label{eqn: NS equation dimensional}\\
&\di u = 0,\quad |u|,|p|\to 0 \mbox{ as }|x|\to 0,\\
&f_X(x,t) = \int_\BT F_X(s,t)\d (x- X(s,t))\,ds,\label{eqn: def of force dimensional}\\
&\pa_t X(s,t) = u(X(s,t),t),\label{eqn: kinematic equation dimensional}\\
&X(s,0) = X_0(s),\quad u(x,0) = u_0(x).\label{eqn: initial data dimensional}
\end{align}
Here $\r$ and $\mu$ are the density and dynamic viscosity of the fluid, respectively, which are assumed to be positive constants;
$f_X(x,t)$ is the elastic force applied to the fluid, which is singularly supported along the curve $X(\BT,t)$;
and $\d$ denotes the Dirac $\d$-function in $\BR^2$.
In \eqref{eqn: def of force dimensional}, $F_X = F_X(s,t)$ denotes the force density in the Lagrangian coordinate.
For simplicity, in this paper, we only study the case of the linear Hookean elasticity, i.e.,
\beq
F_X(s,t) = k X_{ss}(s,t),
\label{eqn: Hooke's law dimensional}
\eeq
where $k>0$ denotes the Hooke's constant, and where $X_{ss}$ denotes the second derivative of $X$ with respect to the $s$-variable.
In other words, each infinitesimal segment of the string is modeled as a Hooke's spring with zero resting length.

\begin{rmk}
To study general elasticity laws, one has \cite{peskin2002immersed}
\[
F_X(s,t) = \pa_s\left(\CT(|X_s(s,t)|)\f{X_s(s,t)}{|X_s(s,t)|}\right),
\]
where $\CT = \CT(q)$ denotes the tension density of the string when an infinitesimal string segment is stretched by an amount of $q\in [0,+\infty)$.
$\CT(q)$ is determined by the elastic energy density $\CE(q)$ of the string material, $\CT(q) = \CE'(q)$.
When an active elastic string is considered, e.g., in some biological models, one may assume spatially inhomogeneous and time-varying elasticity laws such as $\CT = \CT(q,s,t)$.
\end{rmk}

\subsection{Non-dimensionalization}
We first perform non-dimensionalization of the model to simplify the parameters.
Let $L_*,T_*,M_*>0$ denote characteristic length, time, and mass scales respectively, which are to be chosen later.
We introduce dimensionless quantities as follows:
\begin{align*}
& \left(\tilde{x},\,\tilde{t},\,\tilde{u},\,\tilde{p},\,
\tilde{X},\,\tilde{f}_{\tilde{X}},\,\tilde{F}_{\tilde{X}},\,
\tilde{\r},\,\tilde{\mu},\tilde{k}\right) \\
& := \left(\f{x}{L_*},\,\f{t}{T_*},\,\f{u}{L_*/T_*},\, \f{p}{M_*/T_*^2},\, \f{X}{L_*}, \, \f{f_X}{M_*/(L_*T_*^2)}, \, \f{F_X}{M_* L_*/T_*^2}, \, \f{\r}{M_*/L_*^2},\, \f{\mu}{M_*/T_*},\, \f{k}{M_*/T_*^2}\right).
\end{align*}
By a change of variable, one can derive from \eqref{eqn: NS equation dimensional}-\eqref{eqn: Hooke's law dimensional} a dimensionless system in the $(\tilde{x},s,\tilde{t})$-variable:
\begin{align*}
&\tilde{\r}(\pa_{\tilde{t}} \tilde{u} + \tilde{u}\cdot \na_{\tilde{x}} \tilde{u}) + \na_{\tilde{x}} \tilde{p} = \tilde{\mu}\D_{\tilde{x}} \tilde{u} +\tilde{f}_{\tilde{X}},\\
&\mathrm{div}_{\tilde{x}}\, \tilde{u} = 0,\quad |\tilde{u}|,|\tilde{p}|\to 0 \mbox{ as }|\tilde{x}|\to 0,\\
&\tilde{f}_{\tilde{X}}(\tilde{x},\tilde{t}) = \int_\BT \tilde{k} \tilde{X}_{ss}(s,\tilde{t}) \d \big(\tilde{x}- \tilde{X}(s,\tilde{t})\big)\,ds,\\
&\pa_{\tilde{t}} \tilde{X}(s,\tilde{t}) = \tilde{u}\big(\tilde{X}(s,\tilde{t}),\tilde{t}\big),\\
&\tilde{X}(s,0) = \tilde{X}_0(s),\quad \tilde{u}(\tilde{x},0) = \tilde{u}_0 (\tilde{x}),
\end{align*}
where
\[
\tilde{X}_0(s) :=L_*^{-1} X_0(s),\quad \tilde{u}_0(\tilde{x}) := T_*L_*^{-1} u_0 (L_* \tilde{x}).
\]

Now we define $L_*>0$ by
\beqo
\pi L_*^2 = \f12 \int_\BT X_0(s)\times X_0'(s)\,ds,
\eeqo
where the right-hand side represents the area of the planar domain enclosed by the Jordan curve $X_0(\BT)$.
This definition is motivated by the fact that, if the system \eqref{eqn: NS equation dimensional}-\eqref{eqn: Hooke's law dimensional} admits a
well-behaved solution, the area of the domain enclosed by $X(\BT,t)$ should be time-invariant because of the incompressibility of the flow.
We further choose $T_* := \mu k^{-1}$ and $M_* := \mu^2 k^{-1}$ so that $\tilde{\mu} = \tilde{k} = 1$.
Then
\[
\tilde{\rho} = \f{\rho}{M_*/L_*^2} = \f{\rho k}{2\pi \mu^2}\int_\BT X_0(s)\times X_0'(s)\,ds =:Re.
\]
The dimensionless constant $Re$ is exactly the Reynolds number.

In the rest of the paper, we will focus on the dimensionless system.
With the above choice of characteristic scales and with all the tilde symbols omitted for brevity, it can be summarized as
\begin{align}
& Re(\pa_t u + u\cdot \na u) + \na p = \D u +f_X ,\label{eqn: NS equation}\\
&\di u = 0,\quad |u|,|p|\to 0 \mbox{ as }|x|\to 0,\\
&f_X(x,t) = \int_\BT X_{ss}(s,t)\d (x- X(s,t))\,ds,\label{eqn: def of force}\\
&\pa_t X(s,t) = u(X(s,t),t),\label{eqn: kinematic equation}
\end{align}
and the initial condition is given by
\beq
X(s,0) = X_0(s),\quad u(x,0) = u_0(x).\label{eqn: initial data}
\eeq
Our goal is to prove the existence, uniqueness, and various other properties of its solutions.
The main results of the paper will be presented in Section \ref{sec: main result}.
Readers are also referred to Section \ref{sec: organization of the paper} for the organization of the paper.

\subsection{Related studies}
\label{sec: Stokes case revisited}
As is mentioned before, the mathematical formulation of the immersed boundary problem and the numerical immersed boundary method was proposed in early 1970s \cite{peskin1972flow}, and since then, the latter has been extensively investigated and widely applied (see e.g.\;the reviews \cite{griffith2020reviewIB,mittal2005reviewIB,peskin2002immersed,verzicco2023reviewIB} and the references therein).
However, rigorous analytic studies of the immersed boundary problem only started within the past decade.
So far, the existing literature focuses on the system \eqref{eqn: NS equation}-\eqref{eqn: initial data} with $Re = 0$, which is called the \emph{Stokes case} of the 2-D immersed boundary problem, also known as the 2-D \emph{Peskin problem}:
\begin{align}
& -\D u + \na p = f_X ,\label{eqn: Stokes equation}\\
&\di u = 0,\quad |u|,|p|\to 0 \mbox{ as }|x|\to 0,\\
&f_X(x,t) = \int_\BT F_X(s,t)\d (x- X(s,t))\,ds,\label{eqn: def of force Stokes case}\\
&\pa_t X(s,t) = u(X(s,t),t),\label{eqn: kinematic equation Stokes case}
\\
&X(s,0) = X_0(s).\label{eqn: initial data Stokes case}
\end{align}
What makes it mathematically more tractable than the general case is that, it can be reformulated into a contour dynamic equation for $X=X(s,t)$ in the Lagrangian coordinate by virtue of the fundamental solution to the 2-D stationary Stokes equation (i.e., the 2-D Stokeslet)
\beq
G_{ij}(x) := \f{1}{4 \pi}\left(-\ln |x|\d_{ij} + \f{x_i x_j}{|x|^2}\right),\quad i,j = 1,2.
\label{eqn: Stokeslet in 2-D}
\eeq
For instance, in the case $F_X = X_{ss}$, we can obtain that (see e.g.\;\cite{lin2019solvability,mori2019well} for details)
\begin{align*}
\pa_t X  = &\;
\int_\BT G(X(s,t)-X(s',t)) X_{ss}(s',t)\,ds'.
\\
= &\; \f{1}{4\pi}\mathrm{p.v.}\int_\BT \left[-\f{|X'(s')|^2}{|X(s')-X(s)|^2} + \f{2((X(s')-X(s))\cdot X'(s'))^2}{|X(s')-X(s)|^4} \right]\big(X(s')-X(s)\big)\,ds'.
\end{align*}
Here in the second line, for brevity, we wrote $X_s$ as $X'$ and the $t$-dependence is omitted.
By extracting the principal singular part in the integral, we can rewrite it as a fractional heat equation
\beq
\pa_t X(s,t) = -\f14\Lam X(s,t)+ g_X(s,t),
\label{eqn: contour dynamic equation Stokes case}
\eeq
where $\Lam := (-\D)^{1/2}_\BT$, and where $g_X$ is defined in \eqref{eqn: def of g_X} below.
One can show that $g_X$ is in fact a regular forcing term, as long as $X$ admits suitable regularity and satisfies \emph{the well-stretched condition} --- a string configuration $Y(s)$ is said to satisfy the well-stretched condition (or simply, $Y(s)$ is well-stretched) with constant $\lam>0$, if
\beq
\big|Y(s_1)-Y(s_2)\big|\geq \lam |s_1-s_2|_\BT,\quad \forall \,s_1,s_2\in \BT,
\label{eqn: well-stretched condition}
\eeq
where $|\cdot|$ on the left-hand side denotes the norm in $\BR^2$, while $|\cdot|_\BT$ on the right-hand side denotes the distance on $\BT$.
Based on this idea, well-posedness of \eqref{eqn: Stokes equation}-\eqref{eqn: initial data Stokes case} with the linear Hookean elasticity $F_{X} = X_{ss}$ was first proved in the independent works \cite{lin2019solvability} and \cite{mori2019well} for well-stretched initial data $X_0$ that lies in the subcritical spaces $H^{5/2}(\BT)$ and $h^{1,\g}(\BT)$ (the little H\"{o}lder space with $\g\in (0,1)$), respectively.
Besides, among other results, \cite{mori2019well} also showed an instant smoothing property of the solution and a blow-up criterion.
Later, Rodenberg \cite{rodenberg20182d} extended the local well-posedness result in \cite{mori2019well} to the case of general nonlinear elasticity.

After that, several well-posedness results have been obtained in scaling-critical regularity classes.
Garc\'{i}a-Ju\'{a}rez, Mori, and Strain \cite{garcia2020peskin} considered the case where the fluids in the interior and the exterior of the string can have different viscosities, and they proved global well-posedness for medium-size initial data in the critical Wiener space $\CF^{1,1}(\BT)$.
In \cite{gancedo2020global}, Gancedo, Granero-Belinch\'{o}n, and Scrobogna introduced a toy model that captures the normal motion of the string, and its global well-posedness was proved for small initial data in the critical Lipschitz class $W^{1,\infty}(\BT)$.
Chen and Nguyen \cite{chen2021peskin} proved the well-posedness in the critical space $\dot{B}_{\infty,\infty}^{1}(\BT)$, which contains non-Lipschitz functions.
For the 2-D Peskin problem with general nonlinear elasticity, Cameron and Strain \cite{cameron2021critical} proved the local well-posedness in $B^{3/2}_{2,1}(\BT)$, while in a very recent work \cite{GARCIAJUAREZ2025110047}, Garc\'{i}a-Ju\'{a}rez and Haziot obtained the global well-posedness and asymptotic stability for small initial data in some critical space which allows corners.

Some recent studies revealed special nonlinear properties of the 2-D Peskin problem.
In \cite{tong2024tangential}, the first author of this paper introduced the tangential Peskin problem with a Eulerian formulation, which models an infinitely long and straight 1-D elastic string deforming tangentially in a Stokes flow in $\BR^2$.
Global solutions were constructed for initial datum in the energy class, which had been regarded as super-critical in the classic treatment using the Lagrangian formulation.
In our recent paper \cite{tong2024geometric} on the original 2-D Peskin problem, we observed extremum principles for certain geometric quantities of the string, and then proved a global well-posedness result that only imposes a medium-size geometric condition on the initial string shape but no assumption on the size of its stretching.
Connection between the tangential problem and the dynamics of a circular string in the full model was also illustrated in this work.

In addition to those works, the first author of this paper studied the regularized Peskin problem inspired by numerical interests \cite{tong2021regularized}.
Li \cite{li2021stability} studied the case where the elastic string has both stretching and bending energy.
In \cite{GarciaJuarez2025Peskin3D}, Garc\'{i}a-Ju\'{a}rez, Kuo, Mori, and Strain considered the 3-D Peskin problem, where a 2-D closed elastic membrane moves in a 3-D fluid, and they proved its local well-posedness and smoothing of the solution.
In a very recent work \cite{Chen2024wellposedness_quasilinear}, Chen, Hu, and Nguyen obtained local well-posedness in critical spaces of the 2-D and 3-D Peskin problems with nonlinear elasticity as applications of their general regularity theory for quasilinear evolution equations.

Let us mention another two lines of recent research related to the Peskin problem.
It is natural to consider a 1-D elastic string moving in a 3-D Stokes flow.
However, in this case, the string velocity cannot be well-defined since the 3-D Stokeslet is too singular to be integrated along a 1-D curve.
Instead, one can consider the slender body problem where the 1-D string here is replaced by a filament which is a 3-D object with small constant cross-sectional radius; see the recent analytic works in \cite{mori2021accuracy,mori2020theoretical, mori2020theoretical_free_end,ohm2024free,ohm2025angle,ohm2025slender}.
Besides, the tension determination problem was also studied recently as an inextensible version of the 2-D Peskin problem, where an inextensible closed string applies bending forces to the surrounding Stokes flow \cite{garcia2025immersed,kuo2023tension}.

The Peskin problem is known to be mathematically similar to several other evolution free boundary problems.
For instance, the Muskat problem \cite{muskat1934two} arising from the oil industry has attracted enormous attention in the analysis community over the past two decades; see e.g.\;\cite{Alazard2020EndpointST,Ambrose2004WellposednessOT,cameron2018global,Cameron2026111257,chen2022muskat, constantin2016muskat,constantin2012global, cordoba2011,cordoba2013porous,Crdoba2018GlobalWF,deng2017two,
gancedo2019muskat} and the references therein.
Another one is the interface dynamics in a two-phase Stokes flow driven by the surface tension \cite{bohme2024well,bohme2025well,choi2024stability,matioc2021two,matioc2022two,pruss2016moving}.
When the two fluids have the same density and viscosity, this model can be viewed as \eqref{eqn: Stokes equation}-\eqref{eqn: initial data Stokes case} with $F_X(s,t) = \pa_s(X_s(s,t)/|X_s(s,t)|)$ (if one ignores the diversity in the interface and domain geometry).
There is also a counterpart coming with the Navier-Stokes equation; see e.g.\;\cite{denisova1994problem,denisova1995classical, kohne2013qualitative,pruss2010rayleigh,pruss2010two,pruss2011analytic, pruss2016moving,tanaka1993global01011993}.
In contrast to the Peskin problem, in these problems, the dynamics is solely governed by the shape of the interface while its parameterization is not intrinsic to the problem.

To the best of our knowledge, there has been no rigorous analysis result on the general 2-D immersed boundary problem \eqref{eqn: NS equation}-\eqref{eqn: initial data} with the Navier-Stokes equation.

\subsection{Organization of the paper}
\label{sec: organization of the paper}
The rest of this paper is structured as follows.
We define necessary notations in Section \ref{sec: definitions and notations}, and introduce an important decomposition of the flow field in Section \ref{sec: def of basic quantities}.
After that, we will explain the key ideas of our analysis.
In Section \ref{sec: main result}, we introduce the notion of the mild solution to \eqref{eqn: NS equation}-\eqref{eqn: initial data}, and present the main results of the paper.
A priori estimates for each part of the decomposition of the flow field will be established in Section \ref{sec: a priori estimates}.
Using these estimates, we construct a local mild solution by a fixed-point argument in Section \ref{sec: local well-posedness}.
Section \ref{sec: properties of mild solutions} is devoted to studying the regularity and uniqueness of the mild solution.
In Section \ref{sec: regularity of mild solutions}, we derive sharp regularity estimates for the mild solutions in terms of some basic bounds of them.
Based on that, we address in Section \ref{sec: uniqueness and continuation} the uniqueness and continuation of the mild solution, and also prove a blow-up criterion.
In Section \ref{sec: higher regularity}, we further show that the solution enjoys even higher regularity.
In Section \ref{sec: zero Reynolds number limit}, we justify the short-time convergence of the Navier-Stokes case to the Stokes case as $Re\to 0^+$, where the error estimates are derived with the optimal dependence on the Reynolds number.
In Section \ref{sec: local well-posedness for u_0 in L^2 and L^p}, we prove the energy law of the system in the case that the initial flow field additionally satisfies $u_0\in L^2(\BR^2)$.
Using that, in Section \ref{sec: global well-posedness}, we show that, for initial data $(u_0,X_0)$ that is sufficiently close to an equilibrium in a suitable sense, we can actually obtain a global solution.
We collect some parabolic estimates and simple calculus results that will be used throughout the paper in Appendix \ref{sec: parabolic estimates}. 
Lastly, in Appendix \ref{sec: improved estimates for g_X}, we present the proofs of several improved estimates for the nonlinear term $g_X$ in the Stokes case (see \eqref{eqn: contour dynamic equation Stokes case}).

\subsection*{Acknowledgement}
The authors would like to thank De Huang, Jinzi Mac Huang, and Xiaoyutao Luo for helpful discussions.
Both authors are supported by the National Key Research and Development Program of China under the grant No.~2021YFA1001500.


\section{Preliminaries and Main Results}
\label{sec: preliminary}

\subsection{Definitions and notations}
\label{sec: definitions and notations}

Throughout the paper, if not otherwise stated:
\begin{itemize}
\item We will always use $s$ and $s'$ to denote the Lagrangian coordinate that ranges in $\BT$ or its subsets;
\item $x$ and $y$ will always denote the Eulerian coordinate that ranges in $\BR^2$ or its subsets;
\item $t$, $\tau$, $\eta$, etc., are used for the time variables;
\item For $X = X(s)$ defined on $\BT$ or $X = X(s,t)$ defined on $\BT\times [0,T]$, for brevity, we will use $X'$ and $X''$ to denote $X_{s}$ and $X_{ss}$, i.e., the first and second (partial) derivatives of $X$ with respect to the $s$-variable, respectively;
\item $\Lam := (-\D)^{1/2}_\BT$.
\end{itemize}

Let us introduce some standard norms.
\begin{itemize}
\item
Let $u = u(x)$ be defined on $\BR^2$.
For $k\in \BN$, let
\[
\|u\|_{C_x^k(\BR^2)}
:= \sum_{j =0}^k\|\na^j u\|_{L_x^\infty(\BR^2)},
\quad
\|u\|_{\dot{C}_x^{k}(\BR^2)}: = \|\na^ku \|_{L_x^\infty(\BR^2)}.
\]
We also denote
\[
\|u\|_{C_x(\BR^2)} := \|u\|_{C_x^{0}(\BR^2)} = \|u\|_{\dot{C}_x^{0}(\BR^2)}
= \|u\|_{L_x^\infty(\BR^2)}.
\]
With $\al\in (0,1)$, let
\[
\|u\|_{\dot{C}_x^\al(\BR^2)}: = \sup_{x_1,x_2\in \BR^2\atop x_1\neq x_2}\f{|u(x_1)-u(x_2)|}{|x_1-x_2|^\al},
\]
and for $k\in \BN$, define the H\"{o}lder semi-norms and H\"{o}lder norms as
\begin{align*}
\|u\|_{\dot{C}_x^{k+\al}(\BR^2)}
= &\; \|u\|_{\dot{C}_x^{k,\al}(\BR^2)}: =\|\na^k u\|_{\dot{C}_x^{\al}(\BR^2)},\\
\|u\|_{C_x^{k+\al}(\BR^2)}
= &\;\|u\|_{C_x^{k,\al}(\BR^2)}
:= \sum_{j =0}^k\|\na^j u\|_{L_x^\infty(\BR^2)} + \|\na^k u\|_{\dot{C}_x^\al(\BR^2)}.
\end{align*}

\item
Similarly, for $X = X(s)$ defined on $\BT$, and $k\in \BN$, let
\[
\|X\|_{C_s^{k}(\BT)}
:= \sum_{j = 0}^k \|X^{(j)}\|_{L_s^\infty(\BT)},
\quad
\|X\|_{\dot{C}_s^{k}(\BT)}: = \|X^{(k)} \|_{L_s^\infty(\BT)}.
\]
Here $X^{(j)}$ denotes the $j$-th derivative of $X$ with respect to $s$.
We also denote
\[
\|X\|_{C_s(\BT)} := \|X\|_{C_s^0(\BT)} = \|X\|_{\dot{C}_s^0(\BT)}
= \|X\|_{L_s^\infty(\BT)}.
\]
With $\al\in (0,1)$, let
\[
\|X\|_{\dot{C}_s^\al(\BT)}: = \sup_{s_1,s_2\in \T\atop s_1\neq s_2}\f{|X(s_1)-X(s_2)|}{|s_1-s_2|^\al_\BT},
\]
and for $k\in \BN$, define
\begin{align*}
\|X\|_{\dot{C}_s^{k+\al}(\BT)}
= &\; \|X\|_{\dot{C}_s^{k,\al}(\BT)}: = \|X^{(k)} \|_{\dot{C}_s^{\al}(\BT)},\\
\|X\|_{C_s^{k+\al}(\BT)} = \|X\|_{C_s^{k,\al}(\BT)}
:= &\; \sum_{j = 0}^k \|X^{(j)}\|_{L_s^\infty(\BT)} + \|X\|_{\dot{C}_s^{k,\al}(\BT)}.
\end{align*}
Here $|s_1-s_2|_\BT$ denotes the distance between $s_1$ and $s_2$ on $\BT$, which ranges in $[0,\pi]$.
In what follows, we shall simply write it as $|\cdot|$ when it incurs no confusion.

\item
For $u = u(x,t)$ defined on $\BR^2\times [0,T]$ and $p\in [1,\infty]$, we introduce the space-time norms: for $t\in (0,T]$, let
\[
\|u\|_{L_t^\infty L_x^p(\BR^2)}: = \sup_{\eta\in [0,t]} \|u(\cdot,\eta)\|_{L_x^p(\BR^2)}.
\]
Note that we took the supremum in time instead of the essential supremum in the usual definition of the $L^\infty$-norm.
For $\al\in (0,1)$, $\|u\|_{L_t^\infty C_x^\al(\BR^2)}$, $\|u\|_{L_t^\infty \dot{C}_x^\al(\BR^2)}$, $\|u\|_{L_t^\infty C_x^{k,\al}(\BR^2)}$, and $\|u\|_{L_t^\infty \dot{C}_x^{k,\al}(\BR^2)}$, etc., can be defined analogously.

\item
For $X = X(s,t)$ being a function on $\T\times [0,T]$, we can define $\|X\|_{L_t^\infty L^p_s(\T)}$ and $\|X\|_{L_t^\infty C^\al_s(\T)}$, etc., in a similar way for $t\in (0,T]$.
We omit the details.

\item For $X = X(s)$ and $Y = Y(s)$ defined on $\BT$, we denote
\[
\|(X',Y')\|_{L^\infty_s(\BT)}: = \|X'\|_{L^\infty_s(\BT)} + \|Y'\|_{L^\infty_s(\BT)}.
\]
Similar quantities, such as $\|(X',Y')\|_{C^\al_s(\BT)}$ and $\|(u,v)\|_{L^\infty_x(\BR^2)}$, etc., can be defined analogously.

\item We introduce the Morrey norm on $\BR^2$.
With $\mu\in (0,1)$, let
\beq
\|f\|_{M^{1,\mu}(\R^2)}:=\sup_{r>0,\, x\in \BR^2} r^{\mu-2}\|f\|_{L^1(B(x,r))},
\label{eqn: def of Morrey norm}
\eeq
where $B(x,r)$ denotes a disk of radius $r$ centered at $x\in\BR^2$.
By the H\"older's inequality, we have $\|f\|_{M^{1,\mu}(\R^2)}\leq C\|f\|_{L^{2/\mu}(\R^2)} $.

\item
When the context is clear, we may omit $s$ and/or $x$ in the arguments and in the subscripts of the function spaces.
For instance, $\|X(\eta)\|_{\dot{C}^\al}$ denotes $\|X(\cdot,\eta)\|_{C^\al_s(\BT)}$, while $\|u(\tau)\|_{L^p}$ is short for $\|u(\cdot,\tau)\|_{L^p_x(\BR^2)}$, etc.
\end{itemize}

Next, we introduce notations that are specific to the analysis of the problem.
\begin{itemize}
\item Throughout the paper, we denote $\nu: = Re^{-1}$.

\item
Given $X = X(s,t)$ and $x\in \BR^2$, denote
\beq
d(x,t):=\inf_{s\in \BT}{|x-X(s,t)|}.
\label{eqn: def of d(x,t)}
\eeq
\item
Given $X = X(s)$, we follow \cite{mori2019well} to denote
\[
|X|_* : = \sup\{\lam\geq 0:\, |X(s_1)-X(s_2)|\geq \lam |s_1-s_2|_\BT \mbox{ for all }s_1,s_2\in \BT\}.
\]
i.e., it is the largest constant $\lam$ such that
the well-stretched property \eqref{eqn: well-stretched condition} holds for $X$.
It is clear that, if $X\in C^1(\BT)$, $|X|_* \leq \inf_{s\in \BT}|X'(s)|\leq \|X'\|_{L^\infty}$.

\item
If $X=X(s)$ gives a $C^1$-Jordan curve in $\BR^2$ parameterized in the counterclockwise direction, we define its \emph{effective radius} $R>0$ by
\beq
\pi R^2 = \f12\int_{\BT} X(s)\times X'(s)\, ds.
\label{eqn: def of effective radius}
\eeq
Indeed, the right-hand side represents the area of the domain enclosed by the curve $X(\BT)$, so it is positive and thus $R$ is well-defined.

\item For $X= X(s)\in L^2(\BT)$, we define the Fourier expansion of $X(s)$ by
\beq
X(s)=\sum_{n\in\BZ}
\begin{pmatrix}
\cos(ns) & -\sin(ns) \\
\sin(ns) & \cos(ns)
\end{pmatrix}
a_n,
\label{eqn: Fourier series}
\eeq
where for $n\in \BZ$, the Fourier coefficients $a_n\in \BR^2$ are defined by
\beq
a_n := \f{1}{2\pi}\int_\BT \begin{pmatrix}
\cos(ns) & \sin(ns) \\
-\sin(ns) & \cos(ns)
\end{pmatrix}X(s)\,ds =: \CF(X)_n.
\label{eqn: Fourier coefficients}
\eeq
We remark that, using the correspondence between complex numbers and vectors in $\BR^2$, we may rewrite \eqref{eqn: Fourier series} in the form of the more standard Fourier expansion of a complex-valued function:
$X(s) = \sum_{n\in \BZ} e^{ins} a_n$ --- in this formula, both $X(s)$ and $a_n$ should be interpreted as complex numbers.
However, to reduce confusion, we choose not to use this complex notation in this paper.

\item
Assuming the expansion \eqref{eqn: Fourier series}, we define a projection operator $\Pi$ as in \cite[Equation (1.23)]{mori2019well}:
\beq
X^*(s) := a_0 + \begin{pmatrix}
\cos s & -\sin s \\
\sin s & \cos s
\end{pmatrix}
a_1,\quad
\Pi X(s):=X(s)-X^*(s).
\label{eqn: def of projection operator Pi}
\eeq
It will be clear in Section \ref{sec: global well-posedness} that, if $X=X(s)$ represents a string configuration, $X^*(s)$ is the closest equilibrium configuration to $X(s)$ in the $L^2$-sense (not necessarily enclosing the same area), so the $\Pi X(s)$ can be viewed as the deviation of $X(s)$ from an equilibrium.
\end{itemize}

\subsection{Basic quantities and decomposition of the flow field}
\label{sec: def of basic quantities}
Let $G(x)$ be the 2-D Stokeslet defined in \eqref{eqn: Stokeslet in 2-D}.
Given $X = X(s,t)$, we define $u_{X}=u_{X}(x,t) $ to be the Stokes flow in $\BR^2$ induced by the forcing $f_X$ defined in \eqref{eqn: def of force}, i.e.,
\begin{align*}
&-\Delta u_X+\nabla p_X=f_X,\quad \di u_X=0,\\
&
|u_X|,|p_X|\to 0 \mbox{ as }|x|\to +\infty.
\end{align*}
In other words, $u_{X}=\BP(-\Delta)^{-1}f_{X}$, where $\BP$ denotes the Leray projection in $\BR^2$.
It has been shown that \cite[Equations (2.1), (2.4) and (2.6)]{lin2019solvability}
\beq
\begin{split}
u_{X}(x,t) = &\; \int_\BT G(x-X(s',t))X''(s',t)\,ds'\\
= &\; \int_\BT - \pa_{s'}\big[G(x-X(s',t))\big]\big(X'(s',t)-X'(s_{x,t},t)\big)\,ds'.
\end{split}
\label{eqn: expression of u_11}
\eeq
Here $s_{x,t}\in \BT$ is taken such that
\beq
|x- X(s_{x,t},t)| = \inf_{s\in \BT} | x-X(s,t)|.
\label{eqn: def of s_x}
\eeq
If such $s_{x,t}$ is not unique, take an arbitrary one; we shall omit the time variable and simply write $s_{x,t}$ as $s_x$ when it incurs no confusion.
It can also be justified that, for any given $c\in \BR^2$ \cite[Equation (2.2)]{lin2019solvability},
\beq
u_{X}(x,t) = \pv\int_\BT - \pa_{s'}\big[G(x-X(s',t))\big] \big(X'(s',t) -c\big)\,ds'.
\label{eqn: expression of u_11 alternative general c}
\eeq
Taking $c = 0$ yields
\cite[Equations (2.2) and (2.8)]{lin2019solvability}
\beqo
u_{X}(x,t) = \pv\int_\BT - \pa_{s'}\big[G(x-X(s',t))\big] X'(s',t) \,ds'.
\eeqo
Denote
\[
U_{X}(s,t): = u_{X}(X(s,t),t).
\]
In the study of the Stokes case (e.g.\;\cite{lin2019solvability,mori2019well}), it is split into two parts (also see \eqref{eqn: contour dynamic equation Stokes case})
\beq
U_{X}(s,t) = -\f14\Lam X(s,t)+g_{X}(s,t),
\label{eqn: Stokes velocity along the string}
\eeq
where $\Lam = (-\D)^{1/2}_\BT$ and
\beq
g_X(s,t) := \pv\int_\BT - \pa_{s'}\big[G(X(s,t)-X(s',t))\big] X'(s',t) \,ds'
+ \f14\Lam X(s,t).
\label{eqn: def of g_X}
\eeq
It is known that $g_X$ can be regarded as a nonlinear remainder term that is more regular than $-\f14\Lam X$
\cite{lin2019solvability,mori2019well}; also see its estimates in Lemma \ref{lem: improved estimates for g_X} and Lemma \ref{lem: improved estimates for g_X-g_Y} below.

Given $\nu = Re^{-1}>0$ and $X = X(s,t)$, we \emph{informally} let $u_X^{\nu} = u_X^{\nu}(x,t)$ solve the linearized Navier-Stokes equation in $\BR^2\times[0,T]$ with zero initial data and forcing $f_X$ defined in \eqref{eqn: def of force}, i.e.,
\beq
Re \partial_tu_X^{\nu}-\Delta u_X^{\nu}+\nabla p_X^{\nu}=f_X,\quad
\di u_X^{\nu}=0,\quad
u_X^{\nu}|_{t=0}=0.
\label{eqn: def of u_X^nu}
\eeq
Combining this with the definition of $u_X$, we find that
\[
\partial_tu_X^{\nu}-\nu \Delta u_X^{\nu}=-\nu\Delta u_X.
\]
Hence, we \emph{rigorously} define $u_X^\nu$ as follows:
\beq
\begin{split}
u_X^{\nu}(t) = &\; -\nu\int_0^t\Delta e^{\nu (t-\tau)\Delta}u_{X(\tau)}\, d\tau
\\
= &\; u_{X(t)}-e^{\nu t\Delta}u_{X(t)}
-\nu\int_0^t\Delta e^{\nu (t-\tau)\Delta}(u_{X(\tau)}-u_{X(t)})\, d\tau\\
=: &\; u_{X(t)}-e^{\nu t\Delta}u_{X(t)}
+h_X^\nu(t),
\end{split}
\label{eqn: u_X^nu in terms of u_X}
\eeq
where we used the identity $\nu\int_0^t\Delta e^{\nu (t-\tau)\Delta}w\, d\tau = e^{\nu t\Delta}w-w$ with $w=u_{X(t)}$, and where we denoted
\beq
h_X^{\nu}(t):=-\nu\int_0^t\Delta e^{\nu (t-\tau)\Delta}(u_{X(\tau)}-u_{X(t)})\, d\tau.
\label{eqn: def of h_X}
\eeq
For convenience, we also define
\[
H_X^\nu(t): = h_X^{\nu}(t)\circ X(t).
\]

Finally, given $u = u(x,t)$, we define
\beq
B_{\nu}[u](t):= -\int_0^te^{\nu (t-\tau)\Delta}\mathbb{P}(u\cdot\nabla u)(\tau)\, d\tau.
\label{eqn: u_2 tilde}
\eeq

Let us point out that $u_X$, $u_X^\nu$, $h_X^\nu$, and $B_{\nu}[u]$ defined above are natural quantities to study.
In fact, if $(u,X)$ is a sufficiently regular solution to \eqref{eqn: NS equation}-\eqref{eqn: initial data}, it is not difficult to verify the following important decomposition of $u$ (see \eqref{eqn: NS equation}, \eqref{eqn: def of u_X^nu}, \eqref{eqn: u_X^nu in terms of u_X}, and \eqref{eqn: u_2 tilde})
\beq
\begin{split}
u(t)
=&\; u_X^\nu(t) + e^{\nu t\Delta}u_0 +B_{\nu}[u](t)\\
=&\; u_{X(t)}+h_X^{\nu}(t)+e^{\nu t\Delta}(u_0-u_{X(t)}) +B_{\nu}[u](t).
\end{split}
\label{eqn: decomposition of u}
\eeq
This decomposition reveals the structure of the flow field $u$, which inspires the analysis in the rest part of the paper.
Indeed, the first line of \eqref{eqn: decomposition of u} shows that $u$ consists of three parts:
\begin{enumerate}[label = (\roman*)]
\item
$u_X^\nu(t)$ is fully determined by the dynamics of $X$ up to time $t$, so its estimates only depend on $X$ but not on $u$;
\item
$e^{\nu t\D}u_0$ is an explicit part coming from the initial flow field $u_0$, which is easy to control;
\item
and $B_\nu[u]$ arises from the nonlinear term $u\cdot \na u$ in the Navier-Stokes equation, which, due to the smoothing of the heat semigroup, should enjoy better regularity than $u$ itself in the subcritical regime.
\end{enumerate}
In the second line of \eqref{eqn: decomposition of u}, the flow field $u_X^\nu$ obtained from solving the linearized Navier-Stokes equation is further decomposed into a purely Stokesian flow field $u_X$ and the other two error terms (see \eqref{eqn: u_X^nu in terms of u_X}).
Combining \eqref{eqn: decomposition of u} with \eqref{eqn: kinematic equation} and \eqref{eqn: Stokes velocity along the string}, we find that $X$ solves
\begin{align*}
\pa_t X = &\; u\circ X\\
= &\; -\f14\Lam X + g_X
+ \big[h_X^{\nu} + e^{\nu t\Delta}(u_0-u_{X(t)})+B_{\nu}[u]\big]\circ X,
\end{align*}
with the initial condition $X(s,0)=X_0(s)$.
As in the Stokes case (see \eqref{eqn: contour dynamic equation Stokes case}), this can also be treated as a fractional heat equation, which admits smoothing mechanism as well.
By the Duhamel's formula,
\beq
\begin{split}
X(s,t)= &\; e^{-\f{t}4\Lam}X_0(s)\\
&\; +\int_0^t e^{-\f{t-\tau}{4}\Lam}\Big[ g_X(\tau)
+ \big[h_X^{\nu}(\tau) + e^{\nu \tau\Delta}(u_0-u_{X(\tau)})+B_{\nu}[u](\tau)\big]\circ X(\tau)\Big] \,d\tau.
\end{split}
\label{eqn: integral representation of X}
\eeq
In order to bound $X$, it suffices to show that $g_X$ and $[h_X^{\nu} + e^{\nu t\Delta}(u_0-u_{X(t)})+B_{\nu}[u]]\circ X$ have nice regularity estimates in terms of $u$ and $X$.
This will be achieved in Section \ref{sec: a priori estimates}.
In fact,
\begin{itemize}
\item $g_X$ has been well-studied in the Stokes case, and yet we shall provide improved estimates for it;
\item while $[h_X^{\nu} + e^{\nu t\Delta}(u_0-u_{X(t)})+B_{\nu}[u]]\circ X$ turns out to be more regular than $\Lam X$ and $g_X$ under suitable conditions.
\end{itemize}
This eventually enables us to show the existence, uniqueness, and other nice properties of the solutions.

\subsection{Statement of the main results}
\label{sec: main result}

Let us first introduce the notion of the mild solutions, which is directly motivated by the above representations \eqref{eqn: decomposition of u} and \eqref{eqn: integral representation of X} of $u$ and $X$.

\begin{definition}
\label{def: mild solution}
Assume that $u_0= u_0(x)\in L^p(\BR^2)$ for some $p\in (2,\infty)$ and $X_0 = X_0(s)\in C^1(\BT)$, which satisfy $\di u_0 = 0$ and $|X_0|_* >0$.
For some $T\in (0,\infty)$,
\[
(u,X) \in L^\infty([0,T];L^p(\BR^2))\times C([0,T];C^1(\BT))
\]
is called a \emph{(mild) solution} to \eqref{eqn: NS equation}-\eqref{eqn: initial data} on $[0,T]$, if the following holds:
\begin{enumerate}[label = (\roman*)]
\item $\di u(x,t) = 0$ for any $t\in [0,T]$ in the sense of distribution;

\item
\[
\inf_{\eta\in [0,T]}|X(\eta)|_* >0,
\quad \mbox{and}\quad
\sup_{\eta\in (0,T]}\eta^{\f1p} \|X'(\eta)\|_{\dot{C}^{1/p}_s} < +\infty,
\]

\item and $(u,X)$ satisfies that, for any $t\in [0,T]$,
\begin{align}
u(t) = &\; u_X^\nu(t) + e^{\nu t\Delta}u_0 +B_{\nu}[u](t),
\label{eqn: fixed pt equation for u}
\\
X(t) = &\; e^{-\f{t}4 \Lam}X_0+\CI \big[W[u,X; \nu,u_0]\big](t),
\label{eqn: fixed pt equation for X}
\end{align}
where
\beq
W[u,X; \nu,u_0](t): = g_X(t)
+ \big[h_X^{\nu}+e^{\nu t\D}(u_0-u_{X(t)})+B_{\nu}[u]\big] \circ X(t),
\label{eqn: def of w}
\eeq
and
\beq
\CI[W](t):=\int_0^te^{-\f14(t-\tau)\Lam}W(\tau)\,d\tau.
\label{eqn: def of the operator I}
\eeq
%
\end{enumerate}
\end{definition}

\begin{definition}
\label{def: maximal solution}
Let $u_0$ and $X_0$ be given as in Definition \ref{def: mild solution}.
For some $T_*\in (0,+\infty]$, $(u,X) \in L_{loc}^\infty([0,T_*);L^p(\BR^2))\times C_{loc}([0,T_*);C^1(\BT))$ is said to be a \emph{maximal (mild) solution} to \eqref{eqn: NS equation}-\eqref{eqn: initial data}, if the following holds:
\begin{enumerate}[label=(\roman*)]
\item for any $T<T_*$, $(u,X)$ restricted on $[0,T]$ gives a mild solution to \eqref{eqn: NS equation}-\eqref{eqn: initial data} on $[0,T]$ in the sense of Definition \ref{def: mild solution};
\item and $T_*$ is maximal, which means either $T_* = +\infty$, or if $T_*<+\infty$, for any $T\geq T_*$, there is no mild solution to \eqref{eqn: NS equation}-\eqref{eqn: initial data} on $[0,T]$ whose restriction on $[0,T_*)$ coincides with $(u,X)$.
\end{enumerate}
We call $T_*$ the \emph{maximal lifespan} of $(u,X)$.
If $T_* = +\infty$, we say that $(u,X)$ is a global (mild) solution.
If $T_*<+\infty$, we say a \emph{finite-time singularity} occurs, or the solution \emph{blows up} at the finite time $T_*$; in this case, $T_*$ is referred as the \emph{blow-up time} of $(u,X)$.
\end{definition}

Our first theorem states that, for any initial data $(u_0,X_0)$ satisfying the assumptions in Definition \ref{def: mild solution}, there exists a unique maximal solution to \eqref{eqn: NS equation}-\eqref{eqn: initial data}, and in short time, the solution can be quantitatively characterized in terms of the initial data.

\begin{thm}[Maximal mild solution and its short-time characterizations, Proposition \ref{prop: fixed-point solution}, Proposition \ref{prop: uniqueness of mild solution}, and Corollary \ref{cor: maximal solution}]
\label{thm: maximal solution}
Suppose that $u_0= u_0(x)\in L^p(\BR^2)$ for some $p\in (2,\infty)$, $X_0 = X_0(s)\in C^1(\BT)$, and they satisfy that $\di u_0 = 0$ and $|X_0|_* > 0$.
Denote
\beq
M_0:= \|X_0'\|_{L^\infty_s(\BT)},\quad \lam_0: = |X_0|_*,\quad
 Q_0 := \|u_0\|_{L_x^p(\BR^2)} + \lam_0^{-1+\f2p} M_0^2 .
\label{eqn: def of initial constants thm}
\eeq
Then there exists a unique $T_*\in (0,+\infty]$ and a unique maximal mild solution $(u,X)$ to \eqref{eqn: NS equation}-\eqref{eqn: initial data} with its maximal lifespan being $T_*$.
It holds that
\beq
T_* \geq C(p, \nu, \lam_0,  Q_0 , M_0, X_0),
\label{eqn: lower bound for lifespan}
\eeq
where $C(p,\nu,\lam_0, Q_0 ,M_0,X_0)$ denotes a positive constant that depends on $p$, $\nu$, $\lam_0$, $ Q_0 $, $M_0$, and $X_0$.
We will use similar notations in the sequel without further explanation.

In short time, $(u,X)$ enjoys quantitative characterizations in terms of the initial data.
More precisely, there exists some $T>0$ that depends on $p$, $\nu$, $\lam_0$, $ Q_0 $, $M_0$, and $X_0$, such that:
\begin{enumerate}
\item $u \in C_{loc}((0,T]\times \BR^2)$, and
\begin{align*}
&\; \|u\|_{L^\infty_T L^p_x(\BR^2)}
+ \sup_{\eta\in (0,T]} (\nu\eta)^{\f1p}\|u(\cdot,\eta)\|_{L^\infty_x(\BR^2)}
\leq  C\left(\lam_0^{-1+\f2p} M_0^2 +  Q_0 \right),
\end{align*}
where $C$ is a constant only depending on $p$.

\item
\[
\|X'\|_{L^\infty_T L^\infty_s(\BT)}\leq 2M_0,
\quad
\sup_{\eta\in (0,T]} \eta^{\f1p} \|X'(\cdot, \eta)\|_{\dot{C}_s^{1/p}(\BT)}
\leq 4M_0,
\]
and
\[
\inf_{t\in [0,T]}|X(t)|_*\geq \f{\lam_0}2.
\]

\item $X(s,t)\in C^1_{loc}((0,T)\times \BT)$, and it holds pointwise on $(0,T)\times \BT$ that
\beq
\pa_t X = -\f14\Lam X + g_X
+ \big[h_X^{\nu} + e^{\nu t\Delta}(u_0-u_{X(t)})+B_{\nu}[u]\big]\circ X = u\big( X(s,t),t\big).
\label{eqn: equation for X_t}
\eeq
Besides,
\[
\sup_{\eta\in (0,T)} \eta^{\f1p}\|\pa_t X(\cdot, \eta)\|_{L_s^\infty(\BT)}
\leq C(p,\nu,\lam_0, Q_0 ,M_0).
\]
\end{enumerate}

\begin{rmk}
It is well-known that, for the 2-D Navier-Stokes equation, $L^2(\BR^2)$ is the critical regularity for $u_0$, while in the Stokes case of the 2-D immersed boundary problem, the critical regularity for $X_0$ is $\dot{C}^1(\BT)$ and other scaling-equivalent regularity classes, such as $\dot{W}^{1,\infty}(\BT)$, $\dot{B}^1_{\infty,\infty}(\BT)$, and $\dot{H}^{3/2}(\BT)$, etc.\;\cite{chen2021peskin}.
Hence, our results in Theorem \ref{thm: maximal solution} cover critical initial data for $X$ and almost-critical initial data for $u$.

Since the $C^1_s(\BT)$-regularity is critical for $X$, it is conceivable that $T$, the length of the time interval on which we can quantitatively characterize the solution, depends on $X_0$ itself in an implicit manner, but not merely on $M_0$ and $\lam_0$.
See Remark \ref{rmk: choosing a specific gamma} and Remark \ref{rmk: characterization of lifespan T} for more details.
Note that $T$ here can serve as a lower bound for $T_*$, which implies \eqref{eqn: lower bound for lifespan}.
\end{rmk}
\end{thm}

The next theorem provides additional regularity estimates for the solution $(u,X)$ in terms of some basic bounds of it.
Thanks to the definition of the mild solutions, these estimates apply to any compact time interval within the maximal lifespan of $(u,X)$.
They imply that, as $t$ becomes positive, the solution will instantly gain higher regularity than what was stated in Theorem \ref{thm: maximal solution}.

\begin{thm}[Regularity of the solution, Lemma \ref{lem: L inf regularity of u mild solution}, Lemma \ref{lem: time continuity of u in L^p}, Lemma \ref{lem: time derivative and equation}, Lemma \ref{lem: time continuity of u in L inf}, Lemma \ref{lem: intermediate norms vanish as t goes to zero}, Proposition \ref{prop: higher regularity}, and Lemma \ref{lem: volume conservation}]
\label{thm: regularity}
Let $p$ and $(u_0,X_0)$ be given as in Theorem \ref{thm: maximal solution}.
Suppose for some $T\in (0,\infty)$, $(u,X)$ is a mild solution to \eqref{eqn: NS equation}-\eqref{eqn: initial data} on $[0,T]$ in the sense of Definition \ref{def: mild solution}.
Denote
\[
M:= \sup_{\eta\in [0,T]} \|X'(\eta)\|_{L^\infty_s}
+ \sup_{\eta\in (0,T]}\eta^{\f1p} \|X'(\eta)\|_{\dot{C}^{1/p}_s} < +\infty,
\]
\[
\lam:= \inf_{\eta\in [0,T]}|X(\eta)|_* >0,
\]
and
\[
Q:=\sup_{\eta\in [0,T]} \|u(\eta)\|_{L^p_x} < +\infty.
\]
In the above formulas, we omitted the $x$- and $s$-variables.
Then
\begin{enumerate}
\item $u \in C([0,T];L^p(\BR^2))$.
\item In $(0,T)\times \BT$, $\pa_t X(s,t)$ is well-defined pointwise, and \eqref{eqn: equation for X_t} holds pointwise.

\item
\[
\sup_{\eta\in (0,T]} \eta^{\f1p}\|u(\eta)\|_{L_x^\infty(\BR^2)}
+
\sup_{\eta\in (0,T)} \eta^{\f1{p}} \|\pa_t X(\eta)\|_{L_s^\infty(\BT)}
\leq C(T,p,\nu,\lam,M,Q).
\]

\item
For any $\b\in [0,1]$,
\[
\sup_{\eta\in (0,T]} \eta^\b \|X'(\eta)\|_{C^\b_s(\BT)}
\leq C(T, p,\nu,\lam,M,Q),
\]
and for any $\b\in [\f2p,1]$,
\beq
\sup_{\eta\in (0,T)} \eta^\b \|\pa_t X(\eta)\|_{C^\b_s(\BT)}
\leq C(T, p,\nu,\lam,M,Q).
\label{eqn: regularity estimate for X_t}
\eeq
Note that the constants $C$ do not depend on $\b$.

\item
For any $\b\in (0,1)$,
\beqo
\sup_{\eta\in (0,T]} \eta^{1+\b} \|X''(\eta)\|_{C^{\b}_s(\BT)}
+\sup_{\eta\in (0,T)} \eta^{1+\b} \|\pa_t X'(\eta)\|_{C^{\b}_s(\BT)}
\leq C(T,\b,p,\nu,\lam,M,Q).
\eeqo

\item For any $\b\in (0,2)$, $\lim_{t\to 0^+} t^{\b}\|X'(t)\|_{\dot{C}_s^\b(\BT)} = 0$.

\item
For any $\b\in [\f2p,1)$,
\beq
\sup_{\eta\in (0,T]}\eta^{\b}\|u(\eta)\|_{\dot{C}^\b_x(\BR^2)}
\leq C(T, \b, p,\nu,\lam,M,Q),
\label{eqn: regularity estimate for u}
\eeq
and in addition, for any $0 < t_1< t_2\leq T$,
\[
t_1^\b \|u(t_1)-u(t_2)\|_{L^\infty_x(\BR^2)}
\leq C(T, \b, p,\nu,\lam,M,Q) |t_1-t_2|^{\b/2}.
\]
\item For any $t\in (0,T]$,
\beqo
\|\na u(t)\|_{L^\infty_x(\BR^2)}
\leq C(T, p,\nu,\lam,M,Q)  t^{-1}. 
\eeqo

\item
For any $t\in [0,T]$,
\[
\f12\int_\BT X(s,t)\times X'(s,t)\,ds \equiv \f12\int_\BT X_0(s)\times X_0'(s)\,ds,
\]
i.e., the area of the region enclosed by the curve $X(\BT,t)$ is constant in time.
\end{enumerate}

\begin{rmk}
\label{rmk: optimality of higher-order estimates main thm}
These results are optimal in two ways.
\begin{itemize}
\item
The above quantitative bounds for $(u,X)$ 
have the optimal dependence on $t$.
In fact, the time-singularity of the claimed estimates for $X$ coincides with that of the estimates for the linear term $e^{-\f{t}{4}\Lam}X_0$, which in general cannot be improved in view of \eqref{eqn: fixed pt equation for X}.
That in turn leads to the claimed time-singularity of $u$ through \eqref{eqn: fixed pt equation for u}.

We also remark that the condition $\b\geq \f2p$ for \eqref{eqn: regularity estimate for X_t} to hold is optimal as well.
Indeed, for $\b\in (0,1]$,
\begin{align*}
\|e^{\nu t\D}u_0\|_{\dot{C}^\b_x(\BR^2)}\leq &\; C(\nu t)^{-\f1p-\f{\b}{2}}\|u_0\|_{L^p_x(\BR^2)},
\\
\|\pa_t e^{-\f{t}4\Lam}X_0\|_{\dot{C}^\b_s(\BT)}
= &\;\f14\|\Lam e^{-\f{t}4\Lam}X_0\|_{\dot{C}^\b_s(\BT)}
\leq C t^{-\b}\|X_0'\|_{L^\infty_s(\BT)}.
\end{align*}
In view of \eqref{eqn: fixed pt equation for u}, \eqref{eqn: fixed pt equation for X}, and \eqref{eqn: equation for X_t}, $\|\pa_t X(t)\|_{\dot{C}^\b_s(\BT)}\leq Ct^{-\b}$ holds only when $\b\geq \f1p+\f{\b}{2}$, i.e., $\b\geq \f{2}{p}$.
This also explains the range of $\b$ in \eqref{eqn: regularity estimate for u}.

\item

In general, 
the flow field $u$ can at most be Lipschitz in space, since the tangential force along the string gives rise to jump of $\na u$ across the curve $X(\BT)$ \cite{lai2001jumpcondition,mori2019well,Peskin1993improved_volume_conservation}.
This also affects the spatial regularity of $X$ through the term $B_\nu[u]$ in \eqref{eqn: fixed pt equation for X} and \eqref{eqn: def of w}.
As a result, $X$ being in $C^{2,1}_s(\BT)$ seems to be the highest possible regularity one can achieve within the current framework of analysis;
see more detailed explanation in Remark \ref{rmk: optimality of higher-order estimates}.
\end{itemize}
\end{rmk}
\end{thm}

The next theorem is a blow-up criterion.

\begin{thm}[Blow-up criterion, Corollary \ref{cor: maximal solution}]
\label{thm: blow-up criterion}
Let $p$, $(u_0,X_0)$, $(u,X)$, and $T_*$ be given as in Theorem \ref{thm: maximal solution}.
If $T_*<+\infty$, at least one of the following scenarios would occur:
\begin{enumerate}[label=(\alph*)]
\item $\limsup_{t\to T_*^-} \|u(t)\|_{L^p_x(\BR^2)} = +\infty$;

\item $\liminf_{t\to T_*^-} |X(t)|_* = 0$;

\item For any increasing sequence $t_k\to T_*^-$, $\{X'(t_k)\}_k$ is not convergent in $C(\BT)$.
As a result, $X'([0,T_*))$, which is the image of $[0,T_*)$ under the mapping $t\mapsto X'(\cdot,t)$, is not pre-compact in $C(\BT)$.
\end{enumerate}
\end{thm}

As is mentioned in Section \ref{sec: Stokes case revisited}, the Stokes case of the 2-D immersed boundary problem \eqref{eqn: Stokes equation}-\eqref{eqn: initial data Stokes case} has been studied extensively in the literature, which corresponds to the case of zero Reynolds number in \eqref{eqn: NS equation}-\eqref{eqn: initial data}.
It is natural to ask, as the Reynolds number $Re\to 0^+$, whether the solution to \eqref{eqn: NS equation}-\eqref{eqn: initial data} would converge in a suitable sense to the one to the Stokes case.
In the next theorem, we provide an affirmative answer to this question over a short time interval, and also bound the difference between the solutions in these two cases when $Re\ll 1$.
The theorem is stated in terms of the parameter $\nu= Re^{-1}$.

\begin{thm}[The zero-Reynolds-number limit, Proposition \ref{prop: well-posedness Stokes case}, Proposition \ref{prop: zero Reynolds number limit} and Remark \ref{rmk: extend the range of gamma in the zero Reynolds number limit}]
\label{thm: zero Reynolds number limit thm}

Let $p$ and $(u_0,X_0)$ be given as in Theorem \ref{thm: maximal solution}.
Define $M_0$, $\lam_0$, and $ Q_0 $ as in \eqref{eqn: def of initial constants thm}.

Take an arbitrary $\g\in (0,1)$.
There exists
\[
T_\dag = T_\dag(\g,\lam_0, M_0,X_0)\in (0,1]\quad \mbox{and}\quad
\nu_* = \nu_* (\g,p,\lam_0,M_0, Q_0 )\geq 1,
\]
such that the following holds for any $\nu \geq \nu_*$:
\begin{enumerate}
\item There exists a unique mild solution $X_\dag = X_\dag(s,t)$ on $[0,T_\dag]$ to the Stokes case of the 2-D immersed boundary problem \eqref{eqn: Stokes equation}-\eqref{eqn: initial data Stokes case} (see the definition in Remark \ref{rmk: def of mild solution Stokes case}), whose properties are established in Proposition \ref{prop: well-posedness Stokes case}, and there exists a unique mild solution $(u_\nu,X_\nu)$ on $[0,T_\dag]$ to \eqref{eqn: NS equation}-\eqref{eqn: initial data} with $Re = \nu^{-1}$, whose properties have been provided in Theorem \ref{thm: maximal solution} and Theorem \ref{thm: regularity}.

\item $\{(u_\nu,X_\nu)\}_{\nu \geq \nu_*}$ satisfies uniform-in-$\nu$ bounds on $[0,T_\dag]$:
\[
\sup_{t\in [0,T_\dag]}\|X_\nu'(t)\|_{L^\infty_s(\BT)}\leq 3M_0,\quad
\sup_{t\in (0,T_\dag]} t^{\g}\|X_\nu'(t)\|_{\dot{C}^{\g}_s(\BT)} \leq C(\g,\lam_0,M_0),
\]
\[
\inf_{t\in [0,T_\dag]} |X_\nu(t)|_* \geq \f{\lam_0}{2},
\]
and
\[
\|u_\nu\|_{L^\infty_{T_\dag} L^p_x(\BR^2)}\leq C(p,\lam_0,M_0, Q_0 ).
\]

\item
Moreover, for any $t\in (0,T_\dag]$,
\[
\|X_\nu(t)-X_\dag(t)\|_{L^\infty_s(\BT)}
\leq C t\cdot (\nu t)^{-\f1p},
\]
\[
\|(X_\nu-X_\dag)(t)\|_{\dot{C}_s^1(\BT)}
+ t^\g \|(X_\nu-X_\dag)'(t)\|_{\dot{C}^\g_s(\BT)}
\leq Ct \cdot (\nu t)^{-\f{1}{2}-\f1p},
\]
and
\[
\|u_\nu(t)-u_{X_\dag(t)}\|_{L^\infty_x(\BR^2)}
\leq C  \ln \nu \cdot t
\|(X_\nu'(t),X_\dag'(t))\|_{\dot{C}^1_s(\BT)}\cdot (\nu t)^{-\f1p} + C(\nu t)^{-\f1p},
\]
where the constants $C$ depend on $\g$, $p$, $\lam_0$, $M_0$, and $ Q_0 $.
\end{enumerate}
\begin{rmk}
This implies that, given $(u_0,X_0)$, for all sufficiently large $\nu$, the maximal lifespan of $(u_\nu,X_\nu)$ has a uniform-in-$\nu$ positive lower bound.
\end{rmk}

\begin{rmk}
\label{rmk: optimality of the convergence}
Let us highlight that the first two bounds for $X_\nu-X_\dag$ are optimal in terms of its $\nu$- and $t$-dependence.
Indeed, when compared with the dynamics of $X_\dag$, the evolution of $X_\nu$ is also affected by the initial flow field through the term $e^{\nu t\D}u_0$ (see \eqref{eqn: fixed pt equation for X} and \eqref{eqn: def of w}).
It satisfies that
\[
\|e^{\nu t \D}u_0\|_{L^\infty_x(\BR^2)} \leq C (\nu t)^{-\f1p}\|u_0\|_{L^p_x(\BR^2)},
\quad
\|\na e^{\nu t \D}u_0\|_{L^\infty_x(\BR^2)} \leq C (\nu t)^{-\f12-\f1p}\|u_0\|_{L^p_x(\BR^2)}.
\]
Integrating this effect in time, we find that it leads to an error between $X_\nu$ and $X_\dag$ in the exact form of the above bounds.
In a similar spirit, in the last bound for $u_\nu-u_{X_\dag(t)}$, the $t$-dependence is also optimal, while the $\nu$-dependence is optimal up to a logarithmic factor.
\end{rmk}
\end{thm}

If we additionally assume that $u_0\in L^2(\BR^2)$, we can prove the energy law for the system \eqref{eqn: NS equation}-\eqref{eqn: initial data}.
\begin{thm}[Energy law, Proposition \ref{prop: energy law}]
\label{thm: energy law thm}

Let $p$, $(u_0,X_0)$, $(u,X)$, and $T_*$ be given as in Theorem \ref{thm: maximal solution}.
We additionally assume that $u_0\in L^2(\BR^2)$.
Then for any $t\in [0,T_*)$,
\[
\|X'(t)\|_{L_s^2(\BT)}^2 + \nu^{-1}\|u(t)\|_{L_x^2(\BR^2)}^2
+ 2\int_0^t\|\na u(\tau)\|_{L^2_x(\BR^2)}^2\,d\tau
= \|X'_0\|_{L_s^2(\BT)}^2 + \nu^{-1}\|u_0\|_{L_x^2(\BR^2)}^2.
\]
\begin{rmk}
In view of this, we can define the total energy of the system as
\beqo
E(t) := \f12\|X'(\cdot,t)\|_{L_s^2(\BT)}^2 + \f{1}{2\nu}\|u(\cdot,t)\|_{L_x^2(\BR^2)}^2.
\eeqo
The first term is the Hookean elastic energy of the string, while the second term is the kinetic energy in the flow.
Theorem \ref{thm: energy law thm} implies that $E(t)$ is non-increasing in $t$, and the energy loss is purely due to the viscous dissipation in the flow.
\end{rmk}
\end{thm}

Next we prove that, when the initial data is small in a suitable sense, the maximal solution $(u,X)$ to \eqref{eqn: NS equation}-\eqref{eqn: initial data} is global.
The smallness of the initial data is quantified using its deviation from an equilibrium state.
We will see in Section \ref{sec: global well-posedness} that the only equilibria of the system are in the form
\[
u(x)\equiv 0,\quad
X(s) = b_0 + \begin{pmatrix}
\cos s & -\sin s \\
\sin s & \cos s
\end{pmatrix} b_1
\]
for some $b_0,b_1\in \BR^2$ with $|b_1|>0$.
Note that such $X(s)$ represents a circular and evenly-stretched string.

\begin{thm}[Global solution for small initial data, Proposition \ref{prop: global solution}]
\label{thm: global solution for small data general case}
Let $p$, $(u_0,X_0)$, and $(u,X)$ be given as in Theorem \ref{thm: maximal solution}.
We additionally assume that $u_0\in L^2(\BR^2)$, and that $X_0 = X_0(s)$ is parameterized in the counterclockwise direction.
Let $R>0$ be defined by
\beqo
\pi R^2 = \f12\int_{\BT} X_0(s)\times X_0'(s)\, ds.
\eeqo
Then there exists $\e_*>0$, which depends on $p$, $\nu$, and $R$, such that as long as
\[
e_0:=
\|\Pi X_0'\|_{L^\infty_s(\BT)}+ \|u_0\|_{L^2_x(\BR^2)} + \|u_0\|_{L^q_x(\BR^2)} \leq \e_*,
\]
$(u,X)$ is a global solution.
Here $\Pi X_0$ was defined as in \eqref{eqn: def of projection operator Pi}.

Moreover, it holds for all $t>0$ that
\[
\|u(t)\|_{L^p_x(\BR^2)} \leq CR^{1+\f2p},
\]
where $C$ depends on $p$,
\[
\|X'(t)\|_{L^\infty_s(\BT)}\leq CR,\quad
|X(t)|_* \geq cR,
\]
where $C>c>0$ are universal constants, and
\[
\|X'(t)\|_{\dot{C}^{1/p}_s(\BT)}
\leq C(p,\nu,R)\big(1+t^{-1/p}\big).
\]
Lastly,
\[
\|\Pi X'(t)\|_{L^2_s(\BT)} + \nu^{-1/2}\|u(t)\|_{L^2_x(\BR^2)}
\leq C\big(1+\nu^{-1/2}\big)e_0,
\]
where $C$ is universal.

\begin{rmk}
Based on the above bounds, one can derive higher-order estimates for $(u,X)$ by Theorem \ref{thm: regularity}.
However, we should point out that these bounds are not optimal as they fail to fully characterize the property that $(u,X)$ stays $O(e_0)$-close to equilibrium states, which we claim to be true.
To justify such refined estimates for $(u,X)$, we need more sophisticated analysis and lengthy arguments.
We choose to present that in a forthcoming work.
\end{rmk}
\end{thm}


\section{Estimates for the Basic Quantities}
\label{sec: a priori estimates}

This section is devoted to bounding the basic quantities in the decomposition of $u$ which were introduced in Section \ref{sec: def of basic quantities}.
For convenience of the future use, we will summarize all the estimates in Section \ref{sec: summary of estimates}, and then prove them in the subsequent subsections.

\subsection{Summary of the estimates}
\label{sec: summary of estimates}
Let us start by introducing some notations.
Throughout this section, we shall always assume $\g$ to be a real number in $(0,1)$.
Given $\lam,T>0$, denote
\beqo
O^\lam :=
\Big\{X = X(s)\in C^{1,\g}_s(\BT):\, |X|_*\geq \lam \Big\},
\eeqo
and
\beq
\begin{split}
O_T^{\lam}:=&\; \Big\{X = X(s,t)\in C([0,T];C^1(\BT)):\, |X(\cdot,t)|_*\geq \lam
\mbox{ for all }t\in[0,T],\\
&\;\qquad \qquad \quad
\sup_{\eta\in (0,T]}\eta^\g \|X'(\cdot, \eta)\|_{\dot{C}^{\g}_s(\BT)}<+\infty \Big\}.
\end{split}
\label{eqn: def of O_T^M lam}
\eeq
We note that the weight $\eta^\g$ in time is motivated by the smoothing estimates of the fractional heat semi-group in \eqref{eqn: fixed pt equation for X} (also see Lemma \ref{lem: parabolic estimates for fractional Laplace}).

The first three lemmas aim at bounding $u_X$ and $u_X^\nu$; see their definitions in Section \ref{sec: def of basic quantities}.

\begin{lem}
\label{lem: estimate for u_X}
Suppose $X\in O^\lam$.
\begin{enumerate}
\item
For any $t>0$ and $k\in \BN$, it holds that
\begin{align*}
\|\na^k\D e^{t\D}u_X\|_{L^{\infty}_x(\R^2)}
\leq C\lam^{-1-\g}t^{-(2+k-\g)/2}
\|X'\|_{\dot{C}_s^\g} \|X'\|_{L_s^{\infty}},
\end{align*}
\[
\|\na^k \Delta e^{t\Delta}u_{X}\|_{L^1_x(\R^2)}
\leq C\lam^{-\g} t^{-(1+k-\g)/2}\|X'\|_{\dot{C}_s^\g} \|X'\|_{L_s^\infty},
\]
where the constant $C$ depends on $k$ and $\g$, and
\[
\|\na^k\D e^{t\D}u_X\|_{L^{\infty}_x(\R^2)}
\leq C\min\Big(
\lam^{-1}t^{-(2+k)/2}\|X'\|_{L_s^\infty}^2,\; t^{-(3+k)/2}\|X'\|_{L_s^2}^2\Big),
\]
\[
\|\na^k \Delta e^{t\Delta}u_{X}\|_{L^1_x(\R^2)}
\leq C t^{-(1+k)/2} \|X'\|_{L_s^2}^2,
\]
where $C$ depends only on $k$.

\item For any $p\in (2,\infty)$,
\[
\|u_{X}\|_{L_x^p(\R^2)} \leq C\lam^{-1+\f{2}{p}}\|X'\|_{L_s^2}^{\f{4}{p}}
\|X'\|_{L_s^{\infty}}^{2-\f{4}{p}},
\]
where $C$ depends only on $p$.

\item
Moreover, 
\[
\lam\|u_X\|_{L_x^{\infty}(\R^2)}
+\lam^{1+\g}\|u_X\|_{\dot{C}_x^{\g}(\R^2)}
\leq C\|X'\|_{\dot{C}_s^{\gamma}}  \|X'\|_{L_s^{\infty}},
\]
where $C$ depends only on $\g$.
\end{enumerate}
\end{lem}

\begin{lem}
\label{lem: estimate for u_X^nu}
Suppose $X\in O_T^{\lam}$.
Then for any $t\in (0,T]$,
\begin{align*}
&\;
\lam \|u_X^{\nu}(t)\|_{L^{\infty}_x(\R^2)}
+\lam^{1+\g}\|u_X^{\nu}(t)\|_{\dot{C}^{\g}_x(\R^2)}
+\lam \|h_X^{\nu}(t)\|_{L^{\infty}_x(\R^2)}
+\lam^{1+\g} \|h_X^{\nu}(t)\|_{\dot{C}_x^{\g}(\R^2)}\\
\leq &\;
C t^{-\g} \|X'\|_{L^\infty_t L_s^{\infty}}
\sup_{\eta\in (0,t]} \eta^\g \|X'(\eta)\|_{\dot{C}_s^{\g}},
\end{align*}
where the constant $C$ depends only on $\g$.

Moreover, for any $p\in [1,\infty)$, 
\beqo
\|u_X^{\nu}(t)\|_{L^p_x(\R^2)}
\leq C(\nu t)^{\f1{2p}} \lam^{-1+\f1p}
\|X'\|_{L^\infty_t L_s^2}^{\f2p}
\|X'\|_{L^\infty_t L_s^\infty}^{2-\f2p}.
\eeqo
If, additionally, $p\in (2,\infty)$,
\[
\|u_X^{\nu}(t)\|_{L^p_x(\R^2)} + \|h_X^{\nu}(t)\|_{L^p_x(\R^2)}
\leq C \lam^{-1+\f2p}
\|X'\|_{L^\infty_t L_s^2}^{\f4p}
\|X'\|_{L^\infty_t L_s^\infty}^{2-\f4p}.
\]
Here the constants $C$'s depend only on $p$.
\end{lem}

\begin{lem}
\label{lem: time Holder continuity of u_X^nu}
Suppose $X\in O^\lam_T$.
For any $0< t_1< t_2\leq T$,
\[
\|u_X^\nu(t_1)-u_X^\nu(t_2)\|_{L^\infty_x(\BR^2)}
\leq C \big(\nu |t_1-t_2|\big)^{\g/2} t_1^{-\g} \cdot \lam^{-1-\g} \|X'\|_{L^\infty_{t_2} L_s^{\infty}} \sup_{\eta\in (0,t_2]} \eta^\g \|X'(\eta)\|_{\dot{C}_s^{\g}},
\]
where the constant $C$ depends only on $\g$.
\end{lem}

The next two lemmas bound $u_X-u_Y$ and $u_X^\nu-u_Y^\nu$.

\begin{lem}
\label{lem: estimate for u_X-u_Y}
Suppose $X,Y\in O^\lam$, and they satisfy $\|X'-Y'\|_{L_s^\infty}\leq \lam/2$.
Then for any $t>0$, and any $\b\in (0,1]$,
\begin{align*}
&\;\big\|e^{t\D}(u_X-u_Y)\big\|_{L^\infty_x(\BR^2)}\\
\leq &\;C t^{-\b/2} \lam^{-(1-\b)}\|(X',Y')\|_{L_s^{\infty}}\|X'-Y'\|_{L_s^{\infty}} \\
&\; + Ct^{-(1-\g)/2} \lam^{-1-\g}
\big(\|(X',Y')\|_{L_s^\infty} \|(X',Y')\|_{\dot{C}_s^{\g}} \|X-Y\|_{L_s^\infty} + \|(X',Y')\|_{L_s^\infty}^2 \|X-Y\|_{\dot{C}_s^{\g}}\big),
\end{align*}
where the constant $C$ depends on $\g$ and $\b$.

Moreover, for any $t>0$,
\begin{align*}
&\;\big\|e^{t\D}(u_X-u_Y)\big\|_{\dot{C}^\g_x(\BR^2)}\\
\leq &\; C t^{-\g/2} \lam^{-1} \|(X',Y')\|_{L_s^{\infty}} \|X'-Y'\|_{L_s^{\infty}}\\
&\; + C t^{-1/2} \lam^{-1-\g}
\big(\|(X',Y')\|_{L_s^\infty} \|(X',Y')\|_{\dot{C}_s^{\g}} \|X-Y\|_{L_s^\infty} + \|(X',Y')\|_{L_s^\infty}^2 \|X-Y\|_{\dot{C}_s^{\g}}\big),
\end{align*}
where $C$ depends on only $\g$.
\end{lem}

\begin{lem}
\label{lem: estimate for u_X^nu-u_Y^nu in Morrey space}
Suppose $X,Y\in  O_T^{\lam}$ satisfy that $X(s,0) = Y(s,0)$ and $\|X'-Y'\|_{L^\infty_T L_s^\infty}\leq \lam/2$.
Then, with $\mu_0\in [0,1)$, for any $t\in (0,T]$,
\begin{align*}
&\;\|u_X^{\nu}(t)-u_Y^\nu(t)\|_{M^{1,1-\g}(\R^2)}\\
\leq
&\; C \lam^{-1-\g} \|(X',Y')\|_{L^\infty_t L_s^{\infty}} \\
&\; \cdot \Big[ t^{1-\g-\mu_0} \sup_{\eta\in(0,t]}\eta^\g
\|(X'(\eta),Y'(\eta))\|_{\dot{C}_s^\g} \sup_{\eta\in (0,t)} \eta^{\mu_0}\|\pa_t (X-Y)(\eta)\|_{L_s^\infty}\\
&\;\quad
+ \|(X',Y')\|_{L^\infty_t L_s^{\infty}} \|X'-Y'\|_{L^\infty_t L^\infty_s}\Big],
\end{align*}
where $C$ is a constant depending on $\g$ and $\mu_0$.
Here $\|\cdot\|_{M^{1,1-\g}(\BR^2)}$ denotes the Morrey norm; see the definition in Section \ref{sec: definitions and notations}.
\end{lem}

Although we have obtained some basic bounds for $h_X^\nu$ (see its definition in \eqref{eqn: def of h_X}) in Lemma \ref{lem: estimate for u_X^nu}, we will still need the following more refined estimate.

\begin{lem}
\label{lem: estimate for h_X^nu}
Suppose $X\in O_{T}^{\lam}$, and it satisfies $\|X'(t_1)-X'(t_2)\|_{L^\infty}\leq \lam/2$ for any $t_1,t_2\in [0,T]$.
\begin{enumerate}
\item
For $k = 0,1$, if $\g\in (\f{k}{2},1)$, $\g'\in (0,\g]$, and $\mu_0,\mu,\mu'\in [0,1)$, we have that
\beq
\begin{split}
&\; \|\na^k h_X^\nu(t)\|_{L_x^\infty(\BR^2)}
\\
\leq &\; C\left[ (\nu t)^{-\f{k}{2}} \min\left(\lam^{-1},\,(\nu t)^{-\f{1}{2}}\right)
+ \mathds{1}_{\{\nu t\geq 4\lam^2\}} (\nu t)^{-\g} \lam^{-1-k+2\g}  \right] t^{-\g(\mu-\g)}\\
&\;\cdot \|X'\|_{L^\infty_t L_s^\infty}  \left(\sup_{\eta\in (0,t]}\eta^\g \|X'(\eta)\|_{C_s^\g}\right)^{1-\g} \left(\sup_{\eta\in (0,t]}\eta^{\mu}\|\pa_t X(\eta)\|_{\dot{C}_s^{\g}}\right)^{\g} \\
&\; + C (\nu t)^{-\f{1+k-\g'}{2}} \min\left(\lam^{-\g'},\, (\nu t)^{-\f{\g'}{2}}\right)
\left[1+\mathds{1}_{\{\nu t\geq 4\lam^2\}}\int_{2\lam^2/(\nu t)}^1 \zeta^{-\f{1+k}{2}} \,d\zeta \right]\\
&\;\quad \cdot \lam^{-1} \left(t^{1-\g'-\mu_0} \|X'\|_{L^\infty_t L_s^\infty}
\sup_{\eta\in (0,t]}\eta^{\g'} \|X'(\eta)\|_{C_s^{\g'}} \sup_{\eta\in (0,t]}\eta^{\mu_0}\|\pa_t X(\eta)\|_{L_s^\infty} \right.\\
&\;\qquad \qquad \left.+ t^{1-\mu'}\|X'\|_{L^\infty_t L_s^\infty}^2 \sup_{\eta\in (0,t]}\eta^{\mu'}\|\pa_t X(\eta)\|_{\dot{C}_s^{\g'}} \right),
\end{split}
\label{eqn: L inf and W 1 inf estimate for h}
\eeq
where $C$ depends on $k$, $\g$, $\g'$, $\mu_0$, $\mu$, and $\mu'$.

\item Again assume $\g'\in (0,\g]$, and $\mu_0,\mu,\mu'\in [0,1)$.
For any $\al\in (0,1)\cup (1,1+\g']$ satisfying that $\al\leq 2\g$,
\beq
\begin{split}
\|h_X^{\nu}(t)\|_{\dot{C}_x^{\al}(\R^2)}
\leq &\; C \left[(\nu t)^{-\f{\al}{2}}\min\left(\lam^{-1},\,(\nu t)^{-\f{1}{2}}\right)
+ \mathds{1}_{\{\nu t\geq 4\lam^2\}} (\nu t)^{-\g}\lam^{-1-\al+2\g}\right]
t^{-\g(\mu-\g)}
\\
&\;\cdot \|X'\|_{L^\infty_t L_s^\infty} \left(\sup_{\eta\in (0,t]}\eta^\g \|X'(\eta)\|_{C_s^\g }\right)^{1-\g} \left(\sup_{\eta\in (0,t]}\eta^{\mu}\|\pa_t X(\eta)\|_{\dot{C}_s^{\g}}\right)^{\g} \\
&\; + C\left[(\nu t)^{-\f{1+\al-\g'}{2}} \min\left(\lam^{-\g'},\, (\nu t)^{-\f{\g'}{2}} \right)
\lam^{-1}
+ \mathds{1}_{\{\nu t\geq 4\lam^2\}} (\nu t)^{-1}  \lam^{-\al} \right]
\\
&\;\quad \cdot  \left(t^{1-\g'-\mu_0} \|X'\|_{L^\infty_t L_s^\infty}
\sup_{\eta\in (0,t]}\eta^{\g'} \|X'(\eta)\|_{C_s^{\g'}} \sup_{\eta\in (0,t]}\eta^{\mu_0}\|\pa_t X(\eta)\|_{L_s^\infty} \right.\\
&\;\qquad \left.+ t^{1-\mu'}\|X'\|_{L^\infty_t L_s^\infty}^2 \sup_{\eta\in (0,t]}\eta^{\mu'}\|\pa_t X(\eta)\|_{\dot{C}_s^{\g'}} \right),
\end{split}
\label{eqn: Holder estimate for h}
\eeq
where $C$ depends on $\al$, $\g$, $\g'$, $\mu_0$, $\mu$, and $\mu'$.
Here if $\al>1$, $\|h_X^{\nu}(t)\|_{\dot{C}_x^{\al}(\R^2)}$ is interpreted as $\|\na h_X^{\nu}(t)\|_{\dot{C}_x^{\al-1}(\R^2)}$.

\item Besides, for $H_X^\nu(t) = h_X^{\nu}(t)\circ X(t)$, if $\al\in (0,1]$,
\[
\|H_X^{\nu}(t)\|_{\dot{C}_s^{\al}(\T)}
\leq \|h_X^{\nu}(t)\|_{\dot{C}_x^{\al}(\R^2)}
\|X'(t)\|_{L_s^\infty(\BT)}^{\al},
\]
while if $\al \in (1,2)$,
\begin{align*}
\|H_X^{\nu}(t)\|_{\dot{C}^{1,\al-1}_s(\BT)}
\leq &\; \|\na h_X^{\nu}(t)\|_{\dot{C}^{\al-1}_x(\BR^2)} \|X'(t)\|_{L_s^\infty(\BT)}^{\al}
+ \|\na h_X^{\nu}(t)\|_{L^\infty_x(\BR^2)} \|X'(t)\|_{\dot{C}^{\al-1}_s(\BT)}.
\end{align*}
\end{enumerate}
\end{lem}

\begin{lem}
\label{lem: estimate for H_X^nu-H_Y^nu}
Assume $T\leq 1$.
Suppose $X,Y\in O^{\lam}_T$ satisfy that $X(s,0) = Y(s,0)$, and that for any $t_1,t_2\in [0,T]$,
\[
\|X'(t_1)-X'(t_2)\|_{L_s^\infty(\BT)}\leq \lam/2,\quad \|Y'(t_1)-Y'(t_2)\|_{L_s^\infty(\BT)}\leq \lam/2.
\]
Denote
\[
M: = \|X'\|_{L^\infty_T L_s^\infty} + \|Y'\|_{L^\infty_T L_s^\infty} +  \sup_{\eta\in (0,T]} \eta^\g \|X'(\eta)\|_{\dot{C}_s^{\g}(\BT)} + \sup_{\eta\in (0,T]} \eta^\g \|Y'(\eta)\|_{\dot{C}_s^{\g}(\BT)},
\]
and, with some $\mu_0\in[0, \f12]$ and some $\b\in [\g,1)$, 
\[
N_0 := \sup_{\eta\in (0,T)}\eta^{\mu_0}\|\pa_t X(\eta)\|_{L_s^\infty},\quad
N_\b := \sup_{\eta\in (0,T)} \eta^\b \|\pa_t X(\eta)\|_{C_s^\b}.
\]
Define
\beq
\begin{split}
D(t) = &\; D[\g,\b,X,Y](t)\\
:= &\; \|X'-Y'\|_{L^\infty_t L_s^{\infty}} + \sup_{\eta\in (0,t]} \eta^\g \|(X'-Y')(\eta)\|_{C_s^{\g}}
+ \sup_{\eta\in (0,t)} \eta^{\b} \|\pa_t(X-Y)(\eta)\|_{\dot{C}_s^\b},
\end{split}
\label{eqn: def of D distance bewteen X and Y}
\eeq
and
\beq
\begin{split}
A(t) = &\; A[\g, \b,\mu_0, \nu, \lam, X, Y](t)\\
:= &\; \nu^{-1} \left(D(t)+\sup_{\eta\in (0,t]}\eta^{\mu_0} \|\pa_t (X-Y)(\eta)\|_{L_s^\infty}\right)\\
&\;\cdot \Big[\lam^{-3}
M^{3-\g} N_\b^\g \big(1 + \nu^{-1} t^{1-2\mu_0} N_0^2\big) + \lam^{-1} M\\
&\;\quad  + t^{(1-\g)/2-\mu_0} \cdot
\lam^{-2-\g} \nu^{-(1-\g)/2} M^2 N_0
\big( 1 + \nu^{-1} t^{1-2\mu_0} N_0^2\big)^{\f{1+\g}{2}}
\\
&\; \quad + t^{(1-\g)/2-\mu_0} \cdot \lam^{-1-\g}\nu^{-(1-\g)/2} M^2 \Big].
\end{split}
\label{eqn: def of A_J}
\eeq
Then for any $t\in[0,T]$,
\[
\|H_X^\nu(t)-H_Y^\nu(t)\|_{L^\infty_s(\BT)}
\leq
C\nu  A(t),
\]
and, if $\g\in (0,\f12)$,
\[
\|H_X^{\nu}(t)-H_Y^\nu(t)\|_{\dot{C}_s^{2\g}(\T)}
\leq C (\nu t)^{-\g} \nu M^{2\g} A(t).
\]
Here the constants $C$'s depend on $\b$ and $\g$.

\end{lem}

The next two lemmas collect all the estimates we need for $B_\nu[u]$.
\begin{lem}
\label{lem: estimate for B_nu u}
Given $u$, let $B_\nu[u]$ be defined in \eqref{eqn: u_2 tilde}.
The following estimates hold for any $t\in (0,T]$.
\begin{enumerate}
\item
For $q\in (1,\infty]$ and $r\in [1,q]$ satisfying that $0< \f1r <\f12 + \f1q$,
\beq
\|B_{\nu}[u](t)\|_{L^q_x}
\leq C\nu^{-1} (\nu t)^{\f1q-\f12}
\sup_{\eta\in(0,t]}(\nu \eta)^{1-\f1r}\|(u\otimes u)(\eta)\|_{L_x^r}.
\label{eqn: L q estimate for B general}
\eeq
Here we interpret $\f{1}{\infty}$ as $0$, and the constant $C$ depends on $q$ and $r$.
In particular, for $p\in (2,\infty)$,
\begin{align}
\|B_{\nu}[u](t)\|_{L^2_x}
\leq &\; 
C\nu^{-1} \|u\|_{L^\infty_t L_x^2} \sup_{\eta\in(0,t]}(\nu \eta)^{\f12-\f1p}\|u(\eta)\|_{L_x^p},
\label{eqn: L 2 estimate for B}
\\
\|B_{\nu}[u](t)\|_{L^p_x}
\leq &\;
C\nu^{-1} (\nu t)^{\f1p-\f12}\left(\sup_{\eta\in(0,t]}(\nu \eta)^{\f12-\f1p}\|u(\eta)\|_{L_x^p}\right)^2,
\label{eqn: L p estimate for B}
\\
\|B_{\nu}[u](t)\|_{L^\infty_x}
\leq &\; C\nu^{-1} (\nu t)^{-\f12}\sup_{\eta\in(0,t]}(\nu \eta)^{\f12-\f1p}\|u(\eta)\|_{L_x^p}
\sup_{\eta\in(0,t]}(\nu \eta)^{\f12}\|u(\eta)\|_{L_x^\infty},
\label{eqn: L inf estimate for B}
\end{align}
where the constants $C$'s depend on $p$.

\item
For $\g\in (0,1)$,
\begin{align}
\|B_{\nu}[u](t)\|_{\dot{C}_x^{\g}(\R^2)}
\leq &\; C \nu^{-1} (\nu t)^{-(1+\g)/2} \sup_{\eta\in (0,t]} (\nu \eta)^{(1+\g)/2}\|(u\otimes u)(\eta)\|_{L_x^{2/(1-\g)}},
\label{eqn: Holder estimate for B u prelim}
\end{align}
where the constant $C$ depends on $\g$.
Moreover, for $p \in [2,\infty)$,
\beq
\begin{split}
&\;\|B_{\nu}[u](t)\|_{\dot{C}^{\g}_x(\R^2)}\\
\leq &\; C \nu^{-1}(\nu t)^{-(1+\g)/2}
\left(\sup_{\eta\in (0,t]} (\nu \eta)^{\f12-\f1p} \|u(\eta)\|_{L^p_x}+\sup_{\eta\in (0,t]}(\nu \eta)^{\f12}\|u(\eta)\|_{L^\infty_x}\right)^2,
\end{split}
\label{eqn: 6th estimate for B new form general gamma}
\eeq
where the constant $C$ depends on $p$ and $\g$.

\item For $\g\in (0,1)$ and $p\in [2,\infty)$,
\beq
\begin{split}
&\;(\nu t)^{(2+\g)/2} \|B_{\nu}[u](t)\|_{\dot{C}^{1,\g}_x(\BR^2)}
+ \nu t \|B_{\nu}[u](t)\|_{\dot{C}^1_x(\BR^2)}\\
\leq
&\; C \nu^{-1}
\sup_{\eta\in (0,t]} (\nu\eta)^{1/2} \|u(\eta)\|_{L^\infty_x}
\left[\sup_{\eta\in(0,t]}(\nu \eta)^{\f12-\f1p}\|u(\eta)\|_{L_x^p}
+ \sup_{\eta\in (0,t]} (\nu\eta)^{(1+\g)/2} \|u(\eta)\|_{\dot{C}^\g_x}\right],
\end{split}
\label{eqn: 2nd estimate for B}
\eeq
where the constant $C$ depends on $p$ and $\g$.

\item For $\g\in (0,1)$,
\beq
\begin{split}
&\; (\nu t)^{\g/2} \|(B_{\nu}[u]-B_{\nu}[v])(t)\|_{M^{1,1-\g}(\BR^2)}
+ (\nu t)^{1/2} \|(B_{\nu}[u] - B_{\nu}[v])(t)\|_{L^\infty_x(\BR^2)}\\
\leq &\; C\nu^{-1}
\sup_{\eta\in (0,t]}(\nu \eta)^{\g/2}\|(u-v)(\eta)\|_{M^{1,1-\g}}
\sup_{\eta\in (0,t]} (\nu\eta)^{1/2} \|(u(\eta),v(\eta))\|_{L_x^\infty},
\end{split}
\label{eqn: 4th estimate for B L inf}
\eeq
and
\beq
\begin{split}
&\; (\nu t)^{(1+\g)/2} \|(B_{\nu}[u] - B_{\nu}[v])(t)\|_{\dot{C}^\g_x(\BR^2)}\\
\leq &\; C\nu^{-1}
\sup_{\eta\in (0,t]}(\nu \eta)^{\g/2}\|(u-v)(\eta)\|_{M^{1,1-\g}}
\sup_{\eta\in (0,t]} (\nu\eta)^{1/2} \|(u(\eta),v(\eta))\|_{L_x^\infty}.
\end{split}
\label{eqn: 4th estimate for B gamma}
\eeq
Here the constants $C$'s depend only on $\g$.

\end{enumerate}
\end{lem}

\begin{lem}
\label{lem: time Holder continuity of B_u}
Let $\g\in (0,1)$ and $p\in [2,\infty)$.
Then for any $0<t_1<t_2\leq T$,
\begin{align*}
&\; \|B_{\nu}[u](t_2)-B_{\nu}[u](t_1)\|_{L^\infty_x(\BR^2)}\\
\leq &\;
C \nu^{-1} \big(\nu(t_2-t_1)\big)^{\f\g{2}}(\nu t_1)^{-\f{1+\g}{2}}
\left(\sup_{\eta\in (0,t_1]} (\nu \eta)^{\f12-\f1p} \|u(\eta)\|_{L^p_x}+\sup_{\eta\in (0,t_1]}(\nu \eta)^{\f12}\|u(\eta)\|_{L^\infty_x}\right)^2\\
&\;
+ C \nu^{-1} \big(\nu(t_2-t_1)\big)^{\f{\g}2}(\nu t_2)^{-\f{1+\g}{2}}
\left(\sup_{\eta\in (0,t_2]} (\nu \eta)^{\f12-\f1p} \|u(\eta)\|_{L^p_x}+\sup_{\eta\in (0,t_2]}(\nu \eta)^{\f12}\|u(\eta)\|_{L^\infty_x}\right)^2,
\end{align*}
where the constant $C$ depends on $p$ and $\g$.
\end{lem}

Proofs of these lemmas will be presented in the rest of this section.

Finally, we recall the nonlinear term $g_X$ defined in \eqref{eqn: def of g_X}.
Its H\"{o}lder estimates have been studied in several existing works; see e.g.\;\cite[Propositions 3.8, 3.9, 3.12]{mori2019well} and \cite[Lemma 7.15]{Chen2024wellposedness_quasilinear}.
Nevertheless, in order to derive optimal regularity estimates for the solutions, we will need the following improved estimates, which are of independent interest.

\begin{lem}
\label{lem: improved estimates for g_X}
Suppose $X\in O^\lam$.
Let $g_X$ be defined in \eqref{eqn: def of g_X}.
Then
\[
\|g_X\|_{L_s^\infty(\BT)}
\leq C\lam^{-2}\|X'\|_{L_s^\infty}^2 \|X'\|_{\dot{C}_s^\g}.
\]
If $\g\in (0,\f12)$,
\[
\|g_X\|_{\dot{C}_s^{2\g}(\BT)} \leq C \lam^{-2} \|X'\|_{L_s^\infty}\|X'\|_{\dot{C}_s^\g}^2,
\]
and if $\g\in (\f12,1)$,
\[
\|g_X\|_{\dot{C}_s^{1,2\g-1}(\BT)}\leq C\lam^{-3} \|X'\|_{L_s^\infty}^2 \|X'\|_{\dot{C}_s^\g}^2.
\]
Here $C$'s are constants that only depend on $\g$.
\end{lem}
We highlight that the H\"{o}lder estimates above have exactly quadratic dependence on $\|X'\|_{\dot{C}_s^\g}$, which is crucial for deriving the regularity estimates for $(u,X)$ in Section \ref{sec: properties of mild solutions} with the optimal time-singularity at $t=0$.

In a similar spirit, we also have that
\begin{lem}
\label{lem: improved estimates for g_X-g_Y}
Suppose $X,Y\in O^\lam$.
Let $g_X$ and $g_Y$ be defined in \eqref{eqn: def of g_X} in terms of $X$ and $Y$, respectively.
Then
\begin{align*}
\|g_X-g_Y\|_{L_s^\infty(\BT)}
\leq &\; C\lam^{-2} \big(\|X'\|_{L_s^\infty}+\|Y'\|_{L_s^\infty}\big)^2\\
&\;\cdot \Big[\|X'-Y'\|_{\dot{C}_s^\g}
+ \lam^{-1} \big(\|X'\|_{\dot{C}_s^\g}+\|Y'\|_{\dot{C}_s^\g}\big) \|X'-Y'\|_{L_s^\infty} \Big] .
\end{align*}
If $\g\in (0,\f12)$,
\begin{align*}
\|g_X-g_Y\|_{\dot{C}_s^{2\g}(\BT)}
\leq &\; C \lam^{-2}
\big(\|X'\|_{L_s^\infty}+\|Y'\|_{L_s^\infty}\big) \big(\|X'\|_{\dot{C}_s^\g}+\|Y'\|_{\dot{C}_s^\g}\big)\\
&\;\cdot \Big[\|X'-Y'\|_{\dot{C}_s^\g} + \lam^{-1} \big(\|X'\|_{\dot{C}_s^\g}+\|Y'\|_{\dot{C}_s^\g}\big) \|X'-Y'\|_{L_s^\infty}\Big].
\end{align*}
If $\g\in (\f12,1)$,
\begin{align*}
\|(g_X-g_Y)'\|_{\dot{C}_s^{2\g-1}(\BT)}
\leq &\; C\lam^{-3} \big(\|X'\|_{L_s^\infty}+\|Y'\|_{L_s^\infty}\big)^2
\big(\|X'\|_{\dot{C}_s^\g}+\|Y'\|_{\dot{C}_s^\g}\big)\\
&\;\cdot \Big[\|X'-Y'\|_{\dot{C}_s^\g}
+ \lam^{-1} \big(\|X'\| _{\dot{C}_s^\g}+\|Y'\|_{\dot{C}_s^\g}\big) \|X'-Y'\|_{L_s^\infty} \Big].
\end{align*}
Here $C$'s are constants that only depend on $\g$.
\end{lem}
The above two lemmas can be justified as the existing estimates in \cite{Chen2024wellposedness_quasilinear,mori2019well}.
One only has to carefully track the dependence on various norms of $X$, $Y$, and $X-Y$.
Nevertheless, for completeness, we still provide self-consistent new proofs in Appendix \ref{sec: improved estimates for g_X}, which are of independent interest.

\subsection{Estimates for $u_X$ and $u_X^\nu$}
By the definition of $u_X$ (see \eqref{eqn: expression of u_11 alternative general c} and also \eqref{eqn: u_X^nu in terms of u_X}),
\beq
-\Delta e^{t\Delta}u_{X}(x) = \int_{\T}K(x-X(s),t)X''(s)\,ds
= -\int_{\T}\pa_{s}[K(x-X(s))](X'(s)-c)\,ds,
\label{eqn: Laplace of heat kernel applied to u_X}
\eeq
where $c\in \BR^2$ is arbitrary.
Here $K$ denotes the fundamental solution of the linearized Navier-Stokes equation in $\BR^2$:
\beqo
K(x,t):=\f{e^{-|x|^2/(4t)}}{4\pi t}\left(\Id-\f{x\otimes x}{|x|^2}\right)-\f{1-e^{-|x|^2/(4t)}}{2\pi |x|^2}\left(\Id-\f{2x\otimes x}{|x|^2}\right).
\eeqo
It satisfies the following estimates.

\begin{lem}
\label{lem: estimate for K}
$K(x,t)\in C^{\infty}(\R^2\times(0,+\infty))$, and for any $k\in \BN$,
\beqo
|\na_x^k K(x,t)|\leq C(t+|x|^2)^{-(2+k)/2}.
\eeqo
Moreover,
\[
|\na_x^k K(x,t)-\na_y^k K(y,t)|
\leq C|x-y|\left((t+|x|^2)^{-(3+k)/2}+(t+|y|^2)^{-(3+k)/2}\right).
\]
Here the constants $C$'s only depend on $k$.
\begin{proof}
For $z\in(0,+\infty)$, let
\beqo
\va(z):=e^{-z}-\f{1-e^{-z}}{2z},\quad
\psi(z):=\f{1}{z}\left(e^{-z}-\f{1-e^{-z}}{z}\right)
= -\f{1-e^{-z}(1+z)}{z^2} = \f{d}{dz}\left(\f{1-e^{-z}}{z}\right).
\eeqo
We additionally define $\va(0)=1/2$ and $\psi(0)=-1/2$.
It is not difficult to show by induction that, for all $k\in \BN$ and $z>0$,
\[
\left(\f{d}{dz}\right)^k\left(\f{1-e^{-z}}{z}\right)
= (-1)^k\cdot \f{k!}{z^{k+1}}
\left[1-e^{-z}\sum_{j = 0}^k\f{z^j}{j!}\right].
\]
As a result,
\begin{itemize}
\item $\va(z), \psi(z)\in C^{\infty}([0,+\infty))$;
\item and for all $k\in \BN$ and $z\geq 0$,
\[
|\pa_z^k \va(z)|\leq C_k(1+z)^{-1-k}, \quad
|\pa_z^k \psi(z)|\leq C_k(1+z)^{-2-k},
\]
where $C_k$'s are universal constants depending on $k$.
\end{itemize}

With these notations,
\beqo
K(x,t)=\f{\va(|x|^2/(4t))}{4\pi t}\Id-\f{\psi(|x|^2/(4t))}{\pi (4t)^2}x\otimes x.
\eeqo
Then the desired claims follow from induction and direct calculation.
\end{proof}
\end{lem}

The following two lemmas will be useful in the calculation below.
\begin{lem}
\label{lem: lower bound for x-X(s)}
Suppose $|X|_*\geq\lam$. Let $d(x)$ and $s_x$ be defined as in \eqref{eqn: def of d(x,t)} and \eqref{eqn: def of s_x}, respectively.
Then for any $x\in \BR^2$ and $s\in \BT$,
\beq
|x-X(s)|\geq \max\left\{d(x),\,\f{\lam}{2}|s_x-s|_\BT\right\}.
\label{eqn: lower bound for distance of x to a string point}
\eeq
\begin{proof}
It follows from \cite[Equation (A.1)]{lin2019solvability} and the definition of $d(x)$.
\end{proof}
\end{lem}

\begin{lem}
\label{lem: a typical integral}
Suppose $|X|_*\geq\lam$.
For any $x\in\R^2$, $t>0$, and $m>1/2$,
\[
\int_{\T} \f{ds}{(t+|x-X(s)|^2)^m}\leq C\min(\lam^{-1}t^{1/2-m},\, t^{-m}),
\]
where the constant $C$ depends only on $m$.
\begin{proof}
It follows from Lemma \ref{lem: lower bound for x-X(s)} and direct calculation.
\end{proof}
\end{lem}

Now we are ready to bound $u_X$ and $u_X^\nu$.
\begin{proof}[Proof of Lemma \ref{lem: estimate for u_X}]
By \eqref{eqn: Laplace of heat kernel applied to u_X}, for $k \in \BN$,
\beq
\begin{split}
-\na^k \Delta e^{t\Delta}u_{X}(x)_i
= &\; \int_{\T}X_l'(s)\pa_{l}\na^k K_{ij}(x-X(s),t)(X'(s)-X'(s_x))_j\,ds\\
= &\; \int_{\T}X_l'(s)\pa_{l}\na^k K_{ij}(x-X(s),t)X_j'(s)\,ds.
\end{split}
\label{eqn: representation of Delta heat kernel applied to u_X}
\eeq
Here and in what follows, we adopt the Einstein summation convention.

From the first representation, we can derive with Lemma \ref{lem: estimate for K} and Lemma \ref{lem: lower bound for x-X(s)} (also see Lemma \ref{lem: a typical integral}) that, for $k\in \BN$,
\beq
\begin{split}
|\na^k \Delta e^{t\Delta}u_{X}(x)|
\leq &\; C\int_{\T}\frac{|X'(s)||X'(s)-X'(s_x)|}{(t+|x-X(s)|^2)^{(3+k)/2}}\,ds
\leq C \int_{\T}\f{\|X'\|_{L_s^{\infty}}\|X'\|_{\dot{C}_s^\g} |s-s_x|^\g} {(t+\lambda^2|s-s_x|^2)^{(3+k)/2}}\,ds\\
\leq &\; C \lambda^{-1-\g}t^{-(2+k-\g)/2} \|X'\|_{L_s^{\infty}}\|X'\|_{\dot{C}_s^\g},
\end{split}
\label{eqn: estimate for Laplace of heat kernel applied to u_X}
\eeq
where the constant $C$ depends on $k$ and $\g$.
On the other hand, \eqref{eqn: estimate for Laplace of heat kernel applied to u_X} also gives
\[
|\na^k \Delta e^{t\Delta}u_{X}(x)|
\leq C\int_{\T}\frac{\lam^{-\g}\|X'\|_{L_s^{\infty}}\|X'\|_{\dot{C}_s^\g}} {(t+ |x-X(s)|^2)^{(3+k-\g)/2}}\,ds.
\]
Hence,
\[
\|\na^k \Delta e^{t\Delta}u_{X}\|_{L_x^1(\BR^2)}
\leq C \int_{\BR^2}\f{ \lam^{-\g} \|X'\|_{L_s^{\infty}}\|X'\|_{\dot{C}_s^\g}}{(t+ |x|^2)^{(3+k-\g)/2}}\,dx
\leq C\lam^{-\g} t^{-(1+k-\g)/2}\|X'\|_{L_s^{\infty}}\|X'\|_{\dot{C}_s^\g},
\]
where $C$ also depends on $k$ and $\g$.

From the second representation in \eqref{eqn: representation of Delta heat kernel applied to u_X}, we similarly argue that
\begin{align*}
|\na^k \Delta e^{t\Delta}u_{X}(x)|
\leq &\; C\int_{\T}\frac{|X'(s)|^2}{(t+|x-X(s)|^2)^{(3+k)/2}}\,ds
\leq C\int_{\T}\frac{|X'(s)|^2}{(t+\lam^2|s-s_x|^2)^{(3+k)/2}}\,ds
\\
\leq &\; C\min\big(\lambda^{-1}t^{-(2+k)/2}\|X'\|_{L_s^\infty}^2,\, t^{-(3+k)/2}\|X'\|_{L_s^2}^2\big),
\end{align*}
where the constant $C$ only depends on $k$.
On the other hand, integrating this on $\BR^2$ yields that
\[
\|\na^k \Delta e^{t\Delta}u_{X}\|_{L_x^1(\R^2)}
\leq C\int_{\T}|X'(s)|^2\left(\int_{\BR^2}\f1{(t+|x|^2)^{(3+k)/2}}\,dx\right)ds
\leq C t^{-(1+k)/2} \|X'\|_{L_s^2}^2,
\]
where the constant $C$ only depends on $k$ as well.

To prove the estimates for $u_X$, we observe that
\beq
u_{X} =-\int_0^{\infty}\Delta e^{\tau\Delta}u_{X}\, d\tau.
\label{eqn: representation of u_X using time integral}
\eeq
Hence, with $X\in O^{\lam}$,
\begin{align*}
\|u_{X}\|_{L_x^{\infty}(\R^2)}
\leq &\; \int_0^{\infty}\|\Delta e^{\tau\Delta}u_{X}\|_{L_x^{\infty}(\R^2)}\,d\tau\\
\leq &\; C\int_0^{\infty}\min(\lambda^{-1-\gamma}\tau^{-(2-\g)/2},\,\tau^{-3/2}) \|X'\|_{\dot{C}_s^{\gamma}} \|X'\|_{L_s^{\infty}} \, d\tau\\
\leq &\; C\lambda^{-1}\|X'\|_{\dot{C}_s^{\gamma}}\|X'\|_{L_s^{\infty}},
\end{align*}
where the constant $C$ depends on $\g$.
For any $p\in (2,\infty)$,
\begin{align*}
\|u_{X}\|_{L_x^p(\R^2)}
\leq &\; \int_0^{\infty}\|\Delta e^{\tau\Delta}u_{X}\|_{L_x^p(\R^2)}\,d\tau\\
\leq &\; \int_0^{\infty}\|\Delta e^{\tau\Delta}u_{X}\|_{L_x^1(\R^2)}^{\f1p}
\|\Delta e^{\tau\Delta}u_{X}\|_{L_x^\infty(\R^2)}^{1-\f1p}\,d\tau\\
\leq &\; C\int_0^{\infty}
\left[\tau^{-1/2} \|X'\|_{L_s^2}^2\right]^{\f1p}
\left[\min\big(\lambda^{-1}\tau^{-1}\|X'\|_{L_s^\infty}^2,\, \tau^{-3/2}\|X'\|_{L_s^2}^2\big)\right]^{1-\f1p} \, d\tau\\
\leq &\; C\lambda^{-1+\f{2}{p}}\|X'\|_{L_s^2}^{\f{4}{p}}
\|X'\|_{L_s^{\infty}}^{2-\f{4}{p}},
\end{align*}
where the constant $C$ depends on $p$.
Finally, applying Lemma \ref{lem: interpolation} below with $\va \equiv 1$, we also have that
\begin{align*}
\|u_{X}\|_{\dot{C}_x^{\g}(\R^2)}
\leq &\; C\left|\sup_{\tau>0}\tau^{(2-\g)/2}\|\D e^{\tau\D}u_{X}\|_{L_x^{\infty}(\R^2)}\right|^{1-\g}
\left|\sup_{\tau>0}\tau^{(3-\g)/2}\|\na\D e^{\tau\D}u_{X}\|_{L_x^{\infty}(\R^2)}\right|^{\g}\\
\leq &\; C\lam^{-1-\g} \|X'\|_{\dot{C}_s^{\gamma}} \|X'\|_{L_s^{\infty}},
\end{align*}
where the constant $C$ depends on $\g$.

This completes the proof.
\end{proof}

In the proof above, we used the following lemma on H\"{o}lder estimate of an integral term.

\begin{lem}
\label{lem: interpolation}
Assume that $\va:(0,+\infty)\to (0,+\infty)$ is decreasing on $(0,+\infty)$ and it is integrable on any bounded interval.
Then for $\al\in (0,1)$ and any $t>0$,
\begin{align*}
\left\|\int_0^t v(x,\tau)\, d\tau\right\|_{\dot{C}_x^{\alpha}(\R^2)}
\leq &\; \f{C}{t}\int_0^t \va(\tau)\,d\tau \left|\sup_{\eta\in(0,t)}(t-\eta)^{\f{2-\al}{2}} \va(\eta)^{-1} \|v(\cdot,\eta)\|_{L_x^{\infty}(\R^2)}\right|^{1-\al}\\
&\;\cdot \left|\sup_{\eta\in(0,t)}(t-\eta)^{\f{3-\al}{2}}\va(\eta)^{-1}\|\na v(\cdot,\eta)\|_{L_x^{\infty}(\R^2)}\right|^{\al},\\
\left\|\int_0^t V(s,\tau)\, d\tau\right\|_{\dot{C}_s^{\alpha}(\T)}
\leq &\;  \f{C}{t}\int_0^t \va(\tau)\,d\tau \left|\sup_{\eta\in(0,t)}(t-\eta)^{\f{2-\al}{2}} \va(\eta)^{-1}\|V(\cdot,\eta)\|_{L_s^{\infty}(\T)}\right|^{1-\al}\\
&\;\cdot \left|\sup_{\eta\in(0,t)}(t-\eta)^{\f{3-\al}{2}}
\va(\eta)^{-1}\|\pa_s V(\cdot,\eta)\|_{L_s^{\infty}(\T)}\right|^\al,\\
\left\|\int_0^t V(s,\tau)\, d\tau\right\|_{\dot{C}_s^{\alpha}(\T)}
\leq &\; \f{C}{t}\int_0^t \va(\tau)\,d\tau
\left|\sup_{\eta\in(0,t)}(t-\eta)^{1-\al} \va(\eta)^{-1}\|V(\cdot,\eta)\|_{L_s^\infty(\T)}\right|^{1-\al}\\
&\; \cdot \left|\sup_{\eta\in(0,t)}(t-\eta)^{2-\al}
\va(\eta)^{-1}\|\pa_s V(\cdot,\eta)\|_{L_s^{\infty}(\T)}\right|^\al,
\end{align*}
where the constants $C$'s only depend on $\al$.

Similarly,
\begin{align*}
\left\|\int_0^\infty v(x,\tau)\, d\tau\right\|_{\dot{C}_x^{\alpha}(\R^2)}
\leq &\; C\left|\sup_{\eta\in(0,\infty)}\eta^{\f{2-\al}{2}} \|v(\cdot,\eta)\|_{L_x^{\infty}(\R^2)}\right|^{1-\al} \left|\sup_{\eta\in(0,\infty)} \eta^{\f{3-\al}{2}} \|\na v(\cdot,\eta)\|_{L_x^{\infty}(\R^2)}\right|^{\al},
\end{align*}
where $C$ only depends on $\al$.
We also have other two estimates that are parallel to those above.

\begin{proof}
Denote
\begin{align*}
A:= &\; \sup_{\eta\in (0,t)}(t-\eta)^{\f{2-\al}{2}} \va(\eta)^{-1} \|v(\cdot,\eta)\|_{L_x^{\infty}(\R^2)},\\
B:= &\; \sup_{\eta\in (0,t)}(t-\eta)^{\f{3-\al}{2}} \va(\eta)^{-1} \|\na v(\cdot,\eta)\|_{L_x^{\infty}(\R^2)}.
\end{align*}
For any $x,y\in\R^2$ and $\tau\in(0,t)$,
\begin{align*}
|v(x,\tau)-v(y,\tau)|
\leq &\;\min (2\|v(\tau)\|_{L_x^\infty(\R^2)},\,|x-y|\|\nabla v(\tau)\|_{L_x^\infty(\R^2)})\\
\leq &\;
\min\left(2(t-\tau)^{\f{\alpha-2}{2}} \va(\tau) A,\, |x-y|(t-\tau)^{\f{\alpha-3}{2}} \va(\tau) B\right).
\end{align*}
Hence, with $\va$ being decreasing,
\begin{align*}
&\;\left|\int_0^t v(x,\tau)\, d\tau
-\int_0^t v(y,\tau)\, d\tau\right|
\leq \int_0^t|v(x,\tau)-v(y,\tau)|\, d\tau\\
\leq &\;
\int_0^t \min\left(2(t-\tau)^{\f{\alpha-2}{2}} A, \,|x-y|(t-\tau)^{\f{\alpha-3}{2}} B\right)\va(\tau)\, d\tau
\\
\leq &\;
C\int_0^{t/2}\min\left(t^{\f{\alpha-2}{2}}A, \,|x-y|t^{\f{\alpha-3}{2}} B\right) \va(\tau)\, d\tau\\
&\; +
C\int_{t/2}^{t}\min\left((t-\tau)^{\f{\alpha-2}{2}}A, \,|x-y|(t-\tau)^{\f{\alpha-3}{2}} B\right) \va\left(\f{t} {2}\right) d\tau
\\
\leq &\; C\left[ \va\left(\f{t}{2}\right) + \f{1}{t}\int_0^{t/2} \va(\tau)\,d\tau\right] A^{1-\al}\big(|x-y|B\big)^{\al}.
\end{align*}
This implies the desired inequality, since given the monotonicity of $\va$,
\[
\va\left(\f{t}{2}\right) + \f{1}{t}\int_0^{t/2} \va(\tau)\,d\tau
\leq \f{3}{t}\int_0^{t/2} \va(\tau)\,d\tau
\leq \f{3}{t}\int_0^{t} \va(\tau)\,d\tau.
\]

The other claims can be justified similarly.
\end{proof}
\end{lem}

\begin{proof}[Proof of Lemma \ref{lem: estimate for u_X^nu}]

By \eqref{eqn: u_X^nu in terms of u_X} and Lemma \ref{lem: estimate for u_X},
\begin{align*}
&\;\|u_X^{\nu}(t)\|_{L_x^{\infty}(\R^2)}\\
\leq &\; \nu\int_0^{t}\|\Delta e^{\nu (t-\tau)\Delta}u_{X(\tau)}\|_{L^{\infty}_x(\R^2)}\, d\tau\\
\leq &\; C\nu\int_0^{t} \min(\lam^{-1-\g}(\nu (t-\tau))^{-(2-\g)/2},\, (\nu(t-\tau))^{-3/2})
\tau^{-\g} \cdot \|X'(\tau)\|_{L_s^{\infty}} \cdot \tau^\g\|X'(\tau)\|_{\dot{C}_s^{\g}}
\, d\tau\\
\leq &\; C \|X'\|_{L^\infty_t L_s^{\infty}}
\sup_{\eta\in (0,t]} \eta^\g \|X'(\eta)\|_{\dot{C}_s^{\g}}
\\
&\; \cdot \left[\nu \int_0^{t/2} \min(\lam^{-1-\g}(\nu t)^{-(2-\g)/2},\, (\nu t)^{-3/2})
\tau^{-\g}\, d\tau\right.\\
&\;\quad \left. + \nu \int_{t/2}^{t} \min(\lam^{-1-\g}(\nu (t-\tau))^{-(2-\g)/2},\, (\nu(t-\tau))^{-3/2})
t^{-\g}\, d\tau\right]\\
\leq &\; C \|X'\|_{L^\infty_t L_s^{\infty}}
\sup_{\eta\in (0,t]} \eta^\g \|X'(\eta)\|_{\dot{C}_s^{\g}}\\
&\;\cdot \left[ \lam^{-1}t^{-\g} \min(\lam^{-\g}(\nu t)^{\g/2},\, \lam (\nu t)^{-1/2})
+\lam^{-1} t^{-\g} \min(\lam^{-\g} (\nu t)^{\g/2},\, 1) \right]\\
\leq &\; C\lam^{-1}t^{-\g} \min(\lam^{-2} \nu t,\, 1)^{\g/2} \|X'\|_{L^\infty_t L_s^{\infty}}
\sup_{\eta\in (0,t]} \eta^\g \|X'(\eta)\|_{\dot{C}_s^{\g}},
\end{align*}
where the constant $C$ depends only on $\g$.

By Lemma \ref{lem: estimate for u_X} and Lemma \ref{lem: interpolation} with $\va (t)= t^{-\g}$,
\begin{align*}
\|u_X^{\nu}(t)\|_{\dot{C}_x^\g(\R^2)}
\leq &\; C\nu t^{-\g} \left|\sup_{\tau\in (0,t)} (t-\tau)^{(2-\gamma)/2}\cdot \tau^\g \|\Delta e^{\nu (t-\tau)\Delta} u_{X(\tau)}\|_{L_x^\infty(\R^2)}\right|^{1-\gamma}\\
&\;\cdot \left|\sup_{\tau\in (0,t)} (t-\tau)^{(3-\gamma)/2}\cdot \tau^\g \|\nabla\Delta e^{\nu (t-\tau)\Delta}u_{X(\tau)}\|_{L_x^\infty(\R^2)}\right|^{\gamma}\\
\leq &\; C\lam^{-1-\g} t^{-\g} \|X'\|_{L^\infty_t L_s^{\infty}}
\sup_{\eta\in (0,t]} \eta^\g \|X'(\eta)\|_{\dot{C}_s^{\g}}.
\end{align*}

By \eqref{eqn: u_X^nu in terms of u_X} and Lemma \ref{lem: estimate for u_X},
\begin{align*}
\|h_X^{\nu}(t)\|_{L_x^{\infty}(\R^2)}
\leq &\; \|u_X^{\nu}(t)\|_{L_x^{\infty}(\R^2)} +\|(\mathrm{Id}-e^{\nu t\D})u_{X(t)}\|_{L_x^\infty(\R^2)}
\\
\leq &\; \|u_X^\nu(t)\|_{L_x^\infty(\R^2)}
+ C \|u_{X(t)}\|_{L_x^\infty(\R^2)} \\
\leq &\; C\lam^{-1}t^{-\g} \|X'\|_{L^\infty_t L_s^{\infty}}
\sup_{\eta\in (0,t]} \eta^\g \|X'(\eta)\|_{\dot{C}_s^{\g}},
\end{align*}
and similarly,
\[
\|h_X^{\nu}(t)\|_{\dot{C}_x^{\g}(\R^2)}
\leq C\lam^{-1-\g} t^{-\g} \|X'\|_{L^\infty_t L_s^{\infty}}
\sup_{\eta\in (0,t]} \eta^\g \|X'(\eta)\|_{\dot{C}_s^{\g}}.
\]
Here the constants $C$'s only depend on $\g$.

To show the $L^p$-estimate for $u_X^\nu$ with $p\in [1,\infty)$, we apply Lemma \ref{lem: estimate for u_X} to find that, for $\tau\in [0,t]$,
\begin{align*}
&\;\|\Delta e^{\nu (t-\tau)\Delta}u_{X(\tau)}\|_{L^p_x(\R^2)}\, \\
\leq &\; \|\Delta e^{\nu (t-\tau)\Delta}u_{X(\tau)}\|_{L^1_x(\R^2)}^{\f1p}
\|\Delta e^{\nu (t-\tau)\Delta}u_{X(\tau)}\|_{L^\infty_x(\R^2)}^{1-\f1p}\\
\leq &\; C\Big((\nu (t-\tau))^{-\f12} \|X'(\tau)\|_{L_s^2}^2\Big)^{\f1p}\\
&\;\cdot \left[\min\Big(
\lam^{-1}(\nu (t-\tau))^{-1}\|X'(\tau)\|_{L_s^\infty}^2,\; (\nu (t-\tau))^{-\f32}\|X'(\tau)\|_{L_s^2}^2\Big)
\right]^{1-\f1p}\\
\leq &\; C\min\left(\lambda^{-1+\f1p} (\nu (t-\tau))^{-1+\f1{2p}} \|X'(\tau)\|_{L_s^2}^{\f2p}
\|X'(\tau)\|_{L_s^\infty}^{2-\f2p},\;
(\nu (t-\tau))^{-\f32+\f1p}\|X'(\tau)\|_{L_s^2}^2
\right).
\end{align*}
For $p\in [1,\infty)$, we integrate the first bound above in $\tau$ to obtain that
\[
\|u_X^{\nu}(t)\|_{L^p_x(\R^2)}
\leq \nu\int_0^{t}\|\Delta e^{\nu (t-\tau)\Delta}u_{X(\tau)}\|_{L^p_x(\R^2)}\, d\tau
\\
\leq C(\nu t)^{\f1{2p}} \lam^{-1+\f1p}
\|X'\|_{L^\infty_t L_s^2}^{\f2p}
\|X'\|_{L^\infty_t L_s^\infty}^{2-\f2p},
\]
where the constant $C$ depends only on $p$.
If additionally $p\in (2,\infty)$, we use the both bounds above to similarly derive that
\[
\|u_X^{\nu}(t)\|_{L^p_x(\R^2)}
\leq C \lam^{-1+\f2p}
\|X'\|_{L^\infty_t L_s^2}^{\f4p}
\|X'\|_{L^\infty_t L_s^\infty}^{2-\f4p},
\]
where $C$ depends only on $p$.
The $L^p$-estimate for $h_X^\nu$ then follows from \eqref{eqn: u_X^nu in terms of u_X} and Lemma \ref{lem: estimate for u_X}.
\end{proof}

\begin{proof}[Proof of Lemma \ref{lem: time Holder continuity of u_X^nu}]
By \eqref{eqn: u_X^nu in terms of u_X} and Lemma \ref{lem: parabolic estimates}, for arbitrary $0\leq t_1<t_2\leq T$,
\begin{align*}
&\; \|u_X^\nu(t_1)-u_X^\nu(t_2)\|_{L_x^\infty(\BR^2)}\\
\leq &\; \nu\int_{t_1}^{t_2} \|\D e^{\nu (t_2-\tau)\D}u_{X(\tau)}\|_{L_x^\infty} \, d\tau
+\left\| (e^{\nu (t_2-t_1)\D}-\Id) u_{X}^\nu(t_1)\right\|_{L_x^\infty}\\
\leq &\; C\nu\int_{t_1}^{t_2} (\nu(t_2-\tau))^{-1+\g/2}\|u_{X(\tau)}\|_{\dot{C}_x^\g} \, d\tau
+C(\nu(t_2-t_1))^{\g/2}\|u_{X}^\nu(t_1)\|_{\dot{C}_x^\g}\\
\leq &\; 
C\lam^{-1-\g}(\nu(t_2-t_1))^{\g/2} t_1^{-\g} \|X'\|_{L^\infty_{t_2} L_s^{\infty}}
\sup_{\eta\in (0,t_2]} \eta^\g \|X'(\eta)\|_{\dot{C}_s^{\g}} .
\end{align*}
In the last line, we applied Lemma \ref{lem: estimate for u_X} and Lemma \ref{lem: estimate for u_X^nu}.
\end{proof}

\subsection{Estimates for $u_X-u_Y$ and $u^\nu_X-u^\nu_Y$}
\label{sec: estimates for u_X-u_Y}
In order to bound $u_X-u_Y$ and $u^\nu_X-u^\nu_Y$, we 
fix $t>0$, and start with the calculation
\beq
\Delta e^{t\Delta}(u_{X}-u_{Y})
= \int_0^1\f{d}{d\th}\Delta e^{t\Delta}(u_{\th X+(1-\th)Y})\, d\th
= \int_0^1\mathcal{D}[\Delta e^{t\Delta}u]_{\th X+(1-\th)Y}[X-Y]\, d\th.
\label{eqn: Laplace of heat kernel applied to u_X-u_Y}
\eeq
Here by \eqref{eqn: Laplace of heat kernel applied to u_X}, with the time-variable omitted and the Einstein summation convention adopted,
\beq
\mathcal{D}[\Delta e^{t\Delta}u]_{X}[Z](x)_i
= \int_{\T} \pa_{s}[K_{ij}(x-X(s))]Z_j'(s)
-\pa_{s}[Z_l(s)\pa_lK_{ij}(x-X(s))]X_j'(s) \,ds.
\label{eqn: def of first variation of Delta etD u}
\eeq
It satisfies the following estimate.

\begin{lem}
\label{lem: first variation of Laplace of heat kernel applied to u_X}
Suppose $X\in O^\lam$ and $Z\in C^1(\BT)$.
Then for $k\in \BN$,
\begin{align*}
&\; |\na^k\mathcal{D}[\Delta e^{t\Delta}u]_{X}[Z]|\\
\leq &\; C\min(\lam^{-1}t^{-(2+k)/2}, \, t^{-(3+k)/2}) \|X'\|_{L_s^{\infty}}\|Z'\|_{L_s^{\infty}}\\
&\;+C\min\Big(\lam^{-1-\g}t^{-(3+k-\g)/2} \big(\|X'\|_{L_s^\infty}\|X'\|_{\dot{C}_s^\g}\|Z\|_{L_s^\infty}
+\|X'\|_{L_s^\infty}^2\|Z\|_{\dot{C}_s^\g}\big),\\
&\;\qquad\qquad\quad \lam^{-1}t^{-(3+k)/2} \|X'\|_{L_s^\infty}^2 \|Z\|_{L_s^\infty}\Big),
\end{align*}
where the constant $C$ depends on $k$ and $\g$.

\begin{proof}
Let us write $\CD[\Delta e^{t\Delta}u]_{X}[Z]
=I_1[t;X,Z]+I_2[t;X,Z]+I_3[t;X,Z]$, where
\begin{align}
&I_1[t;X,Z](x)_i := \int_{\T}\pa_{s}[K_{ij}(x-X(s),t)]Z_j'(s)\,ds,\label{eqn: formula for I_1}\\
&I_2[t;X,Z](x)_i:=-\int_{\T}Z_l'(s)\pa_lK_{ij}(x-X(s),t) X_j'(s)\,ds,\label{eqn: formula for I_2}\\
&I_3[t;X,Z](x)_i:=\int_{\T}Z_l(s)X_m'(s)\pa_m\pa_lK_{ij}(x-X(s),t) X_j'(s)\,ds.\label{eqn: formula for I_3}
\end{align}
In what follows, we shall omit the $t$-dependence.

For $k \in \BN$,
\beqo
\na^k I_1[X,Z](x)_i=-\int_{\T}X_l'(s)\partial_l \na^k K_{ij}(x-X(s))Z_j'(s)\,ds,
\eeqo
so Lemma \ref{lem: estimate for K} and Lemma \ref{lem: a typical integral} imply that
\begin{align*}
|\na^k I_1[X,Z](x)|\leq &\; C\int_{\T}\f{|X'(s)||Z'(s)|}{(t+|x-X(s)|^2)^{(3+k)/2}}\,ds\\
\leq &\; C\min(\lam^{-1}t^{-(2+k)/2}, \, t^{-(3+k)/2}) \|X'\|_{L_s^\infty}\|Z'\|_{L_s^\infty}.
\end{align*}

Similarly,
\[
|\na^k I_2[X,Z](x)|\leq C\min(\lam^{-1}t^{-(2+k)/2}, \, t^{-(3+k)/2}) \|X'\|_{L_s^\infty}\|Z'\|_{L_s^\infty}.
\]

As for $I_3$, with $s_x$ introduced in \eqref{eqn: def of s_x},
\beq
\begin{split}
\na^k I_3[X,Z](x)_i
= &\; \int_{\T}X_m'(s)\pa_m\pa_l \na^k K_{ij}(x-X(s)) X_j'(s)Z_l(s)\,ds\\
= &\; \int_{\T}X_m'(s)\pa_m\pa_l \na^k K_{ij}(x-X(s)) \big(X_j'(s)Z_l(s)-X_j'(s_x)Z_l(s_x)\big)\,ds.
\end{split}
\label{eqn: formula for I_3 new}
\eeq
So we can analogously derive two estimates
\[
|\na^k I_3[X,Z](x)|\leq C\int_{\T}\frac{|X'(s)|^2|Z(s)|}{(t+|x-X(s)|^2)^{(4+k)/2}}\,ds
\leq 
C\lambda^{-1}t^{-(3+k)/2}\|X'\|_{L_s^\infty}^2\|Z\|_{L_s^\infty},
\]
and (cf.\;\eqref{eqn: estimate for Laplace of heat kernel applied to u_X})
\beq
\begin{split}
|\na^k I_3[X,Z](x)|
\leq &\; C\int_{\T}\f{\|X'\|_{L_s^\infty}\|X' \otimes Z\|_{\dot{C}_s^\g}|s-s_x|^{\g}}{(t+\lam^2|s-s_x|^2)^{(4+k)/2}}\,ds\\
\leq &\; C\lam^{-1-\g}t^{-(3+k-\g)/2}
\big(\|X'\|_{L_s^\infty}\|X'\|_{\dot{C}_s^\g}\|Z\|_{L_s^\infty}
+\|X'\|_{L_s^\infty}^2\|Z\|_{\dot{C}_s^\g}\big).
\end{split}
\label{eqn: the 2nd bound for grad I_3}
\eeq

Summarizing the above estimates, we obtain the desired bound.
\end{proof}
\begin{rmk}
\label{rmk: gamma equals 1}
It is clear from \eqref{eqn: the 2nd bound for grad I_3} that this estimate also holds with $\g = 1$.
\end{rmk}
\end{lem}

\begin{lem}
\label{lem: estimate for Laplace of heat operator applied to u_X-u_Y}

Suppose $X,Y\in O^\lam$, satisfying that $\|X'-Y'\|_{L^{\infty}_s(\BT)}\leq \lam/2$.
For any $t>0$ and $k\in \BN$, it holds that
\begin{align*}
&\;\|\na^k\D e^{t\D}(u_{X}-u_{Y})\|_{L_x^{\infty}(\R^2)}\\
\leq &\; C\min(\lam^{-1}t^{-(2+k)/2}, \, t^{-(3+k)/2}) \|(X',Y')\|_{L_s^\infty}\|X'-Y'\|_{L_s^\infty} \\
&\; + C\min\Big(\lam^{-1-\g}t^{-(3+k-\g)/2}\big(\|(X',Y')\|_{L_s^\infty} \|(X',Y')\|_{\dot{C}_s^\g} \|X-Y\|_{L_s^\infty} + \|(X',Y')\|_{L_s^\infty}^2 \|X-Y\|_{\dot{C}_s^\g}\big),\\
&\;\qquad \qquad \quad \lam^{-1}t^{-(3+k)/2} \|(X',Y')\|_{L_s^\infty}^2 \|X-Y\|_{L_s^\infty}\Big),
\end{align*}
where the constant $C$ depends on $k$ and $\g$.
This estimate also holds for $\g = 1$.

\begin{proof}
By the assumptions, $|\th X+ (1-\th)Y|_* \geq \lam/2$ for any $\th\in [0,1]$.
Then the estimate follows from \eqref{eqn: Laplace of heat kernel applied to u_X-u_Y}, Lemma
\ref{lem: first variation of Laplace of heat kernel applied to u_X}, and Remark \ref{rmk: gamma equals 1}.
\end{proof}
\end{lem}

This enables us to prove Lemma \ref{lem: estimate for u_X-u_Y}.
\begin{proof}[Proof of Lemma \ref{lem: estimate for u_X-u_Y}]
Since
\beq
e^{t\D}(u_X-u_Y) = -\int_t^{\infty}\Delta e^{\tau\Delta}(u_{X}-u_Y)\,d\tau,
\label{eqn: representation of heat kernel applied to u_X-u_Y}
\eeq
we apply Lemma \ref{lem: estimate for Laplace of heat operator applied to u_X-u_Y} to find that, for arbitrary $\b\in (0,1]$,
\begin{align*}
&\; \big\|e^{t\D}(u_X-u_Y)\big\|_{L^\infty_x(\BR^2)}\\
\leq &\; \int_t^\infty \big\|\Delta e^{\tau\Delta}(u_X-u_Y)\big\|_{L_x^\infty(\BR^2)}\,d\tau\\
\leq &\; C \int_t^\infty \lam^{-(1-\b)}\tau^{-1-\b/2}\|(X',Y')\|_{L_s^{\infty}} \|X'-Y'\|_{L_s^{\infty}} \\
&\; \qquad + \lam^{-1-\g}\tau^{-(3-\g)/2} \\
&\;\qquad \quad \cdot \big(\|(X',Y')\|_{L_s^\infty} \|(X',Y')\|_{\dot{C}_s^\g} \|X-Y\|_{L_s^\infty} + \|(X',Y')\|_{L_s^\infty}^2 \|X-Y\|_{\dot{C}_s^\g}\big)
\,d\tau\\
\leq &\; C \lam^{-(1-\b)}t^{-\b/2}\|(X',Y')\|_{L_s^\infty}\|X'-Y'\|_{L_s^\infty} \\
&\; + C\lam^{-1-\g}t^{-(1-\g)/2}
\big(\|(X',Y')\|_{L_s^\infty} \|(X',Y')\|_{\dot{C}_s^\g} \|X-Y\|_{L_s^\infty} + \|(X',Y')\|_{L_s^\infty}^2 \|X-Y\|_{\dot{C}_s^\g}\big),
\end{align*}
where the constant $C$ depends on $\g$ and $\b$.

On the other hand, by Lemma \ref{lem: estimate for Laplace of heat operator applied to u_X-u_Y}, for $k = 0,1$,
\begin{align*}
&\; \|\na^k\D e^{\tau\D}(u_X-u_Y)\|_{L_x^\infty(\R^2)}\\
\leq &\;
C \lam^{-1}\tau^{-(2+k)/2} \|(X',Y')\|_{L_s^\infty} \|X'-Y'\|_{L_s^\infty}\\
&\;
+ C \lam^{-1-\g}\tau^{-(3+k-\g)/2}
\big(\|(X',Y')\|_{L_s^\infty} \|(X',Y')\|_{\dot{C}_s^\g} \|X-Y\|_{L_s^\infty} + \|(X',Y')\|_{L_s^\infty}^2 \|X-Y\|_{\dot{C}_s^\g}\big).
\end{align*}
By \eqref{eqn: representation of heat kernel applied to u_X-u_Y} and the last estimate in Lemma \ref{lem: interpolation}, for $t>0$,
\begin{align*}
&\; \big\|e^{t\D}(u_X-u_Y)\big\|_{\dot{C}^\g_x(\BR^2)}\\
\leq &\; \left|\sup_{\tau>t}\tau^{\f{2-\g}{2}} \|\D e^{\tau\D}(u_X-u_Y)\|_{L_x^{\infty}(\R^2)}\right|^{1-\g}
\left|\sup_{\tau>t}\tau^{\f{3-\g}{2}}\|\na \D  e^{\tau\D}(u_X-u_Y)\|_{L_x^\infty(\R^2)}\right|^\g\\
\leq &\; C \lam^{-1} t^{-\g/2} \|(X',Y')\|_{L_s^\infty} \|X'-Y'\|_{L_s^\infty}\\
&\; + C \lam^{-1-\g} t^{-1/2}
\big(\|(X',Y')\|_{L_s^\infty} \|(X',Y')\|_{\dot{C}_s^\g} \|X-Y\|_{L_s^\infty} + \|(X',Y')\|_{L_s^\infty}^2 \|X-Y\|_{\dot{C}_s^{\g}}\big),
\end{align*}
where $C$ depends only on $\g$.
\end{proof}

Next we prove Lemma \ref{lem: estimate for u_X^nu-u_Y^nu in Morrey space}.
The following estimate will be useful.

\begin{lem}\label{lem: a typical integral in x variable}
For any $x\in \BR^2$, $r>0$, and $m>1$,
\[
\int_{B(x,r)}(t+|y-X(s)|^2)^{-m}\, dy
\leq Cr^2\left(1+\f{r^2}{t}\right)^{m-1}\big(t+r^2+|x-X(s)|^2\big)^{-m},
\]
where the constant $C$ only depends on $m$.

\begin{proof}
By a change of variable,
\[
\int_{B(x,r)}(t+|y-X(s)|^2)^{-m}\, dy
= t^{-m+1} \int_{B\big(\f{x-X(s)}{\sqrt{t}},\f{r}{\sqrt{t}}\big)} (1+|y|^2)^{-m}\, dy.
\]
If $|x-X(s)|\geq 2r$, then
\begin{align*}
\int_{B(x,r)}(t+|y-X(s)|^2)^{-m}\, dy
\leq &\; Ct^{-m+1} \cdot \f{r^2}{t} \left(1+\left|\f{x-X(s)}{\sqrt{t}}\right|^2\right)^{-m}\\
\leq&\; C r^2 \big(t+r^2+|x-X(s)|^2\big)^{-m},
\end{align*}
where the constant $C$ depends on $m$ only.
Otherwise, if $|x-X(s)|\leq 2r$, then
\begin{align*}
\int_{B(x,r)}(t+|y-X(s)|^2)^{-m}\, dy
\leq &\; t^{-m+1} \int_{B\big(0,\f{3r}{\sqrt{t}}\big)} (1+|y|^2)^{-m}\, dy\\
\leq &\; Ct^{-m+1} \min\left(\f{r^2}{t},\,1\right)
\leq Ct^{-m+1} \left(\f{t}{r^2}+1\right)^{-1}\\
= &\; Cr^2 \left(\f{t+r^2}{t}\right)^{m-1}(t+r^2)^{-m}\\
\leq &\; Cr^2 \left(1+\f{r^2}{t}\right)^{m-1}\big(t+r^2+|x-X(s)|^2\big)^{-m}.
\end{align*}
Here $C$ depends on $m$.
\end{proof}
\end{lem}

\begin{proof}[Proof of Lemma \ref{lem: estimate for u_X^nu-u_Y^nu in Morrey space}]

Recall that $\mathcal{D}[\Delta e^{t\Delta}u]_{X}[Z]$ was defined in \eqref{eqn: def of first variation of Delta etD u}.
We first show that, if $X\in O^\lam$ and $Z\in C^1(\BT)$, for arbitrary $x\in \BR^2$ and $r>0$,
\beq
\begin{split}
&\;\int_{B(x,r)}|\CD[\Delta e^{t\D}u]_{X}[Z](y)|\, dy\\
\leq &\; Cr^2\lam^{-\g} t^{-(2-\g)/2}
(t+r^2)^{-1/2} \cdot \lam^{-1} \|X'\|_{L_s^\infty}
(\|X'\|_{C_s^\g}\|Z\|_{L_s^\infty} +\|X'\|_{L_s^\infty}\|Z'\|_{L_s^\infty}),
\end{split}
\label{eqn: local integral of first variation of Laplace of heat kernel applied to u_X}
\eeq
where the constant $C$ depends on $\g$.

We still split $\CD[\Delta e^{t\D}u]_{X}[Z]$ as in the proof of Lemma \ref{lem: first variation of Laplace of heat kernel applied to u_X}.
By \eqref{eqn: formula for I_1}, Lemma \ref{lem: estimate for K}, Lemma \ref{lem: a typical integral in x variable}, and also Lemma \ref{lem: a typical integral},
\begin{align*}
\int_{B(x,r)}|I_1[X,Z](y)|\, dy
\leq &\; Cr^2(1+r^2/t)^{1/2}\int_{\T}\f{|X'(s)||Z'(s)|}{(t+r^2+|x-X(s)|^2)^{3/2}}\,ds\\
\leq &\; Cr^2(1+r^2/t)^{1/2}\cdot \min(\lam^{-1}(t+r^2)^{-1},(t+r^2)^{-3/2}) \|X'\|_{L_s^{\infty}}\|Z'\|_{L_s^\infty}\\
\leq &\; Cr^2(1+r^2/t)^{1/2}\cdot \lam^{-\g}(t+r^2)^{-(3-\g)/2}\|X'\|_{L_s^{\infty}}\|Z'\|_{L_s^{\infty}},
\end{align*}
where $C$ is a universal constant.
Similarly,
\[
\int_{B(x,r)}|I_2[X,Z](y)|\,dy
\leq
Cr^2(1+r^2/t)^{1/2}\cdot \lam^{-\g}(t+r^2)^{-(3-\g)/2}\|X'\|_{L_s^{\infty}}\|Z'\|_{L_s^{\infty}}.
\]
For $I_3$, by \eqref{eqn: formula for I_3 new} and Lemma \ref{lem: lower bound for x-X(s)},
\begin{align*}
|I_3[X,Z](y)|
\leq &\;  C\int_{\T}\f{\|X'\|_{L_s^\infty}\|X' \otimes Z\|_{\dot{C}_s^\g}|s-s_y|^\g}{(t+|y-X(s)|^2)^2}\,ds\\
\leq &\;
C\lam^{-\g}
\big(\|X'\|_{L_s^\infty}\|X'\|_{C_s^\g}\|Z\|_{L_s^\infty} + \|X'\|_{L_s^\infty}^2 \|Z\|_{\dot{C}_s^\g} \big) \int_{\T}\f{|y-X(s)|^{\g}}{(t+|y-X(s)|^2)^{2}}\,ds.
\end{align*}
Further integrating in $y$ and applying Lemma \ref{lem: a typical integral} and Lemma \ref{lem: a typical integral in x variable} gives
\begin{align*}
&\;\int_{B(x,r)}|I_3[X,Z](y)|\, dy  \\
\leq &\; Cr^2(1+r^2/t)^{(2-\g)/2}
\int_{\T}\f{1}{(t+r^2+|x-X(s)|^2)^{2-\g/2}}\,ds\\
&\;\cdot \lam^{-\g}\big(\|X'\|_{L_s^\infty}\|X'\|_{C_s^\g}\|Z\|_{L_s^\infty} + \|X'\|_{L_s^\infty}^2 \|Z\|_{\dot{C}_s^\g} \big)\\
\leq &\; Cr^2\lam^{-1-\g} (1+r^2/t)^{(2-\g)/2}\cdot (t+r^2)^{-(3-\g)/2}\big(\|X'\|_{L_s^\infty}\|X'\|_{C_s^\g}\|Z\|_{L_s^\infty} + \|X'\|_{L_s^\infty}^2 \|Z\|_{\dot{C}_s^\g} \big),
\end{align*}
where the constant $C$ depends on $\g$ only.
Summarizing the above estimates and using the fact that $\lam \leq \|X'\|_{L^\infty}$, we obtain \eqref{eqn: local integral of first variation of Laplace of heat kernel applied to u_X}.
This and \eqref{eqn: Laplace of heat kernel applied to u_X-u_Y} further imply that, if $X,Y\in O^\lam$ satisfy $\|X'-Y'\|_{L^{\infty}}\leq \lam/2$, for any $x\in \BR^2$ and $r>0$,
\begin{align*}
\|\Delta e^{t\Delta}(u_{X}-u_Y)\|_{L^1(B(x,r))}
\leq &\; Cr^2\lam^{-\g}t^{\g/2-1} (t+r^2)^{-1/2} \cdot \lam^{-1} \|(X',Y')\|_{L_s^{\infty}}\\
&\;\cdot  \big(\|(X',Y')\|_{C_s^\g} \|X-Y\|_{L_s^\infty}
+\|(X',Y')\|_{L_s^\infty} \|X'-Y'\|_{L_s^{\infty}}\big),
\end{align*}
where $C$ depends only on $\g$.

By \eqref{eqn: u_X^nu in terms of u_X} and the assumption that $X(s,0) = Y(s,0)$,
\begin{align*}
&\; \|u_X^{\nu}(t)-u_Y^\nu(t)\|_{L^1(B(x,r))}\\
\leq &\; \nu\int_{0}^t \|\Delta e^{\nu (t-\tau)\Delta}(u_{X(\tau)}-u_{Y(\tau)})\|_{L^1(B(x,r))}\,d\tau\\
\leq &\; C\nu\int_{0}^t r^2\lam^{-\g} (\nu (t-\tau))^{\g/2-1}(\nu (t-\tau)+r^2)^{-1/2}\cdot \lam^{-1}\|(X',Y')\|_{L^\infty_t L_s^{\infty}}\\
&\;\qquad \cdot \big(\|(X'(\tau),Y'(\tau))\|_{\dot{C}_s^\g}\|X(\tau)-Y(\tau)\|_{L_s^\infty}
+ \|(X',Y')\|_{L^\infty_t L_s^{\infty}} \|X'-Y'\|_{L^\infty_t L_s^\infty}\big)\,d\tau\\
\leq &\; C\nu r^2\lam^{-1-\g} \|(X',Y')\|_{L^\infty_t L_s^{\infty}} \int_{0}^t (\nu (t-\tau))^{\g/2-1}(\nu (t-\tau)+r^2)^{-1/2} \tau^{-\g} \,d\tau\\
&\;\quad \cdot \sup_{\eta\in(0,t]}\eta^\g
\|(X'(\eta),Y'(\eta))\|_{\dot{C}_s^\g}
\cdot  \sup_{\eta\in (0,t)} \eta^{\mu_0}\|\pa_t (X-Y)(\eta)\|_{L_s^\infty}
\int_0^t \tau^{-\mu_0}\,d\tau
\\
&\; + C\nu r^2\lam^{-1-\g} \|(X',Y')\|_{L^\infty_t L_s^{\infty}}^2 \|X'-Y'\|_{L^\infty_t L^\infty_s} \int_{0}^t (\nu (t-\tau))^{\g/2-1}(\nu (t-\tau)+r^2)^{-1/2} \,d\tau.
\end{align*}
Since
\begin{align*}
&\; \int_0^t (\nu (t-\tau))^{\g/2-1}(\nu (t-\tau)+r^2)^{-1/2}\,d\tau\\
= &\; \int_0^t (\nu \tau)^{\g/2-1}(\nu \tau+r^2)^{-1/2}\,d\tau\\
= &\; \nu^{-1} r^{\g-1}\int_0^{\nu t/r^2} s^{\g/2-1}(s+1)^{-1/2}\,d s
\leq C\nu^{-1} r^{\g-1},
\end{align*}
and
\begin{align*}
&\;\int_{0}^t (\nu (t-\tau))^{\g/2-1}(\nu (t-\tau)+r^2)^{-1/2} \tau^{-\g} \,d\tau\\
\leq &\;C \int_0^{t/2} (\nu t)^{\g/2-1}(\nu t+r^2)^{-1/2}\tau^{-\g} \,d\tau\\
&\;+ C\int_{t/2}^t (\nu (t-\tau))^{\g/2-1}(\nu (t-\tau)+r^2)^{-1/2} t^{-\g} \,d\tau
\\
\leq &\;C (\nu t)^{\g/2-1}(\nu t+r^2)^{-1/2} t^{1-\g}\\
&\;+ Ct^{-\g} \int_0^t (\nu (t-\tau))^{\g/2-1}(\nu (t-\tau)+r^2)^{-1/2}\,d\tau
\\
\leq &\; Ct^{-\g}  \nu^{-1} r^{\g-1},
\end{align*}
we arrive at the estimate
\begin{align*}
&\; \|u_X^{\nu}(t)-u_Y^\nu(t)\|_{L^1(B(x,r))}\\
\leq &\; C r^{1+\g} \lam^{-1-\g} \|(X',Y')\|_{L^\infty_t L_s^{\infty}} \\
&\; \cdot \Big[ t^{1-\g-\mu_0} \sup_{\eta\in(0,t]}\eta^\g
\|(X'(\eta),Y'(\eta))\|_{\dot{C}_s^\g} \cdot \sup_{\eta\in (0,t)} \eta^{\mu_0}\|\pa_t (X-Y)(\eta)\|_{L_s^\infty}\\
&\;\quad
+ \|(X',Y')\|_{L^\infty_t L_s^{\infty}} \|X'-Y'\|_{L^\infty_t L^\infty_s}\Big],
\end{align*}
where the constant $C$ depends on $\g$ and $\mu_0$.
This together with the definition of the Morrey norm in \eqref{eqn: def of Morrey norm} implies the desired claim.
\end{proof}

\subsection{Estimates for $h_X^\nu$ and $H_X^\nu$}

\begin{proof}[Proof of Lemma \ref{lem: estimate for h_X^nu}]

Recall the Gagliardo-Nirenberg interpolation inequality:
\beq
\|X'-Y'\|_{L_s^\infty(\BT)}
\leq C\|X-Y\|_{\dot{C}_s^\g(\BT)}^{\g}\|X'-Y'\|_{C_s^\g(\BT)}^{1-\g}
\leq C\|X-Y\|_{\dot{C}_s^\g(\BT)}^{\g}\|(X',Y')\|_{C_s^\g(\BT)}^{1-\g},
\label{eqn: G-N inequality}
\eeq
where $C$ depends on $\g$.

By Lemma \ref{lem: estimate for Laplace of heat operator applied to u_X-u_Y}, the interpolation inequality \eqref{eqn: G-N inequality}, with $\g'\in (0,\g]$ and $\mu_0,\mu,\mu'\in [0,1)$, and $k \in \BN$,
\begin{align*}
&\;\big\|\na^k \D e^{\nu(t-\tau)\D}(u_{X(\tau)}-u_{X(t)})\big\|_{L_x^{\infty}(\R^2)}
\\
\leq &\; C\min\left(\lam^{-1}(\nu(t-\tau))^{-\f{2+k}2}, \, (\nu(t-\tau))^{-\f{3+k}{2}}\right) \\
&\;\cdot \|X'\|_{L^\infty_t L_s^\infty} \|(X'(\tau),X'(t))\|_{C_s^\g }^{1-\g} \|X(\tau)-X(t)\|_{\dot{C}_s^\g }^{\g} \\
&\; + C\min\left(\lam^{-1-\g'}(\nu(t-\tau))^{-\f{3+k-\g'}{2}},\, \lam^{-1}(\nu(t-\tau))^{-\f{3+k}{2}}\right) \\
&\;\quad \cdot \big(\|X'\|_{L^\infty_t L_s^\infty} \|(X'(\tau),X'(t))\|_{\dot{C}_s^{\g'}} \|X(\tau)-X(t)\|_{L_s^\infty} + \|X'\|_{L^\infty_t L_s^\infty}^2 \|X(\tau)-X(t)\|_{\dot{C}_s^{\g'}}\big)
\\
\leq &\; C\min\left(\lam^{-1}(\nu(t-\tau))^{-\f{2+k}2}, \, (\nu(t-\tau))^{-\f{3+k}{2}}\right) \\
&\;\cdot \|X'\|_{L^\infty_t L_s^\infty} \|(X'(\tau),X'(t))\|_{C_s^\g }^{1-\g} \left(\int_\tau^t \eta^{-\mu} \cdot \eta^\mu\|\pa_t X(\eta)\|_{\dot{C}_s^\g} \,d\eta\right)^{\g} \\
&\; + C\min\left(\lam^{-1-\g'}(\nu(t-\tau))^{-\f{3+k-\g'}{2}},\, \lam^{-1}(\nu(t-\tau))^{-\f{3+k}{2}}\right) \\
&\;\quad \cdot \left(\|X'\|_{L^\infty_t L_s^\infty} \|(X'(\tau),X'(t))\|_{\dot{C}_s^{\g'}}
\int_\tau^t \eta^{-\mu_0} \cdot \eta^{\mu_0}\|\pa_t X(\eta)\|_{L_s^\infty} \,d\eta \right.\\
&\;\qquad \quad\left.+ \|X'\|_{L^\infty_t L_s^\infty}^2 \int_\tau^t \eta^{-\mu'} \cdot \eta^{\mu'}\|\pa_t X(\eta)\|_{\dot{C}_s^{\g'}} \,d\eta \right).
\end{align*}
Observe that, for $\b\in [0,1)$,
\beq
\begin{split}
\int_{\tau}^t \zeta^{-\b} \,d\zeta
\leq &\;
\begin{cases}
(t-\tau)\tau^{-\b}, & \mbox{if } \tau \in [\f{t}{2},t), \\
Ct^{1-\b}, & \mbox{if } \tau \in(0, \f{t}{2}],
\end{cases}\\
\leq &\; C(t-\tau) t^{-\b},
\end{split}
\label{eqn: improved inequality regarding time difference}
\eeq
where $C$ depends on $\b$;
yet, we note that if $\b\in [0,\f12]$, the constant $C$ can be made independent of $\b$.
Hence,
\beq
\begin{split}
&\;\big\|\na^k \D e^{\nu(t-\tau)\D}(u_{X(\tau)}-u_{X(t)})\big\|_{L_x^{\infty}(\R^2)}
\\
\leq &\; C\min\left(\lam^{-1}, \, (\nu(t-\tau))^{-\f{1}{2}}\right) (\nu(t-\tau))^{-\f{2+k}2} \|X'\|_{L^\infty_t L_s^\infty} \\
&\;\cdot \left(\tau^{-\g}\sup_{\eta\in (0,t]}\eta^\g \|X'(\eta)\|_{C_s^\g}\right)^{1-\g} \left(|t-\tau|t^{-\mu}\sup_{\eta\in (0,t]}\eta^{\mu}\|\pa_t X(\eta)\|_{\dot{C}_s^{\g}}\right)^{\g} \\
&\; + C\min\left(\lam^{-1-\g'}(\nu(t-\tau))^{-\f{3+k-\g'}{2}},\, \lam^{-1}(\nu(t-\tau))^{-\f{3+k}{2}}\right)\\
&\;\quad \cdot \left(\|X'\|_{L^\infty_t L_s^\infty}
\cdot \tau^{-\g'}\sup_{\eta\in (0,t]}\eta^{\g'} \|X'(\eta)\|_{C_s^{\g'}}\cdot |t-\tau|t^{-\mu_0}\sup_{\eta\in (0,t]}\eta^{\mu_0}\|\pa_t X(\eta)\|_{L_s^\infty} \right.\\
&\;\qquad \quad \left.+ \|X'\|_{L^\infty_t L_s^\infty}^2  \cdot |t-\tau|t^{-\mu'}\sup_{\eta\in (0,t]}\eta^{\mu'}\|\pa_t X(\eta)\|_{\dot{C}_s^{\g'}} \right),
\end{split}
\label{eqn: bound for grad D e tD u_X(tau)-u_X(t)}
\eeq
where the constant $C$ depends on $k$, $\g$, $\g'$, $\mu_0$, $\mu$, and $\mu'$.

On one hand, by \eqref{eqn: def of h_X}, we directly integrate \eqref{eqn: bound for grad D e tD u_X(tau)-u_X(t)} to find that, for $k = 0,1$,
\begin{align*}
&\; \|\na^k h_X^\nu(t)\|_{L_x^\infty(\BR^2)}\\
\leq &\;\nu\int_0^t\|\na^k \D e^{\nu (t-\tau)\D}(u_{X(\tau)}-u_{X(t)})\|_{L_x^{\infty}(\R^2)}\,d\tau
\\
\leq &\; C\nu^{-\f{k}{2}} t^{-\mu\g} \int_0^t \min\left(\lam^{-1}, \, (\nu(t-\tau))^{-\f{1}{2}}\right) (t-\tau)^{-\f{2+k}2+\g} \tau^{-\g(1-\g)}\,d\tau \\
&\;\cdot \|X'\|_{L^\infty_t L_s^\infty}  \left(\sup_{\eta\in (0,t]}\eta^\g \|X'(\eta)\|_{C_s^\g}\right)^{1-\g} \left(\sup_{\eta\in (0,t]}\eta^{\mu}\|\pa_t X(\eta)\|_{\dot{C}_s^{\g}}\right)^{\g} \\
&\; + C \int_0^t \min\left(\lam^{-\g'},\, (\nu(t-\tau))^{-\f{\g'}{2}}\right) (\nu(t-\tau))^{-\f{1+k-\g'}{2}} \\
&\;\qquad \quad \cdot \lam^{-1} \left(\tau^{-\g'}t^{-\mu_0} \|X'\|_{L^\infty_t L_s^\infty}
\sup_{\eta\in (0,t]}\eta^{\g'} \|X'(\eta)\|_{C_s^{\g'}} \sup_{\eta\in (0,t]}\eta^{\mu_0}\|\pa_t X(\eta)\|_{L_s^\infty} \right.\\
&\;\qquad \qquad \qquad \left.+ t^{-\mu'}\|X'\|_{L^\infty_t L_s^\infty}^2 \sup_{\eta\in (0,t]}\eta^{\mu'}\|\pa_t X(\eta)\|_{\dot{C}_s^{\g'}} \right) d\tau,
\end{align*}
where $C$ depends on $k$, $\g$, $\g'$, $\mu_0$, $\mu$, and $\mu'$.

It remains to calculate the two integrals above.
Recall that $\g\in (\f{k}2,1)$.
If $\nu t \leq 4\lam^2$, we find that
\begin{align*}
&\; \nu^{-\f{k}{2}} t^{-\mu\g} \int_0^t \min\left(\lam^{-1}, \, (\nu(t-\tau))^{-\f{1}{2}}\right) (t-\tau)^{-\f{2+k}2+\g} \tau^{-\g(1-\g)}\,d\tau \\
\leq &\; C\nu^{-\f{k}{2}} t^{-\mu\g}
\int_0^t \lam^{-1} (t-\tau)^{-\f{2+k}2+\g} \tau^{-\g(1-\g)}\,d\tau \\
\leq &\; C\nu^{-\f{k}{2}}t^{-\mu\g} \lam^{-1} t^{-\f{k}2+\g^2}
= C(\nu t)^{-\f{k}{2}} \lam^{-1}  t^{-\g(\mu -\g)},
\end{align*}
where $C$ depends on $k$ and $\g$.
Otherwise, if $\nu t>4\lam^2$,
\begin{align*}
&\; \nu^{-\f{k}{2}} t^{-\mu\g} \int_0^t \min\left(\lam^{-1}, \, (\nu(t-\tau))^{-\f{1}{2}}\right) (t-\tau)^{-\f{2+k}2+\g} \tau^{-\g(1-\g)}\,d\tau
\\
\leq &\; C\nu^{-\f{k}{2}} t^{-\mu\g} \left[\nu^{-\f12}\int_0^{t/2} t^{-\f{3+k}2+\g} \tau^{-\g(1-\g)}\,d\tau
+ \nu^{-\f12}\int_{t/2}^{t-\lam^2/\nu} (t-\tau)^{-\f{3+k}2+\g} t^{-\g(1-\g)}\,d\tau \right]\\
&\; + C\nu^{-\f{k}{2}} t^{-\mu\g} \int_{t-\lam^2/\nu}^t \lam^{-1} (t-\tau)^{-\f{2+k}2+\g} t^{-\g(1-\g)}\,d\tau
\\
\leq &\; C\nu^{-\f{1+k}{2}} t^{-\mu\g} \left[ t^{-\f{1+k}2+\g^2}
+ t^{-\g(1-\g)} \left(t^{-\f{1+k}2+\g} + \left(\f{\lam^2}{\nu}\right)^{-\f{1+k}2+\g} \right)\right]\\
&\; + C\nu^{-\f{k}{2}} \lam^{-1}  t^{-\mu\g-\g(1-\g)} \left(\f{\lam^2}{\nu}\right)^{-\f{k}2+\g}
\\
\leq &\; C\left[(\nu t)^{-\f{1+k}{2}}
+ (\nu t)^{-\g} \lam^{-1-k+2\g} \right] t^{-\g(\mu-\g)}.
\end{align*}
Here we used the facts that $\g\in (\f{k}2,1)$ and
\[
\int_{t/2}^{t-\lam^2/\nu} (t-\tau)^{-\f{3+k}2+\g} \,d\tau
\leq C\left[t^{-\f{1+k}2+\g} + \left(\f{\lam^2}{\nu}\right)^{-\f{1+k}2+\g} \right].
\]
In summary, for $k = 0,1$ and $\g\in (\f{k}{2},1)$,
\begin{align*}
&\; \nu^{-\f{k}{2}} t^{-\mu\g} \int_0^t \min\left(\lam^{-1}, \, (\nu(t-\tau))^{-\f{1}{2}}\right) (t-\tau)^{-\f{2+k}2+\g} \tau^{-\g(1-\g)}\,d\tau \\
\leq &\; C\left[ (\nu t)^{-\f{k}{2}} \min\left(\lam^{-1},\,(\nu t)^{-\f{1}{2}}\right)
+ \mathds{1}_{\{\nu t\geq 4\lam^2\}} (\nu t)^{-\g} \lam^{-1-k+2\g}  \right] t^{-\g(\mu-\g)}.
\end{align*}
where the constant $C$ depends on $k$ and $\g$.
One can similarly calculate that, with arbitrary $\b\in [0,1)$,
\begin{align*}
&\; \int_0^t \min\left(\lam^{-\g'},\, (\nu(t-\tau))^{-\f{\g'}{2}}\right) (\nu(t-\tau))^{-\f{1+k-\g'}{2}} \tau^{-\b}\,d\tau
\\
\leq &\; C (\nu t)^{-\f{1+k-\g'}{2}} \min\left(\lam^{-\g'},\, (\nu t)^{-\f{\g'}{2}}\right)
\cdot t^{1-\b} \left[1+\mathds{1}_{\{\nu t\geq 4\lam^2\}}\int_{2\lam^2/(\nu t)}^1 \zeta^{-\f{1+k}{2}} \,d\zeta \right],
\end{align*}
where the constant $C$ depends on $k$ and $\g'$.
Substituting these results into the above estimate, we obtain the desired bound.

On the other hand, by \eqref{eqn: bound for grad D e tD u_X(tau)-u_X(t)}, for $\al\in (0,1)\cup (1,1+\g']$ satisfying that $\al\leq 2\g$, and any $k\in \BN$,
\begin{align*}
&\;(t-\tau)^{\f{2+k-\al}{2}}
\big\|\na^k \D e^{\nu(t-\tau)\D}(u_{X(\tau)}-u_{X(t)})\big\|_{L_x^{\infty}(\R^2)}
\\
\leq &\; C\min\left(\lam^{-1},(\nu (t-\tau))^{-\f12}\right) \nu^{-\f{2+k}2}(t-\tau)^{\g-\f{\al}{2}} \tau^{-\g(1-\g)}t^{-\mu\g} \\
&\;\cdot  \|X'\|_{L^\infty_t L_s^\infty} \left(\sup_{\eta\in (0,t]}\eta^\g \|X'(\eta)\|_{C_s^\g}\right)^{1-\g} \left(\sup_{\eta\in (0,t]}\eta^{\mu}\|\pa_t X(\eta)\|_{\dot{C}_s^{\g}}\right)^{\g} \\
&\; + C \min\left(\lam^{-1-\g'}(\nu(t-\tau))^{\f{1+\g'-\al}{2}},\, \lam^{-1}(\nu(t-\tau))^{\f{1-\al}{2}}\right)
\nu^{-\f{4+k-\al}{2}}
\\
&\;\quad \cdot \left(\tau^{-\g'}t^{-\mu_0} \|X'\|_{L^\infty_t L_s^\infty}
\sup_{\eta\in (0,t]}\eta^{\g'} \|X'(\eta)\|_{C_s^{\g'}} \sup_{\eta\in (0,t]}\eta^{\mu_0}\|\pa_t X(\eta)\|_{L_s^\infty} \right.\\
&\;\qquad \quad \left.+ t^{-\mu'} \|X'\|_{L^\infty_t L_s^\infty}^2 \sup_{\eta\in (0,t]}\eta^{\mu'}\|\pa_t X(\eta)\|_{\dot{C}_s^{\g'}} \right).
\end{align*}
Let $[\al]$ denote the largest integer that is no greater than $\al$.
We shall apply Lemma \ref{lem: interpolation} with $\va(\tau):= \tau^{-\g}$ to \eqref{eqn: def of h_X}.
Since $\g \geq \f{\al}{2}$, $\g\geq \g'$, and $\al-[\al]\in (0,1)$, we find that
\begin{align*}
&\;\|\na^{[\al]} h_X^{\nu}(t)\|_{\dot{C}_x^{\al-[\al]}(\R^2)}\\
\leq
&\; C\nu t^{-\g} \left|\sup_{\tau\in (0,t)}(t-\tau)^{\f{2+[\al]-\al}{2}}\tau^{\g}\|\na^{[\al]}\D e^{\nu (t-\tau)\Delta}(u_{X(\tau)}-u_{X(t)})\|_{L_x^{\infty}(\R^2)}\right|^{1-(\al-[\al])}\\
&\; \cdot \left|\sup_{\tau \in (0,t)}(t-\tau)^{\f{3+[\al]-\al}2}\tau^{\g} \|\na^{1+[\al]} \D e^{\nu (t-\tau)\Delta}(u_{X(\tau)}-u_{X(t)})\|_{L_x^{\infty}(\R^2)}\right|^{\al-[\al]}
\\
\leq
&\; C\nu^{-\g}t^{-\g(1-\g+\mu)}\sup_{\tau\in (0,t)}\left[ \min\left(\lam^{-1},(\nu (t-\tau))^{-\f12}\right) (\nu(t-\tau))^{\g-\f{\al}{2}}\right]
\\
&\;\cdot \|X'\|_{L^\infty_t L_s^\infty} \left(\sup_{\eta\in (0,t]}\eta^\g \|X'(\eta)\|_{C_s^\g}\right)^{1-\g} \left(\sup_{\eta\in (0,t]}\eta^{\mu}\|\pa_t X(\eta)\|_{\dot{C}_s^{\g}}\right)^{\g} \\
&\; + C\nu^{-1}t^{-\g} \sup_{\tau\in (0,t)}\left[
\tau^{\g}\min\left(\lam^{-1-\g'}(\nu(t-\tau))^{\f{1+\g'-\al}{2}},\, \lam^{-1}(\nu(t-\tau))^{\f{1-\al}{2}}\right) \tau^{-\g'}t^{-\mu_0}\right]
\\
&\;\quad \cdot \|X'\|_{L^\infty_t L_s^\infty}
\sup_{\eta\in (0,t]}\eta^{\g'} \|X'(\eta)\|_{C_s^{\g'}} \sup_{\eta\in (0,t]}\eta^{\mu_0}\|\pa_t X(\eta)\|_{L_s^\infty}\\
&\; + C\nu^{-1}t^{-\g} \sup_{\tau\in (0,t)}\left[
\tau^{\g}\min\left(\lam^{-1-\g'}(\nu(t-\tau))^{\f{1+\g'-\al}{2}},\, \lam^{-1}(\nu(t-\tau))^{\f{1-\al}{2}}\right) t^{-\mu'}\right]
\\
&\;\quad \cdot \|X'\|_{L^\infty_t L_s^\infty}^2 \sup_{\eta\in (0,t]}\eta^{\mu'}\|\pa_t X(\eta)\|_{\dot{C}_s^{\g'}}
\\
\leq
&\; C \left[(\nu t)^{-\f{\al}{2}}\min\left(\lam^{-1},\,(\nu t)^{-\f{1}{2}}\right)
+ \mathds{1}_{\{\nu t\geq 4\lam^2\}} (\nu t)^{-\g}\lam^{-1-\al+2\g}\right]
t^{-\g(\mu-\g)}
\\
&\;\cdot \|X'\|_{L^\infty_t L_s^\infty} \left(\sup_{\eta\in (0,t]}\eta^\g \|X'(\eta)\|_{C_s^\g}\right)^{1-\g} \left(\sup_{\eta\in (0,t]}\eta^{\mu}\|\pa_t X(\eta)\|_{\dot{C}_s^{\g}}\right)^{\g} \\
&\; + C\left[ \min\left(\lam^{-\g'},\, (\nu t)^{-\f{\g'}{2}} \right)
(\nu t)^{-\f{1+\al-\g'}{2}}
\lam^{-1}
+ \mathds{1}_{\{\nu t\geq 4\lam^2\}} (\nu t)^{-1}  \lam^{-\al} \right]
\\
&\;\quad \cdot \left(t^{1-\g'-\mu_0}\|X'\|_{L^\infty_t L_s^\infty}
\sup_{\eta\in (0,t]}\eta^{\g'} \|X'(\eta)\|_{C_s^{\g'}} \sup_{\eta\in (0,t]}\eta^{\mu_0}\|\pa_t X(\eta)\|_{L_s^\infty} \right.\\
&\;\qquad \quad \left.+ t^{1-\mu'}\|X'\|_{L^\infty_t L_s^\infty}^2 \sup_{\eta\in (0,t]}\eta^{\mu'}\|\pa_t X(\eta)\|_{\dot{C}_s^{\g'}} \right).
\end{align*}
Here we used the fact that, since $\g\geq \f{\al}{2}$,
\begin{align*}
\sup_{\eta\in (0,t)}\left[\min\left(\lam^{-1}, \, (\nu\eta)^{-\f{1}{2}}\right) (\nu \eta)^{\g-\f{\al}{2}}\right]
\leq &\;
\begin{cases}
\lam^{-1} (\nu t)^{\g-\f{\al}{2}}, & \mbox{if } \nu t \leq \lam^2, \\
(\nu t)^{\g-\f{1+\al}{2}}, & \mbox{if } \nu t \geq \lam^2\mbox{ and }\g\geq \f{1+\al}{2},\\
\lam^{-1+2\g-\al}, & \mbox{if } \nu t \geq \lam^2\mbox{ and }\g\leq \f{1+\al}{2},
\end{cases}\\
\leq &\;
C(\nu t)^{\g-\f{\al}{2}}\min\left(\lam^{-1},\,(\nu t)^{-\f{1}{2}}\right)
+ C\mathds{1}_{\{\nu t\geq 4\lam^2\}} \lam^{-1-\al+2\g},
\end{align*}
and since $\al\in (0,1)\cup (1,1+\g']$,
\begin{align*}
&\; \sup_{\tau\in (0,t)}\left[\min\left(\lam^{-1-\g'}(\nu(t-\tau))^{\f{1+\g'-\al}{2}},\, \lam^{-1}(\nu(t-\tau))^{\f{1-\al}{2}}\right)\right]\\
= &\; \lam^{-\al}\sup_{\tau\in (0,t)}\left[\min\left(\left(\f{\nu\tau}{\lam^2}\right)^{\f{1+\g'-\al}{2}},\, \left(\f{\nu\tau}{\lam^2}\right)^{\f{1-\al}{2}}\right)\right]\\
= &\;
\begin{cases}
\lam^{-\al} \left(\f{\nu t}{\lam^2}\right)^{\f{1+\g'-\al}{2}},
& \mbox{if } \nu t\leq \lam^2, \\
\lam^{-\al} \left(\f{\nu t}{\lam^2}\right)^{\f{1-\al}{2}},
& \mbox{if } \nu t\geq \lam^2\mbox{ and }\al\in (0,1), \\
\lam^{-\al},
& \mbox{if } \nu t\geq \lam^2\mbox{ and }\al\in (1,1+\g'],
\end{cases}
\\
\leq
&\; C \min\left(\lam^{-\g'},\, (\nu t)^{-\f{\g'}{2}}\right)
(\nu t)^{\f{1+\g'-\al}{2}} \lam^{-1}
+ C \mathds{1}_{\{\nu t\geq 4\lam^2\}} \lam^{-\al}.
\end{align*}
This proves \eqref{eqn: Holder estimate for h}.

The estimates for $H_X^\nu$ follows from Lemma \ref{lem: Holder estimate for composition of functions}.
We only note that, if $\al\in (1,2)$,
\begin{align*}
&\; \|H_X^{\nu}(t)\|_{\dot{C}^{1,\al-1}_s(\BT)}\\
= &\; \|\na h_X^{\nu}(X(\cdot,t),t)X'(\cdot,t)\|_{\dot{C}_s^{\al-1}(\BT)}\\
\leq &\; \|\na h_X^{\nu}(t)\|_{\dot{C}^{\al-1}_x(\BR^2)} \|X'(\cdot,t)\|_{L_s^\infty }^{\al-1} \cdot \|X'(\cdot,t)\|_{L_s^\infty}
+ \|\na h_X^{\nu}(t)\|_{L^\infty_x(\BR^2)} \|X'(\cdot,t)\|_{\dot{C}^{\al-1}_s}.
\end{align*}

\end{proof}

Next we turn to bound $H_X^{\nu}-H_Y^\nu$.
By definition,
\beq
\begin{split}
&\;H_X^{\nu}(t)-H_Y^\nu(t)\\
= &\; h_X^{\nu}(t)\circ X(t)-h_Y^\nu(t)\circ Y(t)\\
= &\; -\nu\int_0^t \big[\Delta e^{\nu (t-\tau)\Delta}(u_{X(\tau)}-u_{X(t)})\big]\circ X(t)
-\big[\Delta e^{\nu (t-\tau)\Delta}(u_{Y(\tau)}-u_{Y(t)})\big]\circ Y(t)\,d\tau.
\end{split}
\label{eqn: H_X - H_Y}
\eeq
We need to estimate the term of the form
\[
\D e^{t\D}(u_{{X+Z}}-u_{X})\circ X-\D e^{t\D}(u_{{Y+W}}-u_{Y})\circ Y.
\]
For that purpose, let us denote
\beq
J(t;X,Z)(s):=\big[\D e^{t\D}(u_{{X}+Z}-u_{X})\big]\circ X=\int_0^1 \CD[\D e^{t\D}u]_{X+\theta Z}[Z]\circ X\,d\theta.
\label{eqn: def of J}
\eeq
Then
\begin{align*}
\big[\Delta e^{t\Delta}(u_{{X+Z}}-u_{X})\big]\circ X
-\big[\Delta e^{t\Delta}(u_{{Y+W}}-u_{Y})\big]\circ Y
= &\; J(t;X,Z) - J(t;Y,W).
\end{align*}

\begin{lem}
\label{lem: bound for J YZ - J YW}
Suppose $X,Y\in O^\lam$, and $Z,W\in C^{1,\g}(\BT)$ satisfy that 
$\|Z'\|_{L_s^\infty(\BT)},\|W'\|_{L_s^\infty(\BT)}\leq \lam/2$.
Then
\begin{align*}
&\; \|J(t;X,Z)-J(t;Y,W)\|_{L_s^\infty(\T)}
+ t^{1/2}(\|X'\|_{L_s^\infty}+\|Y'\|_{L_s^\infty})^{-1}\|J(t;X,Z)-J(t;Y,W)\|_{\dot{C}_s^1(\T)}\\
\leq &\; C
\big(\min(\lam^{-3}t^{-1},\,\lam^{-2}t^{-3/2}) + \lam^{-3}t^{-2}\|Z\|_{L_s^\infty}^2 + \lam^{-3}t^{-3/2}\|Z\|_{L_s^\infty} \big)\\
&\;\quad \cdot
(\|X'\|_{L_s^\infty}+\|Y'\|_{L_s^\infty})^2 \|X'-Y'\|_{L_s^\infty} \|Z'\|_{L_s^\infty}\\
&\; + C \lam^{-1-\g}
\big( t^{-(3-\g)/2} +  t^{-3/2}\|Z\|_{L_s^\infty}^\g\big)
\|X'\|_{L_s^\infty} \|Z\|_{L_s^\infty} \|X'-Y'\|_{\dot{C}_s^\g}\\
&\; + C\lam^{-2-\g}
\big( t^{-(3-\g)/2} + t^{-2}\|Z\|_{L_s^\infty}^{1+\g}\big)
\|Y'\|_{L_s^\infty} \big(\|Y'\|_{\dot{C}_s^\g} + \|Z'\|_{\dot{C}_s^\g}\big) \|X'-Y'\|_{L_s^\infty} \|Z\|_{L_s^\infty}
\\
&\; + C\min(\lam^{-1}t^{-1},\, t^{-3/2}) \|Y'\|_{L_s^\infty} \|Z'-W'\|_{L_s^\infty}\\
&\; + C\min(\lam^{-1-\g}t^{-(3-\g)/2},\,\lam^{-1}t^{-3/2})
\|Y'\|_{L_s^\infty} \big(\|Y'\|_{\dot{C}_s^\g}+\|Z'\|_{\dot{C}_s^\g}+\|W'\|_{\dot{C}_s^\g}\big) \|Z-W\|_{L_s^\infty}\\
&\; + C\min(\lam^{-1-\g}t^{-(3-\g)/2},\,\lam^{-1}t^{-3/2}) \|Y'\|_{L_s^\infty}^2 \|Z-W\|_{C_s^\g},
\end{align*}
where the constant $C$ only depends on $\g$.

\begin{proof}
We start from writing
\beq
J(t;X,Z) - J(t;Y,W)
= \big(J(t;X,Z) - J(t;Y,Z)\big) + \big(J(t;Y,Z)-J(t;Y,W)\big).
\label{eqn: splitting J(X,Z)-J(Y,W)}
\eeq
The second term on the right-hand side can be written as
\[
J(t;Y,Z)-J(t;Y,W)
= \Delta e^{t\Delta}(u_{Y+Z}-u_{Y+W})\circ Y,
\]
so it can be readily bounded by applying Lemma \ref{lem: first variation of Laplace of heat kernel applied to u_X} (also see \eqref{eqn: Laplace of heat kernel applied to u_X-u_Y}).
Indeed,
\begin{align*}
&\; \|J(t;Y,Z)-J(t;Y,W) \|_{L_s^\infty(\BT)}
+t^{1/2} \|Y'\|_{L_s^\infty}^{-1} \|J(t;Y,Z)-J(t;Y,W) \|_{\dot{C}_s^1(\BT)}\\
\leq &\; \int_0^1 \|\CD[\Delta e^{t\Delta}u]_{Y+\th Z+(1-\th)W}[Z-W]\|_{L_x^\infty(\BR^2)} \, d\th\\
&\;+ t^{1/2} \|Y'\|_{L_s^\infty}^{-1}\cdot \|Y'\|_{L_s^\infty} \int_0^1 \|\na \CD[\Delta e^{t\Delta}u]_{Y+\th Z+(1-\th)W}[Z-W]\|_{L_x^\infty(\BR^2)} \, d\th\\
\leq &\; C\min(\lam^{-1}t^{-1},\, t^{-3/2})\big(\|Y'+Z'\|_{L_s^\infty}+\|Y'+W'\|_{L_s^\infty}\big) \|Z'-W'\|_{L_s^{\infty}}\\
&\; + C\min(\lam^{-1-\g}t^{-(3-\g)/2},\,\lam^{-1}t^{-3/2})
\big(\|Y'+Z'\|_{L_s^\infty}+\|Y'+W'\|_{L_s^\infty}\big) \\
&\;\quad \cdot \big(\|Y'+Z'\|_{\dot{C}_s^\g}+\|Y'+W'\|_{\dot{C}_s^\g}\big) \|Z-W\|_{L_s^\infty}\\
&\; + C\min(\lam^{-1-\g}t^{-(3-\g)/2},\,\lam^{-1}t^{-3/2})
\big(\|Y'+Z'\|_{L_s^\infty}+\|Y'+W'\|_{L_s^\infty}\big)^2\|Z-W\|_{C_s^\g},
\end{align*}
where the constant $C$ depends on $\g$.
Here we used the fact that, by assumption, $|Y+\th Z + (1-\th)W|_*\geq \lam/2$.
Since $\|Z'\|_{L_s^\infty},\|W'\|_{L_s^\infty}\leq \lam/2\leq \|Y'\|_{L_s^\infty}$, this further simplifies to
\beq
\begin{split}
&\; \|J(t;Y,Z)-J(t;Y,W) \|_{L_s^\infty(\BT)}
+t^{1/2} \|Y'\|_{L_s^\infty}^{-1} \|J(t;Y,Z)-J(t;Y,W) \|_{\dot{C}_s^1(\BT)}\\
\leq &\; C\min(\lam^{-1}t^{-1},\, t^{-3/2}) \|Y'\|_{L_s^\infty} \|Z'-W'\|_{L_s^\infty}\\
&\; + C\min(\lam^{-1-\g}t^{-(3-\g)/2},\,\lam^{-1}t^{-3/2})
\|Y'\|_{L_s^\infty} \big(\|Y'\|_{\dot{C}_s^\g}+\|Z'\|_{\dot{C}_s^\g}+\|W'\|_{\dot{C}_s^\g}\big) \|Z-W\|_{L_s^\infty}\\
&\; + C\min(\lam^{-1-\g}t^{-(3-\g)/2},\,\lam^{-1}t^{-3/2}) \|Y'\|_{L_s^\infty}^2 \|Z-W\|_{C_s^\g}.
\end{split}
\label{eqn: bound for J YZ - J YW}
\eeq

To bound the first term $(J(t;X,Z) - J(t;Y,Z))$ on the right-hand side of \eqref{eqn: splitting J(X,Z)-J(Y,W)}, we define $I_m[t;X,Z]$ $(m = 1,2,3)$ as in the proof of Lemma \ref{lem: first variation of Laplace of heat kernel applied to u_X} (see \eqref{eqn: formula for I_1}-\eqref{eqn: formula for I_3}), and further denote
\[
J_m(t,\theta;X,Z)(s): =I_m[t;X+\theta Z,Z]\circ X(s).
\]
Then
\begin{align*}
J(t;X,Z)-J(t;Y,Z)
= &\; \int_0^1
\CD[\Delta e^{t\Delta}u]_{{X}+\theta Z}[Z]\circ X - \CD[\Delta e^{t\Delta}u]_{Y+\theta Z}[Z]\circ Y\,d \th\\
= &\; \sum_{m = 1}^3 \int_0^1
J_m(t,\theta;X,Z) - J_m(t,\theta;Y,Z)\,d \th.
\end{align*}

By definition, with the $t$-dependence omitted,
\[
J_1(t,\theta;X,Z)(s)_i
=-\int_{\T}(X_l'(s')+\th Z_l'(s'))\pa_l K_{ij}(X(s)-X(s')-\th Z(s'))Z_j'(s')\,ds',
\]
so
\begin{align*}
&\;(J_1(t,\theta;X,Z)-J_1(t,\theta;Y,Z))(s)_i\\
=&\; -\int_{\T}(X_l'(s')- Y_l'(s'))\pa_l K_{ij}(X(s)-X(s')-\th Z(s')) Z_j'(s')\,ds'\\
&\;-\int_{\T}(Y_l'(s')+\th Z_l'(s'))
\pa_l K_{ij}(x-\th Z(s'))\big|_{x=Y(s)-Y(s')}^{X(s)-X(s')} Z_j'(s')\,ds'\\
=: &\; (J_{1,1}+J_{1,2})(t,\theta;X,Y,Z)(s)_i.
\end{align*}
By Lemma \ref{lem: estimate for K} and Lemma \ref{lem: a typical integral},
\begin{align*}
|J_{1,1}(t,\th;X,Y,Z)(s)|
\leq &\; C
\int_{\T}\f{\|X'-Y'\|_{L_s^{\infty}}\|Z'\|_{L_s^{\infty}}}{(t+|X(s)-X(s')-\th Z(s')|^2)^{3/2}}\,ds'\\
\leq &\; C\min(\lam^{-1}t^{-1},\,t^{-3/2})\|X'-Y'\|_{L_s^{\infty}}\|Z'\|_{L_s^{\infty}},
\end{align*}
where $C$ is a universal constant.
Here we used the fact that $|X+\th Z|_*\geq \lam/2$.

By Lemma \ref{lem: estimate for K},
\begin{align*}
&\;\Big|\pa_l K_{ij}(x-\th Z(s'))\big|_{x=Y(s)-Y(s')}^{X(s)-X(s')}\Big| \\
\leq &\; C|X(s)-X(s')-Y(s)+Y(s')|\\
&\;
\cdot \Big[\big(t+|X(s)-X(s')-\th Z(s')|^2\big)^{-2}+\big(t+|Y(s)-Y(s')-\th Z(s')|^2\big)^{-2}\Big],
\end{align*}
where we have
\[
|X(s)-X(s')-Y(s)+Y(s')| \leq\|X'-Y'\|_{L_s^{\infty}}|s-s'|.
\]
Hence,
\begin{align*}
&\;|J_{1,2}(t,\th;X,Y,Z)(s)|\\
\leq &\; C \|Y'+\th Z'\|_{L_s^\infty} \|Z'\|_{L_s^{\infty}}\\
&\;\cdot
\int_{\T}\f{\|X'-Y'\|_{L_s^{\infty}}|s-s'| }{(t+|X(s)-X(s')-\th Z(s')|^2)^{2}}
+\f{\|X'-Y'\|_{L_s^{\infty}}|s-s'|}{(t+|Y(s)-Y(s')-\th Z(s')|^2)^{2}}\,ds'.
\end{align*}
For future use, we calculate the following general integral with $\b\in [0,3)$ by Lemma \ref{lem: a typical integral}:
\beq
\begin{split}
&\; \int_{\T}\f{|s-s'|^\b}{(t+|X(s)-X(s')-\th Z(s')|^2)^{2}}\,ds'\\
\leq
&\; C\int_{\T}\f{\lam^{-\b}|X(s)+\th Z(s)-X(s')-\th Z(s')|^\b}{(t+|X(s)-X(s')-\th Z(s')|^2)^{2}}\,ds'\\
\leq
&\; C\lam^{-\b} \int_{\T}\f{1}{(t+|X(s)-X(s')-\th Z(s')|^2)^{2-\b/2}}
+ \f{(\th |Z(s)|)^\b}{(t+|X(s)-X(s')-\th Z(s')|^2)^{2}}\,ds'\\
\leq
&\; C \Big[\min(\lam^{-1-\b}t^{-(3-\b)/2},\,\lam^{-\b}t^{-(4-\b)/2}) + \lam^{-1-\b}t^{-3/2}\|Z\|_{L_s^\infty}^\b\Big],
\end{split}
\label{eqn: integral along a deformed curve}
\eeq
where the constant $C$ may depend on $\b\in [0,3)$.
Hence, with $\b = 1$,
\begin{align*}
&\;|J_{1,2}(t,\th;X,Y,Z)(s)|\\
\leq &\; C\big(\min(\lam^{-2}t^{-1},\,\lam^{-1}t^{-3/2}) + \lam^{-2}t^{-3/2}\|Z\|_{L_s^\infty}\big)
\|Y'\|_{L_s^\infty}\|X'-Y'\|_{L_s^{\infty}}\|Z'\|_{L_s^{\infty}}.
\end{align*}
Here we used the inequality $\|Z'\|_{L_s^\infty}\leq \|Y'\|_{L_s^\infty}$ as before to find that $\|Y'+\th Z'\|_{L_s^\infty} \leq C\|Y'\|_{L_s^\infty}$.
Therefore,
\beq
\begin{split}
&\; \|J_1(t,\theta;X,Z)-J_1(t,\theta;Y,Z)\|_{L_s^\infty(\BT)}\\
\leq &\;
C\big(\min(\lam^{-2}t^{-1},\,\lam^{-1}t^{-3/2}) + \lam^{-2}t^{-3/2}\|Z\|_{L_s^\infty}\big)
\|Y'\|_{L_s^\infty}\|X'-Y'\|_{L_s^{\infty}}\|Z'\|_{L_s^{\infty}},
\end{split}
\label{eqn: L inf estimate for difference of J_1}
\eeq
where $C$ is a universal constant.
Furthermore, from
\begin{align*}
\pa_s J_1(t,\theta;X,Z)(s)_i
=-\int_{\T}(X_l'(s')+\th Z_l'(s'))X_m'(s)\pa_l\pa_m K_{ij}(X(s)-X(s')-\th Z(s'))Z_j'(s')\,ds',
\end{align*}
one can similarly derive that
\beq
\begin{split}
&\; \|J_1(t,\th;X,Z)-J_1(t,\th;Y,Z)\|_{\dot{C}_s^1(\BT)}\\
\leq &\;
C\big(\min(\lam^{-2}t^{-3/2},\,\lam^{-1}t^{-2}) + \lam^{-2}t^{-2}\|Z\|_{L_s^\infty}\big)
(\|X'\|_{L_s^\infty}+\|Y'\|_{L_s^\infty})^2\|X'-Y'\|_{L_s^{\infty}}\|Z'\|_{L_s^{\infty}},
\end{split}
\label{eqn: C 1 estimate for difference of J_1}
\eeq
where $C$ is also universal.
We omit the details.

Since
\[
J_2(t,\th;X,Z)(s)_i
=
-\int_{\T}Z_l'(s')\pa_l K_{ij}(X(s)-X(s')-\th Z(s'))(X_j'(s')+\th Z_j'(s'))\,ds'
\]
has a similar form as $J_1$, the above estimates \eqref{eqn: L inf estimate for difference of J_1} and
\eqref{eqn: C 1 estimate for difference of J_1} also apply to $J_2$.

Finally (cf.\;\eqref{eqn: formula for I_3 new}),
\beqo
\begin{split}
&\; J_3(t,\theta;X,Z)(s)_i\\
= &\; \int_{\T}
Z_l(s')(X_m'(s')+\th Z_m'(s'))\pa_m\pa_lK_{ij}(X(s)-X(s')-\th Z(s')) (X_j'(s')+\theta Z_j'(s'))\,ds'\\
= &\; \int_{\T}
(X_m'(s')+\th Z_m'(s'))\pa_m\pa_lK_{ij}(X(s)-X(s')-\th Z(s')) \\
&\;\qquad \cdot \big[Z_l(s') (X_j'(s')+\theta Z_j'(s')) - Z_l(s) (X_j'(s)+\theta Z_j'(s))\big]\,ds'.
\end{split}
\eeqo
Using the last formula, we derive that
\begin{align*}
&\; (J_3(t,\theta;X,Z)-J_3(t,\theta;Y,Z))(s)_i\\
= &\; \int_{\T}
(X_m'(s')+\th Z_m'(s'))\pa_m\pa_lK_{ij}(X(s)-X(s')-\th Z(s')) \\
&\;\qquad \cdot \big[Z_l(s') (X_j'(s')-Y_j'(s')) - Z_l(s) (X_j'(s)-Y_j'(s))\big]\,ds'\\
&\; + \int_{\T}
(X_m'(s')-Y_m'(s'))\pa_m\pa_lK_{ij}(X(s)-X(s')-\th Z(s')) \\
&\;\qquad \quad \cdot \big[Z_l(s') (Y_j'(s')+\theta Z_j'(s')) - Z_l(s) (Y_j'(s)+\theta Z_j'(s))\big]\,ds'\\
&\; + \int_{\T}
(Y_m'(s')+\th Z_m'(s'))\pa_m\pa_lK_{ij}(x-\th Z(s'))\big|_{x = Y(s)-Y(s')}^{X(s)-X(s')} \\
&\;\qquad \quad \cdot \big[Z_l(s') (Y_j'(s')+\theta Z_j'(s')) - Z_l(s) (Y_j'(s)+\theta Z_j'(s))\big]\,ds'.
\end{align*}
Proceeding as before, we find that
\begin{align*}
&\; \|J_3(t,\theta;X,Z)-J_3(t,\theta;Y,Z)\|_{L_s^\infty(\BT)}\\
&\; + t^{1/2}(\|X'\|_{L_s^\infty}+\|Y'\|_{L_s^\infty})^{-1} \|J_3(t,\theta;X,Z)-J_3(t,\theta;Y,Z)\|_{\dot{C}_s^1(\BT)}\\
\leq &\; C\int_{\T}
\|X'\|_{L_s^\infty} \cdot \f{1}{(t+|X(s)-X(s')-\th Z(s')|^2)^2} \\
&\; \qquad \cdot \big[|s-s'|\|Z'\|_{L_s^\infty}\|X'-Y'\|_{L_s^\infty} + \|Z\|_{L_s^\infty}|s-s'|^\g \|X'-Y'\|_{\dot{C}_s^\g}\big]\,ds'\\
&\; + C\int_{\T}
\|X'-Y'\|_{L_s^\infty} \cdot \f{1}{(t+|X(s)-X(s')-\th Z(s')|^2)^2} \\
&\;\qquad \quad \cdot \big[|s-s'|\|Z'\|_{L_s^\infty}\|Y'\|_{L_s^\infty} + \|Z\|_{L_s^\infty}|s-s'|^\g \|Y'+\th Z'\|_{\dot{C}_s^\g}\big] \,ds'\\
&\; + C \int_{\T}
\|Y'\|_{L_s^\infty} \left(\f{|X(s)-X(s')-Y(s)+Y(s')|}{(t+|X(s)-X(s')-\th Z(s')|^2)^{5/2}}+\f{|X(s)-X(s')-Y(s)+Y(s')|}{(t+|Y(s)-Y(s')-\th Z(s')|^2)^{5/2}}\right)\\
&\;\qquad \quad \cdot \big[|s-s'|\|Z'\|_{L_s^\infty}\|Y'\|_{L_s^\infty} + \|Z\|_{L_s^\infty}|s-s'|^\g \|Y'+\th Z'\|_{\dot{C}_s^\g}\big] \,ds'\\
\leq &\; C(\|X'\|_{L_s^\infty}+\|Y'\|_{L_s^\infty}) \|X'-Y'\|_{L_s^\infty}\|Z'\|_{L_s^\infty}
\int_{\T} \f{|s-s'|}{(t+|X(s)-X(s')-\th Z(s')|^2)^2} \,ds'\\
&\; + C\|X'\|_{L_s^\infty} \|Z\|_{L_s^\infty} \|X'-Y'\|_{\dot{C}_s^\g}
\int_{\T} \f{|s-s'|^\g }{(t+|X(s)-X(s')-\th Z(s')|^2)^2} \,ds'\\
&\; + C\|X'-Y'\|_{L^\infty}\|Z\|_{L_s^\infty} \|Y'+\th Z'\|_{\dot{C}_s^\g}
\int_{\T} \f{|s-s'|^\g}{(t+|X(s)-X(s')-\th Z(s')|^2)^2} \,ds'\\
&\; + C\|Y'\|_{L_s^\infty}^2 \|X'-Y'\|_{L_s^\infty} \|Z'\|_{L_s^\infty}\\
&\;\quad \cdot \int_{\T}
\f{|s-s'|^2}{t^{1/2}(t+|X(s)-X(s')-\th Z(s')|^2)^2}
+\f{|s-s'|^2}{t^{1/2}(t+|Y(s)-Y(s')-\th Z(s')|^2)^2} \,ds'\\
&\; + C\|Y'\|_{L_s^\infty}\|Y'+\th Z'\|_{\dot{C}_s^\g} \|X'-Y'\|_{L_s^\infty} \|Z\|_{L_s^\infty}\\
&\;\quad \cdot \int_{\T}
\f{|s-s'|^{1+\g}}{t^{1/2}(t+|X(s)-X(s')-\th Z(s')|^2)^2}
+ \f{|s-s'|^{1+\g}}{t^{1/2}(t+|Y(s)-Y(s')-\th Z(s')|^2)^2} \,ds'.
\end{align*}
Here $C$ is a universal constant.
By virtue of \eqref{eqn: integral along a deformed curve},
\begin{align*}
&\; \|J_3(t,\theta;X,Z)-J_3(t,\theta;Y,Z)\|_{L_s^\infty(\BT)} \\
&\; + t^{1/2}(\|X'\|_{L_s^\infty}+\|Y'\|_{L_s^\infty})^{-1}\|J_3(t,\theta;X,Z)-J_3(t,\theta;Y,Z)\|_{\dot{C}_s^1(\BT)}\\
\leq &\; C
\big(\min(\lam^{-2}t^{-1},\,\lam^{-1}t^{-3/2}) + \lam^{-2}t^{-3/2}\|Z\|_{L_s^\infty}\big)
(\|X'\|_{L_s^\infty}+\|Y'\|_{L_s^\infty})\|X'-Y'\|_{L_s^\infty}\|Z'\|_{L_s^\infty}\\
&\; + C
\big( \lam^{-1-\g}t^{-(3-\g)/2} + \lam^{-1-\g}t^{-3/2}\|Z\|_{L_s^\infty}^\g\big)\\
&\;\quad \cdot
\big[\|X'\|_{L_s^\infty} \|Z\|_{L_s^\infty} \|X'-Y'\|_{\dot{C}_s^\g}
+ \|X'-Y'\|_{L_s^\infty}\|Z\|_{L_s^\infty} \|Y'+\th Z'\|_{\dot{C}_s^\g}\big]
\\
&\; + C
\big(\min(\lam^{-3}t^{-1},\,\lam^{-2}t^{-3/2}) + \lam^{-3}t^{-2}\|Z\|_{L_s^\infty}^2\big)
\|Y'\|_{L_s^\infty}^2 \|X'-Y'\|_{L_s^\infty} \|Z'\|_{L_s^\infty}\\
&\; + C
\big(\lam^{-2-\g}t^{-(3-\g)/2} + \lam^{-2-\g}t^{-2}\|Z\|_{L_s^\infty}^{1+\g}\big)
\|Y'\|_{L_s^\infty}\|Y'+\th Z'\|_{\dot{C}_s^\g} \|X'-Y'\|_{L_s^\infty} \|Z\|_{L_s^\infty}\\
\leq &\; C
\big(\min(\lam^{-3}t^{-1},\,\lam^{-2}t^{-3/2}) + \lam^{-3}t^{-2}\|Z\|_{L_s^\infty}^2 + \lam^{-3}t^{-3/2}\|Z\|_{L_s^\infty} \big)\\
&\;\quad \cdot
(\|X'\|_{L_s^\infty}+\|Y'\|_{L_s^\infty})^2 \|X'-Y'\|_{L_s^\infty} \|Z'\|_{L_s^\infty}\\
&\; + C \lam^{-1-\g}
\big( t^{-(3-\g)/2} +  t^{-3/2}\|Z\|_{L_s^\infty}^\g\big)
\|X'\|_{L_s^\infty} \|Z\|_{L_s^\infty} \|X'-Y'\|_{\dot{C}_s^\g}\\
&\; + C\lam^{-2-\g}
\big( t^{-(3-\g)/2} + t^{-2}\|Z\|_{L_s^\infty}^{1+\g}\big)
\|Y'\|_{L_s^\infty}\|Y'+\th Z'\|_{\dot{C}_s^\g} \|X'-Y'\|_{L_s^\infty} \|Z\|_{L_s^\infty},
\end{align*}
In the last inequality, we used the fact that $t^{-3/2}\|Z\|_{L_s^\infty}^\g\leq C(t^{-(3-\g)/2}
 + t^{-2}\|Z\|_{L_s^\infty}^{1+\g})$.
Here the constant $C$ may depend on $\g$.

Combining this with \eqref{eqn: L inf estimate for difference of J_1} and \eqref{eqn: C 1 estimate for difference of J_1}, we conclude that
\begin{align*}
&\;\|J(t;X,Z)-J(t;Y,Z)\|_{L_s^\infty(\T)}\\
&\; + t^{1/2}(\|X'\|_{L_s^\infty}+\|Y'\|_{L_s^\infty})^{-1}\|J(t;X,Z)-J(t;Y,Z)\|_{\dot{C}_s^1(\T)}\\
\leq &\; C
\big(\min(\lam^{-3}t^{-1},\,\lam^{-2}t^{-3/2}) + \lam^{-3}t^{-2}\|Z\|_{L_s^\infty}^2 + \lam^{-3}t^{-3/2}\|Z\|_{L_s^\infty} \big)\\
&\;\quad \cdot
(\|X'\|_{L_s^\infty}+\|Y'\|_{L_s^\infty})^2 \|X'-Y'\|_{L_s^\infty} \|Z'\|_{L_s^\infty}\\
&\; + C \lam^{-1-\g}
\big( t^{-(3-\g)/2} +  t^{-3/2}\|Z\|_{L_s^\infty}^\g\big)
\|X'\|_{L_s^\infty} \|Z\|_{L_s^\infty} \|X'-Y'\|_{\dot{C}_s^\g}\\
&\; + C\lam^{-2-\g}
\big( t^{-(3-\g)/2} + t^{-2}\|Z\|_{L_s^\infty}^{1+\g}\big)
\|Y'\|_{L_s^\infty} \big(\|Y'\|_{\dot{C}_s^\g} + \|Z'\|_{\dot{C}_s^\g}\big) \|X'-Y'\|_{L_s^\infty} \|Z\|_{L_s^\infty},
\end{align*}
where $C$ depends on $\g$.
This together with \eqref{eqn: splitting J(X,Z)-J(Y,W)} and \eqref{eqn: bound for J YZ - J YW} implies the desired estimate.
\end{proof}
\end{lem}

Now we can bound $H_X^{\nu}-H_Y^\nu$.
\begin{proof}[Proof of Lemma \ref{lem: estimate for H_X^nu-H_Y^nu}]
Let $J$ be defined as in \eqref{eqn: def of J}.
Denote
\beq
\begin{split}
&\; \CJ[\tau,t,\nu;X,Y](s)\\
:= &\; \big[J(\nu (t-\tau);X(t),X(\tau)-X(t))
-J(\nu (t-\tau);Y(t),Y(\tau)-Y(t))\big](s).
\end{split}
\label{eqn: def of script J}
\eeq
Then by \eqref{eqn: H_X - H_Y} and \eqref{eqn: def of J}, 
\beq
H_X^{\nu}(t)-H_Y^\nu(t)
= -\nu\int_0^t \CJ[\tau,t,\nu;X,Y]\,d\tau. 
\label{eqn: difference of H is integral of CJ}
\eeq

Under the assumptions on $X$ and $Y$, we apply Lemma \ref{lem: bound for J YZ - J YW} to \eqref{eqn: def of script J} to find that, for $0<\tau<t\leq T$,
\beq
\begin{split}
&\; \|\CJ[\tau,t,\nu;X,Y]\|_{L_s^\infty(\BT)} \\
&\; + (\nu(t-\tau))^{1/2}(\|X'(t)\|_{L_s^\infty}+\|Y'(t)\|_{L_s^\infty})^{-1}\|\CJ[\tau,t,\nu;X,Y]\|_{\dot{C}_s^1(\BT)}\\
\leq &\; C
\big[ \lam^{-3}(\nu (t-\tau))^{-1}
+ \lam^{-3}(\nu (t-\tau))^{-2}\|X(\tau)-X(t)\|_{L_s^\infty}^2\big]\\
&\;\quad \cdot
(\|X'(t)\|_{L_s^\infty}+\|Y'(t)\|_{L_s^\infty})^2 \|X'(t)-Y'(t)\|_{L_s^\infty} \|X'(\tau)-X'(t)\|_{L_s^\infty}\\
&\; + C \lam^{-1-\g}
\big[ (\nu (t-\tau))^{-(3-\g)/2} +  (\nu (t-\tau))^{-3/2}\|X(\tau)-X(t)\|_{L_s^\infty}^\g\big]\\
&\;\quad \cdot \|X'(t)\|_{L_s^\infty}\|X(\tau)-X(t)\|_{L_s^\infty} \|X'(t)-Y'(t)\|_{\dot{C}_s^\g}\\
&\; + C\lam^{-2-\g}
\big[ (\nu (t-\tau))^{-(3-\g)/2} + (\nu (t-\tau))^{-2}\|X(\tau)-X(t)\|_{L_s^\infty}^{1+\g}\big]\\
&\;\quad \cdot
\|Y'(t)\|_{L_s^\infty} \big(\|Y'(t)\|_{\dot{C}_s^\g} + \|X'(\tau)-X'(t)\|_{\dot{C}_s^\g}\big) \|X'(t)-Y'(t)\|_{L_s^\infty} \|X(\tau)-X(t)\|_{L_s^\infty}
\\
&\; + C \lam^{-1}(\nu (t-\tau))^{-1} \|Y'(t)\|_{L_s^\infty} \|X'(\tau)-X'(t)-Y'(\tau)+Y'(t)\|_{L_s^{\infty}}\\
&\; + C \lam^{-1-\g}(\nu (t-\tau))^{-(3-\g)/2} \|Y'(t)\|_{L_s^\infty} \|X(\tau)-X(t)-Y(\tau)+Y(t)\|_{L_s^\infty}\\
&\;\quad \cdot
\big(\|Y'(t)\|_{\dot{C}_s^\g}+\|X'(\tau)-X'(t)\|_{\dot{C}_s^\g}+\|Y'(\tau)-Y'(t)\|_{\dot{C}_s^\g}\big) \\
&\; + C \lam^{-1-\g}(\nu (t-\tau))^{-(3-\g)/2} \|Y'(t)\|_{L_s^\infty}^2  \|X(\tau)-X(t)-Y(\tau)+Y(t)\|_{C_s^\g},
\end{split}
\label{eqn: estimate for cal J}
\eeq
where the constant $C$ depends on $\g$.
By \eqref{eqn: improved inequality regarding time difference}, with $\mu_0\in [0,\f12]$,
\beq
\begin{split}
\|X(\tau)-X(t)\|_{L_s^\infty}
\leq &\; \int_{\tau}^t \|\pa_t X(\eta)\|_{L_s^\infty}\,d\eta \\
\leq &\;\int_{\tau}^t \zeta^{-\mu_0} \,d\zeta \cdot
\sup_{\eta\in (0,t)}\eta^{\mu_0}\|\pa_t X(\eta)\|_{L_s^\infty}
\leq C (t-\tau) t^{-\mu_0} N_0,
\end{split}
\label{eqn: L inf estimate for X tau - X t}
\eeq
and similarly,
\[
\|X(\tau)-X(t)-Y(\tau)+Y(t)\|_{L_s^\infty}
\leq C (t-\tau)t^{-\mu_0} \sup_{\eta\in (0,t)}\eta^{\mu_0} \|\pa_t (X-Y)(\eta)\|_{L_s^\infty}.
\]
Here the constants $C$'s are universal.
Moreover, with $\b\in [\g,1)$,
\begin{align*}
\|X(\tau)-X(t)-Y(\tau)+Y(t)\|_{C_s^{\b}}
\leq &\; \int_\tau^t \zeta^{-\b}\,d\zeta \cdot \sup_{\eta\in (0,t)}\eta^{\b}\|\pa_t(X-Y)(\eta)\|_{C_s^\b}\\
\leq &\; C(t-\tau)t^{-\b} \cdot \sup_{\eta\in (0,t)}\eta^{\b}\|\pa_t(X-Y)(\eta)\|_{C_s^\b},
\end{align*}
where $C$ depends on $\b$.
This together with the previous estimate and the interpolation inequality implies that
\begin{align*}
&\;\|X(\tau)-X(t)-Y(\tau)+Y(t)\|_{C_s^\g}\\
\leq &\; C(t-\tau)t^{-\g-\mu_0(1-\f{\g}{\b})} \left(\sup_{\eta\in (0,t)}\eta^{\mu_0} \|\pa_t (X-Y)(\eta)\|_{L_s^\infty}+ \sup_{\eta\in (0,t)}\eta^{\b}\|\pa_t(X-Y)(\eta)\|_{C_s^\b}\right),
\end{align*}
where the constant $C$ depends on $\b$ and $\g$.
Also by the interpolation inequality, with $\al\in (0,\g]$ satisfying $(1-\g)(1+\al) + \g\b = 1$, we have that
\beq
\begin{split}
\|X(\tau)-X(t)\|_{\dot{C}_s^1}
\leq &\; C\|X(\tau)-X(t)\|_{\dot{C}_s^{1,\al}}^{1-\g}
\|X(\tau)-X(t)\|_{\dot{C}_s^{\b}}^\g\\
\leq &\; C\left(\tau^{-\al} \sup_{\eta\in [\tau,t]}\eta^{\al}\|X'(\eta)\|_{\dot{C}_s^\al}\right)^{1-\g}
\left(\int_\tau^t \zeta^{-\b}\cdot \zeta^\b\|\pa_t X(\zeta)\|_{\dot{C}_s^{\b}}\,d\zeta\right)^\g\\
\leq &\; C\left(\tau^{-\al} M\right)^{1-\g}
\left(|t-\tau| \tau^{-\b} N_\b \right)^\g
\\
= &\; C|t-\tau|^\g \tau^{-\g} M^{1-\g} N_\b^\g,
\end{split}
\label{eqn: C 1 estimate for X tau - X t}
\eeq
and similarly,
\begin{align*}
&\; \|X(\tau)-X(t)-Y(\tau)+Y(t)\|_{\dot{C}_s^1}\\
\leq &\; C \left(\tau^{-\al}\sup_{\eta\in (0,t]}\eta^{\al}\|(X'-Y')(\eta)\|_{\dot{C}_s^\al}\right)^{1-\g}
\left(|t-\tau| \tau^{-\b} \sup_{\eta\in (0,t)} \eta^\b \|\pa_t (X-Y)(\eta)\|_{\dot{C}_s^{\b}}\right)^\g
\\
\leq &\; C|t-\tau|^{\g} \tau^{-\g}\\
&\; \cdot \left(\|X'-Y'\|_{L^\infty_t L_s^{\infty}} + \sup_{\eta\in (0,t]} \eta^\g \|(X'-Y')(\eta)\|_{C_s^{\g}}
+ \sup_{\eta\in (0,t)} \eta^{\b} \|\pa_t(X-Y)(\eta)\|_{\dot{C}_s^\b}\right)\\
= &\; C|t-\tau|^{\g} \tau^{-\g} D(t),
\end{align*}
where the constants $C$'s depend on $\b$ and $\g$, and where $D(t)$ was defined in \eqref{eqn: def of D distance bewteen X and Y}.
Finally, under the assumption that $X(s,0) = Y(s,0)$, we also find that
\begin{align*}
&\; \|X'(t)-Y'(t)\|_{L_s^\infty}\\
\leq &\; C \|(X'-Y')(t)\|_{C_s^\al}^{1-\g} \|(X-Y)(t)\|_{\dot{C}_s^\b}^{\g} \\
\leq &\; C
\left(t^{-\al}\sup_{\eta\in (0,t]}\eta^\al \|(X'-Y')(\eta)\|_{C_s^\al}\right)^{1-\g}
\left(\int_0^t \zeta^{-\b}\,d\zeta \cdot
\sup_{\eta\in (0,t)}\eta^\b\|\pa_t(X-Y)(\eta)\|_{\dot{C}_s^\b}\right)^{\g}\\
\leq &\; C\left(\|X'-Y'\|_{L^\infty_t L_s^\infty} + \sup_{\eta\in (0,t]} \eta^\g \|(X'-Y')(\eta)\|_{C_s^{\g}}
+ \sup_{\eta\in (0,t)} \eta^{\b} \|\pa_t(X-Y)(\eta)\|_{\dot{C}_s^\b}\right)\\
= &\; C D(t),
\end{align*}
where $C$ depends on $\b$ and $\g$.

Putting all these estimates into \eqref{eqn: estimate for cal J} gives that
\begin{align*}
&\; \|\CJ[\tau,t,\nu;X,Y]\|_{L_s^\infty} \\ 
&\; + (\nu(t-\tau))^{1/2}(\|X'(t)\|_{L_s^\infty}+\|Y'(t)\|_{L_s^\infty})^{-1} \|\CJ[\tau,t,\nu;X,Y]\|_{\dot{C}_s^1}
\\
\leq &\; C\lam^{-3}(\nu (t-\tau))^{-1}
\Big[ 1
+ (\nu (t-\tau))^{-1}\big((t-\tau)t^{-\mu_0} N_0\big)^2\Big]\\
&\;\quad \cdot
M^2 \cdot D(t) \cdot |t-\tau|^\g \tau^{-\g} M^{1-\g} N_\b^\g\\
&\; + C \lam^{-1-\g}(\nu (t-\tau))^{-(3-\g)/2}
\Big[ 1 +  (\nu (t-\tau))^{-\g/2}\big((t-\tau)t^{-\mu_0} N_0\big)^\g\Big]\\
&\;\quad \cdot M\cdot (t-\tau)t^{-\mu_0} N_0 \cdot  t^{-\g} D(t)\\
&\; + C\lam^{-2-\g} (\nu (t-\tau))^{-(3-\g)/2}
\Big[ 1 + (\nu (t-\tau))^{-(1+\g)/2} \big((t-\tau)t^{-\mu_0} N_0\big)^{1+\g}\Big]\\
&\;\quad \cdot
M\cdot \tau^{-\g} M\cdot D(t) \cdot (t-\tau)t^{-\mu_0} N_0
\\
&\; + C \lam^{-1}(\nu (t-\tau))^{-1} M\cdot |t-\tau|^{\g} \tau^{-\g} D(t)\\
&\; + C \lam^{-1-\g}(\nu (t-\tau))^{-(3-\g)/2} M
\cdot (t-\tau)t^{-\mu_0} \sup_{\eta\in (0,t]}\eta^{\mu_0} \|\pa_t (X-Y)(\eta)\|_{L_s^\infty}
\cdot
\tau^{-\g} M \\
&\; + C \lam^{-1-\g}(\nu (t-\tau))^{-(3-\g)/2} M^2 \\
&\;\quad  \cdot (t-\tau)t^{-\g-\mu_0(1-\f{\g}{\b})} \left(\sup_{\eta\in (0,t)}\eta^{\mu_0} \|\pa_t (X-Y)(\eta)\|_{L_s^\infty}+ D(t)\right)
\\
\leq &\; C (t-\tau)^{-1+\g}\tau^{-\g} D(t)\cdot \lam^{-3}\nu^{-1}
M^{3-\g}  N_\b^\g \Big[ 1
+ \nu^{-1} t^{1-2\mu_0} N_0^2\Big]\\
&\; + C(t-\tau)^{-(1-\g)/2} \tau^{-\g} t^{-\mu_0} D(t) \cdot
\lam^{-2-\g} \nu^{-(3-\g)/2} M^2 N_0
\Big[ 1 + \nu^{-1} t^{1-2\mu_0} N_0^2\Big]^{\f{1+\g}{2}}
\\
&\; + C(t-\tau)^{-1+\g}\tau^{-\g} D(t)\cdot \lam^{-1} \nu^{-1}  M\\
&\; + C(t-\tau)^{-(1-\g)/2}
\tau^{-\g} t^{-\mu_0} \sup_{\eta\in (0,t)}\eta^{\mu_0} \|\pa_t (X-Y)(\eta)\|_{L_s^\infty}\cdot \lam^{-1-\g}\nu^{-(3-\g)/2} M^2\\
&\; + C (t-\tau)^{-(1-\g)/2} t^{-\g-\mu_0(1-\f{\g}{\b})} \left(D(t)+\sup_{\eta\in (0,t)}\eta^{\mu_0} \|\pa_t (X-Y)(\eta)\|_{L_s^\infty}\right)\cdot  \lam^{-1-\g}\nu^{-(3-\g)/2}M^2,
\end{align*}
where $C$ depends on $\b$ and $\g$.
By assumption $t\leq 1$, so
\beq
\begin{split}
&\; \|\CJ[\tau,t,\nu;X,Y]\|_{L_s^\infty}\\
&\; + (\nu(t-\tau))^{1/2}(\|X'\|_{L^\infty_t L_s^\infty}+\|Y'\|_{L^\infty_t L_s^\infty})^{-1} \|\CJ[\tau,t,\nu;X,Y]\|_{\dot{C}_s^1}\\
\leq &\; C(t-\tau)^{-1+\g}\tau^{-\g} A(t),
\end{split}
\label{eqn: L inf bound for J}
\eeq
where $A(t)$ was defined in \eqref{eqn: def of A_J}, and where $C$ depends on $\b$ and $\g$.

Now by \eqref{eqn: difference of H is integral of CJ},
\[
\|H_X^{\nu}(t)-H_Y^\nu(t)\|_{L^\infty_s}
\leq \nu\int_0^t \|\CJ[\tau,t,\nu;X,Y]\|_{L_s^\infty} \,d\tau
\leq C\nu A(t).
\]
Applying Lemma \ref{lem: interpolation} and \eqref{eqn: L inf bound for J}, we also find for $\g\in (0,\f12)$ that
\begin{align*}
&\;\|H_X^{\nu}(t)-H_Y^\nu(t)\|_{\dot{C}^{2\g}_s(\T)}\\
\leq
&\; C\nu t^{-\g} \left|\sup_{\tau \in (0,t)}(t-\tau)^{1-\g} \tau^\g
\|\CJ[\tau,t,\nu;X,Y]\|_{L_s^\infty(\T)}\right|^{1-2\g}\\
&\;\cdot
\left|\sup_{\tau \in (0,t)}(t-\tau)^{3/2-\g} \tau^\g
\|\CJ[\tau,t,\nu;X,Y]\|_{\dot{C}_s^1(\T)}\right|^{2\g}\\
\leq
&\; Ct^{-\g} \nu^{1-\g} (\|X'\|_{L^\infty_t L^\infty_s}+\|Y'\|_{L^\infty_t L^\infty_s})^{2\g}\\
&\; \cdot \sup_{\tau \in (0,t)}(t-\tau)^{1-\g}\tau^\g
\Big[\|\CJ[\tau,t,\nu;X,Y]\|_{L_s^\infty(\T)} \\
&\;\qquad \qquad \qquad \qquad +(\nu (t-\tau))^{1/2} (\|X'\|_{L^\infty_t L^\infty_s}+\|Y'\|_{L^\infty_t L^\infty_s})^{-1}\|\CJ[\tau,t,\nu;X,Y]\|_{\dot{C}_s^1(\T)}\Big]\\
\leq &\; C t^{-\g} \nu^{1-\g} (\|X'\|_{L^\infty_t L^\infty_s}+\|Y'\|_{L^\infty_t L^\infty_s})^{2\g} A(t).
\end{align*}
Here the constants $C$'s depend on $\b$ and $\g$.
\end{proof}

\subsection{Estimates for $B_{\nu}[u]$}
Recall that $B_\nu[u]$ was defined in \eqref{eqn: u_2 tilde}.
\begin{proof}[Proof of Lemma \ref{lem: estimate for B_nu u}]
We rewrite $B_{\nu}[u]$ as
\beqo
B_{\nu}[u](t)=-\int_0^t e^{\nu (t-\tau)\Delta}\BP\di(u\otimes u)(\tau)\, d\tau.
\eeqo
Since $\BP = \Id + \na(-\D)^{-1}\di$, the standard parabolic estimates give that,
for
$q\in (1,\infty]$ and $r\in [1,q]$,
\beqo
\|e^{\nu (t-\tau)\Delta}\BP\di(u\otimes u)(\tau)\|_{L^q_x}
\leq C(\nu(t-\tau))^{-\f12-\f1r+\f1q} \|(u\otimes u)(\tau)\|_{L^{r}_x}.
\eeqo
Here we interpret $\f{1}{\infty}$ as $0$, and the constant $C$ depends on $q$ and $r$.
Hence, if additionally $0< \f1r <\f12 + \f1q$,
\begin{align*}
\|B_{\nu}[u](t)\|_{L^q_x}
\leq &\; C \int_0^t (\nu(t-\tau))^{-\f12-\f1r+\f1q} (\nu\tau)^{\f1r-1}\, d\tau \cdot \sup_{\eta\in(0,t]}(\nu \eta)^{1-\f1r}\|(u\otimes u)(\eta)\|_{L_x^r}\\
\leq &\; C\nu^{-1} (\nu t)^{\f1q-\f12}
\sup_{\eta\in(0,t]}(\nu \eta)^{1-\f1r}\|(u\otimes u)(\eta)\|_{L_x^r},
\end{align*}
which gives \eqref{eqn: L q estimate for B general}.

By Lemma \ref{lem: interpolation} with $\va(\tau) = \tau^{-(1+\g)/2}$,
\begin{align*}
\| B_{\nu}[u](t)\|_{\dot{C}_x^{\g}}
\leq &\; Ct^{-(1+\g)/2}\left|\sup_{\tau\in [0,t]}(t-\tau)^{\f{2-\g}{2}} \tau^{(1+\g)/2}\| e^{\nu (t-\tau)\Delta}\BP\di(u\otimes u)(\tau)\|_{L_x^\infty}\right|^{1-\g}\\
&\;\cdot \left|\sup_{\tau\in [0,t]}(t-\tau)^{\f{3-\g}{2}}
\tau^{(1+\g)/2} \|\na e^{\nu (t-\tau)\Delta}\BP\di(u\otimes u)(\tau)\|_{L_x^\infty}\right|^{\g}.
\end{align*}
Then by Lemma \ref{lem: parabolic estimates},
\beq
\begin{split}
&\; \|B_{\nu}[u](t)\|_{\dot{C}_x^{\g}}\\
\leq &\; Ct^{-(1+\g)/2}\left|\sup_{\tau\in (0,t]}(t-\tau)^{\f{2-\g}{2}} \cdot
(\nu (t-\tau))^{-(2-\g)/2} \tau^{(1+\g)/2}
\|(u\otimes u)(\tau)\|_{M^{1,1-\g}}\right|^{1-\g}\\
&\;\cdot \left|\sup_{\tau\in (0,t]}(t-\tau)^{\f{3-\g}{2}}
\cdot (\nu (t-\tau))^{-(3-\g)/2} \tau^{(1+\g)/2} \|(u\otimes u)(\tau)\|_{M^{1,1-\g}}\right|^{\g}\\
\leq &\; C \nu^{-1} t^{-(1+\g)/2} \sup_{\tau\in (0,t]} \tau^{(1+\g)/2} \|(u\otimes u) (\tau)\|_{M^{1,1-\g}}\\
\leq &\; C \nu^{-1} (\nu t)^{-(1+\g)/2} \sup_{\tau\in (0,t]} (\nu \tau)^{(1+\g)/2} \|(u\otimes u) (\tau)\|_{ L_x^{2/(1-\g)}},
\end{split}
\label{eqn: Holder bound for B u in term of Morrey norm}
\eeq
where $C$ depends on $\g$.
This proves \eqref{eqn: Holder estimate for B u prelim}.
Based on this, for $p\in [2,\infty)$, if $\g \geq 1-\f4p$, i.e.\;$\f{4}{1-\g}\geq p$, we may use the interpolation inequality to further obtain that
\beq
\begin{split}
&\; \|B_{\nu}[u](t)\|_{\dot{C}_x^{\g}(\R^2)}\\
\leq &\; C \nu^{-1}(\nu t)^{-(1+\g)/2}
\left(\sup_{\eta\in (0,t]} (\nu \eta)^{\f12-\f1p} \|u(\eta)\|_{L^p_x}\right)^{\f{(1-\g)p}{2}}
\left(\sup_{\eta\in (0,t]}(\nu \eta)^{\f12}\|u(\eta)\|_{L^\infty_x}\right)^{2-\f{(1-\g)p}{2}}\\
\leq &\;
C \nu^{-1}(\nu t)^{-(1+\g)/2}
\left(\sup_{\eta\in (0,t]} (\nu \eta)^{\f12-\f1p} \|u(\eta)\|_{L^p_x} + \sup_{\eta\in (0,t]}(\nu \eta)^{1/2}\|u(\eta)\|_{L^\infty_x}\right)^2.
\end{split}
\label{eqn: Holder estimate for B u after interpolation}
\eeq
The general case $\g\in (0,1)$ follows by interpolating this with \eqref{eqn: L inf estimate for B}, which gives \eqref{eqn: 6th estimate for B new form general gamma}.

To show the higher-order estimates in \eqref{eqn: 2nd estimate for B}, we observe from \eqref{eqn: u_2 tilde} that
\begin{align*}
B_{\nu}[u](t) = &\; -\int_{t/2}^t e^{\nu (t-\tau)\Delta}\mathbb{P}\di(u\otimes u)(\tau)\, d\tau
+ e^{\f{\nu t}{2}\Delta}\left[-\int_0^{t/2}e^{\nu (t/2-\tau)\Delta}\mathbb{P}(u\cdot\na u)(\tau)\, d\tau\right]\\
= &\; -\int_{t/2}^t e^{\nu (t-\tau)\Delta}\mathbb{P}\di (u\otimes u)(\tau)\, d\tau
+ e^{\f{\nu t}{2}\Delta}\big(B_{\nu}[u](t/2)\big).
\end{align*}
Applying Lemma \ref{lem: interpolation} with $\va(t) = 1$ to the first term, we find that
\begin{align*}
\|B_{\nu}[u](t)\|_{\dot{C}^{1,\g}_x}
\leq &\; C \left|\sup_{\tau\in(t/2,t) }(t-\tau)^{\f{2-\g}{2}} \|\na e^{\nu (t-\tau)\Delta}\BP\di(u\otimes u)(\tau)\|_{L_x^\infty}\right|^{1-\g}\\
&\;\cdot
\left|\sup_{\tau\in (t/2,t)}(t-\tau)^{\f{3-\g}{2}} \|\na^2 e^{\nu (t-\tau)\Delta}\BP\di(u\otimes u)(\tau)\|_{L_x^\infty}\right|^{\g}\\
&\; + \big\| e^{\f{\nu t}{2}\Delta}\big(B_{\nu}[u](t/2)\big)\big\|_{\dot{C}^{1,\g}_x}.
\end{align*}
By Lemma \ref{lem: parabolic estimates} and \eqref{eqn: L inf estimate for B}, for $p\in (2,\infty)$,
\begin{align*}
\|B_{\nu}[u](t)\|_{\dot{C}^{1,\g}_x}
\leq &\; C \left|\sup_{\tau\in(t/2,t) }(t-\tau)^{\f{2-\g}{2}} (\nu(t-\tau))^{-\f{2-\g}{2}}
\|(u\otimes u)(\tau)\|_{\dot{C}^\g_x}\right|^{1-\g}\\
&\;\cdot
\left|\sup_{\tau\in (t/2,t)}(t-\tau)^{\f{3-\g}{2}} (\nu(t-\tau))^{-\f{3-\g}{2}} \|(u\otimes u)(\tau)\|_{\dot{C}^\g_x}\right|^{\g}\\
&\;
+ C(\nu t)^{-(1+\g)/2} \big\|B_{\nu}[u](t/2)\big\|_{L^\infty_x}
\\
\leq &\; C \nu^{-1} \sup_{\tau\in(t/2,t)} \|(u\otimes u)(\tau)\|_{\dot{C}^\g_x} \\
&\; + C\nu^{-1} (\nu t)^{-(2+\g)/2}\sup_{\eta\in(0,t]}(\nu \eta)^{\f12-\f1p}\|u(\eta)\|_{L_x^p}\cdot
\sup_{\eta\in(0,t]}(\nu \eta)^{1/2}\|u(\eta)\|_{L_x^\infty}
\\
\leq
&\; C\nu^{-1} (\nu t)^{-(2+\g)/2}
\sup_{\eta\in (0,t]} (\nu\eta)^{1/2} \|u(\eta)\|_{L^\infty_x}\\
&\; \cdot \left[\sup_{\eta\in(0,t]}(\nu \eta)^{\f12-\f1p}\|u(\eta)\|_{L_x^p}
+ \sup_{\eta\in (0,t]} (\nu\eta)^{(1+\g)/2} \|u(\eta)\|_{\dot{C}^\g_x}\right].
\end{align*}
This bounds the first term on the left-hand side of \eqref{eqn: 2nd estimate for B}.
Similarly,
\begin{align*}
\|B_{\nu}[u](t)\|_{\dot{C}^{1}_x}
\leq &\; \int_{t/2}^t \big\|\na e^{\nu (t-\tau)\Delta}\mathbb{P}\di (u\otimes u)(\tau)\big\|_{L^\infty_x}\, d\tau
+ \big\| e^{\f{\nu t}{2}\Delta}\big(B_{\nu}[u](t/2)\big)\big\|_{\dot{C}^{1}_x}
\\
\leq &\; C \int_{t/2}^t (\nu(t-\tau))^{-\f{2-\g}{2}}
\|(u\otimes u)(\tau)\|_{\dot{C}^\g_x}\,d\tau
+ C(\nu t)^{-1/2} \big\|B_{\nu}[u](t/2)\big\|_{L^\infty_x}
\\
\leq &\; C \nu^{-1} (\nu t)^{\g/2} \sup_{\tau\in(t/2,t)} \|(u\otimes u)(\tau)\|_{\dot{C}^\g_x} \\
&\; + C\nu^{-1} (\nu t)^{-1}\sup_{\eta\in(0,t]}(\nu \eta)^{\f12-\f1p}\|u(\eta)\|_{L_x^p}\cdot
\sup_{\eta\in(0,t]}(\nu \eta)^{1/2}\|u(\eta)\|_{L_x^\infty}
\\
\leq
&\; C\nu^{-1} (\nu t)^{-1}
\sup_{\eta\in (0,t]} (\nu\eta)^{1/2} \|u(\eta)\|_{L^\infty_x}\\
&\; \cdot \left[\sup_{\eta\in(0,t]}(\nu \eta)^{\f12-\f1p}\|u(\eta)\|_{L_x^p}
+ \sup_{\eta\in (0,t]} (\nu\eta)^{(1+\g)/2} \|u(\eta)\|_{\dot{C}^\g_x}\right].
\end{align*}
This bounds the second term on the left-hand side of \eqref{eqn: 2nd estimate for B}.
This completes the proof of \eqref{eqn: 2nd estimate for B} for $p\in (2,\infty)$.
The case of $p = 2$ then follows from the interpolation inequality and the Young's inequality.

By Lemma \ref{lem: parabolic estimates},
\begin{align*}
\|B_\nu[u](t)\|_{M^{1,1-\g}(\BR^2)}
\le &\; \int_0^t \left\|e^{\nu (t-\tau)\Delta}\BP\di(u\otimes u)(\tau)\right\|_{M^{1,1-\g}(\BR^2)} d\tau\\
\leq &\; C\int_0^t (\nu (t-\tau))^{-1/2} \big\|(u\otimes u)(\tau)\big\|_{M^{1,1-\g}}\, d\tau\\
\leq &\; C\nu^{-1} (\nu t)^{-\g/2} \sup_{\eta\in(0,t]}(\nu\eta)^{(1+\g)/2}\|(u\otimes u)(\eta)\|_{M^{1,1-\g}},
\end{align*}
and in a similar spirit,
\beqo
\begin{split}
\|B_{\nu}[u](t)\|_{L^\infty_x}
\leq &\; \int_0^t \big\|e^{\nu (t-\tau)\Delta}\BP\di(u\otimes u)(\tau)\big\|_{L^\infty_x}\, d\tau\\
\leq &\; C \int_0^t (\nu(t-\tau))^{-(2-\g)/2} \|(u\otimes u)(\tau)\|_{M^{1,1-\g}} \, d\tau\\
\leq &\; C\nu^{-1} (\nu t)^{-1/2}
\sup_{\tau\in (0,t]} (\nu \tau)^{(1+\g)/2} \|(u\otimes u) (\tau)\|_{M^{1,1-\g}}.
\end{split}
\eeqo
In view of these inequalities as well as
\[
\|u_1\otimes u_2 \|_{M^{1,1-\g}(\BR^2)}
\leq C \|u_1\|_{M^{1,1-\g}(\BR^2)} \|u_2\|_{L^\infty_x(\BR^2)},
\]
we find that
\begin{align*}
&\; (\nu t)^{\g/2} \|(B_{\nu}[u]-B_{\nu}[v])(t)\|_{M^{1,1-\g}(\BR^2)}
+ (\nu t)^{1/2} \|(B_{\nu}[u] - B_{\nu}[v])(t)\|_{L^\infty_x(\BR^2)}\\
\leq &\; C\nu^{-1} \sup_{\eta\in(0,t]}(\nu\eta)^{(1+\g)/2}
\big\|(u\otimes u - v\otimes v)(\eta)\big\|_{M^{1,1-\g}}
\\
\leq &\; C\nu^{-1}
\sup_{\eta\in (0,t]}(\nu \eta)^{\g/2}\|(u-v)(\eta)\|_{M^{1,1-\g}}
\sup_{\eta\in (0,t]} (\nu\eta)^{1/2} \|(u(\eta),v(\eta))\|_{L_x^\infty},
\end{align*}
where $C$ depends on $\g$.
This proves \eqref{eqn: 4th estimate for B L inf}.
One can justify \eqref{eqn: 4th estimate for B gamma} in a similar manner using the second last line of \eqref{eqn: Holder bound for B u in term of Morrey norm}.
\end{proof}

Finally, we prove Lemma \ref{lem: time Holder continuity of B_u} on the time-continuity of $B_\nu[u]$.

\begin{proof}[Proof of Lemma \ref{lem: time Holder continuity of B_u}]

If $t_1 \leq t_2/2$, by \eqref{eqn: L inf estimate for B} in Lemma \ref{lem: estimate for B_nu u},
\begin{align*}
&\; \|B_{\nu}[u](t_2)-B_{\nu}[u](t_1)\|_{L^\infty_x}\\
\leq &\; \|B_{\nu}[u](t_2)\|_{L^\infty_x} + \|B_{\nu}[u](t_1)\|_{L^\infty_x}\\
\leq &\; C\nu^{-1} (\nu t_1)^{-\f12}\sup_{\eta\in(0,t_1]}(\nu \eta)^{\f12-\f1p}\|u(\eta)\|_{L_x^p}\cdot
\sup_{\eta\in(0,t_1]}(\nu \eta)^{\f12}\|u(\eta)\|_{L_x^\infty}\\
&\; + C\nu^{-1} (\nu t_2)^{-\f12}\sup_{\eta\in(0,t_2]}(\nu \eta)^{\f12-\f1p}\|u(\eta)\|_{L_x^p}\cdot
\sup_{\eta\in(0,t_2]}(\nu \eta)^{\f12}\|u(\eta)\|_{L_x^\infty}
\\
\leq &\; C\nu^{-1} (\nu (t_2-t_1))^{\f\g{2}} (\nu t_1)^{-\f{1+\g}2}
\left(\sup_{\eta\in(0,t_1]}(\nu \eta)^{\f12-\f1p}\|u(\eta)\|_{L_x^p}
+
\sup_{\eta\in(0,t_1]}(\nu \eta)^{\f12}\|u(\eta)\|_{L_x^\infty}\right)^2\\
&\; + C\nu^{-1} (\nu (t_2-t_1))^{\f{\g}2} (\nu t_2)^{-\f{1+\g}{2}}
\left(\sup_{\eta\in(0,t_2]}(\nu \eta)^{\f12-\f1p}\|u(\eta)\|_{L_x^p}+
\sup_{\eta\in(0,t_2]}(\nu \eta)^{\f12}\|u(\eta)\|_{L_x^\infty}\right)^2.
\end{align*}

If $t_1\in (t_2/2, t_2)$, by definition,
\[
B_{\nu}[u](t_2)-B_{\nu}[u](t_1)
= -\int_{t_1}^{t_2}e^{\nu (t_2-\tau)\Delta}\BP\di(u\otimes u)(\tau)\, d\tau
+ \big(e^{\nu(t_2-t_1)\D}-\Id\big)B_{\nu}[u](t_1).
\]
By Lemma \ref{lem: parabolic estimates} (also see \eqref{eqn: Holder bound for B u in term of Morrey norm}) and \eqref{eqn: 6th estimate for B new form general gamma} in Lemma \ref{lem: estimate for B_nu u},
and also using the assumption $t_1\geq t_2/2$,
\begin{align*}
&\; \|B_{\nu}[u](t_2)-B_{\nu}[u](t_1)\|_{L^\infty_x }\\
\leq &\; \int_{t_1}^{t_2}\|e^{\nu (t_2-\tau)\Delta}\BP\di(u\otimes u)(\tau)\|_{L^\infty_x}\, d\tau
+ \big\|\big(e^{\nu(t_2-t_1)\D}-\Id\big) B_{\nu}[u](t_1)\big\|_{L^\infty_x}\\
\leq &\; C\int_{t_1}^{t_2} (\nu(t_2-\tau))^{-(2-\g)/2}\|(u\otimes u)(\tau)\|_{M^{1,1-\g}}\, d\tau
+ C (\nu(t_2-t_1))^{\g/2} \|B_{\nu}[u](t_1)\|_{\dot{C}^\g_x}\\
\leq &\; C\int_{t_1}^{t_2} (\nu(t_2-\tau))^{-(2-\g)/2} (\nu \tau)^{-(1+\g)/2} \, d\tau \cdot \sup_{\eta\in (0,t_2]} (\nu \eta)^{(1+\g)/2} \|(u\otimes u)(\eta)\|_{L^{2/(1-\g)}_x}\\
&\; + C \nu^{-1} (\nu(t_2-t_1))^{\g/2}(\nu t_1)^{-(1+\g)/2} \\ &\;\quad \cdot \left(\sup_{\eta\in (0,t_1]} (\nu \eta)^{\f12-\f1p} \|u(\eta)\|_{L^p_x}
+\sup_{\eta\in (0,t_1]}(\nu \eta)^{\f12}\|u(\eta)\|_{L^\infty_x}\right)^2.
\end{align*}
When $\g\geq 1-\f4p$, which means $\f{4}{1-\g}\geq p$, we use interpolation inequality as in \eqref{eqn: Holder estimate for B u after interpolation} to find that
\begin{align*}
&\; \|B_{\nu}[u](t_2)-B_{\nu}[u](t_1)\|_{L^\infty_x }\\
\leq &\; C \nu^{-1} (\nu(t_2-t_1))^{\f{\g}2}(\nu t_2)^{-\f{1+\g}{2}}
\left(\sup_{\eta\in (0,t_2]} (\nu \eta)^{\f12-\f1p} \|u(\eta)\|_{L^p_x}+\sup_{\eta\in (0,t_2]}(\nu \eta)^{\f12}\|u(\eta)\|_{L^\infty_x}\right)^2.
\end{align*}
Observe that lowering the value of $\g$ makes the bound larger, so this holds true for any $\g\in (0,1)$.

Combining these two cases, we prove the desired estimate.
\end{proof}

\section{Construction of the Local Solution}
\label{sec: local well-posedness}

In this section, we shall apply a fixed-point argument to construct a local mild solution to \eqref{eqn: NS equation}-\eqref{eqn: initial data}.
The main result of this section is Proposition \ref{prop: fixed-point solution}.

\subsection{Setup of the fixed-point argument}
\label{sec: fixed-pt mapping}
Throughout this section, we consider the initial data $X_0 = X_0(s)\in C^1(\BT)$, and $u_0 = u_0(x)\in L^p(\BR^2)$ for some $p\in (2,\infty)$, which satisfies $\di u_0 = 0$ and $|X_0|_*>0$.
Denote
\beq
M_0:= \|X_0'\|_{C(\BT)},\quad \lam_0: = |X_0|_*,\quad
Q_0 := \|u_0\|_{L^p(\BR^2)} + \lam_0^{-1+\f2p} M_0^2.
\label{eqn: def of initial constants}
\eeq
Clearly, $0<\lam_0\leq \inf_{s\in \BT} |X_0'(s)|\leq M_0$.
Here we put a second term in the definition of $ Q_0 $ just to avoid possible issues in the special case $u_0=0$.

Fix
\beq
\g\in \left[ \f2p,1\right) \cap \left(1-\f2p,1\right).
\label{eqn: range of gamma}
\eeq
In particular $\g>\f12$.
Let
\beq
\r(t) = \r_{X_0,\g}(t): = \sup_{\tau\in (0,t]}\tau^{\g}\big\|e^{-\f{\tau}{4} \Lam}X_0'\big\|_{\dot{C}^\g_s(\BT)}.
\label{eqn: def of rho}
\eeq
This is well-defined since $\tau^{\g}\|e^{-\f{\tau}{4} \Lam}X_0'\|_{C^\g_s(\BT)}\leq C\|X_0'\|_{L^\infty_s(\BT)}$, with $C$ depending on $\g$ only.
$\r(t)$ is continuous and increasing on $(0,+\infty)$.
Moreover, $\lim_{t\to 0^+}\r(t) = 0$.
Indeed,
\beq
\begin{split}
\tau^{\g}\big\|e^{-\f{\tau}{4} \Lam}X_0'\big\|_{\dot{C}^\g(\BT)}
\leq &\;
\tau^{\g}\big\|e^{-\f{\sqrt{\tau}}{4} \Lam} X_0'\big\|_{\dot{C}^\g(\BT)}
+ \tau^{\g}\Big\|\big(e^{-\f{\sqrt{\tau}}{4} \Lam}- e^{-\f{\tau}{4} \Lam}\big) X_0'\Big\|_{\dot{C}^\g(\BT)}\\
\leq &\;
C\tau^{\g/2}\| X_0'\|_{C(\BT)}
+ C\Big\|\big(e^{-\f{\sqrt{\tau}-\tau}{4} \Lam}- \mathrm{Id}\big) X_0'\Big\|_{C(\BT)},
\end{split}
\label{eqn: time weighted higher order norm vanish at 0}
\eeq
which implies $\lim_{\tau\to 0^+}\tau^{\g}\|e^{-\f{\tau}{4} \Lam}X_0'\|_{\dot{C}^\g(\BT)} = 0$.

Define
\beq
\phi(t) = \phi_{X_0}(t):= \sup_{\tau\in [0,t]}\big\|X_0-e^{-\f{\tau}{4}\Lam}X_0\big\|_{\dot{C}_s^1(\BT)}.
\label{eqn: def of phi}
\eeq
It is also clear that $\phi(t)$ is continuous and increasing on $[0,+\infty)$, and $\lim_{t\to 0^+}\phi(t) = 0$.

With $T\in (0,1)$ and $Q>0$ to be determined, we define
\beq
\begin{split}
\CU := &\; \Big\{u(x,t) \in L^\infty([0,T]; L^p(\BR^2))\cap C_{loc}((0,T]\times \BR^2):
\\
&\;\qquad \quad
\di u = 0\mbox{ for all }t\in[0,T],\;
\|u\|_{L^\infty_T L^p_x(\BR^2)} \leq Q,\\
&\;\qquad \quad
\sup_{\eta\in (0,T]} (\nu\eta)^{\f1p}\|u(\cdot,\eta)\|_{L^\infty_x(\BR^2)}
\leq Q,\; \sup_{\eta\in (0,T]} (\nu\eta)^{\g}\|u(\cdot,\eta)\|_{\dot{C}^{\g}_x(\BR^2)}
\leq Q,\\
&\;\qquad \quad
\sup_{x\in \BR^2} \sup_{0<t_1<t_2\leq T} (\nu t_1)^\g \cdot \f{|u(x,t_1)-u(x,t_2)|}{(\nu(t_1-t_2))^{\g/2}}
\leq Q\bigg\}.
\end{split}
\label{eqn: def of set CU}
\eeq
With $N>0$ to be determined, and $\s: = \f12(\f12-\f1p)$, 
we let
\beq
\begin{split}
\CX := &\; \Big\{X = X(s,t)\in C([0,T];C^1(\BT))\cap C^1_{loc}((0,T)\times \BT):\; X(s,0) = X_0(s),\\
&\;\qquad \quad
t^\g \|X(\cdot, t)\|_{\dot{C}_s^{1,\g}(\BT)}
+\|X(\cdot,t)-e^{-\f{t}4\Lam}X_0\|_{C_s^1(\BT)}
\leq 2\big(\r(t)+ t^{\s}\big)
\mbox{ for all }t\in(0,T],\\
&\;\qquad \quad
\sup_{\eta\in (0,T)} \eta^{\f1p}\|\pa_t X(\cdot, \eta)\|_{L_s^\infty(\BT)}+
\sup_{\eta\in (0,T)} \eta^\g \|\pa_t X(\cdot, \eta)\|_{C_s^{\g}(\BT)}\leq N
\Big\}.
\end{split}
\label{eqn: def of set CX}
\eeq

\begin{rmk}
\label{rmk: choice of sigma}
The form of the bounds $2(\r(t)+t^{\s})$ in the definition of $\CX$ will be clear below; see \eqref{eqn: estimate for Y in C1gamma} and \eqref{eqn: bounds for Y in C1 and C1gamma}.
Actually, one may take $\s$ to be an arbitrary number in $(0,\f12-\f1p)$.
This fact will be used in the proof of Proposition \ref{prop: continuation} below.
\end{rmk}

\begin{rmk}\label{rmk: time continuity in C^1 topology}
In fact, in the definition of $\CX$, the time-continuity of $X\in \CX$ in the $C^1(\BT)$-topology on $[0,T]$ is guaranteed by the quantitative estimates in \eqref{eqn: def of set CX}.
Indeed, the time-continuity of $X\in \CX$ at $t = 0$ is a consequence of the bound for $\|X(t)-e^{-\f{t}4\Lam}X_0\|_{C_s^1(\BT)}$ and the fact $\lim_{t\to 0^+}\r(t) = 0$, while the time-continuity at $t\in (0,T]$ follows from the interpolation inequalities that are similar to \eqref{eqn: L inf estimate for X tau - X t} and \eqref{eqn: C 1 estimate for X tau - X t}.

As a corollary, $t\mapsto X(t)$ is continuous on $(0,T]$ in the $C^{1,\b}(\BT)$-topology for any $\b\in (0,\g)$.
Indeed, this follows from the (local) $\dot{C}^{1,\g}$-bound on $X(t)$ and the time-continuity of $X$ in the $C^1(\BT)$-topology.
\end{rmk}

The main result of this section is as follows.
\begin{prop}
\label{prop: fixed-point solution}
Suppose $u_0 = u_0(x)\in L^p(\BR^2)$ for some $p\in (2,\infty)$, $X_0 = X_0(s)\in C^1(\BT)$, and they satisfy that $\di u_0 = 0$ and $|X_0|_* >0$.

Let $M_0$, $\lam_0$, and $ Q_0 $ be defined as in \eqref{eqn: def of initial constants}.
Fix $\g\in [\f{2}{p},1)\cap (1-\f2p,1)$, and define $\r=\r(t)$ as in \eqref{eqn: def of rho}.
Take $Q,\,N>0$ as in \eqref{eqn: choice of Q} and \eqref{eqn: choice of N} below respectively.
Then for some $T>0$, which depends on $\g,\,p,\,\nu,\,\lam_0,\, Q_0 ,\, M_0$ and $X_0$, there exists a unique $(u,X)\in \CU\times \CX$ satisfying \eqref{eqn: fixed pt equation for u} and \eqref{eqn: fixed pt equation for X} for all $t\in [0,T]$, where $\CU$ and $\CX$ are defined as in \eqref{eqn: def of set CU} and \eqref{eqn: def of set CX}, respectively.
In addition,
\begin{enumerate}
\item it holds pointwise on $(0,T)\times \BT$ that
\beq
\pa_t X = -\f14\Lam X + g_X
+ \big[h_X^{\nu} + e^{\nu t\Delta}(u_0-u_{X(t)})+B_{\nu}[u]\big]\circ X = u\big( X(s,t),t\big);
\label{eqn: equation for X_t as a fixed point}
\eeq
\item
\[
\|X'\|_{L^\infty_T L^\infty_s(\BT)}\leq 2M_0,
\quad
\inf_{t\in [0,T]}|X(t)|_*\geq \f{\lam_0}2,
\]
and for any $\g'\in [0,\g]$,
\[
\sup_{\eta\in (0,T]} \eta^{\g'} \|X'(\cdot, \eta)\|_{\dot{C}_s^{\g'}(\BT)}
\leq 4M_0.
\]
\end{enumerate}

In particular, $(u,X)$ is a mild solution in the sense of Definition \ref{def: mild solution} to \eqref{eqn: NS equation}-\eqref{eqn: kinematic equation} on $[0,T]$ with the initial condition \eqref{eqn: initial data}.

\begin{rmk}
\label{rmk: choosing a specific gamma}
The short-time characterizations of the mild solution in Theorem \ref{thm: maximal solution} follow by taking e.g.,\;$\g:= \max(\f12+\f1p,1-\f1p)$ and $\g' := \f1p$.
\end{rmk}
\begin{rmk}
\label{rmk: characterization of lifespan T}
As is mentioned before, the $C^1_s(\BT)$-regularity is critical for $X$, so it is natural that $T$ depends implicitly on $X_0$, but not just on its $C^1$-seminorm $M_0$ and its well-stretched constant $\lam_0$.
We point out that such implicit constraints on $T$ arise in (and only in) \eqref{eqn: smallness condition on T}, \eqref{eqn: spatial Holder estimate for v}, \eqref{eqn: bound for time continuity for v}, \eqref{eqn: C 1 gamma estimate for deviation from the semigroup solution}-\eqref{eqn: bounds for Y in C1 and C1gamma}, and \eqref{eqn: final bound for distance bewteen output solutions} below.
They essentially require $(\r_{X_0,\g}(T)+T^{\s})$ and $\phi_{X_0}(T)$ to be smaller than certain constants that depend on $\g,\,p,\,\nu,\,\lam_0,\, Q_0 $, and $M_0$.
Observe that the mappings $X_0\mapsto \r_{X_0,\g}$ and $X_0\mapsto \phi_{X_0}$ are continuous from $C^1_s(\BT)$ to $C([0,+\infty))$, so $T$ can be made to depend continuously on $X_0$ in the $C^1_s(\BT)$-topology.
This fact will be used in Section \ref{sec: global well-posedness}.
\end{rmk}
\begin{rmk}
At this moment, we only claim the uniqueness of such $(u,X)$ in $\CU\times \CX$, but not the uniqueness of the mild solution.
The latter will be addressed in Proposition \ref{prop: uniqueness of mild solution}.
\end{rmk}
\end{prop}

We will prove Proposition \ref{prop: fixed-point solution} by a fixed-point argument.
We start by introducing the metrics on $\CU$ and $\CX$.

For $u,v\in \CU$, define
\beq
d_\CU(u,v) := \sup_{t\in (0,T]} (\nu t)^{\f\g2+\f1p-\f12} \|(u-v)(\cdot, t)\|_{M^{1,1-\g}(\BR^2)},
\label{eqn: def of d_U}
\eeq
where $M^{1,1-\g}(\BR^2)$ denotes the Morrey norm on $\BR^2$ defined in \eqref{eqn: def of Morrey norm}.
Since $\|f\|_{M^{1,1-\g}(\BR^2)}\leq \|f\|_{L_x^{\f{2}{1-\g}}(\BR^2)}$ where $\f{2}{1-\g}\in (p,\infty)$, one can show that $d_\CU(u,v)$ is a well-defined metric on $\CU$ by interpolation.

For $X,Y\in\CX$, we let
\beq
\begin{split}
d_\CX(X,Y) := &\; \|X'-Y'\|_{L_T^\infty L_s^\infty(\BT)}
+ \sup_{\eta\in (0,T]} \eta^\g \|(X'-Y')(\cdot, \eta)\|_{\dot{C}^\g_s(\BT)}\\
&\; + \sup_{\eta\in (0,T)} \eta^{\f1p} \|\pa_t (X-Y)(\cdot, \eta)\|_{L^\infty_s(\BT)}
+ \sup_{\eta\in (0,T)} \eta^\g \|\pa_t (X-Y)(\cdot, \eta)\|_{C^\g_s(\BT)}.
\end{split}
\label{eqn: def of d_X}
\eeq
Given the definition of $\CX$, this is a well-defined metric on $\CX$.

\begin{lem}
$\CU\times \CX$ equipped with the metric
\beq
d_{\CU\times \CX}((u,X),(v,Y)): = d_\CU(u,v) + d_\CX(X,Y)
\label{eqn: metric on U times X}
\eeq
is a complete metric space.
\begin{proof}
It suffices to show that $(\CU,d_\CU)$ and $(\CX,d_\CX)$ are both complete.
This stems from the compactness provided by the regularity assumptions on the elements in $\CU$ and $\CX$.

Take a Cauchy sequence $\{u_n\}_{n\in \BZ_+}\subset \CU$ under the metric $d_\CU$.
By definition \eqref{eqn: def of set CU}, $\{u_n\}_{n\in \BZ_+}$ is uniformly bounded in $L^\infty([0,T];L^p(\BR^2))$, and it is also uniformly locally H\"{o}lder continuous in $\BR^2\times (0,T]$.
By Arzel\`{a}-Ascoli theorem and Banach-Alaoglu theorem,
there exists a subsequence $\{u_{n_k}\}_{k\in \BZ_+}$ and $u_*:\BR^2\times (0,T]\to \BR^2$ such that, as $k\to +\infty$, $u_{n_k}$ converges to $u_*$ in the weak-$*$ topology of $L^\infty([0,T]; L^p(\BR^2))$, and also converges to $u_*$ locally uniformly in $\BR^2\times (0,T]$.
It is then straightforward to verify that $u_*\in \CU$.

Next we show that $u_n\to u_*$ as $n\to +\infty$ under the metric $d_\CU$.
For any $x\in \BR^2$, $t\in (0,T]$ and $r>0$, by the dominated convergence theorem,
\begin{align*}
r^{-1-\g}\int_{B(x,r)}|u_{n_k}-u_*|(y,t)\,dy
= &\;\lim_{l\to +\infty} r^{-1-\g}\int_{B(x,r)}|u_{n_k}-u_{n_l}|(y,t)\,dy\\
\leq &\; \limsup_{l\to +\infty}\|(u_{n_k}-u_{n_l})(\cdot,t)\|_{M^{1,1-\g}(\BR^2)}.
\end{align*}
This implies
\[
d_\CU (u_{n_k},u_*)
\leq
\limsup_{l\to +\infty}d_\CU (u_{n_k},u_{n_l}),
\]
and thus $\lim_{k\to +\infty}d_\CU(u_{n_k},u_*) = 0$.
Since $\{u_n\}_{n\in \BZ_+}$ is a Cauchy sequence under the metric $d_\CU$, the whole sequence must converge to $u_*\in \CU$.
This shows that $(\CU,d_\CU)$ is a complete metric space.

Now take a Cauchy sequence $\{X_n\}_{n\in \BZ_+}\subset \CX$ under the metric $d_\CX$.
Let $X_*$ be the $C_{loc}^1((0,T)\times \BT)$-limit of $\{X_n\}$; indeed, the definition of $d_\CX$ guarantees that this limit is well-defined.
We additionally define $X_*(s,0) = X_0(s)$, and let $X_*(s,T)$ be the $C^1(\BT)$-limit of $X_*(\cdot,t)$ as $t\to T^-$.
Note that the well-definedness of $X_*(s,T)$ and the time-continuity of $X_*$ in the $C^1(\BT)$-topology is implied by the quantitative estimates for $X_*$ (see Remark \ref{rmk: time continuity in C^1 topology}).
It is then not difficult to verify that $X_*\in \CX$, and $d_\CX(X_n,X_*) \to 0$ as $n\to +\infty$.
Therefore, $(\CX,d_\CX)$ is complete.
\end{proof}
\end{lem}

Given $(u,X)\in \CU\times \CX$, using the notations in Section \ref{sec: def of basic quantities}, we let (cf.\;\eqref{eqn: fixed pt equation for u})
\beq
v(t) := u_X^\nu(t) + e^{\nu t\Delta}u_0 +B_{\nu}[u](t).
\label{eqn: def of v}
\eeq
and (cf.\;\eqref{eqn: fixed pt equation for X})
\beq
Y(t):=e^{-\f{t}4 \Lam}X_0+\CI \big[W[u,X; \nu,u_0]\big](t),
\label{eqn: def of Y}
\eeq
where $W$ and $\CI[W]$ were defined in \eqref{eqn: def of w} and \eqref{eqn: def of the operator I}, respectively.
Define the solution mapping
\[
\CT:\,(u,X)\mapsto (v,Y).
\]
In what follows, we will show that $\CT$ is a contraction mapping from $\CU \times \CX$ to itself, provided that $T$,
$Q$, and $N$ are suitably chosen.
This gives a fixed-point of $\CT$, and thus a local mild solution to \eqref{eqn: NS equation}-\eqref{eqn: initial data} in the sense of Definition \ref{def: mild solution}.

\subsection{Bounds for the mapping $\CT$}
\label{sec: bounds for the mapping T}
In this subsection, we will show that by choosing $T$, $Q$, and $N$ suitably, $\CT$ maps $\CU\times \CX$ into itself.

Take an arbitrary $(u,X)\in \CU\times \CX$.
For $t\in [0,T]$,
\[
\|X(\cdot, t)-X_0\|_{\dot{C}_s^1(\BT)}
\leq \|X(\cdot, t)-e^{-\f{t}{4}\Lam}X_0\|_{\dot{C}_s^1(\BT)}
+ \|X_0-e^{-\f{t}{4}\Lam}X_0\|_{\dot{C}_s^1(\BT)}.
\]
Recall that $\r = \r(t)$ and $\phi = \phi(t)$ were defined in \eqref{eqn: def of rho} and \eqref{eqn: def of phi}, respectively.
We choose $T\in(0,1)$ sufficiently small, which depends on $X_0$, $\g$, and $\lam_0$, such that
\beq
2\big(\r(T)+ T^{\s}\big) \leq \f{\lam_0}{16},\quad \mbox{and}\quad \phi(T)\leq \f{\lam_0}{16}.
\label{eqn: smallness condition on T}
\eeq
As a result, by the preceding estimate and \eqref{eqn: def of set CX}, for all $t\in [0,T]$,
\beq
\|X(\cdot, t)-X_0\|_{\dot{C}^1(\BT)}\leq \f{\lam_0}{8},
\label{eqn: X is close to initial data}
\eeq
which implies
\beq
|X(\cdot,t)|_*\geq |X_0|_*- \f{\lam_0}{8} \geq \f{\lam_0}{2}=:\lam.
\label{eqn: lower bound for |X|_* for elements in CX}
\eeq
In view of \eqref{eqn: def of O_T^M lam} and the definition of $\CX$ in \eqref{eqn: def of set CX}, $\CX\subset O_T^{\lam}$ with $\lam = \f{\lam_0}2$.
We also have the simple bounds
\beq
\|X'\|_{L^\infty_T L^\infty_s(\BT)}\leq M_0+\f{\lam_0}{8}\leq 2M_0,
\quad
\sup_{\eta\in (0,T]} \eta^\g \|X'(\cdot, \eta)\|_{\dot{C}_s^{\g}(\BT)}
\leq \f{\lam_0}{16} \leq \f{M_0}{16},
\label{eqn: simple bound for X in CX}
\eeq
and by interpolation, for all $\g'\in [0,\g]$,
\beq
\sup_{\eta\in (0,T]} \eta^{\g'} \|X'(\cdot, \eta)\|_{\dot{C}_s^{\g'}(\BT)}
\leq \left(\sup_{\eta\in (0,T]} \eta^{\g} \|X'(\cdot, \eta)\|_{\dot{C}_s^{\g}}\right)^{\f{\g'}{\g}}
\big(2\|X'\|_{L^\infty_T L_s^\infty}\big)^{1-\f{\g'}{\g}}
\leq 4M_0.
\label{eqn: simple bound 2 for X in CX}
\eeq
Moreover, by \eqref{eqn: def of set CX}, for any $t_1,t_2\in [0,T]$,
\[
\|X'(\cdot, t_1)-X'(\cdot, t_2)\|_{L_s^\infty(\BT)}\leq \f{\lam_0}{4} = \f{\lam}{2}.
\]

Let $(v,Y):=\CT (u,X)$.
We start from estimating $v$.
By \eqref{eqn: def of v}, \eqref{eqn: simple bound for X in CX}, Lemma \ref{lem: estimate for u_X^nu}, and \eqref{eqn: L p estimate for B} in Lemma \ref{lem: estimate for B_nu u}, for any $t\in [0,T]$,
\beq
\begin{split}
\|v(\cdot, t)\|_{L^p_x(\BR^2)}
\leq &\; \|u_X^\nu(t)\|_{L_x^p(\BR^2)}
+ \|e^{\nu t\D}u_0\|_{L_x^p(\BR^2)} + \|B_\nu[u](t)\|_{L^p_x(\BR^2)}\\
\leq &\; C\lam^{-1+\f2p} \|X'\|_{L^\infty_t L_s^2}^{\f4p} \|X'\|_{L^\infty_t L_s^\infty}^{2-\f4p}
+ \|u_0\|_{L_x^p}\\
&\; + C\nu^{-1} (\nu t)^{\f1p-\f12}
\left(\sup_{\eta\in(0,t]}(\nu \eta)^{\f12-\f1p}\|u(\eta)\|_{L_x^p}\right)^2
\\
\leq &\; C_0 \lam_0^{-1+\f2p} M_0^2 +  Q_0  + C \nu^{-1} (\nu t)^{\f12-\f1p}  Q^2,
\end{split}
\label{eqn: L p estimate for v}
\eeq
where $C_0$ and $C$ are universal constants depending only on $p$.
Here we used the fact that $\lam = \lam_0/2$.

Similarly, by \eqref{eqn: def of v}, Lemma \ref{lem: estimate for u_X^nu}, and \eqref{eqn: L inf estimate for B} in Lemma \ref{lem: estimate for B_nu u}, with $\g'=\f1{2p} \in (0,\g)$, for any $t\in (0,T]$,
\beq
\begin{split}
&\; (\nu t)^{\f1p}\|v(\cdot, t)\|_{L_x^{\infty}(\BR^2)}\\
\leq &\; (\nu t)^{\f1p} \Big(\|u_X^\nu(t)\|_{L_x^\infty(\BR^2)}
+ \|e^{\nu t\D}u_0\|_{L_x^\infty(\BR^2)} + \|B_\nu[u](t)\|_{L_x^\infty(\BR^2)}\Big)\\
\leq &\; C(\nu t)^{\f1p}\cdot  \lam^{-1} t^{-\g'} \|X'\|_{L^\infty_t L_s^{\infty}}
\sup_{\eta\in (0,t]} \eta^{\g'} \|X'(\eta)\|_{\dot{C}_s^{\g'}} + C(\nu t)^{\f1p} (\nu t)^{-\f1p} \|u_0\|_{L_x^p}
\\
&\; + C (\nu t)^{\f1p}\cdot \nu^{-1} (\nu t)^{-\f12}\sup_{\eta\in(0,t]}(\nu \eta)^{\f12-\f1p}\|u(\eta)\|_{L_x^p}\cdot
\sup_{\eta\in(0,t]}(\nu \eta)^{\f12}\|u(\eta)\|_{L_x^\infty}
\\
\leq &\; C(\nu t)^{\f1p}\cdot  \lam^{-1} t^{-\f{1}{2p}} M_0^2
+ C_1  Q_0
+ C \nu^{-1} (\nu t)^{\f12-\f1p}Q^2,
\end{split}
\label{eqn: spatial L inf estimate for v}
\eeq
where $C$ and $C_1$ are universal constants that only depend on $p$.
In the last line, we used \eqref{eqn: simple bound for X in CX}, \eqref{eqn: simple bound 2 for X in CX}, and the fact $\g' = \f{1}{2p}$.
Moreover,
\beq
\begin{split}
&\; (\nu t)^{\g}\|v(\cdot, t)\|_{\dot{C}_x^\g(\BR^2)}\\
\leq &\; (\nu t)^\g \Big(\|u_X^\nu(t)\|_{\dot{C}_x^\g(\BR^2)}
+ \|e^{\nu t\D}u_0\|_{\dot{C}_x^\g(\BR^2)} + \|B_\nu[u](t)\|_{\dot{C}_x^\g(\BR^2)}\Big)\\
\leq &\; C(\nu t)^\g \lam^{-1-\g} t^{-\g} \|X'\|_{L^\infty_t L_s^{\infty}}
\sup_{\eta\in (0,t]} \eta^\g \|X'(\eta)\|_{\dot{C}_s^\g} + C(\nu t)^\g (\nu t)^{-\f1p-\f\g2} \|u_0\|_{L_x^p}
\\
&\; + C (\nu t)^\g \cdot \nu^{-1}(\nu t)^{-(1+\g)/2}
\left(\sup_{\eta\in (0,t]} (\nu \eta)^{\f12-\f1p} \|u(\eta)\|_{L^p_x}+\sup_{\eta\in (0,t]}(\nu \eta)^{1/2}\|u(\eta)\|_{L^\infty_x}\right)^2
\\
\leq &\; C\nu^\g \lam^{-1-\g} M_0 \big(\r(T)+ T^{\s}\big)
+ C_2 (\nu t)^{\f\g2-\f1p}  Q_0
+ C \nu^{-1}(\nu t)^{\f{1+\g}{2}-\f2p} Q^2,
\end{split}
\label{eqn: spatial Holder estimate for v}
\eeq
where $C$ and $C_2$ are universal constants that depend on $p$ and $\g$.

Regarding the time-continuity of $v$, we derive by Lemma \ref{lem: time Holder continuity of u_X^nu}, Lemma \ref{lem: time Holder continuity of B_u}, and Lemma \ref{lem: parabolic estimates} that, for any $x\in \BR^2$ and $0 < t_1< t_2\leq T$,
\beq
\begin{split}
&\;|v(x, t_1)-v(x, t_2)|\\
\leq &\; \|u_X^\nu(t_1)-u_X^\nu(t_2)\|_{L^\infty_x} + \|(e^{\nu (t_2-t_1)\D}-\mathrm{Id})e^{\nu t_1\D}u_0 \|_{L^\infty_x} + \|B_\nu[u](t_1) - B_\nu[u](t_2)\|_{L^\infty_x}\\
\leq &\;
C\big(\nu |t_1-t_2|\big)^{\g/2} t_1^{-\g} \cdot \lam^{-1-\g} \|X'\|_{L^\infty_T L_s^{\infty}} \sup_{\eta\in (0,T]} \eta^\g \|X'(\eta)\|_{\dot{C}_s^{\g}} \\
&\; + C(\nu|t_1-t_2|)^{\g/2} (\nu t_1)^{-\f1{p}-\f{\g}{2}} \|u_0\|_{L^p_x} \\
&\; +
C \nu^{-1} \big(\nu(t_2-t_1)\big)^{\g/2}(\nu t_1)^{-\f{1+\g}{2}}
\left(\sup_{\eta\in (0,t_1]} (\nu \eta)^{\f12-\f1p} \|u(\eta)\|_{L^p_x}+\sup_{\eta\in (0,t_1]}(\nu \eta)^{\f12}\|u(\eta)\|_{L^\infty_x}\right)^2\\
&\;
+ C \nu^{-1} \big(\nu(t_2-t_1)\big)^{\g/2}(\nu t_2)^{-\f{1+\g}{2}}
\left(\sup_{\eta\in (0,t_2]} (\nu \eta)^{\f12-\f1p} \|u(\eta)\|_{L^p_x}+\sup_{\eta\in (0,t_2]}(\nu \eta)^{\f12}\|u(\eta)\|_{L^\infty_x}\right)^2\\
\leq &\;
C\big(\nu |t_1-t_2|\big)^{\g/2} (\nu t_1)^{-\g}
\left[ \nu^\g \lam^{-1-\g} M_0 \big(\r(T)+T^{\s}\big)
+ (\nu t_1)^{\f{\g}{2}-\f1p} Q_0
+ \nu^{-1} (\nu t_2)^{\f{1+\g}{2}-\f2p} Q^2
\right],
\end{split}
\label{eqn: bound for time continuity for v prelim}
\eeq
where the constant $C$ depends on $p$ and $\g$.
Since $p\in (2,\infty)$ and $\g \geq \f2p$, for any $0< t_1 < t_2\leq T$,
\beq
\begin{split}
&\;(\nu t_1)^\g \cdot \f{|v(x, t_1)-v(x, t_2)|}{(\nu|t_1-t_2|)^{\g/2}}\\
\leq &\;
C \nu^\g \lam^{-1-\g} M_0 \big(\r(T)+T^{\s}\big)
+ C_3(\nu T)^{\f{\g}{2}-\f1p} Q_0
+ C\nu^{-1} (\nu T)^{\f{1+\g}{2}-\f2p} Q^2,
\end{split}
\label{eqn: bound for time continuity for v}
\eeq
where $C$ and $C_3$ are universal constants that depend on $p$ and $\g$.

In view of these bounds, we shall choose from now on
\beq
Q: =
\begin{cases}
C_0 \lam_0^{-1+\f2p} M_0^2 + (C_1 + 1)  Q_0 , & \mbox{if } \g>\f2p, \\
C_0 \lam_0^{-1+\f2p} M_0^2 + (C_1 + C_2+ C_3 + 1)  Q_0 , & \mbox{if } \g = \f2p.
\end{cases}
\label{eqn: choice of Q}
\eeq
where $C_0$, $C_1$, $C_2$ and $C_3$ were defined in \eqref{eqn: L p estimate for v}-
\eqref{eqn: bound for time continuity for v}, respectively.
We note that if $\g = \f{2}{p}$, $C_2$ and $C_3$ only depend on $p$, so $Q$ is always determined by $p$, $\lam_0$, $M_0$, and $ Q_0 $ only.
Since $p\in (2,\infty)$ and $\g \in[\f2p,1)$, by combining this choice of $Q$ with the estimates above, we may further assume $T \ll 1$ to achieve that
\[
\sup_{t\in [0,T]}\|v(t)\|_{L^p_x(\BR^2)}
\leq Q,\quad
\sup_{t\in (0,T]} (\nu t)^{\f1p}\|v(t)\|_{L^\infty_x(\BR^2)}
\leq Q,\quad
\sup_{t\in (0,T]} (\nu t)^\g \|v(t)\|_{\dot{C}^\g_x(\BR^2)}
\leq Q,
\]
and
\beq
\sup_{x\in \BR^2} \sup_{0<t_1<t_2\leq T} (\nu t_1)^\g \cdot \f{|v(x,t_1)-v(x,t_2)|}{(\nu(t_1-t_2))^{\g/2}}
\leq Q.
\label{eqn: time continuity of v}
\eeq
In summary, we can guarantee $v\in \CU$ by choosing $Q$ as in \eqref{eqn: choice of Q} and assuming $T\ll 1$, where the required smallness of $T$ depends on $\nu$, $p$, $\g$, $\lam_0$, $M_0$, $ Q_0 $, and $X_0$.

Next we turn to estimating $Y$ given in \eqref{eqn: def of Y}.
It suffices to bound $W[u,X; \nu,u_0]$ defined in \eqref{eqn: def of w}; for brevity, we shall write it as $W[u,X]$ or even $W$ in the sequel.

For that purpose, we first estimate $h_X^\nu$ by incorporating the bounds for $X$ with Lemma \ref{lem: estimate for h_X^nu}.
\begin{lem}
\label{lem: L inf and W 1 inf estimate for h_X^nu with optimal power of nu}
Suppose $T\leq 1$.
Let $X\in \CX$.
Then for any $t\in [0,T]$,
\[
\|h_X^\nu(t)\|_{L_x^\infty(\BR^2)}
\leq C (\nu t)^{-\f1p}
\lam^{-1+\f2p}
\left(M_0^{2-\g} N^{\g} +\lam^{-1}M_0^2N \cdot t^{\f1p} \right),
\]
and
\beq
\begin{split}
\|\na h_X^\nu(t)\|_{L_x^\infty(\BR^2)}
\leq &\; C(\nu t)^{-\f{1}{2}-\f1p} \lam^{-1+\f2p}
\left[ 1
+ \mathds{1}_{\{\nu t\geq 4\lam^2\}} \big(\lam^{-2}\nu t\big)^{\f{1}{2}+\f1p-\g} \right]\\
&\;\cdot\left( M_0^{2-\g} N^{\g} +\lam^{-1} M_0^2N \cdot t^{\f1p}\right),
\end{split}
\label{eqn: W 1 inf estimate for h}
\eeq
where the constant $C$ depends on $p$ and $\g$.

As a result,
\begin{align*}
\|h_X^\nu(t)\|_{\dot{C}_x^\g(\BR^2)}
\leq &\; C(\nu t)^{-\f\g 2-\f1p} \lam^{-1+\f2p}  \left[1
+ \mathds{1}_{\{\nu t\geq 4\lam^2\}}
 \big(\lam^{-2} \nu t\big)^{\g(\f1 2+\f1 p-\g)}\right] \\
&\; \cdot
 \left( M_0^{2-\g} N^{\g}
+ \lam^{-1} M_0^2N \cdot t^{\f1p}\right),
\end{align*}
where the constant $C$ depends on $p$ and $\g$.

\begin{proof}
We apply \eqref{eqn: L inf and W 1 inf estimate for h} in Lemma \ref{lem: estimate for h_X^nu} with the parameters there chosen as follows:
\[
(\g',\mu_0,\mu,\mu'):=\left(1-\f2p,\, \f1p,\, \g,\, 1-\f1p\right).
\]
Then, for $k = 0,1$,
\begin{align*}
&\; \|\na^k h_X^\nu(t)\|_{L_x^\infty(\BR^2)}
\\
\leq &\; C\left[ (\nu t)^{-\f{k}{2}} \min\left(\lam^{-1},\,(\nu t)^{-\f{1}{2}}\right)
+ \mathds{1}_{\{\nu t\geq 4\lam^2\}} (\nu t)^{-\g} \lam^{-1-k+2\g}  \right] M_0^{2-\g} N^{\g} \\
&\; + C (\nu t)^{-\f{k}{2}-\f1p} \min\left(\lam^{-1+\f2p},\, (\nu t)^{-\f12+\f1p}\right) t^{\f1p} \cdot \left[1+\mathds{1}_{\{\nu t\geq 4\lam^2\}}\int_{2\lam^2/(\nu t)}^1 \zeta^{-\f{1+k}{2}} \,d\zeta \right]\\
&\;\quad \cdot \lam^{-1} \left(M_0^2N + M_0^2 \sup_{\eta\in (0,t]}\eta^{\mu'}\|\pa_t X(\eta)\|_{\dot{C}_s^{\g'}} \right),
\end{align*}
where $C$ depends on $k$, $p$, and $\g$.
By interpolation and the fact $t\leq 1$,
\begin{align*}
&\; \sup_{\eta\in (0,t]}\eta^{\mu'}\|\pa_t X(\eta)\|_{\dot{C}_s^{\g'}} = \sup_{\eta\in (0,t]}\eta^{\g'+\f1p}\|\pa_t X(\eta)\|_{\dot{C}_s^{\g'}}\\
\leq &\; \left( 2\sup_{\eta\in (0,t]}\eta^{\f1p}\|\pa_t X(\eta)\|_{L_s^\infty}\right)^{1-\f{\g'}{\g}}
\left(t^{\f1p}\sup_{\eta\in (0,t]}\eta^{\g}\|\pa_t X(\eta)\|_{\dot{C}_s^{\g}}\right)^{\f{\g'}{\g}}
\leq 2N.
\end{align*}
Hence,
\begin{align*}
\|\na^k h_X^\nu(t)\|_{L_x^\infty(\BR^2)}
\leq &\; C\left[ (\nu t)^{-\f{k}{2}-\f1p} \lam^{-1+\f2p}
+ \mathds{1}_{\{\nu t\geq 4\lam^2\}} (\nu t)^{-\g} \lam^{-1-k+2\g}  \right] M_0^{2-\g} N^{\g} \\
&\; + C (\nu t)^{-\f{k}{2}-\f1p} \min\left(\lam^{-1+\f2p},\, (\nu t)^{-\f12+\f1p}\right) t^{\f1p} \cdot \lam^{-1} M_0^2N\\
&\;\quad  \cdot \left[1+\mathds{1}_{\{\nu t\geq 4\lam^2\}}\int_{2\lam^2/(\nu t)}^1 \zeta^{-\f{1+k}{2}} \,d\zeta \right].
\end{align*}
If $k = 0$, with in mind that $\g >\f1p$, we have
\[
\|h_X^\nu(t)\|_{L_x^\infty(\BR^2)}
\leq C(\nu t)^{-\f1p}  \lam^{-1+\f2p} M_0^{2-\g} N^{\g}
+ C (\nu t)^{-\f1p} \lam^{-1+\f2p}\cdot t^{\f1p} \cdot \lam^{-1} M_0^2N.
\]
If $k = 1$,
\begin{align*}
\|\na h_X^\nu(t)\|_{L_x^\infty(\BR^2)}
\leq &\; C\left[ (\nu t)^{-\f{1}{2}-\f1p} \lam^{-1+\f2p}
+ \mathds{1}_{\{\nu t\geq 4\lam^2\}} (\nu t)^{-\g} \lam^{-2+2\g}  \right] M_0^{2-\g} N^{\g} \\
&\; + C (\nu t)^{-\f12-\f1p} \lam^{-1+\f2p} \min\big(1,\, \lam^2(\nu t)^{-1}\big)^{\f12-\f1p} t^{\f1p} \cdot \lam^{-1} M_0^2N
\ln \left(2+\f{\nu t}{\lam^2}\right).
\end{align*}
Since
\[
\min\big(1,\, \lam^2(\nu t)^{-1}\big)^{\f12-\f1p} \ln \left(2+\f{\nu t}{\lam^2}\right)
\leq C\left(1+\f{\nu t}{\lam^2}\right)^{-\f12+\f1p} \ln \left(2+\f{\nu t}{\lam^2}\right)\leq C
\]
for some constant $C$ that only depends on $p$, we obtain that
\begin{align*}
\|\na h_X^\nu(t)\|_{L_x^\infty(\BR^2)}
\leq &\; C\left[ (\nu t)^{-\f{1}{2}-\f1p} \lam^{-1+\f2p}
+ \mathds{1}_{\{\nu t\geq 4\lam^2\}} (\nu t)^{-\g} \lam^{-2+2\g}  \right] M_0^{2-\g} N^{\g} \\
&\; + C (\nu t)^{-\f{1}{2}-\f1p} \lam^{-1+\f2p} \cdot \lam^{-1} M_0^2N \cdot t^{\f1p}
\\
\leq
&\; C(\nu t)^{-\f{1}{2}-\f1p} \lam^{-1+\f2p}
\left[ 1
+ \mathds{1}_{\{\nu t\geq 4\lam^2\}} \big(\lam^{-2}\nu t\big)^{\f{1}{2}+\f1p-\g} \right]\\
&\;\cdot\left( M_0^{2-\g} N^{\g} +\lam^{-1} M_0^2N \cdot t^{\f1p}\right),
\end{align*}
where $C$ depends on $p$ and $\g$.

Finally, the H\"{o}lder estimate follows from the interpolation inequality
\begin{align*}
\|h_X^\nu(t)\|_{\dot{C}_x^\g(\BR^2)}
\leq &\;
\big(2\|h_X^\nu(t)\|_{L_x^\infty(\BR^2)}\big)^{1-\g}\|\na h_X^\nu(t)\|_{L_x^\infty(\BR^2)}^{\g}.
\end{align*}
\end{proof}
\end{lem}

Recall that $W[u,X]$ was defined in \eqref{eqn: def of w}.
By applying Lemma \ref{lem: estimate for u_X}, \eqref{eqn: L inf estimate for B} in Lemma \ref{lem: estimate for B_nu u}, Lemma \ref{lem: improved estimates for g_X}, Lemma \ref{lem: L inf and W 1 inf estimate for h_X^nu with optimal power of nu}, and Lemma \ref{lem: parabolic estimates}, we derive with $\g':= \f1p\in (0,\g)$ that, for $t\in (0,T]$,
\beq
\begin{split}
&\; \|W[u,X](t)\|_{L^\infty_s(\BT)}\\
\leq
&\;\|g_X(t)\|_{L^\infty_s(\BT)}
+ \big\|h_X^{\nu}(t)+ e^{\nu t\D}(u_0-u_{X(t)})+ B_{\nu}[u](t) \big\|_{L^\infty_x(\BR^2)}\\
\leq &\; C\lam^{-2} \|X'(t)\|_{L_s^\infty}^2 \|X'(t)\|_{\dot{C}_s^{\g'}}
+ C(\nu t)^{-\f1p} \lam^{-1+\f2p} \left( M_0^{2-\g} N^{\g}
+ \lam^{-1} M_0^2N \cdot t^{\f1p}\right)\\
&\; + C (\nu t)^{-\f1p} \| u_0-u_{X(t)} \|_{L^p_x} \\
&\; + C\nu^{-1} (\nu t)^{-\f12}\sup_{\eta\in(0,t]}(\nu \eta)^{\f12-\f1p}\|u(\eta)\|_{L_x^p}\cdot
\sup_{\eta\in(0,t]}(\nu \eta)^{\f12}\|u(\eta)\|_{L_x^\infty}
\\
\leq &\; Ct^{-\f1p}\lam^{-2} M_0^3
+ C(\nu t)^{-\f1p} \lam^{-1+\f2p} \left( M_0^{2-\g} N^{\g}
+ \lam^{-1} M_0^2N \cdot t^{\f1p}\right)\\
&\; + C (\nu t)^{-\f1p}\left(  Q_0  + \lam^{-1+\f2p}M_0^2\right)
+ C\nu^{-1}(\nu t)^{\f12-\f2p}Q^2,
\end{split}
\label{eqn: L inf estimate for w}
\eeq
where the constant $C$ depends on $p$ and $\g$.
Combining this with \eqref{eqn: def of Y} and Lemma \ref{lem: parabolic estimates for fractional Laplace}, we find that
\beq
\begin{split}
\|Y(\cdot,t)- e^{-\f{t}4 \Lam}X_0\|_{L^\infty_s(\BT)}
\leq &\; \int_0^t \|W[u,X](\tau)\|_{L^\infty_s}\, d\tau\\
\leq &\;
Ct^{1-\f1p}\left[\lam^{-2} M_0^3
+ \nu^{-\f1p} \lam^{-1+\f2p} \left( M_0^{2-\g} N^{\g}
+ \lam^{-1} M_0^2N \cdot t^{\f1p}\right)\right.\\
&\;\left. \qquad \quad + \nu^{-\f1p}\left(  Q_0  + \lam^{-1+\f2p}M_0^2\right)
+  \nu^{-\f12-\f2p} t^{\f12-\f1p}Q^2\right],
\end{split}
\label{eqn: L inf difference from initial data}
\eeq
where $C$ depends on $p$ and $\g$.

Take $\b\in (0,1)$ to be chosen later.
Recall that $\g>\f12$, so $0<\f{\b}2< \g$.
By Lemma \ref{lem: improved estimates for g_X}, the interpolation inequality, and the definition of $\CX$,
\beq
\begin{split}
\|g_X(t)\|_{\dot{C}^\b_s(\BT)}
\leq &\; C \lam^{-2} \|X'(t)\|_{L_s^\infty}\|X'(t)\|_{\dot{C}^{\b/2}_s}^2
\\
\leq &\; C \lam^{-2} \|X'(t)\|_{L_s^\infty}^{3-\f{\b}{\g}}\|X'(t)\|_{\dot{C}^{\g}_s}^{\f{\b}{\g}}
\leq Ct^{-\b} \lam^{-2} M_0^{3-\f{\b}{\g}} \left(\sup_{\eta\in (0,t]} \eta^\g \|X'(\eta)\|_{\dot{C}^{\g}_s}\right)^{\f{\b}{\g}}
\\
\leq &\; Ct^{-\b} \lam^{-2} M_0^{3-\f{\b}{\g}} \big(\r(t)+t^{\s}\big)^{\f{\b}{\g}},
\end{split}
\label{eqn: C beta estimate for g_X}
\eeq
where the constant $C$ depends on $\b$ and $\g$.

By Lemma \ref{lem: parabolic estimates} and Lemma \ref{lem: estimate for u_X},
\beq
\big\|e^{\nu t\D}(u_0-u_{X(t)})\big\|_{\dot{C}_x^\g(\BR^2)}
\leq
C (\nu t)^{-\f{\g}2-\f1p} \|u_0-u_{X(t)}\|_{L^p_x}
\leq C (\nu t)^{-\f{\g}2-\f1p} \left(  Q_0  + \lam^{-1+\f2p} M_0^2\right),
\label{eqn: C gamma estimate for semigroup and initial term}
\eeq
where $C$ depends on $p$ and $\g$, and similarly,
\[
\big\|\na e^{\nu t\D}(u_0-u_{X(t)})\big\|_{L^\infty_x(\BR^2)}
\leq C (\nu t)^{-\f12-\f1p} \left(  Q_0  + \lam^{-1+\f2p} M_0^2\right),
\]
where $C$ only depends on $p$.

Finally, by \eqref{eqn: 6th estimate for B new form general gamma} in Lemma \ref{lem: estimate for B_nu u},
\beq
\begin{split}
\|B_{\nu}[u](t)\|_{\dot{C}^{\g}_x(\BR^2)}
\leq
&\; C\nu^{-1}(\nu t)^{-\f{1+\g}2} \left(\sup_{\eta\in (0,t]} (\nu \eta)^{\f12-\f1p} \|u(\eta)\|_{L^p_x}
+\sup_{\eta\in (0,t]}(\nu \eta)^{\f12}\|u(\eta)\|_{L^\infty_x}\right)^2
\\
\leq &\; C \nu^{-1} (\nu t)^{\f{1-\g}2-\f2p} Q^2,
\end{split}
\label{eqn: C gamma estimate for B}
\eeq
where $C$ depends on $p$ and $\g$.
Now by \eqref{eqn: def of Y} and Lemma \ref{lem: parabolic estimates for Duhamel term of fractional Laplace}, with $\b:=\f{1+\g}{2}\in (\g,1)$, it holds for any $t\in [0,T]$ that
\beq
\begin{split}
&\; \|Y(\cdot,t)- e^{-\f{t}4 \Lam}X_0\|_{\dot{C}_s^1(\BT)}
+ t^\g \|Y(\cdot, t) - e^{-\f{t}4 \Lam}X_0\|_{\dot{C}^{1,\g}_s(\BT)}\\
\leq &\; C \sup_{\eta\in (0,t]} \eta^\b \|g_X(\eta)\|_{\dot{C}^\b_s(\BT)}
+  Ct^{\f12-\f1p} \sup_{\eta\in (0,t]} \eta^{\f12+\f1p} \big\|\big[h_X^{\nu} + e^{\nu \eta\Delta}(u_0-u_{X(\eta)}) \big]\circ X(\eta)\big\|_{\dot{C}^1_s(\BT)}\\
&\;
+  C\sup_{\eta\in (0,t]} \eta^\g \big\| B_\nu[u] \circ X(\eta)\big\|_{\dot{C}^{\g}_s(\BT)}.
\end{split}
\label{eqn: estimate for Y-semigroup solution prelim}
\eeq
We remark that we treated different terms on the right-hand side differently not only to help close the bound in \eqref{eqn: estimate for Y in C1gamma} below, but also to achieve the optimal $\nu$-dependence of the estimate in the last two terms, which will be useful in Section \ref{sec: zero Reynolds number limit}.
Combining this with \eqref{eqn: L inf difference from initial data} and all the estimates and also applying Lemma \ref{lem: L inf and W 1 inf estimate for h_X^nu with optimal power of nu} and Lemma \ref{lem: Holder estimate for composition of functions}, we will find that
\beq
\begin{split}
&\; \|Y(\cdot,t)- e^{-\f{t}4 \Lam}X_0\|_{C^1_s(\BT)}
+ t^{\g}\|Y(\cdot,t)- e^{-\f{t}4 \Lam}X_0\|_{\dot{C}^{1,\g}_s(\BT)}
\\
\leq &\; \|Y(\cdot,t)- e^{-\f{t}4 \Lam}X_0\|_{L^\infty_s(\BT)}
+ C \sup_{\eta\in (0,t]} \eta^\b \|g_X(\eta)\|_{\dot{C}^\b_s(\BT)}
\\
&\;
+ Ct^{\f12-\f1p} \sup_{\eta\in (0,t]} \eta^{\f12+\f1p}
\big\|\na h_X^\nu(\eta) + \na e^{\nu \eta\Delta}(u_0-u_{X(\eta)})\big\|_{L^\infty_x(\BT)} \|X'\|_{L^\infty_t L^\infty_s}\\
&\;
+  C\sup_{\eta\in (0,t]} \eta^\g \| B_\nu[u](\eta)\|_{\dot{C}^{\g}_x(\BR^2)} \|X'\|_{L^\infty_t L^\infty_s}^\g
\\
\leq &\; Ct^{1-\f1p}\left[\lam^{-2} M_0^3
+ \nu^{-\f1p} \lam^{-1+\f2p} \left( M_0^{2-\g} N^{\g}
+ \lam^{-1} M_0^2N \cdot t^{\f1p}\right)\right.\\
&\;\left. \qquad \quad + \nu^{-\f1p}\left(  Q_0  + \lam^{-1+\f2p}M_0^2\right)
+  \nu^{-\f12-\f2p} t^{\f12-\f1p}Q^2\right]
\\
&\; + C\lam^{-2} M_0^{3-\f{\b}{\g}} \big(\r(t)+t^{\s}\big)^{\f{\b}{\g}}
+ Ct^{\f12-\f1p} \nu^{-\f12-\f1p} \left(  Q_0  + \lam^{-1+\f2p} M_0^2\right) M_0 \\
&\; + C t^{\f12-\f1p} \nu^{-\f{1}{2}-\f1p} \lam^{-1+\f2p}
\left[ 1
+ \mathds{1}_{\{\nu t\geq 4\lam^2\}} \big(\lam^{-2}\nu t\big)^{\f{1}{2}+\f1p-\g} \right]\\
&\;\quad \cdot\left( M_0^{2-\g} N^{\g} +\lam^{-1} M_0^2N \cdot t^{\f1p}\right)M_0
\\
&\; + C t^{\f{1+\g}2-\f2p} \nu^{-\f{1+\g}2-\f2p} Q^2 M_0^\g.
\end{split}
\label{eqn: C 1 gamma estimate for Y minus semigroup solution}
\eeq
where $C$ is a constant depending on $p$ and $\g$.
Since $\g > 1-\f2p$, $\g\geq \f2p$, and $t\leq 1$, in short we have
\beq
\begin{split}
&\; \|Y(\cdot,t)- e^{-\f{t}4 \Lam}X_0\|_{C^1_s(\BT)}
+ t^{\g}\|Y(\cdot,t)- e^{-\f{t}4 \Lam}X_0\|_{\dot{C}^{1,\g}_s(\BT)}
\\
\leq &\; C(\g,p,\nu,\lam_0, Q_0 ,M_0,N)\left[ \left(\r(t)+t^{\s}\right)^{\b/\g}
+ t^{\f12-\f1p}\right].
\end{split}
\label{eqn: C 1 gamma estimate for deviation from the semigroup solution}
\eeq
Recall that the notation $C = C(\g,p,\nu,\lam_0, Q_0 ,M_0,N)$ represents a universal constant that only depends on $\g$, $p$, $\nu$, $\lam_0$, $ Q_0 $, $M_0$, and $N$.
Here we used the facts that $\lam = \lam_0/2$ and that $Q$ depends only on $p$, $\lam_0$, $M_0$, and $ Q_0 $.
Thus it follows that, for any $t\in (0,T]$,
\beq
\begin{split}
t^{\g}\|Y(\cdot,t)\|_{\dot{C}^{1,\g}_s(\BT)}
\leq &\; t^{\g}\|Y(\cdot,t)- e^{-\f{t}4 \Lam}X_0\|_{\dot{C}^{1,\g}_s(\BT)}
+ t^{\g}\| e^{-\f{t}4 \Lam}X_0\|_{\dot{C}^{1,\g}_s(\BT)}
\\
\leq &\; C(\g,p,\nu,\lam_0, Q_0 ,M_0,N)\left[ \left(\r(t)+t^{\s}\right)^{\b/\g}
+ t^{\f12-\f1p}\right] + \r(t).
\end{split}
\label{eqn: estimate for Y in C1gamma}
\eeq
Since $\b > \g $ and $\s \in (0,\f12-\f1p)$,
we may take $T$ sufficiently small, which depends on $\g$, $p$, $\nu$, $\lam_0$, $ Q_0 $, $M_0$, $N$, and $X_0$, so that for all $t\in (0,T]$,
\beq
\|Y(\cdot,t)- e^{-\f{t}4 \Lam}X_0\|_{C^1_s(\BT)}
+ t^{\g}\|Y(\cdot,t)\|_{\dot{C}^{1,\g}_s(\BT)}
\leq 2\big(\r(t)+t^{\s}\big).
\label{eqn: bounds for Y in C1 and C1gamma}
\eeq
For future use, we point out that, with this choice of $T$ as well as \eqref{eqn: smallness condition on T}, $Y$ also satisfies \eqref{eqn: simple bound for X in CX} and \eqref{eqn: simple bound 2 for X in CX}.

To bound $\pa_t Y$, we first justify
\beq
\pa_t Y = -\f14\Lam Y + W[u,X].
\label{eqn: equation for Y}
\eeq
Observe from \eqref{eqn: Stokes velocity along the string}, \eqref{eqn: u_X^nu in terms of u_X}, \eqref{eqn: def of w}, and \eqref{eqn: def of v} that
\beq
W[u,X](t) = v\big(X(t),t\big)+\f14 \Lam X(t).
\label{eqn: alternative expression of W}
\eeq
By the regularity of $v$ shown in \eqref{eqn: spatial Holder estimate for v}, the time-continuity of $v$ shown in \eqref{eqn: time continuity of v}, and the time-continuity of $X\in \CX$ noted in Remark \ref{rmk: time continuity in C^1 topology}, we can show that $t\mapsto W[u,X](t)$ is continuous in the $C(\BT)$-topology on $(0,T)$.
From \eqref{eqn: C beta estimate for g_X}-\eqref{eqn: C gamma estimate for B} and Lemma \ref{lem: L inf and W 1 inf estimate for h_X^nu with optimal power of nu}, we can derive a local $C^\b$-bound for $W[u,X]$ with $\b\in [\f2p,\g]$ as follows (cf.\;\eqref{eqn: C 1 gamma estimate for Y minus semigroup solution}):
\beq
\begin{split}
&\;\|W[u,X](t)\|_{\dot{C}_s^\b(\BT)}\\
\leq &\; Ct^{-\b} \lam^{-2} M_0^3
\\
&\; + C(\nu t)^{-\f\b 2-\f1p} \lam^{-1+\f2p}
\left[1+ \mathds{1}_{\{\nu t\geq 4\lam^2\}}
\big(\lam^{-2} \nu t\big)^{\b(\f1 2+\f1 p-\g)}\right] \\
&\;\quad  \cdot\left( M_0^{2-\g} N^\g + \lam^{-1} M_0^2 N \cdot t^{\f1p}\right) M_0^\b\\
&\;+ C(\nu t)^{-\f{\b}2-\f1p} \left(  Q_0  + \lam^{-1+\f2p} M_0^2\right)M_0^\b
+ C \nu^{-1} (\nu t)^{\f{1-\b}2-\f2p} Q^2 M_0^\b.
\end{split}
\label{eqn: C gamma estimate for W}
\eeq
Since $\b\geq \f2p$, Lemma \ref{lem: parabolic estimates for Duhamel term of fractional Laplace} gives the local $C^{1,\b}$-bound for $\CI[W]$ on $(0,T)$:
\beq
\|\CI[W](t)\|_{\dot{C}_s^1(\BT)}
+ t^\b \|\CI[W](t)\|_{\dot{C}_s^{1,\b}(\BT)}
\leq C\sup_{\eta\in (0,t]} \eta^\b \|W[u,X](\eta)\|_{\dot{C}_s^\b(\BT)}.
\label{eqn: C 1 gamma estimate for IW}
\eeq
Hence, by Lemma \ref{lem: parabolic estimates for Duhamel term of fractional Laplace}, $\pa_t Y(s,t)$ is well-defined pointwise for any $t\in (0,T)$, and \eqref{eqn: equation for Y} holds.

Now by \eqref{eqn: equation for Y}, the interpolation inequalities, and the fact $T\leq 1$,
\beq
\begin{split}
&\; \sup_{\eta\in (0,T)} \eta^{\f1p} \|\pa_t Y(\cdot, \eta)\|_{L^\infty_s(\BT)}
+ \sup_{\eta\in (0,T)} \eta^\g \|\pa_t Y(\cdot, \eta)\|_{C^\g_s(\BT)}
\\
\leq &\; C\sup_{\eta\in (0,T]} \eta^{\f1p} \|\Lam Y(\eta)\|_{L^\infty_s(\BT)}
+ \eta^{\f1p} \big\|W[u,X](\eta)\big\|_{L^\infty_s(\BT)}\\
&\; + C\sup_{\eta\in (0,T]} \eta^\g \|\Lam Y( \eta)\|_{C^\g_s(\BT)}
+ \eta^\g \big\|W[u,X](\eta)\big\|_{C^\g_s(\BT)}
\\
\leq &\; C\sup_{\eta\in (0,T]} \eta^{\f1p} \| Y'(\eta)\|_{\dot{C}^{1/p}_s(\BT)}
+ C\sup_{\eta\in (0,T]} \eta^\g \|Y'(\eta)\|_{\dot{C}^\g_s(\BT)}\\
&\; + C\sup_{\eta\in (0,T]} \eta^{\f1p} \big\|W[u,X](\eta)\big\|_{L^\infty_s(\BT)}
+ C\sup_{\eta\in (0,T]} \eta^\g \big\|W[u,X](\eta)\big\|_{\dot{C}^\g_s(\BT)},
\end{split}
\label{eqn: estimate for Y_t prelim}
\eeq
where $C$ depends on $p$ and $\g$.
We use this, \eqref{eqn: L inf estimate for w}, \eqref{eqn: C 1 gamma estimate for Y minus semigroup solution}, \eqref{eqn: C gamma estimate for W} and \eqref{eqn: C 1 gamma estimate for IW} (with $\b:=\g$ there), as well as the interpolation inequalities to find that, for $t\in (0,T]$ with $T\leq 1$,
\begin{align*}
&\;\sup_{\eta\in (0,T)}\eta^{\f1p}\|\pa_t Y(\cdot,\eta)\|_{L^\infty_s(\BT)}
+ \sup_{\eta\in (0,T)}\eta^{\g}\|\pa_t Y(\cdot,t)\|_{C^{\g}_s(\BT)}
\\
\leq &\; C\|X_0'\|_{L^\infty_s(\BT)}
+ C\sup_{\eta\in (0,T]} \eta^{\f1p} \big\|W[u,X](\eta)\big\|_{L^\infty_s(\BT)}
+ C\sup_{\eta\in (0,T]} \eta^\g \big\|W[u,X](\eta)\big\|_{\dot{C}^\g_s(\BT)}
\\
\leq
&\; C \lam^{-2} M_0^3
+ C\nu^{-\f1p} \lam^{-1+\f2p} \left( M_0^{2-\g} N^{\g}
+ \lam^{-1} M_0^2N \cdot T^{\f1p}\right)\\
&\; + C \nu^{-\f1p}\left(  Q_0  + \lam^{-1+\f2p}M_0^2\right)
+ C \nu^{-\f12-\f2p} T^{\f12-\f1p}Q^2
\\
&\; + C\nu^{-\f\g 2-\f1p}T^{\f\g 2-\f1p} \lam^{-1+\f2p}
\left[1+ \mathds{1}_{\{\nu T\geq 4\lam^2\}}
\big(\lam^{-2} \nu T\big)^{\g(\f1 2+\f1 p-\g)}\right] \\
&\;\quad  \cdot\left( M_0^{2-\g} N^\g + \lam^{-1} M_0^2 N \cdot T^{\f1p}\right) M_0^\g\\
&\;+ C\nu^{-\f{\g}2-\f1p} T^{\f{\g}2-\f1p} \left(  Q_0  + \lam^{-1+\f2p} M_0^2\right)M_0^\g
+ C \nu^{-\f{1+\g}2-\f2p} T^{\f{1+\g}2-\f2p} Q^2 M_0^\g
\\
\leq &\;C(\g,p,\nu,\lam_0, Q_0 ,M_0)\left(1+N^\g + T^{\f1p}N\right).
\end{align*}
By taking $T$ to be sufficiently small, which depends on $\g$, $p$, $\nu$, $\lam_0$, $M_0$, and $ Q_0 $, and also using Young's inequality, we can obtain that
\begin{align*}
&\;\sup_{\eta\in (0,T)}\eta^{\f1p}\|\pa_t Y(\cdot,\eta)\|_{L^\infty_s(\BT)}
+ \sup_{\eta\in (0,T)}\eta^{\g}\|\pa_t Y(\cdot,t)\|_{C^{\g}_s(\BT)}
\\
\leq &\; \f12 N + C_*(\g,p,\nu,\lam_0, Q_0 ,M_0).
\end{align*}
Then we choose
\beq
N: = 2C_*(\g,p,\nu,\lam_0, Q_0 ,M_0),
\label{eqn: choice of N}
\eeq
which implies
\[
\sup_{\eta\in (0,T)}\eta^{\f1p}\|\pa_t Y(\cdot,\eta)\|_{L^\infty_s(\BT)}
+ \sup_{\eta\in (0,T)}\eta^{\g}\|\pa_t Y(\cdot,t)\|_{C^{\g}_s(\BT)}
\leq N.
\]
We remind that these quantitative estimates for $Y$ are enough to justify as in Remark \ref{rmk: time continuity in C^1 topology} that $Y\in C([0,T];C^1(\BT))$.

Finally, let us show that $Y\in C^1_{loc}((0,T)\times \BT)$.
Since $t\mapsto \|\pa_t Y(t)\|_{C^\g(\BT)}$ and $t\mapsto \|Y'(t)\|_{C^\g(\BT)}$ are locally bounded in $(0,T)$, $t\mapsto Y(t)$ is locally continuous in the $C^{1,\b}(\BT)$-topology for any $\b\in [0,\g)$ on $(0,T)$ by interpolation;
see \eqref{eqn: L inf estimate for X tau - X t}, \eqref{eqn: C 1 estimate for X tau - X t}, and Remark \ref{rmk: time continuity in C^1 topology}.
Hence, $t\mapsto \Lam Y$ is locally continuous in the $C(\BT)$-topology on $(0,T)$.
By \eqref{eqn: equation for Y}  and the continuity of $t\mapsto W[u,X]$ obtained from \eqref{eqn: alternative expression of W}, this further implies that $\pa_t Y \in C_{loc}((0,T)\times \BT)$.
This together with the fact that $Y'\in C([0,T];C(\BT))$ gives that $Y\in C^1_{loc}((0,T)\times \BT)$.

Combining all the above characterizations of $Y$, we obtain that $Y\in \CX$.

In summary, by defining $Q$ and $N$ as in \eqref{eqn: choice of Q} and \eqref{eqn: choice of N} respectively, and then taking $T$ to be sufficiently small, which depends on $\g,\,p,\,\nu,\,\lam_0,\, Q_0 ,\, M_0$ and $X_0$, we can guarantee that, whenever $(u,X)\in \CU\times \CX$, then $(v,Y): = \CT (u,X) \in \CU\times \CX$.

\subsection{Contraction mapping and the proof of Proposition \ref{prop: fixed-point solution}}
\label{sec: contraction mapping}
Next we show that, by taking $T$ to be smaller if necessary, we can make $\CT$ a contraction mapping from $\CU\times \CX$ to itself under the metric \eqref{eqn: metric on U times X}.

Take $(u_i,X_i)\in \CU\times \CX$ ($i = 1,2$), and let $(v_i,Y_i):= \CT(u_i,X_i)$.
It is known that $X_i\in O^{\lam}_T$ with $\lam$ defined above, and they satisfy \eqref{eqn: simple bound for X in CX} and \eqref{eqn: simple bound 2 for X in CX}.
Moreover, by the definition of $\CX$ and \eqref{eqn: smallness condition on T},
\[
\|X_1'-X_2'\|_{L^\infty_T L^\infty_s}
\leq
\|X_1'- e^{-\f{t}4 \Lam}X_0'\|_{L^\infty_T L^\infty_s}
+\|X_2'- e^{-\f{t}4 \Lam}X_0'\|_{L^\infty_T L^\infty_s}
\leq \f{\lam_0}{8} < \f{\lam}2,
\]
and by \eqref{eqn: X is close to initial data}, for any $t_1,t_2\in [0,T]$,
\[
\|X_i'(t_1)-X_i'(t_2)\|_{L^\infty_s}
\leq
\|X_i'(t_1)- X_0'\|_{L^\infty_s}
+\|X_i'(t_2)- X_0'\|_{L^\infty_s}
\leq \f{\lam_0}4 = \f{\lam}2.
\]

Since $X_1(s,0) = X_2(s,0)$, by \eqref{eqn: def of v}, Lemma \ref{lem: estimate for u_X^nu-u_Y^nu in Morrey space} (with $\mu_0 = \f1p$ there) and \eqref{eqn: 4th estimate for B L inf} in Lemma \ref{lem: estimate for B_nu u}, for any $t\in (0,T]$,
\begin{align*}
&\; \|v_1(t)-v_2(t)\|_{M^{1,1-\g}(\R^2)}\\
\leq &\;\|u_{X_1}^{\nu}(t)-u_{X_2}^\nu(t)\|_{M^{1,1-\g}(\R^2)}
+ \|B_\nu[u_1](t)-B_\nu[u_2](t)\|_{M^{1,1-\g}(\R^2)} \\
\leq
&\; C \lam^{-1-\g} M_0^2
\Big[ t^{1-\g-\f1p}\sup_{\eta\in (0,t]} \eta^{\f1p}\|\pa_t (X_1-X_2)(\eta)\|_{L_s^\infty}
+  \|X_1'-X_2'\|_{L^\infty_t L^\infty_s}\Big]\\
&\; +
C(\nu t)^{-\f{\g}2} \nu^{-1}
\sup_{\eta\in(0,t]} (\nu \eta)^{\f{\g}2} \|(u_1-u_2)(\eta)\|_{M^{1,1-\g}} \sup_{\eta\in(0,t]} (\nu \eta)^{\f12}\|(u_1(\eta),u_2(\eta))\|_{L_x^\infty},
\end{align*}
where the constant $C$ depends on $p$ and $\g$.
Using the bounds for $(u_i,X_i)$, we obtain that
\begin{align*}
&\; (\nu t)^{\f{\g}{2}+\f1p-\f12}\|v_1(t)-v_2(t)\|_{M^{1,1-\g}(\R^2)}\\
\leq
&\; C\nu^{\f{\g}{2}+\f1p-\f12} \lam^{-1-\g} M_0^2
\Big[ t^{\f{1-\g} 2}\sup_{\eta\in (0,t]} \eta^{\f1p}\|\pa_t (X_1-X_2)(\eta)\|_{L_s^\infty}
+  t^{\f{\g}2+\f1p-\f12}\|X_1'-X_2'\|_{L^\infty_t L^\infty_s}\Big]\\
&\; +
C\nu^{-1}(\nu t)^{\f12-\f1p}Q
\sup_{\eta\in(0,t]} (\nu \eta)^{\f{\g}{2}+\f1p-\f12} \|(u_1-u_2)(\eta)\|_{M^{1,1-\g}}.
\end{align*}
This, together with the facts that $p\in (2,\infty)$ and $\g\in(1-\f2p,1)$, implies that
\beq
\begin{split}
&\;d_\CU(v_1,v_2)\\
\leq &\; C\big(T^{\f{1-\g} 2} +  T^{\f\g2+\f1p-\f12}\big) \nu^{\f{\g}{2}+\f1p-\f12} \lam^{-1-\g} M_0^2 \cdot d_\CX(X_1,X_2)
+ CT^{\f12-\f1p}  \nu^{-\f12-\f1p} Q \cdot d_\CU(u_1,u_2),
\end{split}
\label{eqn: dist between v_1 and v_2}
\eeq
where $C$ depends on $p$ and $\g$.

Next we bound $d_\CX(Y_1,Y_2)$.
By \eqref{eqn: equation for Y}, we argue as in \eqref{eqn: C 1 gamma estimate for IW} and \eqref{eqn: estimate for Y_t prelim} that
\beq
\begin{split}
d_\CX(Y_1,Y_2)
\leq &\; \|Y_1'-Y_2'\|_{L_T^\infty L_s^\infty(\BT)}
+ \sup_{\eta\in (0,T]} \eta^\g \|(Y_1'-Y_2')(\cdot, \eta)\|_{\dot{C}^\g_s(\BT)}\\
&\; + C\sup_{\eta\in (0,T]} \eta^{\f1p} \big\|W[u_1,X_1](\eta)-W[u_2,X_2](\eta)\big\|_{L^\infty_s(\BT)}\\
&\; + C\sup_{\eta\in (0,T]} \eta^\g \big\|W[u_1,X_1](\eta)-W[u_2,X_2](\eta)\big\|_{\dot{C}^\g_s(\BT)}
\\
\leq &\; C\sup_{\eta\in (0,T]} \eta^{\f1p} \big\|W[u_1,X_1](\eta)-W[u_2,X_2](\eta)\big\|_{L^\infty_s(\BT)}\\
&\; + C\sup_{\eta\in (0,T]} \eta^\g \big\|W[u_1,X_1](\eta)-W[u_2,X_2](\eta)\big\|_{\dot{C}^\g_s(\BT)}, 
\end{split}
\label{eqn: estimate for Y_1 - Y_2}
\eeq
where the constant $C$ depends on $p$ and $\g$.

By Lemma \ref{lem: parabolic estimates} and Lemma \ref{lem: Holder estimate for composition of functions}, we derive from \eqref{eqn: def of w} that
\beq
\begin{split}
&\; \big\|W[u_1,X_1](t)-W[u_2,X_2](t) \big\|_{L^\infty_s(\BT)}\\
\leq &\; \|g_{X_1}(t)-g_{X_2}(t)\|_{L^\infty_s(\BT)}
+ \|H_{X_1}^\nu(t)-H_{X_2}^\nu(t)\|_{L^\infty_s(\BT)}\\
&\;
+ \big\|\big[e^{\nu t\D}(u_0-u_{X_1(t)})+B_{\nu}[u_1]\big] \circ X_1(t)-\big[e^{\nu t\D}(u_0-u_{X_1(t)})+B_{\nu}[u_1]\big] \circ X_2(t) \big\|_{L^\infty_s(\BT)}\\
&\;
+ \big\|\big[e^{\nu t\D}(-u_{X_1(t)}+u_{X_2(t)})+B_{\nu}[u_1]-B_{\nu}[u_2]\big] \circ X_2(t) \big\|_{L^\infty_s(\BT)}
\\
\leq &\;\|g_{X_1}(t)-g_{X_2}(t)\|_{L^\infty_s(\BT)}
+ \|H_{X_1}^\nu(t)-H_{X_2}^\nu(t)\|_{L^\infty_s(\BT)}\\
&\;
+ \Big(\big\|\na e^{\nu t\D}(u_0-u_{X_1(t)})\big\|_{L^\infty_x(\BR^2)}
+ \|\na B_{\nu}[u_1](t)\|_{L^\infty_x(\BR^2)}\Big) \|X_1(t)-X_2(t)\|_{L^\infty_s(\BT)}\\
&\; + \big\|e^{\nu t\D}(u_{X_1(t)}-u_{X_2(t)})\big\|_{L^\infty_x(\BR^2)}
+ \|B_{\nu}[u_1](t) - B_{\nu}[u_2](t)\|_{L^\infty_x(\BR^2)}.
\end{split}
\label{eqn: L inf estimate for w - w}
\eeq

In order to bound $H_{X_1}^\nu-H_{X_2}^\nu$, we shall apply Lemma \ref{lem: estimate for H_X^nu-H_Y^nu} with $(\g,\b,\mu_0)$ there taken to be $(\f{\g}2,\g,\f1p)$, and then $(M,N_0,N_{\b})$ there can be replaced by $(CM_0,N,N)$ due to the assumptions on $X_i\in \CX$, \eqref{eqn: simple bound for X in CX}, \eqref{eqn: simple bound 2 for X in CX}, and the interpolation inequalities.
Then we will find with $T\leq 1$ that
\beq
\begin{split}
&\; \|H_{X_1}^{\nu}(t)-H_{X_2}^\nu(t)\|_{L^\infty_s(\BT)}
+(\nu t)^{\g/2} M_0^{-\g} \|H_{X_1}^{\nu}(t)-H_{X_2}^\nu(t)\|_{\dot{C}_s^\g(\T)}\\
\leq &\; C d_\CX(X_1,X_2)\\
&\;\cdot \Big[ \lam^{-3}
M_0^{3-\g/2} N^{\g/2} \big( 1 + \nu^{-1} T^{1-\f2p} N^2\big)
+ \lam^{-1} M_0\\
&\;\quad + t^{\f12-\f{\g}{4}-\f1p} \cdot
\lam^{-2-\g/2} \nu^{-(1-\g/2)/2} M_0^2 N
\big( 1 + \nu^{-1} T^{1-\f2p} N^2\big)^{\f{2+\g}{4}}
\\
&\; \quad + t^{\f12-\f{\g}{4}-\f1p} \cdot \lam^{-1-\g/2}\nu^{-(1-\g/2)/2} M_0^2 \Big]
\\
\leq &\; C(\g, p,\nu,\lam,M_0,N) \cdot d_\CX(X_1,X_2)
\Big(1 + t^{\f12-\f{\g}{4}-\f1p}\Big).
\end{split}
\label{eqn: estimate for H_X_1 - H_X_2}
\eeq

Applying Lemma \ref{lem: estimate for u_X}, Lemma \ref{lem: estimate for u_X-u_Y} (with $(\b,\g)$ there taken to be $(\f1p,\g)$), \eqref{eqn: 2nd estimate for B} and \eqref{eqn: 4th estimate for B L inf} in Lemma \ref{lem: estimate for B_nu u}, Lemma \ref{lem: improved estimates for g_X-g_Y} (with $\g = \f{1}{2p}$ there), we derive from \eqref{eqn: L inf estimate for w - w} that, for $t\in (0,T]$ with $T\leq 1$,
\begin{align*}
&\; \big\|W[u_1,X_1](t)-W[u_2,X_2](t) \big\|_{L^\infty_s(\BT)}\\
\leq &\;C \lam^{-2} M_0^2 \Big[\|X_1'(t)-X_2'(t)\|_{\dot{C}_s^{\f{1}{2p}}}
+ \lam^{-1} \big(\|X_1'(t)\|_{\dot{C}_s^{\f{1}{2p}}} +\|X_2'(t)\|_{\dot{C}_s^{\f{1}{2p}}}\big) \|X_1'(t)-X_2'(t)\|_{L_s^\infty} \Big]\\
&\; + C(\g, p,\nu,\lam,M_0,N) \cdot d_\CX(X_1,X_2)
\Big(1 + t^{\f12-\f{\g}{4}-\f1p}\Big)
\\
&\;
+ C\Big[(\nu t)^{-\f12-\f1p}\|u_0-u_{X_1(t)}\|_{L^p_x}\\
&\;\qquad
+ (\nu t)^{-1} \nu^{-1} (\nu t)^{\f12-\f1p} Q
\big((\nu t)^{\f12-\f1p}Q
+ (\nu t)^{\f{1-\g}{2}} Q\big)\Big]
\|X_1(t)-X_2(t)\|_{L^\infty_s}\\
&\;
+ C (\nu t)^{-\f1{2p}} \lam^{-1+\f1p}M_0 \|X_1'(t)-X_2'(t)\|_{L_s^{\infty}}
\\
&\;
+ C(\nu t)^{-\f{1-\g}{2}} \lam^{-1-\g}
\Big(t^{-\g}M_0^2 \|X_1(t)-X_2(t)\|_{L_s^\infty}
+ M_0^2 \|X_1(t)-X_2(t)\|_{\dot{C}_s^\g}\Big)
\\
&\; + C(\nu t)^{-\f12}\nu^{-1}
(\nu t)^{\f12 -\f1p}Q \cdot \sup_{\eta\in (0,t]}(\nu \eta)^{\g/2}\|(u_1-u_2)(\eta)\|_{M^{1,1-\g}}
\\
\leq
&\;C \lam^{-2} M_0^2 \Big[t^{-\f{1}{2p}} d_\CX(X_1,X_2)
+ \lam^{-1} \cdot t^{-\f1{2p}}M_0 d_\CX(X_1,X_2) \Big]\\
&\; + C(\g, p,\nu,\lam,M_0,N) \cdot d_\CX(X_1,X_2)
\Big(1 + t^{\f12-\f{\g}{4}-\f1p}\Big)
\\
&\;
+ C(\nu t)^{-\f12-\f1p} \Big[ Q_0  + \lam^{-1+\f{2}{p}} M_0^2
+ \nu^{-1} Q^2 \big((\nu T)^{\f12-\f1p} + (\nu T)^{\f{1-\g}{2}}\big)\Big]\cdot t^{1-\f1p}d_\CX(X_1,X_2)\\
&\; + C (\nu t)^{-\f1{2p}} \lam^{-1+\f1p}M_0 \cdot d_\CX(X_1,X_2)
\\
&\;
+ C(\nu t)^{-\f{1-\g}2} \lam^{-1-\g}
\Big(t^{-\g}M_0^2 \cdot t^{1-\f1p}d_\CX(X_1,X_2)
+ M_0^2 \cdot t^{1-\g}d_\CX(X_1,X_2)\Big)
\\
&\; + C\nu^{-1}
(\nu t)^{-\f1p} Q \cdot (\nu t)^{\f12-\f1p}\sup_{\eta\in (0,t]}(\nu \eta)^{\f\g2+\f1p-\f12}\|(u_1-u_2)(\eta)\|_{M^{1,1-\g}},
\end{align*}
where the constant $C$ depends on $p$ and $\g$.
In the last inequality, we used the following estimates due to the fact $X_1(s,0) = X_2(s,0)$:
\beq
\|X_1(t)-X_2(t)\|_{L_s^\infty}
\leq \int_0^t \eta^{-\f1p} \cdot \eta^{\f1p}\|\pa_t (X_1-X_2)(\eta)\|_{L_s^\infty}\,d\eta
\leq Ct^{1-\f1p} d_\CX(X_1,X_2),
\label{eqn: L inf difference of X_1 and X_2}
\eeq
and similarly,
\beq
\|X_1(t)-X_2(t)\|_{\dot{C}_s^\g}
\leq Ct^{1-\g} d_\CX(X_1,X_2).
\label{eqn: C gamma difference of X_1 and X_2}
\eeq
Hence, under the assumptions $0<t<T\leq 1$ and $\g\in[\f2p,1)$,
\begin{align*}
&\; \big\|W[u_1,X_1](t)-W[u_2,X_2](t) \big\|_{L^\infty_s(\BT)}\\
\leq
&\; C(\g, p,\nu,\lam_0, Q_0 , M_0,N) \cdot d_\CX(X_1,X_2)
\Big(1 + t^{\f12-\f{\g}{4}-\f1p}+t^{-\f1{2p}} + t^{\f12-\f2p} + t^{\f{1-\g}{2}-\f1p}\Big)
\\
&\; + C \nu^{-\f12-\f2p} Q\cdot d_\CU(u_1,u_2) t^{\f12-\f2p}
\\
\leq
&\; C(\g, p,\nu,\lam_0,  Q_0 , M_0,N) \cdot d_\CX(X_1,X_2)
\Big(t^{-\f1{2p}} + t^{\f{1-\g}{2}-\f1p}\Big)
+ C\nu^{-\f12-\f2p} Q \cdot d_\CU(u_1,u_2) t^{\f12-\f2p},
\end{align*}
which gives
\beq
\begin{split}
&\;\sup_{\eta\in (0,T]} \eta^{\f1p}\big\|W[u_1,X_1](\eta)-W[u_2,X_2](\eta)\big\|_{L^\infty_s(\BT)}\\
\leq &\;  C(\g, p,\nu,\lam_0, Q_0 , M_0,N)
\big[ d_\CX(X_1,X_2)+d_\CU(u_1,u_2)\big]
\Big(T^{\f1{2p}} + T^{\f{1-\g}{2}}\Big).
\end{split}
\label{eqn: L inf difference of w}
\eeq

Similarly,
\begin{align*}
&\; \big\|W[u_1,X_1](t)-W[u_2,X_2](t) \big\|_{\dot{C}_s^\g(\BT)}\\
\leq &\; \|g_{X_1}(t)-g_{X_2}(t)\|_{\dot{C}_s^\g(\BT)}
+ \|H_{X_1}^\nu(t)-H_{X_2}^\nu(t)\|_{\dot{C}_s^\g(\BT)}\\
&\;
+ \big\|\big[e^{\nu t\D}(u_0-u_{X_1(t)})\big] \circ X_1(t)-\big[e^{\nu t\D}(u_0-u_{X_2(t)})\big] \circ X_2(t) \big\|_{\dot{C}_s^\g(\BT)}\\
&\; + \|B_{\nu}[u_1]\circ X_1(t) - B_{\nu}[u_1]\circ X_2(t)\|_{\dot{C}_s^\g(\BT)}
+ \| (B_{\nu}[u_1] - B_{\nu}[u_2]) \circ X_2(t) \|_{ \dot{C}_s^\g(\BT)}.
\end{align*}

By Lemma \ref{lem: improved estimates for g_X-g_Y},
\begin{align*}
&\; \|g_{X_1}(t) - g_{X_2}(t)\|_{ \dot{C}_s^\g(\BT)}\\
\leq &\; C \lam^{-2} M_0 \big(\|X_1'(t)\|_{\dot{C}_s^{\g/2}}+\|X_2'(t)\|_{\dot{C}_s^{\g/2}}\big)\\
&\;\cdot \Big[\|X_1'(t)-X_2'(t)\|_{\dot{C}_s^{\g/2}} + \lam^{-1} \big(\|X'_1(t)\|_{\dot{C}_s^{\g/2}}+\|X_2'(t)\|_{\dot{C}_s^{\g/2}}\big) \|X_1'(t)-X_2'(t)\|_{L_s^\infty}\Big]
\\
\leq &\; C \lam^{-2} M_0\cdot  t^{-\f\g2} \|(X_1',X_2')\|_{L^\infty_t L^\infty_s}^{\f12}
\left[ \sup_{\eta\in (0,t]} \eta^{\g} \big(\|X_1'(\eta)\|_{\dot{C}_s^\g}+\|X_2'(\eta)\|_{\dot{C}_s^\g}\big)\right]^{\f12} \\
&\;\cdot \Big[t^{-\g/2}  \sup_{\eta\in (0,t]}\eta^{\g/2}\|X_1'(\eta)-X_2'(\eta)\|_{\dot{C}_s^{\g/2}} + \lam^{-1} t^{-\g/2} M_0 \|X_1'-X_2'\|_{L^\infty_t L_s^\infty}\Big]
\\
\leq &\; C \lam^{-3} M_0^{5/2} \cdot  d_\CX(X_1,X_2) t^{-\g}
\big( \r(t)+ t^{\s}\big)^{\f12},
\end{align*}
where $C$ depends on $\g$.
In the last line, we used the bound for $t^{\g}\|X_i'(t)\|_{\dot{C}_s^\g}$ in the definition of $\CX$.
The H\"older estimate for $(H_{X_1}^\nu-H_{X_2}^\nu)$ in \eqref{eqn: estimate for H_X_1 - H_X_2} reads
\[
\|H_{X_1}^\nu(t) - H_{X_2}^\nu(t)\|_{ \dot{C}_s^\g(\BT)}
\leq
C(\g, p,\nu,\lam,M_0,N) \cdot d_\CX(X_1,X_2)
\Big(1 + t^{\f12-\f{\g}{4}-\f1p}\Big)t^{-\f\g2}.
\]
By Lemma \ref{lem: estimate for u_X}, Lemma \ref{lem: estimate for u_X-u_Y}, Lemma \ref{lem: parabolic estimates}, Lemma \ref{lem: Holder estimate for composition of functions}, \eqref{eqn: L inf difference of X_1 and X_2}, and \eqref{eqn: C gamma difference of X_1 and X_2}, for any $t\in (0,T]$ with $T\leq 1$,
\begin{align*}
&\; \left\|e^{\nu t\D}(u_0-u_{X_1(t)})\circ X_1(t)- e^{\nu t\D}(u_0-u_{X_2(t)})\circ X_2(t) \right\|_{\dot{C}_s^\g(\BT)}\\
\leq
&\; \left\|e^{\nu t\D}(u_0-u_{X_1(t)})\circ X_1(t)- e^{\nu t\D}(u_0-u_{X_1(t)})\circ X_2(t) \right\|_{\dot{C}_s^\g(\BT)}\\
&\; + \left\|e^{\nu t\D}(u_{X_1(t)}-u_{X_2(t)})\circ X_2(t) \right\|_{\dot{C}_s^\g(\BT)}\\
\leq &\;
\big\|e^{\nu t\D}(u_0-u_{X_1(t)})\big\|_{\dot{C}_x^{1}} \|X_1(t)-X_2(t)\|_{\dot{C}^\g_s}\\
&\; +
C \big\|e^{\nu t\D}(u_0-u_{X_1(t)})\big\|_{ \dot{C}_x^{1,\g}} \|(X_1'(t),X_2'(t))\|_{L^\infty_s}^\g \|X_1(t)-X_2(t)\|_{L^\infty_s}  \\
&\; + \big\|e^{\nu t\D}(u_{X_1(t)}-u_{X_2(t)}) \big\|_{\dot{C}_x^\g(\BR^2)} \|X_2'(t)\|_{L_s^\infty}^\g\\
\leq &\;
C(\nu t)^{-\f12-\f1p} \left(\|u_0\|_{L_x^p} + \|u_{X_1(t)}\|_{L_x^p}\right)
\cdot t^{1-\g}d_\CX(X_1,X_2)\\
&\;+
C(\nu t)^{-\f{1+\g}2-\f1p} \left(\|u_0\|_{L_x^p} + \|u_{X_1(t)}\|_{L_x^p}\right)
\cdot  M_0^\g \cdot t^{1-\f1p}d_\CX(X_1,X_2)\\
&\; + C\Big[
(\nu t)^{-\g/2} \lam^{-1} M_0 \|X_1'(t)-X_2'(t)\|_{L_s^{\infty}} \\
&\;\qquad + (\nu t)^{-1/2} \lam^{-1-\g} \cdot
M_0^2 \big(t^{-\g} \|X_1(t)-X_2(t)\|_{L_s^\infty} + \|X_1(t)-X_2(t)\|_{\dot{C}_s^\g}\big)
\Big] M_0^\g
\\
\leq &\;
C\big( Q_0  + \lam^{-1+\f{2}{p}} M_0^2\big) \Big[(\nu t)^{-\f12-\f1p}
\cdot t^{1-\g} + (\nu t)^{-\f{1+\g}2-\f1p} \cdot  M_0^\g \cdot t^{1-\f1p}\Big] d_\CX(X_1,X_2)\\
&\; + C\Big[
(\nu t)^{-\g/2} \lam^{-1} M_0
+ (\nu t)^{-1/2} \lam^{-1-\g} \cdot
M_0^2 \cdot t^{1-\g-\f1p}
\Big] M_0^\g\cdot d_\CX(X_1,X_2)
\\
\leq &\;
C(\g,p,\nu,\lam, Q_0 ,M_0)\cdot d_\CX(X_1,X_2) t^{\f12-\f1p-\g}.
\end{align*}
By \eqref{eqn: 2nd estimate for B} in Lemma \ref{lem: estimate for B_nu u}, Lemma \ref{lem: parabolic estimates} and Lemma \ref{lem: Holder estimate for composition of functions}, for any $t\in (0,T]$ with $T\leq 1$,
\begin{align*}
&\; \big\|B_{\nu}[u_1]\circ X_1(t)- B_{\nu}[u_1]\circ X_2(t) \big\|_{\dot{C}_s^\g(\BT)}\\
\leq &\;
\big\|B_{\nu}[u_1](t)\big\|_{\dot{C}_x^{1}} \|X_1(t)-X_2(t)\|_{\dot{C}^\g_s} +
C \big\|B_{\nu}[u_1](t)\big\|_{ \dot{C}_x^{1,\g}} \|(X_1',X_2')\|_{L^\infty_T L^\infty_s}^\g \|X_1(t)-X_2(t)\|_{L^\infty_s} \\
\leq &\;
C \nu^{-1}
\sup_{\eta\in (0,t]} (\nu\eta)^{1/2} \|u_1(\eta)\|_{L^\infty_x}
\left[\sup_{\eta\in(0,t]}(\nu \eta)^{\f12-\f1p}\|u_1(\eta)\|_{L_x^p}
+ \sup_{\eta\in (0,t]} (\nu\eta)^{(1+\g)/2} \|u_1(\eta)\|_{\dot{C}^\g_x}\right] \\
&\; \cdot \left[(\nu t)^{-1}\|X_1(t)-X_2(t)\|_{\dot{C}^\g_s}
+ (\nu t)^{-(2+\g)/2} \cdot  M_0^\g \cdot \|X_1(t)-X_2(t)\|_{L^\infty_s}\right]
\\
\leq &\;
C \nu^{-1} (\nu t)^{\f12-\f1p}Q
\left[(\nu t)^{\f12-\f1p} Q
+ (\nu t)^{\f{1-\g}{2}}Q\right] \\
&\; \cdot \left[(\nu t)^{-1}\cdot t^{1-\g} d_\CX(X_1,X_2)
+ (\nu t)^{-\f{2+\g}{2}} \cdot  M_0^\g \cdot t^{1-\f1p} d_\CX(X_1,X_2)\right]
\\
\leq &\;
C(\g,p,\nu,\lam_0, Q_0 ,M_0)
\cdot d_\CX(X_1,X_2)t^{1-\f1p -\f{3\g}{2}}.
\end{align*}
Finally, by \eqref{eqn: 4th estimate for B gamma} in Lemma \ref{lem: estimate for B_nu u} and Lemma \ref{lem: Holder estimate for composition of functions},
\begin{align*}
&\; \big\| \big(B_{\nu}[u_1]- B_{\nu}[u_2]\big)\circ X_2(t) \big\|_{\dot{C}_s^\g(\BT)}\\
\leq &\; \big\|B_{\nu}[u_1](t) - B_{\nu}[u_2](t)\big\|_{\dot{C}_x^\g(\BR^2)} \|X_2'\|_{L^\infty_T L^\infty_s}^\g\\
\leq &\; C\nu^{-1}(\nu t)^{-\f{1+\g}2}
\sup_{\eta\in (0,t]}(\nu \eta)^{\g/2}\|(u_1-u_2)(\eta)\|_{M^{1,1-\g}}
\sup_{\eta\in (0,t]} (\nu\eta)^{1/2} \|(u_1(\eta),u_2(\eta))\|_{L_x^\infty}\cdot M_0^\g
\\
\leq &\; C\nu^{-1}(\nu t)^{-\f{1+\g}2} \cdot (\nu t)^{\f12-\f1p} d_\CU(u_1,u_2)
\cdot (\nu t)^{\f12-\f1p} Q \cdot M_0^\g
\\
\leq &\;  C\nu^{-\f{1+\g}2-\f2p}
Q M_0^\g\cdot d_\CU(u_1,u_2) t^{\f{1-\g}2-\f2p}.
\end{align*}
Summarizing all the estimates and using the assumption $T\leq 1$ yields that
\beq
\begin{split}
&\; \big\|W[u_1,X_1](t)-W[u_2,X_2](t) \big\|_{\dot{C}_s^\g(\BT)}\\
\leq &\; C \lam^{-3} M_0^{5/2} \cdot  d_\CX(X_1,X_2) t^{-\g}
\big( \r(t)+ t^\s\big)^{\f12}
+ C(\g, p,\nu,\lam_0, Q_0 ,M_0,N) \cdot d_\CX(X_1,X_2)
t^{\f12-\f1p-\g}
\\
&\; + C\nu^{-\f{1+\g}2-\f2p}
Q M_0^\g\cdot d_\CU(u_1,u_2) t^{\f{1-\g}2-\f2p},
\end{split}
\label{eqn: C gamma estimate for W_1 - W_2}
\eeq
which gives, under the assumptions $T\leq 1$ and $\g\in [\f2p,1)$,
\begin{align*}
&\; \sup_{\eta\in (0,T]} \eta^\g \big\|W[u_1,X_1](\eta)-W[u_2,X_2](\eta) \big\|_{\dot{C}_s^\g(\BT)}\\
\leq &\; C(\g, p,\nu,\lam_0, Q_0 ,M_0,N)
\big[ d_\CX(X_1,X_2)  + d_\CU(u_1,u_2)\big]
\left[\big( \r(T)+ T^\s\big)^{\f12}
+T^{\f12-\f1p} \right].
\end{align*}

Putting this and \eqref{eqn: L inf difference of w} into \eqref{eqn: estimate for Y_1 - Y_2}, and also using \eqref{eqn: dist between v_1 and v_2}, and the definitions of $Q$ and $N$ in \eqref{eqn: choice of Q} and \eqref{eqn: choice of N} respectively, we obtain that
\beq
\begin{split}
&\;d_\CU(v_1,v_2)+d_\CX(Y_1,Y_2)
\\
\leq &\;  C(\g, p,\nu,\lam_0, Q_0 , M_0,N)\left(T^{\f{1-\g} 2} +  T^{\f\g2+\f1p-\f12} + T^{\f12-\f1p}+ T^{\f1{2p}}
+\big( \r(T)+ T^\s\big)^{\f12}
\right)\\
&\;\cdot
\big[d_\CU(u_1,u_2) + d_\CX(X_1,X_2)\big].
\end{split}
\label{eqn: final bound for distance bewteen output solutions}
\eeq
Since $\g\in (1-\f2p,1)$, $\s>0$, and $\lim_{t\to 0^+}\r(t) = 0$, we may take $T$ smaller if necessary, such that for all $u_1,u_2\in \CU$, $X_1,X_2\in \CX$, and $(v_i,Y_i):=\CT(u_i,X_i)$ $(i = 1,2)$, we have $(v_i,Y_i)\in \CU\times \CX$, and
\[
d_{\CU\times \CX}((v_1,Y_1),(v_2,Y_2))
\leq \f12 d_{\CU\times \CX}((u_1,X_1),(u_2,X_2)),
\]
so $\CT$ is a contraction mapping from $\CU\times \CX$ to itself.

Therefore, $\CT$ admits a unique fixed-point in $\CU\times \CX$, which we denote by $(u,X)$ with slight abuse of the notations.
It is clearly a mild local solution on $[0,T]$ to \eqref{eqn: NS equation}-\eqref{eqn: initial data} in the sense of Definition \ref{def: mild solution}, where $T\leq 1$ essentially depends on $\g,\,p,\,\nu,\,\lam_0,\, Q_0 ,\,M_0$, and $X_0$.
We remark that the implicit dependence of $T$ on $X_0$ stems from the requirement that $(\r(T)+T^\s)$ and $\phi(T)$ should be smaller than some constants that only depend on $\g,\,p,\,\nu,\,\lam_0,\, Q_0 $, and $M_0$.

Finally, in the statement of Proposition \ref{prop: fixed-point solution}, \eqref{eqn: equation for X_t as a fixed point} follows from \eqref{eqn: equation for Y} and the fact that $(u,X)$ is a fixed-point of $\CT$, while the other bounds for $X$ in the statement have been established in \eqref{eqn: lower bound for |X|_* for elements in CX}-\eqref{eqn: simple bound 2 for X in CX}.

\section{Properties of the Mild Solutions}
\label{sec: properties of mild solutions}
Throughout this section, we let $u_0 = u_0(x)\in L^p(\BR^2)$ for some $p\in (2,\infty)$ and $X_0 = X_0(s)\in C^1(\BT)$, which satisfy $\di u_0 = 0$ and $|X_0|_* >0$.
Suppose for some $T>0$, $(u,X)$ is a mild solution to \eqref{eqn: NS equation}-\eqref{eqn: initial data} on $[0,T]$ in the sense of Definition \ref{def: mild solution};
note that it is not necessarily obtained from the construction in Section \ref{sec: local well-posedness}.

Our goal in this section is to establish several fundamental properties of $(u,X)$.
We will show that the mild solution is unique (Proposition \ref{prop: uniqueness of mild solution}) and enjoys estimates in various regularity class (see Lemma \ref{lem: L inf regularity of u mild solution}, Lemma \ref{lem: time continuity of u in L^p}, Lemma \ref{lem: time derivative and equation}, and Proposition \ref{prop: higher regularity}).
Continuation of the mild solution will be studied in Proposition \ref{prop: continuation}.
We will show in Corollary \ref{cor: maximal solution} that there exists a unique maximal mild solution, and in the case a finite-time singularity occurs (see Definition \ref{def: maximal solution}), we provide characterization of the solution when it approaches the blow-up time.

\subsection{Regularity of the mild solutions}
\label{sec: regularity of mild solutions}

Throughout the section, with abuse of the notations, we denote that
\beq
M:= \sup_{\eta\in [0,T]} \|X'(\eta)\|_{L^\infty_s}
+ \sup_{\eta\in (0,T]}\eta^{\f1p} \|X'(\eta)\|_{\dot{C}^{1/p}_s} < +\infty,
\label{eqn: def of M}
\eeq
\beq
\lam:= \inf_{\eta\in [0,T]}|X(\eta)|_* >0,
\label{eqn: def of lambda}
\eeq
and
\beq
Q:=\sup_{\eta\in [0,T]} \|u(\eta)\|_{L^p_x} < +\infty.
\label{eqn: def of Q}
\eeq
These quantities are well-defined thanks to the definition of the mild solutions (see Definition \ref{def: mild solution}).

\begin{lem}
\label{lem: L inf regularity of u mild solution}
\beq
\sup_{\eta\in (0,T]} \eta^{\f1p}\|u(\eta)\|_{L_x^\infty(\BR^2)}
\leq C(T, p,\nu,\lam,M,Q).
\label{eqn: L inf regularity of u mild solution}
\eeq
\begin{proof}
Recall that $u$ satisfies the representation \eqref{eqn: fixed pt equation for u}.

On one hand, by Lemma \ref{lem: estimate for u_X^nu} and Lemma \ref{lem: parabolic estimates}, for $t\in (0,T]$,
\[
\big\|u_X^\nu(t) + e^{\nu t\Delta}u_0\big\|_{L^p_x(\BR^2)}
\leq
C\lam^{-1+\f2p} \|X'\|_{L^\infty_t L_s^\infty}^2 + \|u_0\|_{L^p_x}
\leq C\lam^{-1+\f2p}M^2 + Q,
\]
where $C$ depends on $p$, and
\[
\big\|u_{X}^\nu(t) + e^{\nu t\Delta}u_0\big\|_{L^\infty_x(\BR^2)}
\leq
Ct^{-\f1p} \lam^{-1}\|X'\|_{L^\infty_t L_s^{\infty}}
\sup_{\eta\in (0,t]} \eta^{\f1p} \|X'(\eta)\|_{\dot{C}_s^{1/p}}
+ C(\nu t)^{-\f1p}\|u_0\|_{L^p_x},
\]
which gives
\[
(\nu t)^{\f1p}
\big\|u_{X}^\nu(t) + e^{\nu t\Delta}u_0\big\|_{L^\infty_x(\BR^2)}
\leq
C \nu^{\f1p} \lam^{-1}M^2 + CQ,
\]
where $C$ depends on $p$. 
By interpolation, for any $q\in [p,\infty]$ and any $t\in (0,T]$,
\beq
(\nu t)^{\f1p -\f1q}\big\|u_{X}^\nu(t) + e^{\nu t\Delta}u_0\big\|_{L^q_x(\BR^2)}
\leq C(p,\nu,\lam,M,Q).
\label{eqn: L q estimate for a part of u}
\eeq

On the other hand, by \eqref{eqn: L q estimate for B general} in Lemma \ref{lem: estimate for B_nu u}, for any $q,r\in [p,\infty]$ satisfying $\f1q \in ( \f1r+\f1p - \f12, \f1r+\f1p]$,
\[
(\nu t)^{\f1p-\f1q}\|B_{\nu}[u](t)\|_{L^q_x}
\leq C\nu^{-1} (\nu t)^{\f1p-\f12}
\sup_{\eta\in(0,t]}(\nu \eta)^{1-\f1r-\f1p}
\|(u\otimes u) (\eta)\|_{L_x^{\f{pr}{p+r}}},
\]
which together with \eqref{eqn: fixed pt equation for u} and \eqref{eqn: L q estimate for a part of u} implies
\beq
(\nu t)^{\f1p-\f1q}\|u(t)\|_{L^q_x}
\leq C(p,\nu,\lam,M,Q)+ C\nu^{-1} (\nu t)^{\f12-\f1p} \|u\|_{L^\infty_t L_x^p}
\sup_{\eta\in(0,t]} (\nu \eta)^{\f1p-\f1r}
\|u(\eta)\|_{L_x^r}.
\label{eqn: bootstrap formula for u}
\eeq
Here $C$ depends on $p$, $q$, and $r$.
If $p>4$, we can simply take $(r,q) = (p,\infty)$ and that gives \eqref{eqn: L inf regularity of u mild solution}.
Otherwise, we first take $r = r_0:= p$ and choose $q = q_1$ such that $\f1{q_1} = \f1p - \f12(\f12-\f1p) \in (\f2p-\f12,\f1p]$, which leads to
\[
\sup_{\eta\in (0,T]}(\nu \eta)^{\f1p-\f1q}\|u(\eta)\|_{L^q_x}
\leq C(T, p,\nu,\lam,M,Q).
\]
Then we bootstrap with \eqref{eqn: bootstrap formula for u}, each time by taking $r_k = q_{k-1}$ and $\f1{q_k} = \max(0,\f1{r_k} - \f12(\f12-\f1p))$ there.
After finite steps of iteration, we end up having \eqref{eqn: L inf regularity of u mild solution}.
\end{proof}
\end{lem}

\begin{lem}
\label{lem: time continuity of u in L^p}
$u \in C([0,T];L^p(\BR^2))$.

\begin{proof}
Recall that $p\in (2,\infty)$ and $u$ satisfies \eqref{eqn: fixed pt equation for u}. 

Since $u_0\in L^p(\BR^2)$, $t\mapsto e^{\nu t \D}u_0$ is continuous in $L^p(\BR^2)$.

The time-continuity in $L^p(\BR^2)$ of $u_X^\nu$ at $t = 0$ follows from Lemma \ref{lem: estimate for u_X^nu}, while the time-continuity at some $t\in (0,T]$ can be justified by interpolating the estimate in Lemma \ref{lem: time Holder continuity of u_X^nu} and, say, an $L^2$-estimate of $u_X^\nu(t)$ due to Lemma \ref{lem: estimate for u_X^nu}:
\beqo
\|u_X^{\nu}(t)\|_{L^2_x(\R^2)}
\leq C(\nu t)^{\f1{4}} \lam^{-\f12}
\|X'\|_{L^\infty_t L_s^2} \|X'\|_{L^\infty_t L_s^\infty}.
\eeqo

Regarding $B_\nu[u]$, its time-continuity in $L^p(\BR^2)$ at $t = 0$ follows from \eqref{eqn: L p estimate for B} in Lemma \ref{lem: estimate for B_nu u}:
\[
\|B_{\nu}[u](t)\|_{L^p_x}
\leq C\nu^{-1} (\nu t)^{\f1p-\f12}\left(\sup_{\eta\in(0,t]}(\nu \eta)^{\f12-\f1p}\|u(\eta)\|_{L_x^p}\right)^2
\leq C\nu^{-1} (\nu t)^{\f12-\f1p}\|u\|_{L^\infty_t L_x^p}^2.
\]
Its time-continuity at some $t\in (0,T]$ follows by interpolating the estimate in Lemma \ref{lem: time Holder continuity of B_u} and, say, an $L^{p/2}$-estimate for $B_\nu[u]$ from \eqref{eqn: L q estimate for B general} in Lemma \ref{lem: estimate for B_nu u}:
\[
\|B_{\nu}[u](t)\|_{L^{p/2}_x}
\leq C\nu^{-1} (\nu t)^{\f{2}{p}-\f12}
\sup_{\eta\in(0,t]}(\nu \eta)^{1-\f{2}{p}}\|(u\otimes u)(\eta)\|_{L_x^{p/2}}
\leq C\nu^{-1} (\nu t)^{\f12}\|u\|_{L^\infty_t L_x^p}^2.
\]

Combining them, we conclude that $u \in C([0,T];L^p(\BR^2))$.
\end{proof}
\end{lem}

\begin{lem}
\label{lem: time derivative and equation}
For any $t\in (0,T)$, $\pa_t X(s,t)$ is well-defined pointwise, and the equation
\beq
\pa_t X(s,t) = -\f14\Lam X(s,t) + W[u,X;\nu,u_0](s,t) = u(X(s,t),t)
\label{eqn: differential equation of X in the mild solution}
\eeq
holds.
As a result,
\[
\sup_{\eta\in (0,T)} \eta^{\f1{p}} \|\pa_t X(\eta)\|_{L_s^\infty(\BT)}
\leq C(T,p,\nu,\lam,M,Q).
\]
\begin{proof}
Proceeding as in \eqref{eqn: spatial Holder estimate for v}, we obtain that, for any $t\in (0,T]$,
\beq
\begin{split}
&\;  \|u(t)\|_{\dot{C}_x^{1/p}(\BR^2)}\\
\leq &\; \|u_X^\nu(t)\|_{\dot{C}_x^{1/p}(\BR^2)}
+ \|e^{\nu t\D}u_0\|_{\dot{C}_x^{1/p}(\BR^2)} + \|B_\nu[u](t)\|_{\dot{C}_x^{1/p}(\BR^2)} \\
\leq &\; C \lam^{-1-\f1p} t^{-\f1p} \|X'\|_{L^\infty_t L_s^{\infty}}
\sup_{\eta\in (0,t]} \eta^{\f1p} \|X'(\eta)\|_{\dot{C}_s^{1/p}}
+ C (\nu t)^{-\f3{2p}} \|u_0\|_{L_x^p}
\\
&\; + C \nu^{-1}(\nu t)^{-(1+\f1p)/2}
\left(\sup_{\eta\in (0,t]} (\nu \eta)^{\f12-\f1p} \|u(\eta)\|_{L^p_x}+\sup_{\eta\in (0,t]}(\nu \eta)^{1/2}\|u(\eta)\|_{L^\infty_x}\right)^2
\\
\leq &\; C \lam^{-1-\f1p} t^{-\f1p} M^2
+ C (\nu t)^{-\f3{2p}} Q
+ C \nu^{-1} (\nu t)^{\f12-\f5{2p}}
\left(Q + \sup_{\eta\in (0,t]}(\nu \eta)^{\f1p}\|u(\eta)\|_{L^\infty_x}\right)^2.
\end{split}
\label{eqn: C 1/p estimate for u}
\eeq
Thanks to Lemma \ref{lem: L inf regularity of u mild solution} and the fact $p>2$,
\beq
\sup_{\eta\in (0,T]} \eta^{\f{3}{2p}} \|u(\eta)\|_{\dot{C}_x^{1/p}}
\leq C(T, p,\nu,\lam,M,Q)
< +\infty.
\label{eqn: a singular Holder bound for u}
\eeq
Interpolating this with the $L^p$-time-continuity of $u$ due to Lemma \ref{lem: time continuity of u in L^p}, we know that, for any $\b\in [0,\f1p)$, $t\mapsto u(x,t)$ is continuous in $(0,T]$ in the $C^\b_x(\BR^2)$-topology.

The rest of the argument is similar to that right after \eqref{eqn: alternative expression of W}.
From the $C^1_s(\BT)$-time-continuity of $t\mapsto X(s,t)$ as well as the local $C^{1/p}_s(\BT)$-bound of $X(s,t)$, one can deduce that $t\mapsto X(t)$ is continuous in $(0,T]$ in the $C^\b_s(\BT)$-topology for any $\b\in [0,\f1p)$.
Observe that (cf.\;\eqref{eqn: alternative expression of W})
\beqo
W[u,X] = u\circ X +\f14 \Lam X,
\eeqo
so \eqref{eqn: a singular Holder bound for u}, Lemma \ref{lem: L inf regularity of u mild solution}, and the bound for $X$ imply that
\[
\sup_{\eta\in (0,T]} \eta^{\f{1}{p}} \|W[u,X](\eta)\|_{L_s^\infty} < +\infty,
\quad
\sup_{\eta\in (0,T]} \eta^{\f{3}{2p}} \|W[u,X](\eta)\|_{\dot{C}_s^{1/p}} < +\infty,
\]
Since $p>2$, by Lemma \ref{lem: parabolic estimates for Duhamel term of fractional Laplace},
\[
\sup_{\eta\in (0,T]} \big\|\CI[W](\eta)\big\|_{L_s^\infty}
< +\infty,
\quad
\sup_{\eta\in (0,T]} \eta^{\f{3}{2p}} \big\|\CI[W](\eta)\big\|_{\dot{C}_s^{1,\f{1}{p}}} < +\infty.
\]
Moreover, the regularity and time-continuity of $u$ and $X$ implies that $t\mapsto W[u,X]$ is continuous in the $C_s(\BT)$-topology in $(0,T]$.
Hence, by Lemma \ref{lem: parabolic estimates for Duhamel term of fractional Laplace} again, $\pa_t \CI [W](t)$ is well-defined pointwise for all $t\in (0,T)$, and
\[
\pa_t \CI[W](t) = -\f14\Lam \CI[W](t) + W[u,X](t).
\]
Then the desired claims follow from the representation \eqref{eqn: fixed pt equation for X} of $X$.

Lastly, the $L^\infty$-bound for $\pa_t X$ follows from Lemma \ref{lem: L inf regularity of u mild solution} and Lemma \ref{lem: Holder estimate for composition of functions}:
\[
\sup_{\eta\in (0,T)} \eta^{\f1{p}} \|\pa_t X(\eta)\|_{L_s^\infty(\BT)}
\leq
\sup_{\eta\in (0,T)} \eta^{\f1{p}} \|u(\eta)\|_{L_x^\infty}
\leq C(T,p,\nu,\lam,M,Q).
\]
\end{proof}
\begin{rmk}
For future use, we remark that, by Lemma \ref{lem: Holder estimate for composition of functions} and \eqref{eqn: a singular Holder bound for u}, one can similarly show that
\beq
\begin{split}
\sup_{\eta\in (0,T)} \eta^{\f3{2p}} \|\pa_t X(\eta)\|_{C_s^{1/p}(\BT)}
\leq &\;
\sup_{\eta\in (0,T)} \eta^{\f3{2p}}
\left(\|u(\eta)\|_{L_x^\infty}
+ \|u(\eta)\|_{\dot{C}_x^{1/p}} \|X'(\eta)\|_{L^\infty_s}^{1/p}\right)\\
\leq &\; C(T, p,\nu,\lam,M,Q).
\end{split}
\label{eqn: C 1/p estimate for X_t}
\eeq
\end{rmk}
\end{lem}

In the next lemma, we shall prove that the mild solution $(u,X)$ admits improved regularity than what is dictated in Definition \ref{def: mild solution}.

\begin{lem}
\label{lem: C gamma estimate for mild solutions}
For any $\g\in [\f2p,1)$,
\begin{align*}
&\; \sup_{\eta\in (0,T]}\eta^\g
\|u(\eta)\|_{\dot{C}^{\g}_x(\BR^2)}
+ \sup_{\eta\in (0,T]}\eta^\g\|X'(\eta)\|_{C^{\g}_s(\BT)}
+ \sup_{\eta\in (0,T)}\eta^\g\|\pa_t X(\eta)\|_{C_s^{\g}(\BT)}\\
\leq &\; C(T, \g, p,\nu,\lam,M,Q).
\end{align*}
\begin{proof}

We first bound a good part of the representation \eqref{eqn: fixed pt equation for X}.
Thanks to \eqref{eqn: 6th estimate for B new form general gamma} in Lemma \ref{lem: estimate for B_nu u} and Lemma \ref{lem: L inf regularity of u mild solution}, for any $\b\in [\f2p,1)$, and any $t\in (0,T]$,
\beq
\begin{split}
&\; \big\|e^{\nu t\D}u_0+B_{\nu}[u](t)\big\|_{\dot{C}_x^{\b}(\BR^2)}
\\
\leq &\; C (\nu t)^{-\f1p-\f{\b}{2}} \|u_0\|_{L^p_x}
+ \nu^{-1}(\nu t)^{\f{1-\b}{2}-\f2p}
\left(\|u\|_{L^\infty_T L^p_x}+\sup_{\tau\in (0,T]}(\nu \tau)^{\f1p}\|u(\tau)\|_{L^\infty_x}\right)^2
\\
\leq &\; C(T,\b, p,\nu,\lam,M,Q)
\left(t^{-\f1p-\f{\b}{2}}
+ t^{\f{1-\b}{2}-\f2p}\right).
\end{split}
\label{eqn: C gamma estimate for good part of W}
\eeq
Hence, by Lemma \ref{lem: parabolic estimates for Duhamel term of fractional Laplace} and Lemma \ref{lem: Holder estimate for composition of functions},
\begin{align*}
&\; \sup_{\eta\in (0,T]} \eta^\b \left\| \CI\Big[\big(e^{\nu\eta\D}u_0+B_{\nu}[u]\big) \circ X(\eta)\Big]\right\|_{\dot{C}^{1,\b}_s(\BT)}
\\
\leq &\; C\sup_{\eta\in (0,T]} \eta^\b \Big\|\big(e^{\nu \eta \D}u_0+B_{\nu}[u]\big) \circ X(\eta)\Big\|_{\dot{C}^{\b}_s(\BT)}
\\
\leq &\; C \sup_{\eta\in (0,T]} \eta^\b \big\|e^{\nu \eta\D}u_0+B_{\nu}[u](\eta)\big\|_{\dot{C}_x^{\b}} \|X'\|_{L^\infty_{T} L_s^\infty}^\b
\\
\leq &\; C(T,\b, p,\nu,\lam,M,Q).
\end{align*}

Next we handle the other terms in \eqref{eqn: fixed pt equation for X}.
By Lemma \ref{lem: improved estimates for g_X}, for any $\b\in (0,1)$,
\[
\big\|g_{X}(t)\big\|_{\dot{C}_s^{\b}(\BT)}\\
\leq C \lam^{-2} \|X'(t)\|_{L_s^\infty} \|X'(t)\|_{\dot{C}_s^{\b/2}}^2.
\]
By \eqref{eqn: Holder estimate for h} in Lemma \ref{lem: estimate for h_X^nu}, with the parameter there chosen as
\[
(\al,\g,\g',\mu_0,\mu,\mu'):=\left(\b,\, \f{\b}{2},\, \f{\b}{2},\, \f1p,\, \f\b2+\f1p,\, \f\b2+\f1p\right),
\]
we derive that
\begin{align*}
&\; \|H_X^\nu(t)\|_{\dot{C}_s^{\b}(\BT)}
\leq \|h_X^\nu(t)\|_{\dot{C}_x^{\b}(\BR^2)} \| X'(t)\|_{L_s^\infty(\BT)}^\b
\\
\leq &\; C (\nu t)^{-\f{\b}{2}} \lam^{-1} t^{-\f{\b}{2p}} \|X'\|_{L^\infty_t L_s^\infty}
\\
&\;\cdot \left(\sup_{\eta\in (0,t]}\eta^{\f{\b}{2}} \|X'(\eta)\|_{C_s^{\b/2}}\right)^{1-\f{\b}{2}} \left(\sup_{\eta\in (0,t]}\eta^{\f{\b}{2}+\f1p}\|\pa_t X(\eta)\|_{\dot{C}_s^{\b/2}}\right)^{\f{\b}{2}} \cdot M^\b \\
&\; + C (\nu t)^{-\f{1}{2}-\f\b4}
\min\left(\lam^{-\f{\b}{2}},\, (\nu t)^{-\f{\b}4} \right)\lam^{-1}
\\
&\;\quad \cdot  \left(t^{1-\f{\b}{2}-\f1p} \|X'\|_{L^\infty_t L_s^\infty}
\sup_{\eta\in (0,t]}\eta^{\f{\b}{2}} \|X'(\eta)\|_{C_s^{\b/2}} \sup_{\eta\in (0,t]}\eta^{\f1p}\|\pa_t X(\eta)\|_{L_s^\infty} \right.\\
&\;\qquad \left.+ t^{1-\f{\b}{2}-\f1p}\|X'\|_{L^\infty_t L_s^\infty}^2 \sup_{\eta\in (0,t]}\eta^{\f{\b}{2}+\f1p}\|\pa_t X(\eta)\|_{\dot{C}_s^{\b/2}} \right)\cdot M^\b
\\
\leq &\; C t^{-\f{\b}{2}-\f{\b}{2p}} \nu^{-\f{\b}{2}} \lam^{-1} M^{1+\b}
\left(\sup_{\eta\in (0,t]}\eta^{\f{\b}{2}} \|X'(\eta)\|_{C_s^{\b/2}}+\sup_{\eta\in (0,t]}\eta^{\f{\b}{2}+\f1p}\|\pa_t X(\eta)\|_{\dot{C}_s^{\b/2}}\right) \\
&\; + C(T,p,\nu,\lam,M,Q) \cdot
t^{\f{1}{2}-\f1p-\f{3\b}{4}}
\nu^{-\f{1}{2}-\f\b4} \lam^{-1-\f{\b}{2}}
\\
&\;\quad \cdot  \left(
\sup_{\eta\in (0,t]}\eta^{\f{\b}{2}} \|X'(\eta)\|_{C_s^{\b/2}}
+ \sup_{\eta\in (0,t]}\eta^{\f{\b}{2}+\f1p}\|\pa_t X(\eta)\|_{\dot{C}_s^{\b/2}} \right)\cdot M^\b,
\end{align*}
where $C$ depends on $\b$ and $p$.
By Lemma \ref{lem: estimate for u_X} and Lemma \ref{lem: parabolic estimates}, for any $\b\in (0,1)$,
\begin{align*}
\big\|e^{\nu t\D}u_{X(t)}\circ X(t)\big\|_{\dot{C}_s^\b(\BT)}
\leq &\; \big\|e^{\nu t\D}u_{X(t)}\big\|_{\dot{C}_x^\b} \|X'\|_{L^\infty_{T} L^\infty_s}^\b
\\
\leq &\; C(\nu t)^{-\f{\b}4} \|u_{X(t)}\|_{\dot{C}_x^{\b/2}} M^\b\\
\leq &\; C t^{-\f{3\b}4} \nu^{-\f{\b}4} \cdot \lam^{-1-\f\b2} M^{1+\b} \sup_{\eta\in (0,T]} \eta^{\b/2}\|X'(\eta)\|_{\dot{C}^{\b/2}_s}.
\end{align*}

Now we take $\b = \f2p$.
Then Lemma \ref{lem: parabolic estimates for Duhamel term of fractional Laplace} and the estimates above (also recall \eqref{eqn: C 1/p estimate for X_t}) give that
\begin{align*}
&\; \sup_{\eta\in (0,T]}\eta^{\f2p} \Big\|\CI[g_X+H_X^\nu - e^{\nu \eta\D}u_{X(\eta)}\circ X(\eta)\big]\Big\|_{\dot{C}^{1,2/p}_s(\BT)}\\
\leq &\;  C\sup_{\eta\in (0,T]}\eta^{\f2p}
\big\|g_X(\eta)+H_X^\nu(\eta)-e^{\nu \eta\D}u_{X(\eta)}\circ X(\eta)\big\|_{\dot{C}^{2/p}_s}\\
\leq &\; C(T, p,\nu,\lam,M,Q).
\end{align*}
Therefore,
\begin{align*}
\sup_{\eta\in (0,T]}\eta^{\f2p} \|X'(\eta)\|_{C^{2/p}_s}
\leq &\; C\|X_0'\|_{L^\infty_s} + \sup_{\eta\in (0,T]} \eta^{\f2p} \left\| \CI\Big[\big(e^{\nu\eta\D}u_0+B_{\nu}[u]\big) \circ X(\eta)\Big]\right\|_{\dot{C}^{1,2/p}_s}\\
&\; + \sup_{\eta\in (0,T]}\eta^{\f2p} \Big\|\CI[g_X+H_X^\nu - e^{\nu \eta\D}u_{X(\eta)}\circ X(\eta)\big]\Big\|_{\dot{C}^{1,2/p}_s}
\\
\leq &\; C(T, p,\nu,\lam,M,Q).
\end{align*}
By Lemma \ref{lem: time derivative and equation}, \eqref{eqn: C gamma estimate for good part of W}, as well as the argument in \eqref{eqn: C 1/p estimate for u} and \eqref{eqn: a singular Holder bound for u}, we can further show that
\[
\sup_{\eta\in (0,T]} \eta^{\f2p} \|u\|_{\dot{C}^{2/p}_x(\BR^2)}
\leq C(T, p,\nu,\lam,M,Q),
\]
and thus
\[
\sup_{\eta\in (0,T)} \eta^{\f2p}\|\pa_t X(\eta)\|_{C_s^{2/p}(\BT)}
\leq C(T, p,\nu,\lam,M,Q).
\]

We may bootstrap with this improved regularity.
For instance, now we can take arbitrary $\b\in (\f2p,\f4p)\cap (0,1)$ and apply the preceding estimates once again.
For any fixed $\g\in [\f2p,1)$, it suffices to repeat the above argument for finitely many times to reach the desired conclusion.
\end{proof}
\end{lem}

\begin{lem}
\label{lem: time continuity of u in L inf}
For any $\g\in [\f2p,1)$, and any $0 < t_1< t_2\leq T$,
\[
t_1^\g \|u(t_1)-u(t_2)\|_{L^\infty_x}
\leq C(T, \g, p,\nu,\lam,M,Q) |t_1-t_2|^{\g/2}.
\]

\begin{proof}
One may argue as in \eqref{eqn: bound for time continuity for v prelim}, and then apply \eqref{eqn: def of M}-\eqref{eqn: def of Q}, Lemma \ref{lem: L inf regularity of u mild solution} and Lemma \ref{lem: C gamma estimate for mild solutions} to find that, for any $0 < t_1< t_2\leq T$,
\begin{align*}
&\;\|u(t_1)-u(t_2)\|_{L^\infty_x}\\
\leq &\;
C\big(\nu |t_1-t_2|\big)^{\g/2} t_1^{-\g} \cdot \lam^{-1-\g} \|X'\|_{L^\infty_T L_s^{\infty}} \sup_{\eta\in (0,T]} \eta^\g \|X'(\eta)\|_{\dot{C}_s^{\g}} \\
&\; + C(\nu|t_1-t_2|)^{\g/2} (\nu t_1)^{-\f1{p}-\f{\g}{2}} \|u_0\|_{L^p_x} \\
&\; +
C \nu^{-1} \big(\nu(t_2-t_1)\big)^{\g/2}(\nu t_1)^{-\f{1+\g}{2}}
\left(\sup_{\eta\in (0,t_1]} (\nu \eta)^{\f12-\f1p} \|u(\eta)\|_{L^p_x}+\sup_{\eta\in (0,t_1]}(\nu \eta)^{\f12}\|u(\eta)\|_{L^\infty_x}\right)^2\\
&\;
+ C \nu^{-1} \big(\nu(t_2-t_1)\big)^{\g/2}(\nu t_2)^{-\f{1+\g}{2}}
\left(\sup_{\eta\in (0,t_2]} (\nu \eta)^{\f12-\f1p} \|u(\eta)\|_{L^p_x}+\sup_{\eta\in (0,t_2]}(\nu \eta)^{\f12}\|u(\eta)\|_{L^\infty_x}\right)^2\\
\leq &\;
C(T, \g, p,\nu,\lam,M,Q) |t_1-t_2|^{\g/2} t_1^{-\g}.
\end{align*}
\end{proof}
\end{lem}

\begin{lem}
\label{lem: intermediate norms vanish as t goes to zero}
For any $\b\in (0,1)$,
$\lim_{t\to 0^+} t^{\b}\|X'(t)\|_{\dot{C}_s^\b(\BT)} = 0$.

\begin{proof}
Take an arbitrary $\g\in (\b,1)$.
By the time-continuity of $X$ and Lemma \ref{lem: C gamma estimate for mild solutions},
\[
\sup_{\eta\in (0,T]} \eta^{\g}\|X'(\eta)\|_{\dot{C}^\g_s} < +\infty,
\quad
\lim_{t\to 0^+}\|X'(t)-X_0'\|_{L^\infty_s} = 0.
\]
On the other hand, $X_0\in C^1(\BT)$ implies (see \eqref{eqn: time weighted higher order norm vanish at 0})
\[
\lim_{t\to 0^+}\|X_0' - e^{-\f{t}{4}\Lam}X_0'\|_{L_s^\infty} = 0,\quad
\lim_{t\to 0^+} t^{\b}\|e^{-\f{t}{4}\Lam}X_0'\|_{\dot{C}^\b_s} = 0.
\]
Hence,
\begin{align*}
&\;t^{\b}\|X'(t)\|_{\dot{C}^\b_s(\BT)}\\
\leq &\; t^{\b}\|e^{-\f{t}{4}\Lam}X_0'\|_{\dot{C}^\b_s}
+ t^{\b}\|X'(t) - e^{-\f{t}{4}\Lam}X_0'\|_{\dot{C}^\b_s}
\\
\leq &\; t^{\b}\|e^{-\f{t}{4}\Lam}X_0'\|_{\dot{C}^\b_s}
+ C\left(t^{\g}\|X'(t)\|_{\dot{C}^\g_s}
+ t^{\g}\|e^{-\f{t}{4}\Lam}X_0'\|_{\dot{C}^\g_s}\right)^{\f\b\g}
\|X'(t) - e^{-\f{t}{4}\Lam}X_0'\|_{L_s^\infty}^{1-\f\b\g}
\\
\leq &\; t^{\b}\|e^{-\f{t}{4}\Lam}X_0'\|_{\dot{C}^\b_s}
+ C\left(t^{\g}\|X'(t)\|_{\dot{C}^\g_s}
+ \|X_0'\|_{L_s^\infty}\right)^{\f\b\g}
\left(\|X'(t) - X_0'\|_{L_s^\infty}
+ \|X_0' - e^{-\f{t}{4}\Lam}X_0'\|_{L_s^\infty}\right)^{1-\f\b\g},
\end{align*}
which converges to $0$ as $t\to 0^+$.
\end{proof}
\end{lem}

\subsection{Uniqueness and continuation of the mild solutions}
\label{sec: uniqueness and continuation}

Now we prove the uniqueness of the mild solutions.

\begin{prop}
\label{prop: uniqueness of mild solution}
Let $u_0 = u_0(x)\in L^p(\BR^2)$ for some $p\in (2,\infty)$ and $X_0 = X_0(s)\in C^1(\BT)$, which satisfy $\di u_0 = 0$ and $|X_0|_* >0$.
Suppose that, for some $T>0$, $(u_1,X_1)$ and $(u_2,X_2)$ are two mild solutions to \eqref{eqn: NS equation}-\eqref{eqn: initial data} on $[0,T]$ in the sense of Definition \ref{def: mild solution}.
Then $u_1(t)\equiv u_2(t)$ and $X_1 (t)\equiv X_2(t)$ for all $t\in [0,T]$.

\begin{rmk}
As a corollary, the solution constructed in the preceding subsections does not depend on the choice of $\g$, and it is the unique mild solution.
\end{rmk}

\begin{proof}
Fix $\g \in [\f2p,1)\cap (1-\f2p,1)$.
By virtue of Lemma \ref{lem: L inf regularity of u mild solution}, Lemma \ref{lem: time derivative and equation} and Lemma \ref{lem: C gamma estimate for mild solutions}, we can denote
\begin{align*}
M_* := &\; \sup_{\eta\in [0,T]} \|X_1'(\eta)\|_{L^\infty_s}
+ \sup_{\eta\in (0,T]}\eta^{\g} \|X_1'(\eta)\|_{\dot{C}^{\g}_s} \\
&\;
+ \sup_{\eta\in [0,T]} \|X_2'(\eta)\|_{L^\infty_s}
+ \sup_{\eta\in (0,T]}\eta^{\g} \|X_2'(\eta)\|_{\dot{C}^{\g}_s}
< +\infty,
\end{align*}
\[
\lam_*: = \min\left\{\inf_{t\in [0,T]}|X_1(t)|_*,\,
\inf_{t\in [0,T]}|X_2(t)|_*\right\}>0,
\]
\begin{align*}
N_*:=&\;
\sup_{\eta\in (0,T)}\eta^{\f1p}\|\pa_t X_1(\eta)\|_{L_s^\infty}
 + \sup_{\eta\in (0,T)} \eta^{\g} \|\pa_t X_1(\eta)\|_{C_s^{\g}}\\
&\;+ \sup_{\eta\in (0,T)}\eta^{\f1p}\|\pa_t X_2(\eta)\|_{L_s^\infty}
+ \sup_{\eta\in (0,T)} \eta^{\g} \|\pa_t X_2(\eta)\|_{C_s^{\g}}
< +\infty.
\end{align*}
and
\begin{align*}
Q_*:= &\; \sup_{\eta\in [0,T]} \|u_1(\eta)\|_{L^p_x}
+ \sup_{\eta\in (0,T]} (\nu\eta)^{\f1p}\|u_1(\eta)\|_{L^\infty_x}
+
\sup_{\eta\in (0,T]} (\nu\eta)^{\g}\|u_1(\eta)\|_{\dot{C}^{\g}_x}\\
&\; + \sup_{\eta\in [0,T]} \|u_2(\eta)\|_{L^p_x}
+ \sup_{\eta\in (0,T]} (\nu\eta)^{\f1p}\|u_2(\eta)\|_{L^\infty_x}
+ \sup_{\eta\in (0,T]} (\nu\eta)^{\g}\|u_2(\eta)\|_{\dot{C}^{\g}_x}
< +\infty.
\end{align*}

Clearly, by \eqref{eqn: fixed pt equation for u} and \eqref{eqn: fixed pt equation for X}, $(u_1-u_2,X_1-X_2)$ satisfies on $[0,T]$ that
\begin{align*}
(u_1-u_2)(t) = &\; \big(u_{X_1}^\nu - u_{X_2}^\nu\big)(t) + \big(B_{\nu}[u_1] - B_{\nu}[u_2]\big)(t),\\
(X_1-X_2)(t) = &\; \CI \big[W[u_1,X_1]-W[u_2,X_2]\big](t).
\end{align*}
For convenience, we define for $0\leq t_1 < t_2\leq T$ that (cf.\;\eqref{eqn: def of d_U})
\[
\psi_u(t_1,t_2) := \sup_{\eta\in (t_1,t_2]} \big(\nu (\eta-t_1)\big)^{\f\g2+\f1p-\f12} \|(u_1-u_2)(\eta)\|_{M^{1,1-\g}(\BR^2)},
\]
and (cf.\;\eqref{eqn: def of d_X})
\begin{align*}
\psi_X(t_1,t_2) := &\; \sup_{\eta\in [t_1,t_2]}\|(X_1'-X_2')(\eta)\|_{L_s^\infty(\BT)}
+ \sup_{\eta\in (t_1,t_2]} (\eta-t_1)^\g \|(X_1'-X_2')(\eta)\|_{\dot{C}^\g_s(\BT)}\\
&\; + \sup_{\eta\in (t_1,t_2)} (\eta-t_1)^{\f1p}
\|\pa_t (X_1-X_2)(\eta)\|_{L^\infty_s(\BT)}
+ \sup_{\eta\in (t_1,t_2)} (\eta-t_1)^\g
\|\pa_t (X_1-X_2)(\eta)\|_{C^\g_s(\BT)}.
\end{align*}
Given the estimates for $(u_i,X_i)$, they are well-defined functions.

Take $T_0\in (0,T]$ with $T_0\leq 1$, such that for any $t\in [0,T_0]$,
\beq
\|X_1'(t)-X_2'(t)\|_{L^\infty_s} \leq \f{\lam_*}{2},
\label{eqn: C^1 difference between X_1 and X_2}
\eeq
and that, for any $t_1,t_2\in [0,T_0]$,
\beq
\|X_1'(t_1)-X_1'(t_2)\|_{L^\infty_s} \leq \f{\lam_*}{2},
\quad
\|X_2'(t_1)-X_2'(t_2)\|_{L^\infty_s} \leq \f{\lam_*}{2}.
\label{eqn: C^1 difference of X_1 and X_2 between different times}
\eeq
This can be achieved since $X_1,X_2\in C([0,T];C^1(\BT))$.

Arguing as before (see \eqref{eqn: dist between v_1 and v_2}), we find that
\begin{align*}
\psi_u(0,T_0)
\leq &\; C\left(T_0^{\f{1-\g} 2} +  T_0^{\f\g2+\f1p-\f12}\right) \nu^{\f{\g}{2}+\f1p-\f12} \lam_*^{-1-\g} M_*^2 \cdot \psi_X(0,T_0)\\
&\; + CT_0^{\f12-\f1p}  \nu^{-\f12-\f1p} Q_* \cdot \psi_u(0,T_0),
\end{align*}
where $C$ depends on $p$ and $\g$.

On the other hand
(see \eqref{eqn: estimate for Y_1 - Y_2}),
\begin{align*}
\psi_X(0,T_0)
\leq &\; C\sup_{\eta\in (0,T_0]} \eta^{\f1p} \big\|W[u_1,X_1](\eta)-W[u_2,X_2](\eta)\big\|_{L^\infty_s(\BT)}\\
&\; + C\sup_{\eta\in (0,T_0]} \eta^\g \big\|W[u_1,X_1](\eta)-W[u_2,X_2](\eta)\big\|_{\dot{C}^\g_s(\BT)}.
\end{align*}
We can analogously bound that (see \eqref{eqn: L inf difference of w})
\begin{align*}
&\;\sup_{\eta\in (0,T_0]} \eta^{\f1p}\big\|W[u_1,X_1](\eta)-W[u_2,X_2](\eta)\big\|_{L^\infty_s(\BT)}\\
\leq &\; C(\g, p,\nu,\lam_*,Q_*, M_*,N_*)
\big[ \psi_X(0,T_0)+\psi_u(0,T_0)\big]
\Big(T_0^{\f1{2p}} + T_0^{\f{1-\g}{2}}\Big).
\end{align*}
Besides (cf.\;\eqref{eqn: C gamma estimate for W_1 - W_2}), for any $t\in (0,T_0]$,
\begin{align*}
&\; \big\|W[u_1,X_1](t)-W[u_2,X_2](t) \big\|_{\dot{C}_s^\g(\BT)}\\
\leq &\; \|g_{X_1}(t) - g_{X_2}(t)\|_{ \dot{C}_s^\g(\BT)}
+ C(\g, p,\nu,\lam_*,Q_*,M_*,N_*) \cdot \psi_X(0,T_0) \cdot t^{\f12-\f1p-\g}
\\
&\; + C\nu^{-\f{1+\g}2-\f2p}
Q_* M_*^\g\cdot \psi_u(0,T_0)\cdot t^{\f{1-\g}2-\f2p}.
\end{align*}
Thanks to Lemma \ref{lem: improved estimates for g_X-g_Y},
\begin{align*}
&\; \|g_{X_1}(t) - g_{X_2}(t)\|_{ \dot{C}_s^\g(\BT)}\\
\leq &\; C \lam_*^{-2} M_* \big(\|X_1'(t)\|_{\dot{C}_s^{\g/2}}+\|X_2'(t)\|_{\dot{C}_s^{\g/2}}\big)\\
&\;\cdot \Big[\|X_1'(t)-X_2'(t)\|_{\dot{C}_s^{\g/2}} + \lam_*^{-1} \big(\|X'_1(t)\|_{\dot{C}^{\g/2}_s}+\|X_2'(t)\|_{\dot{C}^{\g/2}_s}\big) \|X_1'(t)-X_2'(t)\|_{L^\infty_s}\Big]
\\
\leq &\; Ct^{-\g} \lam_*^{-2} M_* \cdot t^{\g/2} \big(\|X_1'(t)\|_{\dot{C}_s^{\g/2}}+\|X_2'(t)\|_{\dot{C}_s^{\g/2}}\big) \cdot \lam_*^{-1} M_* \psi_X(0,T_0),
\end{align*}
with $C$ depending on $\g$ only.
This gives
\begin{align*}
&\; \sup_{\eta\in (0,T_0]} \eta^\g \big\|W[u_1,X_1](\eta)-W[u_2,X_2](\eta) \big\|_{\dot{C}_s^\g(\BT)}\\
\leq &\; C(\g,\lam_*,M_*) \psi_X(0,T_0)
\sup_{\eta\in (0,T_0]} \eta^{\g/2}\big(\|X_1'(\eta)\|_{\dot{C}_s^{\g/2}} +\|X_2'(\eta)\|_{\dot{C}_s^{\g/2}}\big).
\\
&\; + C(\g, p,\nu,\lam_*,Q_*,M_*,N_*)
\big[ \psi_X(0,T_0)  + \psi_u(0,T_0)\big] T_0^{\f12-\f1p}.
\end{align*}
Plugging them into the estimate for $\psi_X(0,T_0)$, and then combining with the estimate for $\psi_u(0,T_0)$, we finally obtain that
\begin{align*}
&\; \psi_X(0,T_0)+\psi_u(0,T_0)\\
\leq &\; \big(\psi_X(0,T_0)+\psi_u(0,T_0)\big)\\
&\;\cdot \Big[C_1(\g,\lam_*,M_*) \sup_{\eta\in (0,T_0]} \eta^{\g/2} \big(\|X_1'(\eta)\|_{\dot{C}_s^{\g/2}}+\|X_2'(\eta)\|_{\dot{C}_s^{\g/2}}\big)
\\
&\;\quad + C_2(\g, p,\nu,\lam_*,Q_*, M_*,N_*)
\Big(T_0^{\f1{2p}} + T_0^{\f{1-\g}{2}} +  T_0^{\f\g2+\f1p-\f12}\Big)\Big].
\end{align*}
By virtue of Lemma \ref{lem: intermediate norms vanish as t goes to zero} and the fact $\g\in (1-\f2p,1)$, we may take $T_0$ smaller if necessary, so that
\begin{align*}
&\; C_1(\g,\lam_*,M_*) \sup_{\eta\in (0,T_0]} \eta^{\g/2}\big(\|X_1'(\eta)\|_{\dot{C}^{\g/2}}+\|X_2'(\eta)\|_{\dot{C}^{\g/2}}\big)
\\
&\; + C_2(\g, p,\nu,\lam_*,Q_*, M_*,N_*)
\Big(T_0^{\f1{2p}} + T_0^{\f{1-\g}{2}}
+ T_0^{\f\g2+\f1p-\f12}\Big)
\leq \f12.
\end{align*}
This implies that $X_1(s,t)\equiv X_2(s,t)$ and $u_1(x,t)\equiv u_2(x,t)$ for all $t\in [0,T_0]$.

If $T_0<T$, we shall repeat this argument with the new initial time $T_0$.
Let us only sketch it here.
We first derive that, with $T_1\in (T_0,T]$ to be chosen later, for any $t_1,t_2\in [T_0,T_1]$ with $t_1<t_2$,
\[
\|X_i'(t_1)-X_i'(t_2)\|_{L^\infty_s}
\leq C\|X_i(t_1)-X_i(t_2)\|_{\dot{C}_s^\g}^\g
\|(X_i'(t_1),X_i'(t_2))\|_{\dot{C}_s^\g}^{1-\g}
\leq  C|t_1-t_2|^\g t_1^{-\g} M_*^{1-\g} N_*^\g.
\]
This together with the fact $X_1(T_0) = X_2(T_0)$ implies that there exists $\d = \d(T_0,\lam_*,M_*,N_*)$, such that as long as $T_1\leq T_0 + \d$, \eqref{eqn: C^1 difference between X_1 and X_2} holds for any $t\in [T_0,T_1]$ and \eqref{eqn: C^1 difference of X_1 and X_2 between different times} holds for any $t_1,t_2\in [T_0,T_1]$.
Now viewing $T_0$ as the new initial data, the same argument as above leads to
\begin{align*}
&\; \psi_X(T_0,T_1)+\psi_u(T_0,T_1)\\
\leq &\; \big( \psi_X(T_0,T_1)+\psi_u(T_0,T_1) \big)\\
&\;\cdot \Big[C_3(\g,\lam_*,M_*) \sup_{\eta\in (T_0,T_1]} (\eta-T_0)^{\g/2} \big(\|X_1'(\eta)\|_{\dot{C}_s^{\g/2}}+\|X_2'(\eta)\|_{\dot{C}_s^{\g/2}}\big)
\\
&\;\quad + C_4(\g, p,\nu,\lam_*,Q_*, M_*,N_*)
\Big((T_1-T_0)^{\f1{2p}} + (T_1-T_0)^{\f{1-\g}{2}} +  (T_1-T_0)^{\f\g2+\f1p-\f12}\Big)\Big].
\end{align*}
This time, instead of applying Lemma \ref{lem: intermediate norms vanish as t goes to zero}, we observe that
\[
\sup_{\eta\in (T_0,T_1]} (\eta-T_0)^{\g/2} \left(\|X_1'(\eta)\|_{\dot{C}_s^{\g/2}}+\|X_2'(\eta)\|_{\dot{C}_s^{\g/2}}\right)
\leq C\left(\f{T_1-T_0}{T_0}\right)^{\g/2} M_*.
\]
Hence, there exists $\d_0 = \d_0(T_0,\g, p,\nu,\lam_*,Q_*, M_*,N_*) \leq \d$, such that as long as $T_1\leq T_0+\d_0$, it holds that
\begin{align*}
&\;C_3(\g,\lam_*,M_*) \sup_{\eta\in (T_0,T_1]} (\eta-T_0)^{\g/2} \big(\|X_1'(\eta)\|_{\dot{C}_s^{\g/2}}+\|X_2'(\eta)\|_{\dot{C}_s^{\g/2}}\big)
\\
&\; + C_4(\g, p,\nu,\lam_*,Q_*, M_*,N_*)
\Big((T_1-T_0)^{\f1{2p}} + (T_1-T_0)^{\f{1-\g}{2}} +  (T_1-T_0)^{\f\g2+\f1p-\f12}\Big) \leq \f12.
\end{align*}
This gives $\psi_X(T_0,T_1)+\psi_u(T_0,T_1) = 0$,
and thus $X_1(t)\equiv X_2(t)$ and $u_1(t)\equiv u_2(t)$ in $t\in [T_0,T_1]$.
Then we repeat this argument, and each time we may extend the time interval by $\d_0$ if it has not exceeded $T$.
This leads to the conclusion that $X_1(t)\equiv X_2(t)$ and $u_1(t)\equiv u_2(t)$ for all $t\in [0,T]$.
\end{proof}
\end{prop}

The next proposition is concerned with continuation of the mild solutions.

\begin{prop}
\label{prop: continuation}
Let $u_0 = u_0(x)\in L^p(\BR^2)$ for some $p\in (2,\infty)$ and $X_0 = X_0(s)\in C^1(\BT)$, which satisfy $\di u_0 = 0$ and $|X_0|_* >0$.
Suppose that for some $T>0$, $(u,X)$ is a mild solution to \eqref{eqn: NS equation}-\eqref{eqn: initial data} on $[0,T]$ in the sense of Definition \ref{def: mild solution}.
Then it can be extended to a longer time interval, i.e., for some $\tilde{T}>T$, there exists $(\tilde{u},\tilde{X})$ forming a mild solution to \eqref{eqn: NS equation}-\eqref{eqn: initial data} on $[0,\tilde{T}]$ and satisfying $(u,X)=(\tilde{u},\tilde{X})$ on $[0,T]$.

\begin{proof}

Fix $\g\in [\f2p,1)\cap (1-\f2p,1)$.
By the definition of the mild solution as well as Lemma \ref{lem: C gamma estimate for mild solutions}, $u(T)\in L_x^p\cap C_x^\g(\BR^2)$ and $X(T)\in C_s^{1,\g}(\BT)$, and they satisfy $\di u(T) = 0$ and $|X(T)|_*>0$.
Denote
\[
M_1:= \|X'(T)\|_{C^\g_s(\BT)},\quad \lam_1: = |X(T)|_*>0,
\quad
Q_1:=\|u(T)\|_{L_x^p(\BR^2)}+\|u(T)\|_{L^\infty_x(\BR^2)} + \lam_1^{-1+\f2p}M_1^2.
\]
Then we apply Proposition \ref{prop: fixed-point solution} with the same $p$ and $\g$ to construct a local mild solution starting from the initial data $(u(T),X(T))$.
More precisely, there exists $T_1\in (0,1)$ and a mild solution $(u_1,X_1)$ on $[0,T_1]$ to \eqref{eqn: NS equation}-\eqref{eqn: kinematic equation} with the initial condition $u_1(x,0) = u(x,T)$, and $X_1(s,0) = X(s,T)$.
Thanks to Proposition \ref{prop: fixed-point solution}, it is known that
\begin{itemize}
\item
\[
\|u_1\|_{L^\infty_{T_1} L^p_x(\BR^2)} \leq C Q_1,\quad
\|X_1'\|_{L^\infty_{T_1} L^\infty_s(\BT)} < +\infty, \quad
\inf_{t\in [0,T_1]}|X_1(t)|_* > \lam_1/2,
\]
where the constant $C$ only depends on $p$ (see \eqref{eqn: choice of Q} and the remark that follows it);

\item and with $\s : = \f12(\f12-\f1p) \in (0,\g)$,
\beq
t^{\g} \|X_1(t)\|_{\dot{C}_s^{1,\g}(\BT)}
+ \|X_1(t)-e^{-\f{t}4\Lam}X(T)\|_{C_s^1(\BT)}
\leq 4\big(\r(t)+ t^{\s}\big),
\label{eqn: basic bound for X_1}
\eeq
where $\r(t)$ is defined in terms of $X_1(s,0) = X(s,T)$ (cf.\;\eqref{eqn: def of rho}):
\[
\r(t) = \r_{X(T),\g}(t):= \sup_{\tau\in (0,t]}\tau^{\g}\big\|e^{-\f{\tau}{4} \Lam}X'(T)\big\|_{\dot{C}^\g_s(\BT)}.
\]
\end{itemize}

Our goal is to show
\beq
\|X_1'\|_{L^\infty_{T_1} C^{1/p}_s(\BT)} < +\infty.
\label{eqn: claim for improved regularity of X_1}
\eeq

Since $X(T)\in C_s^{1,\g}(\BT)$, we find that
$\r(t)\leq t^\g M_1$.
Combining this with \eqref{eqn: basic bound for X_1}, we derive that, for any $t\in (0,T_1]$,
\beq
\begin{split}
\|X_1(t)\|_{\dot{C}_s^{1,\s}(\BT)}
\leq &\; \big\|X_1(t)-e^{-\f{t}4\Lam}X(T)\big\|_{\dot{C}_s^{1,\s}(\BT)}
+ \big\|e^{-\f{t}4\Lam}X(T)\big\|_{\dot{C}_s^{1,\s}(\BT)}\\
\leq &\; C\big\|X_1(t)-e^{-\f{t}4\Lam}X(T)\big\|_{\dot{C}_s^{1,\g}}^{\f{\s}{\g}}
\big\|X_1(t)-e^{-\f{t}4\Lam}X(T)\big\|_{\dot{C}_s^1}^{1-\f{\s}{\g}}
+ \|X(T)\|_{\dot{C}_s^{1,\s}}
\\
\leq &\; C\left(\|X_1(t)\|_{\dot{C}_s^{1,\g}} + \|X(T)\|_{\dot{C}_s^{1,\g}}\right)^{\f{\s}{\g}}
\big(\r(t)+ t^\s\big)^{1-\f{\s}{\g}}
+ \|X(T)\|_{\dot{C}_s^{1,\g}}
\\
\leq &\; C\left[t^{-\g}\big(t^\g M_1+ t^\s\big) + M_1\right]^{\f{\s}{\g}}
\big(t^\g M_1+ t^\s\big)^{1-\f{\s}{\g}}
+ M_1
\\
\leq &\; C t^{-\s}\big(t^\g M_1+ t^\s\big)
+ M_1,
\end{split}
\label{eqn: C 1 sigma estimate for X_1}
\eeq
where $C$ is a universal constant.
Since $\s <\g$ and $T_1\leq 1$,
\beq
\|X_1\|_{L^\infty_{T_1} \dot{C}_s^{1,\s}(\BT)} \leq C(1+M_1),
\label{eqn: improved bound for X_1}
\eeq
where $C$ is universal.
If $\s \geq \f1p$, this already implies \eqref{eqn: claim for improved regularity of X_1}.
In what follows, we shall handle the case $\s < \f1p$.

First, we show $u_1\in L^\infty_{T_1} L^\infty_x(\BR^2)$ by arguing as in the proof of Lemma \ref{lem: L inf regularity of u mild solution}.
Recall that, for $t\in [0,T_1]$,
\beq
u_1(t) = u_{X_1}^\nu(t) + e^{\nu t\Delta}u(T) +B_{\nu}[u_1](t).
\label{eqn: equation satisfied by u_1}
\eeq
By Lemma \ref{lem: estimate for u_X^nu} and Lemma \ref{lem: parabolic estimates}, for $t\in [0,T_1]$ with $T_1\leq 1$,
\[
\big\|u_{X_1}^\nu(t) + e^{\nu t\Delta}u(T)\big\|_{L^p_x(\BR^2)}
\leq
C \lam_1^{-1+\f2p} \|X_1'\|_{L^\infty_t L_s^\infty}^2 + \|u(T)\|_{L^p_x}
\leq C(p,\lam_1, M_1, Q_1),
\]
and
\[
\big\|u_{X_1}^\nu(t) + e^{\nu t\Delta}u(T)\big\|_{L^\infty_x(\BR^2)}
\leq
C\lam_1^{-1} \|X_1'\|_{L^\infty_t C_s^{\s}}^2
+ \|u(T)\|_{L^\infty_x} \leq C(p,\lam_1, M_1, Q_1).
\]
By interpolation, for any $q\in [p,\infty]$,
\beq
\big\|u_{X_1}^\nu(t) + e^{\nu t\Delta}u(T)\big\|_{L^\infty_{T_1} L^q_x(\BR^2)}
\leq C(p,\lam_1, M_1, Q_1).
\label{eqn: L q estimate for a part of u_1}
\eeq
On the other hand, by \eqref{eqn: L q estimate for B general} in Lemma \ref{lem: estimate for B_nu u}, for any $q,r\in [p,\infty]$ satisfying $\f1q \in ( \f1r+\f1p - \f12, \f1r+\f1p]$ and $\f1p+\f1r\leq 1$,
\[
\|B_{\nu}[u_1](t)\|_{L^q_x}
\leq C\nu^{-1} (\nu t)^{\f1q-\f12}
\sup_{\eta\in(0,t]}(\nu \eta)^{1-\f1r-\f1p}
\|(u_1\otimes u_1) (\eta)\|_{L_x^{\f{pr}{p+r}}},
\]
where $C$ depends on $p$, $q$, and $r$.
Combining this with \eqref{eqn: equation satisfied by u_1} and \eqref{eqn: L q estimate for a part of u_1} yields that
\beqo
\|u_1\|_{L^\infty_{T_1} L^q_x}
\leq C(p,\lam_1, M_1, Q_1) + C\nu^{-1} (\nu T_1)^{\f1q+\f12-\f1r-\f1p}
\|u_1\|_{L^\infty_{T_1} L_x^p}\|u_1\|_{L^\infty_{T_1} L_x^r}.
\eeqo
By choosing suitable pairs of $q$ and $r$, we can bootstrap as in the proof of Lemma \ref{lem: L inf regularity of u mild solution} to show that
\beq
\|u_1\|_{L^\infty_{T_1} L_x^\infty(\BR^2)}
\leq C(p,\nu,\lam_1, M_1, Q_1).
\label{eqn: uniform L inf estimate for u_1}
\eeq
We omit the details.

In order to show \eqref{eqn: claim for improved regularity of X_1} in the case $\s < \f1p$, one may proceed as in the proofs of Lemma \ref{lem: L inf regularity of u mild solution}, Lemma \ref{lem: time derivative and equation}, and Lemma \ref{lem: C gamma estimate for mild solutions} to further raise the regularity of $(u_1,X_1)$, but we shall present a different approach here.

Let $q\in (2,\infty)$ such that $\f1q = \f12(\f12-\f1p) = \s$.
By the assumption $\s < \f1p$, we have $q > p$, so $u(T)\in L_x^p\cap C_x^\g(\BR^2)$ implies $u(T)\in L_x^q(\BR^2)$.
Take $\g'\in [\f2q,1)\cap (1-\f2q,1)$ such that $\g'\geq \g$; we have $X(T)\in C^{\g'}_s(\BT)$ thanks to Lemma \ref{lem: C gamma estimate for mild solutions}.
Now we may apply Proposition \ref{prop: fixed-point solution} with the new parameters $q$ and $\g'$ to construct another local mild solution starting from $(u(T),X(T))$.
This time we choose the parameter $\s$ in the definition \eqref{eqn: def of set CX} of $\CX$ to be $\f1p\in (0,\f12-\f1q)$; this is valid due to Remark \ref{rmk: choice of sigma}.
In this way, we obtain $\tilde{T}_1\in (0,1)$ and a mild solution $(\tilde{u}_1,\tilde{X}_1)$ on $[0,\tilde{T}_1]$ to \eqref{eqn: NS equation}-\eqref{eqn: kinematic equation} with the initial condition
$\tilde{u}_1(x,0) = u(x,T)$, and $\tilde{X}_1(s,0) = X(s,T)$.
We also know that
\beqo
\|\tilde{u}_1\|_{L^\infty_{\tilde{T}_1} L^q_x(\BR^2)}
+ \|\tilde{X}_1'\|_{L^\infty_{\tilde{T}_1} L^\infty_s(\BT)}
<+\infty, \quad
\inf_{t\in [0,\tilde{T}_1]} |\tilde{X}_1(t)|_* > 0,
\eeqo
and for $t\in (0,\tilde{T}_1]$,
\[
t^{\g'} \|\tilde{X}_1(t)\|_{\dot{C}_s^{1,\g'}(\BT)}
+ \|\tilde{X}_1(t)-e^{-\f{t}4\Lam}X(T)\|_{C_s^1(\BT)}
\leq 4\big(\tilde{\r}(t)+ t^{1/p}\big),
\]
where
\[
\tilde{\r}(t):=\sup_{\tau\in (0,t]}\tau^{\g'}\big\|e^{-\f{\tau}{4} \Lam}X'(T)\big\|_{\dot{C}^{\g'}_s(\BT)}.
\]
By interpolation,
\[
\sup_{\eta\in (0,\tilde{T}_1]}\eta^{\f1q} \|\tilde{X}_1'(\eta)\|_{\dot{C}_s^{1/q}(\BT)} < +\infty.
\]
Furthermore, by arguing as in \eqref{eqn: C 1 sigma estimate for X_1} and \eqref{eqn: improved bound for X_1}, we find that
\beq
\tilde{X}_1\in L^\infty_{\tilde{T}_1} C^{1,\f1p}_s (\BT).
\label{eqn: improved C 1/p regularity of tilde X_1}
\eeq

On the other hand, the known bounds for $(u_1,X_1)$ and \eqref{eqn: uniform L inf estimate for u_1} give
\[
\|u_1\|_{L^\infty_{T_1} L^q_x(\BR^2)}
+
\|X_1'\|_{L^\infty_{T_1} L^\infty_s(\BT)}
+
\sup_{\eta\in (0,T_1]}\eta^{\f1q} \|X_1'(t)\|_{\dot{C}_s^{1/q}(\BT)}
<+\infty, \quad
\inf_{t\in [0,T_1]} |X_1(t)|_* > 0.
\]
Therefore, both $(u_1,X_1)$ and $(\tilde{u}_1,\tilde{X}_1)$ are mild solutions on $[0,T_1]\cap [0,\tilde{T}_1]$ to \eqref{eqn: NS equation}-\eqref{eqn: kinematic equation} with the initial data $(u(x,T),X(s,T))\in L^q_x(\BR^2)\times C^1_s(\BT)$ in the sense of Definition \ref{def: mild solution}.
By Proposition \ref{prop: uniqueness of mild solution}, we must have $\tilde{X}_1(t)\equiv X_1(t)$ on $[0,T_1]\cap [0,\tilde{T}_1]$.
Therefore, in the case $T_1\leq \tilde{T}_1$, \eqref{eqn: improved C 1/p regularity of tilde X_1} implies \eqref{eqn: claim for improved regularity of X_1}, while if $T_1> \tilde{T}_1$, \eqref{eqn: claim for improved regularity of X_1} follows from \eqref{eqn: basic bound for X_1} and \eqref{eqn: improved C 1/p regularity of tilde X_1}.

In either case, we have shown \eqref{eqn: claim for improved regularity of X_1}.
Let $\tilde{T} := T+T_1$, and
\[
\tilde{u}(x,t) :=
\begin{cases}
u(x,t), & \mbox{if } t\in [0,T], \\
u_1(x,t-T), & \mbox{if } t\in [T,\tilde{T}],
\end{cases}\quad
\tilde{X}(s,t) :=
\begin{cases}
X(s,t), & \mbox{if } t\in [0,T], \\
X_1(s,t-T), & \mbox{if } t\in [T,\tilde{T}].
\end{cases}
\]
It is straightforward to verify that $(\tilde{u},\tilde{X})$ is a mild solution to \eqref{eqn: NS equation}-\eqref{eqn: initial data} on $[0,\tilde{T}]$, which is an extension of $(u,X)$.
\end{proof}

\begin{rmk}
The lifespan $T_1$ of the extended part $(u_1,X_1)$ only depends on $\g$, $p$, $\nu$, $|X(T)|_*$, $\|u(T)\|_{L^p_x(\BR^2)}$, and $\|X'(T)\|_{C_s^\g(\BT)}$.
Indeed, since $X(T)\in C^{1,\g}_s(\BT)$, the corresponding $\r$ and $\phi$ defined in terms of $X(T)$ (see \eqref{eqn: def of rho} and \eqref{eqn: def of phi}) should satisfy
\[
\r(t)\leq t^\g \|X'(T)\|_{\dot{C}^\g_s(\BT)},
\quad
\phi(t)\leq Ct^\g\|X'(T)\|_{\dot{C}^\g_s(\BT)},
\]
where $C$ only depends on $\g$.
Then the claim follows from Proposition \ref{prop: fixed-point solution} and Remark \ref{rmk: characterization of lifespan T}.

\end{rmk}
\end{prop}

A corollary of Proposition \ref{prop: uniqueness of mild solution} and Proposition \ref{prop: continuation} is the existence and characterization of a maximal mild solution (see Definition \ref{def: maximal solution}).

\begin{cor}
\label{cor: maximal solution}
Let $u_0 = u_0(x)\in L^p(\BR^2)$ for some $p\in (2,\infty)$ and $X_0 = X_0(s)\in C^1(\BT)$, which satisfy $\di u_0 = 0$ and $|X_0|_* >0$.
Then there exists a unique $T_*\in (0,+\infty]$ and a unique maximal mild solution $(u,X)$ to \eqref{eqn: NS equation}-\eqref{eqn: initial data} with its maximal lifespan being $T_*$.

Moreover, if $T_*<+\infty$, at least one of the following scenarios would occur:
\begin{enumerate}[label=(\alph*)]
\item $\limsup_{t\to T_*^-} \|u(t)\|_{L^p_x(\BR^2)} = +\infty$;

\item $\liminf_{t\to T_*^-} |X(t)|_* = 0$;

\item For any increasing sequence $t_k\to T_*^-$, $\{X'(t_k)\}_k$ is not convergent in $C(\BT)$.
As a result, $X'([0,T_*))$, which is the image of $[0,T_*)$ under the mapping $t\mapsto X'(\cdot,t)$, is not pre-compact in $C(\BT)$.
\end{enumerate}

\begin{proof}
The existence and uniqueness of $T_*$ and the maximal solution $(u,X)$ immediately follows from Proposition \ref{prop: uniqueness of mild solution} and Proposition \ref{prop: continuation}.
We will only prove the characterizations of solution when it approaches the finite blow-up time $T_*$.

We prove by contradiction.
Suppose none of the above-mentioned three scenarios occurs as $t\to T_*^-$.
Take an increasing sequence $t_k\to T_*^-$, such that $\{X'(t_k)\}_{k\in \BZ_+}$ is convergent in $C(\BT)$.
Since the limit has mean zero, we shall denote it as $X_\circ'(s)$;
without loss of generality, we also assume $X_\circ$ to have mean zero.
By assumption, with abuse of the notations,
\[
Q:=\sup_{k\in \BZ_+} \|u(t_k)\|_{L^p_x(\BR^2)}<+\infty,\quad
\lam:=\inf_{k\in \BZ_+} |X(t_k)|_*>0.
\]
It is not difficult to show $|X_\circ|_* \geq \lam$.
Note that the mean of $X_\circ$ is irrelevant when defining $|X_\circ|_*$.

Take $\g\in [\f2p,1)\cap (1-\f2p,1)$.
Let $\r_k$ and $\phi_k$ be defined in terms of $X(t_k)$ as in \eqref{eqn: def of rho} and \eqref{eqn: def of phi}:
\begin{align*}
\r_k(t) := \sup_{\tau\in (0,t]}\tau^{\g}\big\|e^{-\f{\tau}{4} \Lam}X'(t_k)\big\|_{\dot{C}^{\g}_s(\BT)},\quad
\phi_k(t) := \sup_{\tau\in [0,t]}\big\|X(t_k)-e^{-\f{\tau}{4}\Lam}X(t_k)\big\|_{\dot{C}_s^1(\BT)}.
\end{align*}
Also define
\begin{align*}
\r_\circ(t) := \sup_{\tau\in (0,t]}\tau^{\g}\big\|e^{-\f{\tau}{4} \Lam}X_\circ'\big\|_{\dot{C}^{\g}_s(\BT)},\quad
\phi_\circ(t) := \sup_{\tau\in [0,t]}\big\|X_\circ-e^{-\f{\tau}{4}\Lam}X_\circ\big\|_{\dot{C}_s^1(\BT)}.
\end{align*}
They are all continuous increasing functions.
Note again that $\r_\circ$ and $\phi_\circ$ do not rely on the mean of $X_\circ$.
Since $\lim_{k\to +\infty}\|X'(t_k)- X'_\circ\|_{L^\infty_s(\BT)} = 0$, we find that
\[
M: = \sup_{k\in \BZ_+} \|X'(t_k)\|_{C_s(\BT)}<+\infty,
\]
and
\beq
\r_k(t)\rightrightarrows \r_\circ(t),\mbox{ and } \phi_k(t)\rightrightarrows \phi_\circ(t)\mbox{ for } t\in [0,+\infty).
\label{eqn: uniform convergence of rho and phi}
\eeq

Now we apply Proposition \ref{prop: fixed-point solution} with the parameters $p$ and $\g$ to construct local solutions starting from the initial datum $(u(t_k), X(t_k))$.
Namely, for each $k\in \BZ_+$, there exists $T_k>0$ and mild solutions $(u_k,X_k)$ on $[0,T_k]$ to \eqref{eqn: NS equation}-\eqref{eqn: kinematic equation} with the initial condition $u_k(x,0) = u(x,t_k)$, and $X_k(s,0) = X(s,t_k)$.
By virtue of Remark \ref{rmk: characterization of lifespan T},
$T_k$ should depend on $\g,\,p,\,\nu,\,\lam,\,Q$, and $M$, and additionally, we need $(\r_k(T_k)+T_k^{\s})$ and $\phi_k(T_k)$ to be smaller than some constants that depend on $\g,\,p,\,\nu,\,\lam,\,Q$, and $M$.
By \eqref{eqn: uniform convergence of rho and phi}, we may choose $\{T_k\}_{k\in \BZ_+}$, such that $T_\circ:=\inf_{k\in \BZ_+}T_k > 0$.

Now take $k$ sufficiently large so that $T_*- t_k < T_\circ$.
Define
\[
\tilde{u}(x,t) :=
\begin{cases}
u(x,t), & \mbox{if } t\in [0,t_k], \\
u_k(x,t-t_k), & \mbox{if } t\in [t_k,t_k+T_k],
\end{cases}\quad
\tilde{X}(s,t) :=
\begin{cases}
X(s,t), & \mbox{if } t\in [0,t_k], \\
X_k(s,t-t_k), & \mbox{if } t\in [t_k,t_k+T_k].
\end{cases}
\]
By arguing as in Proposition \ref{prop: continuation}, we can show that $(\tilde{u},\tilde{X})$ gives a mild solution to \eqref{eqn: NS equation}-\eqref{eqn: initial data} on $[0,t_k+T_k]$.
Since $t_k+T_k \geq t_k+T_\circ > T_*$ for sufficiently large $k$, this contradicts with the maximality of $T_*$.

This completes the proof.
\end{proof}
\end{cor}

\subsection{Higher-order regularity}
\label{sec: higher regularity}
As is introduced at the beginning of this section, we still assume $(u,X)$ to be the unique mild solution to \eqref{eqn: NS equation}-\eqref{eqn: initial data} on $[0,T]$, and let $M$, $\lam$, and $Q$ be defined in \eqref{eqn: def of M}-\eqref{eqn: def of Q}, respectively.
In this subsection, we will show that $X = X(s,t)$ will instantly become $C_s^{2,\al}(\BT)$ for any $\al\in (0,1)$, and $u = u(x,t)$ will instantly gain Lipschitz regularity in space, with their high-order norms admitting time-singularities of optimal orders at $t = 0$.

We first derive a higher-order estimates for $H_X^\nu$, which is an extension of Lemma \ref{lem: L inf and W 1 inf estimate for h_X^nu with optimal power of nu}.

\begin{lem}
\label{lem: estimate for h_X^nu with time singularity}
Let $\g\in (\f12,1)$.
Suppose that $X\in L^\infty([0,T];C_s^{1,\g}(\BT))$, satisfying that $|X|_*\geq \lam$ and $\pa_t X\in L^\infty([0,T];C_s^\g(\BT))$.
Denote
\begin{align*}
Z[\g,\nu,\lam,X,t]
:=
&\; \nu^{-\g} \lam^{-1} \|X'\|_{L^\infty_t L_s^\infty} \|X'\|_{L^\infty_t C_s^\g }^{1-\g} \|\pa_t X\|_{L^\infty_t \dot{C}_s^{\g}}^\g \\
&\; +(\nu t)^{-\f{1+\g}{2}} \lam^{-1-\g} t\cdot \|X'\|_{L^\infty_t L_s^\infty}
\\
&\;\quad \cdot
\left(\|X'\|_{L^\infty_t C_s^\g} \|\pa_t X\|_{L^\infty_t L_s^\infty}
+ \|X'\|_{L^\infty_t L_s^\infty} \|\pa_t X\|_{L^\infty_t \dot{C}_s^{\g}} \right).
\end{align*}
Then
\begin{align*}
\|\na h_X^{\nu}(t)\|_{\dot{C}_x^{2\g-1}(\R^2)}
\leq &\; C Z[\g,\nu,\lam,X,t],\\
\|\na h_X^\nu(t)\|_{L_x^\infty(\BR^2)}
\leq &\;  C(\nu t)^{\g-\f12} Z[\g,\nu,\lam,X,t],
\end{align*}
and
\[
\|H_X^{\nu}(t)\|_{\dot{C}^{1,2\g-1}(\BT)}
\leq C \left( \|X'\|_{L^\infty_t L_s^\infty(\BT)}^{2\g}
+ (\nu t)^{\g-\f12} \|X'\|_{L^\infty_t \dot{C}^{2\g-1}_s(\BT)}\right) Z[\g,\nu,\lam,X,t],
\]
where the constants $C$ depend on $\g$.

\begin{proof}
By \eqref{eqn: Holder estimate for h} in Lemma \ref{lem: estimate for h_X^nu}, with the parameters there taken as $(\al,\g',\mu_0,\mu,\mu'):= (2\g,\g,0,0,0)$, we obtain that
\begin{align*}
\|\na h_X^{\nu}(t)\|_{\dot{C}_x^{2\g-1}(\R^2)}
\leq &\; C \left[(\nu t)^{-\g}\min\left(\lam^{-1},\,(\nu t)^{-\f{1}{2}}\right)
+ \mathds{1}_{\{\nu t\geq 4\lam^2\}} (\nu t)^{-\g}\lam^{-1}\right]
t^{\g}
\\
&\;\cdot \|X'\|_{L^\infty_t L_s^\infty} \|X'\|_{L^\infty_t C_s^\g }^{1-\g} \|\pa_t X\|_{L^\infty_t \dot{C}_s^{\g}}^\g \\
&\; + C\left[(\nu t)^{-\f{1+\g}{2}} \min\left(\lam^{-\g},\, (\nu t)^{-\f{\g}{2}} \right)
\lam^{-1}
+ \mathds{1}_{\{\nu t\geq 4\lam^2\}} (\nu t)^{-1}  \lam^{-2\g} \right]
\\
&\;\quad \cdot  t\left(\|X'\|_{L^\infty_t L_s^\infty}
\|X'\|_{L^\infty_t C_s^\g} \|\pa_t X\|_{L^\infty_t L_s^\infty}
+ \|X'\|_{L^\infty_t L_s^\infty}^2 \|\pa_t X\|_{L^\infty_t \dot{C}_s^{\g}} \right)
\\
\leq &\; C (\nu t)^{-\g} \lam^{-1} t^\g \|X'\|_{L^\infty_t L_s^\infty} \|X'\|_{L^\infty_t C_s^\g }^{1-\g} \|\pa_t X\|_{L^\infty_t \dot{C}_s^{\g}}^\g \\
&\; + C(\nu t)^{-\f{1+\g}{2}} \lam^{-1-\g} t
\\
&\;\quad \cdot  \left(\|X'\|_{L^\infty_t L_s^\infty}
\|X'\|_{L^\infty_t C_s^\g} \|\pa_t X\|_{L^\infty_t L_s^\infty}
+ \|X'\|_{L^\infty_t L_s^\infty}^2 \|\pa_t X\|_{L^\infty_t \dot{C}_s^{\g}} \right)\\
\leq &\; C Z[\g,\nu,\lam,X,t],
\end{align*}
where $C$ depends on $\g$.

Similarly, by \eqref{eqn: L inf and W 1 inf estimate for h} in Lemma \ref{lem: estimate for h_X^nu}, with the parameters there taken as $(\g',\mu_0,\mu,\mu'):=(\g,0,0,0)$,
\begin{align*}
\|\na h_X^\nu(t)\|_{L_x^\infty(\BR^2)}
\leq &\; C (\nu t)^{-\f{1}{2}} \lam^{-1}  t^{\g}
\|X'\|_{L^\infty_t L_s^\infty} \|X'\|_{L^\infty_t C_s^\g}^{1-\g} \|\pa_t X\|_{L^\infty_t \dot{C}_s^{\g}}^{\g} \\
&\; + C (\nu t)^{-\f{2-\g}{2}} \min\left(\lam^{-\g},\, (\nu t)^{-\f{\g}{2}}\right)
\ln \left(2+\f{\nu t}{\lam^2}\right)\\
&\;\quad \cdot \lam^{-1}t \left( \|X'\|_{L^\infty_t L_s^\infty}
\|X'\|_{L^\infty_t C_s^{\g}} \|\pa_t X\|_{L^\infty_t L_s^\infty}
+ \|X'\|_{L^\infty_t L_s^\infty}^2 \|\pa_t X \|_{L^\infty_t \dot{C}_s^{\g}} \right)
\\
\leq &\;  C(\nu t)^{\g-\f12} Z[\g,\nu,\lam,X,t],
\end{align*}
where $C$ depends on $\g$.
Here we used the fact that
\[
\min\left(\lam^{-\g},\, (\nu t)^{-\f{\g}{2}}\right)
\ln \left(2+\f{\nu t}{\lam^2}\right)
\leq C\lam^{-\g}
\left(1+\f{\nu t}{\lam^2}\right)^{-\f{\g}{2}}
\ln \left(2+\f{\nu t}{\lam^2}\right)\leq C\lam^{-\g}
\]
where $C$ only depends on $\g$.

The estimate for $H_X^\nu$ follows from Lemma \ref{lem: estimate for h_X^nu}.
\end{proof}
\end{lem}

The next lemma is concerned with higher-order estimates for $B_\nu[u]$.

\begin{lem}
\label{lem: estimate for B_nu u with time singularity}
For any $\b\in (0,1)$ and $t\in [0,T]$,
\[
(\nu t)^{-\b/2} \|B_{\nu}[u](t)\|_{\dot{C}^{1}_x(\BR^2)}
+ \|B_{\nu}[u](t)\|_{\dot{C}^{1,\b}_x(\BR^2)}
\leq C \nu^{-1}\|u\|_{L^\infty_t L_x^\infty(\BR^2)}  \|u\|_{L^\infty_t\dot{C}_x^\b(\BR^2)},
\]
where $C$ depends on $\b$.

\begin{proof}
It follows from
\begin{align*}
\|B_{\nu}[u](t)\|_{\dot{C}^{1}_x}
\leq &\; \int_{0}^t \big\|\na e^{\nu (t-\tau)\Delta}\mathbb{P}\di (u\otimes u)(\tau)\big\|_{L^\infty_x}\, d\tau
\\
\leq &\; C \int_{0}^t (\nu(t-\tau))^{-\f{2-\b}{2}}
\|(u\otimes u)(\tau)\|_{\dot{C}^\b_x}\,d\tau
\\
\leq &\; C \nu^{-1} (\nu t)^{\b/2} \sup_{\tau\in(0,t)} \|(u\otimes u)(\tau)\|_{\dot{C}^\b_x},
\end{align*}
and by Lemma \ref{lem: interpolation},
\begin{align*}
\|B_{\nu}[u](t)\|_{\dot{C}^{1,\b}_x}
\leq &\; C \left|\sup_{\tau\in(0,t) }(t-\tau)^{\f{2-\b}{2}} (\nu(t-\tau))^{-\f{2-\b}{2}}
\|(u\otimes u)(\tau)\|_{\dot{C}^\b_x}\right|^{1-\b}\\
&\;\cdot
\left|\sup_{\tau\in (0,t)}(t-\tau)^{\f{3-\b}{2}} (\nu(t-\tau))^{-\f{3-\b}{2}} \|(u\otimes u)(\tau)\|_{C^\b_x}\right|^{\b}
\\
\leq &\; C \nu^{-1} \sup_{\tau\in(0,t)} \|(u\otimes u)(\tau)\|_{\dot{C}^\b_x}.
\end{align*}

\end{proof}
\end{lem}

\begin{lem}
\label{lem: smoothing lemma for X' in 2gamma}
Let $\g\in (\f12,1)\cap [\f2p,1)$.
For any $t\in (0,T]$,
\[
\sup_{\eta\in (0,T]}\eta^{2\g} \|X''(\eta)\|_{C^{2\g-1}_s(\BT)} + \sup_{\eta\in (0,T)}\eta^{2\g} \|\pa_t X'(\eta)\|_{C^{2\g-1}_s(\BT)}
\leq C(T, \g, p,\nu,\lam,M,Q),
\]
where as before, $C(T, \g, p,\nu,\lam,M,Q)$ denotes a constant depending on $T,\,\g,\,p,\,\nu,\,\lam,\,M$, and $Q$.

\begin{proof}
Take an arbitrary $t_0\in (0,T]$.
Since $\g \geq \f2p$, the results in Section \ref{sec: regularity of mild solutions} imply that
\begin{itemize}
\item $X\in C^1(\BT\times [t_0/2,T])$ and $u\in C(\BR^2\times [t_0/2,T])$;

\item $\inf_{\eta\in [t_0/2,T]} |X(\eta)|_* \geq \lam$, and
\begin{align*}
&\; t_0^\g \sup_{\eta\in [t_0/2,T]}\|X'(\eta)\|_{C^{\g}_s(\BT)}
+
t_0^{\g}\sup_{\eta\in [t_0/2,T)}
\|\pa_t X(\eta)\|_{C_s^{\g}(\BT)}
+
t_0^{1/p}\sup_{\eta\in [t_0/2,T)}\|\pa_t X(\eta)\|_{L_s^\infty(\BT)} \\
\leq &\; C(T, \g, p,\nu,\lam,M,Q);
\end{align*}
\item and, moreover,
\begin{align*}
t_0^{1/p}\sup_{\tau\in [t_0/2,T]}\|u(\tau)\|_{L^\infty_x(\BR^2)}
+ t_0^{\g}\sup_{\tau\in [t_0/2,T]}\|u(\tau)\|_{\dot{C}^\g_x(\BR^2)}
\leq &\; C(T, \g, p,\nu,\lam,M,Q).
\end{align*}

\end{itemize}
It is not difficult to verify by definition that, $(u,X)$ satisfies on $t\in [t_0/2,T]$ that (cf.\;\eqref{eqn: u_X^nu in terms of u_X}-\eqref{eqn: u_2 tilde}, and \eqref{eqn: fixed pt equation for u}-\eqref{eqn: def of the operator I})
\begin{align}
u(t)
=&\; u_{X(t)}+\tilde{h}_{X,t_0/2}^{\nu}(t)+e^{\nu (t-t_0/2)\Delta}(u(t_0/2)-u_{X(t)}) +\tilde{B}_{\nu,t_0/2}[u](t),
\label{eqn: formula for u new initial time}\\
X(t)= &\; e^{-\f14(t-t_0/2) \Lam}X(t_0/2)+\tilde{\CI}_{t_0/2} \big[\tilde{W}[u,X; \nu,u(t_0/2),t_0/2]\big](t),
\label{eqn: formula for X new initial time}
\end{align}
and it holds on $[t_0/2,T)$ that (cf.\;\eqref{eqn: differential equation of X in the mild solution})
\beq
\pa_t X(s,t) = -\f14\Lam X(s,t) + \tilde{W}[u,X;\nu,u(t_0/2),t_0/2](t) = u(X(s,t),t).
\label{eqn: X_t equation new form}
\eeq
Here we denoted, with $0\leq t_*\leq t$,
\begin{align*}
\tilde{h}_{X,t_*}^{\nu}(t):= &\; -\nu\int_{t_*}^t\Delta e^{\nu (t-\tau)\Delta}(u_{X(\tau)}-u_{X(t)})\, d\tau,
\\
\tilde{H}_{X,t_*}^{\nu}(t):=&\; \tilde{h}_{X,t_*}^{\nu}(t)\circ X(t),
\\
\tilde{B}_{\nu,t_*}[u](t):= &\; -\int_{t_*}^te^{\nu (t-\tau)\Delta}\mathbb{P}(u\cdot\nabla u)(\tau)\, d\tau,
\\
\tilde{W}[u,X; \nu,\tilde{u}_0,t_*](t)
:= &\; g_X(t) + \tilde{H}_{X,t_*}^{\nu}(t)
+ \big[e^{\nu (t-t_*)\D}(\tilde{u}_0-u_{X(t)})+\tilde{B}_{\nu,t_*}[u]\big] \circ X(t),
\end{align*}
and
\[
\tilde{\CI}_{t_*}[\tilde{W}](t):=\int_{t_*}^te^{-\f14(t-\tau)\Lam}\tilde{W}(\tau)\,d\tau.
\]
In what follows, with $(u,X)$ and $t_0$ given, we shall write $\tilde{W}[u,X; \nu,u(t_0/2),t_0/2]$ as $\tilde{W}[u,X]$ or even $\tilde{W}$ for brevity.

Denote $\al := 2\g-1$.
We shall derive a $C^{2,\al}_s(\BT)$-estimate for $X(t_0)$ using the new representation.
By \eqref{eqn: formula for X new initial time}, Lemma \ref{lem: parabolic estimates for fractional Laplace}, and Lemma \ref{lem: parabolic estimates for Duhamel term of fractional Laplace} with $\va(t)= t^{-\g}$, 
\beq
\begin{split}
\|X''(t_0)\|_{C^{\al}_s(\BT)}
\leq &\; C\big\|e^{-\f{t_0}8 \Lam}X'(t_0/2)\big\|_{\dot{C}^{1,\al}_s(\BT)}
+ C\big\|\tilde{\CI}_{t_0/2} \big[\tilde{W}[u,X]\big](t_0)\big\|_{\dot{C}^{2,\al}_s(\BT)}
\\
\leq &\; Ct_0^{-\g}\|X'(t_0/2)\|_{C^\g_s(\BT)}
+ Ct_0^{-\g}
\sup_{\tau\in[t_0/2,t_0]} (\tau-t_0/2)^{\g} \|\tilde{W}[u,X](\tau)\|_{\dot{C}_s^{1,\al}(\T)}.
\end{split}
\label{eqn: C 2 alpha estimate for X at t_0}
\eeq
By Lemma \ref{lem: Holder estimate for composition of functions}, for any $t\in [\f{t_0}{2},t_0]$,
\begin{align*}
\|\tilde{W}[u,X](t)\|_{\dot{C}^{1,\al}_s(\BT)}
\leq
&\;\|g_X(t)\|_{C^{1,\al}_s}
+ \|\tilde{H}_{X,t_0/2}^{\nu}(t)\|_{\dot{C}^{1,\al}_s}\\
&\;
+ \big\|e^{\nu (t-t_0/2)\D}(u(t_0/2)-u_{X(t)})+ \tilde{B}_{\nu,t_0/2}[u](t)\big\|_{\dot{C}^1_x} \|X'(t)\|_{\dot{C}^{\al}_s}\\
&\;
+ \big\|e^{\nu (t-t_0/2)\D}(u(t_0/2)-u_{X(t)})+ \tilde{B}_{\nu,t_0/2}[u](t)\big\|_{\dot{C}^{1,\al}_x} \|X'(t)\|_{L^\infty_s}^{1+\al}.
\end{align*}
We shall estimate these terms one by one.

By Lemma \ref{lem: improved estimates for g_X}, 
for any $t\in [\f{t_0}{2},t_0]$,
\[
\|g_X(t)\|_{\dot{C}_s^{1,\al}(\BT)}
\leq C\lam^{-3} \|X'(t)\|_{L_s^\infty}^2 \|X'(t)\|_{\dot{C}_s^\g}^2
\leq C(T, \g, p,\nu,\lam,M,Q) t_0^{-2\g}.
\]
With $t_0/2$ viewed as the new initial time, we apply Lemma \ref{lem: estimate for h_X^nu with time singularity} to obtain that, for any $t\in [\f{t_0}{2},t_0]$,
\begin{align*}
&\;\|\tilde{H}_{X,t_0/2}^{\nu}(t)\|_{\dot{C}^{1,\al}_s(\BT)}\\
\leq &\; C \left( \|X'\|_{L^\infty_t L_s^\infty}^{2\g}
+ (\nu (t-t_0/2))^{\g-\f12} \sup_{\eta\in [t_0/2,t]} \|X'(\eta)\|_{\dot{C}^{2\g-1}_s}\right)\\
&\;\cdot \left[
\nu^{-\g} \lam^{-1} \|X'\|_{L^\infty_t L_s^\infty} \sup_{\eta\in [t_0/2,t]}\|X'(\eta)\|_{C_s^\g}^{1-\g} \sup_{\eta\in [t_0/2,t]}\|\pa_t X(\eta)\|_{\dot{C}_s^{\g}}^\g \right. \\
&\; \quad +(\nu (t-t_0/2))^{-\f{1+\g}{2}} \lam^{-1-\g} (t-t_0/2) \|X'\|_{L^\infty_t L_s^\infty}
\\
&\;\left.\qquad \cdot
\left(\sup_{\eta\in [t_0/2,t]}\|X'(\eta)\|_{C_s^\g} \sup_{\eta\in [t_0/2,t]}\|\pa_t X(\eta)\|_{L_s^\infty}
+ \|X'\|_{L^\infty_t L_s^\infty} \sup_{\eta\in [t_0/2,t]}\|\pa_t X(\eta)\|_{\dot{C}_s^{\g}} \right)
\right]
\\
\leq &\; C(T, \g, p,\nu,\lam,M,Q)
\left( 1 + (t-t_0/2)^{\g-\f12} t_0^{-(2\g-1)} \right)
\left( t_0^{-\g} + (t-t_0/2)^{\f{1-\g}{2}} t_0^{-\g-\f1p}
\right).
\end{align*}
Since $\g \in(\f12, 1)$ and $t_0\leq T$, for $t\in [\f{t_0}{2},t_0]$,
\[
\|\tilde{H}_{X,t_0/2}^{\nu}(t)\|_{\dot{C}^{1,\al}_s(\BT)}
\leq C(T, \g, p,\nu,\lam,M,Q) t_0^{-2\g}.
\]

By Lemma \ref{lem: estimate for u_X} (with $\g = \f1p$ there) and Lemma \ref{lem: parabolic estimates}, for $t\in [\f{t_0}{2},t_0]$,
\begin{align*}
&\; (\nu (t-t_0/2))^{1/2}\big\|e^{\nu (t-t_0/2)\D}(u(t_0/2)-u_{X(t)}) \big\|_{\dot{C}^1_x(\BR^2)} \\
&\; + (\nu (t-t_0/2))^{\g}
\big\|e^{\nu (t-t_0/2)\D}(u(t_0/2)-u_{X(t)}) \big\|_{\dot{C}^{1,\al}_x(\BR^2)}
\\
\leq
&\; C \left(\|u(t_0/2)\|_{L^\infty_x} +\|u_{X(t)}\|_{L^\infty_x}\right) \\
\leq
&\; C
\left[t_0^{-\f1p} C(T, \g, p,\nu,\lam,M,Q) +t^{-\f1p}\cdot \lam^{-1} M^2\right] \\
\leq
&\; C(T, \g, p,\nu,\lam,M,Q) t_0^{-1/p}.
\end{align*}
So for $t\in [\f{t_0}{2},t_0]$,
\begin{align*}
&\; \big\|e^{\nu (t-t_0/2)\D}(u(t_0/2)-u_{X(t)})\big\|_{\dot{C}^1_x} \|X'(t)\|_{\dot{C}^{\al}_s}\\
&\; + \big\|e^{\nu (t-t_0/2)\D}(u(t_0/2)-u_{X(t)})\big\|_{\dot{C}^{1,\al}_x} \|X'(t)\|_{L^\infty_s}^{1+\al}
\\
\leq &\; C(T, \g, p,\nu,\lam,M,Q) (t-t_0/2)^{-\g} t_0^{-1/p}\left[
(t-t_0/2)^{\g-\f12} t_0^{-\al} + 1\right].
\end{align*}
By Lemma \ref{lem: estimate for B_nu u with time singularity} and the interpolation inequality,
\begin{align*}
&\;  (\nu (t-t_0/2))^{-\al/2}
\big\| \tilde{B}_{\nu,t_0/2}[u](t)\big\|_{\dot{C}^1_x(\BR^2)}
+ \big\|\tilde{B}_{\nu,t_0/2}[u](t)\big\|_{\dot{C}^{1,\al}_x(\BR^2)}\\
\leq
&\; C \nu^{-1}\sup_{\eta\in [t_0/2,t]}\|u(\eta)\|_{L_x^\infty(\BR^2)}
\sup_{\eta\in [t_0/2,t]} \|u(\eta)\|_{\dot{C}_x^\al(\BR^2)} \\
\leq
&\; C \nu^{-1}\cdot t_0^{-\f1p} C(T, \g, p,\nu,\lam,M,Q)\cdot t_0^{-\al-\f1p(1-\f{\al}{\g})} C(T, \g, p,\nu,\lam,M,Q),
\end{align*}
and thus, for $t\in [\f{t_0}{2},t_0]$,
\begin{align*}
&\; \big\| \tilde{B}_{\nu,t_0/2}[u](t)\big\|_{\dot{C}^1_x} \|X'(t)\|_{\dot{C}^{\al}_s}
+ \big\|\tilde{B}_{\nu,t_0/2}[u](t)\big\|_{\dot{C}^{1,\al}_x} \|X'(t)\|_{L^\infty_s}^{1+\al}\\
\leq &\;
C\left[(\nu (t-t_0/2))^{\al/2} \|X'(t)\|_{\dot{C}^{\al}_s}
+ \|X'(t)\|_{L^\infty_s}^{1+\al}\right]
\nu^{-1}\cdot t_0^{-\al-\f1p(2-\f{\al}{\g})} C(T, \g, p,\nu,\lam,M,Q)
\\
\leq &\;
C(T, \g, p,\nu,\lam,M,Q) \left(t_0^{-\al/2} + 1\right)
t_0^{-\al-\f1p(2-\f{\al}{\g})}
\\
\leq &\;
C(T, \g, p,\nu,\lam,M,Q) t_0^{-2\g}.
\end{align*}
In the last line, we used the facts that $\g\in [\f2p,1)$, $\al = 2\g-1$, and $t_0\leq T$.
Combining all these bounds into the estimate for $\tilde{W}$, we find that, for $t\in [\f{t_0}{2},t_0]$,
\beq
\begin{split}
\|\tilde{W}[u,X](t)\|_{\dot{C}^{1,\al}_s(\BT)}
\leq &\;
C(T, \g, p,\nu,\lam,M,Q) t_0^{-2\g}\\
&\;
+ C(T, \g, p,\nu,\lam,M,Q) (t-t_0/2)^{-\g}
t_0^{-1/p}\left(t_0^{-\g+\f12} + 1\right).
\end{split}
\label{eqn: C 1 alpha estimate for tilde W}
\eeq

Since $p>2$, $\g\in (\f12,1)$, and $t_0\leq T$,
\begin{align*}
&\; t_0^{-\g} \sup_{\tau\in[t_0/2,t_0]} (\tau-t_0/2)^{\g} \|\tilde{W}[u,X](\tau)\|_{\dot{C}^{1,\al}_s(\BT)}\\
\leq
&\; C(T, \g, p,\nu,\lam,M,Q) t_0^{-2\g}
+ C(T, \g, p,\nu,\lam,M,Q) t_0^{-\g} \cdot t_0^{-1/p}
\left(t_0^{-\g+\f12} + 1\right).
\\
\leq
&\; C(T, \g, p,\nu,\lam,M,Q) t_0^{-2\g}.
\end{align*}
Hence, by \eqref{eqn: C 2 alpha estimate for X at t_0},
\[
\|X''(t_0)\|_{C^{\al}_s(\BT)}
\leq C(T, \g, p,\nu,\lam,M,Q) t_0^{-2\g}.
\]
Note that $t_0\in (0,T]$ is arbitrary.
This together with \eqref{eqn: X_t equation new form} and \eqref{eqn: C 1 alpha estimate for tilde W} that, with $t_0\in (0,T)$,
\[
\|\pa_t X'(t_0)\|_{C^{\al}_s(\BT)}
\leq C\left(\|\Lam X'(t_0)\|_{C^{\al}_s}
+ \|\tilde{W}[u,X](t_0)\|_{\dot{C}^{1,\al}_s}\right)
\leq C(T, \g, p,\nu,\lam,M,Q) t_0^{-2\g}.
\]
This proves the desired estimate.
\end{proof}
\end{lem}

To study the higher-order regularity of $u$, we need the following Lipschitz estimate for the Stokes flow field $u_X$ under the assumption that $X\in C_s^{2,\al}(\BT)$, which can be viewed as an extension of Lemma \ref{lem: estimate for u_X}.
Let us mention that a similar estimate on the log-Lipschitz regularity of $u_X$ was established under slightly weaker assumptions in \cite[Lemma 2.2]{tong2018stokes}.
\begin{lem}
\label{lem: Lipschitz continuity of flow field}
Suppose $X = X(s)\in C^{2,\al}(\BT)$ for some $\al\in (0,1)$ and $|X|_* = \lam>0$.
Then $u_X = u_X(x)$ is Lipschitz in $\BR^2$, satisfying that, for any $x,y\in \BR^2$,
\[
|u_X(x)-u_X(y)|
\leq
C\lam^{-2} \|X'\|_{L_s^\infty}^{1+\f{\al}{1+\al}}\|X''\|_{\dot{C}_s^\al}^{\f{1}{1+\al}} |x-y|,
\]
where $C$ depends on $\al$.

\begin{proof}
It suffices to bound $\na u_X(x)$ when $x \not \in X(\BT)$.
This is because Lemma \ref{lem: estimate for u_X} implies the continuity of $u_X$ across $X(\BT)$, and the points on $X(\BT)$ can always be approximated by those outside $X(\BT)$.

For $\d\in (0,1]$ to be chosen later, we take $\chi_\d:\BT\to [0,1]$ to be an even smooth cut-off function, satisfying that $\chi_\d(s)\equiv 1$ in $[-\d,\d]$, $\chi_\d(s)\equiv 0$ in $\BT\setminus[-2\d,2\d]$, $\chi_\d$ is decreasing on $[0,\pi]$, and $|\chi_\d'(s)|\leq C\d^{-1}$ for some universal $C>0$.

By \eqref{eqn: expression of u_11}, with $s_x$ defined as in \eqref{eqn: def of s_x} and with $\d\in (0,1]$ to be determined,
\begin{align*}
&\; \pa_k u_{X,i}(x)\\
= &\; \int_{s_x-\pi}^{s_x+\pi} \pa_k G_{ij}(x-X(s')) X_j''(s')\cdot \chi_\d(s'-s_x)\,ds'\\
&\;
+\int_{s_x-\pi}^{s_x+\pi} \pa_k G_{ij}(x-X(s')) X_j''(s')\big[1-\chi_\d(s'-s_x)\big]\,ds'
\\
= &\; \int_{s_x-\pi}^{s_x+\pi} \pa_k G_{ij}(x-X(s')) \big(X_j''(s')-X_j''(s_x)\big)\cdot \chi_\d(s'-s_x)\,ds'\\
&\; + X_j''(s_x) \int_{s_x-\pi}^{s_x+\pi}
\Big[\pa_k G_{ij}(x-X(s')) - \pa_k G_{ij}\big(x-X(s_x)-(s'-s_x)X'(s_x)\big)\Big] \chi_\d(s'-s_x)\,ds'\\
&\; + X_j''(s_x) \int_{s_x-\pi}^{s_x+\pi}
\pa_k G_{ij}\big(x-X(s_x)-(s'-s_x)X'(s_x)\big) \cdot \chi_\d(s'-s_x)\,ds'\\
&\; + \int_{s_x-\pi}^{s_x+\pi} \pa_{kl} G_{ij}(x-X(s')) X_l'(s') X_j'(s')\big[1-\chi_\d(s'-s_x)\big]\,ds'\\
&\;
+\int_{s_x-\pi}^{s_x+\pi} \pa_k G_{ij}(x-X(s')) X_j'(s')\cdot \chi_\d'(s'-s_x)\,ds'\\
=: &\; V_1+V_2+V_3+V_4+V_5.
\end{align*}
Here we performed integration by parts to obtain the last two terms $V_4$ and $V_5$.

It is not difficult to verify by \eqref{eqn: Stokeslet in 2-D} that (cf.\;Lemma \ref{lem: estimate for K}), for any $k\in \BZ_+$,
\[
|\na^k G(x)|\leq C_k |x|^{-k},
\]
and for any $k\in \BN$,
\[
|\na^k G(x)-\na^k G(y)|
\leq C_k|x-y|\big(|x|^{-k-1}+|y|^{-k-1}\big).
\]
Lemma \ref{lem: lower bound for x-X(s)} implies that,
\[
|x-X(s')|\geq \lam|s_x-s'|.
\]
Moreover, with $d=d(x)$ defined as in \eqref{eqn: def of d(x,t)},
\[
|x-X(s_x)-(s'-s_x)X'(s_x)| = \big(d^2+(s'-s_x)^2|X'(s_x)|^2\big)^{1/2} \geq \lam|s_x-s'|,
\]
and
\[
|X(s')-X(s_x)-(s'-s_x)X'(s_x)|
=\left|\int_{s_x}^{s'} X'(s'')-X'(s_x) \, ds''\right|
\leq C|s'-s_x|^{1+\al}\|X'\|_{\dot{C}_s^\al}.
\]

Using these estimates, we can derive that 
\begin{align*}
|V_1|\leq &\;
C\int_{s_x-2\d}^{s_x+2\d} |x-X(s')|^{-1}|s'-s_x|^\al \|X''\|_{\dot{C}_s^\al}\,ds' \\
\leq &\;
C\|X''\|_{\dot{C}_s^\al}
\int_{s_x-2\d}^{s_x+2\d} \lam^{-1} |s'-s_x|^{-1+\al} \,ds'
\leq 
C \lam^{-1}\|X''\|_{\dot{C}_s^\al} \d^{\al},
\end{align*}
where $C$ depends on $\al$.
Similarly,
\begin{align*}
|V_2|
\leq &\;
C\|X''\|_{L_s^\infty} \int_{s_x-2\d}^{s_x+2\d} |s'-s_x|^{1+\al}\|X'\|_{\dot{C}_s^\al}\\
&\;\qquad \qquad \qquad \qquad \cdot \big( |x-X(s')|^{-2} + |x-X(s_x)-(s'-s_x)X'(s_x)|^{-2}\big)\,ds'\\
\leq &\;
C\|X'\|_{\dot{C}_s^\al}\|X''\|_{L_s^\infty}
\int_{s_x-2\d}^{s_x+2\d} |s'-s_x|^{1+\al}\cdot \lam^{-2}|s'-s_x|^{-2}\,ds'\\
\leq &\; C\lam^{-2} \|X'\|_{L_s^\infty}\|X''\|_{\dot{C}_s^\al} \d^{\al}.
\end{align*}
In the last line, we applied the interpolation inequality.
Moreover,
\begin{align*}
|V_4|
\leq &\; C\int_{[s_x-\pi,s_x+\pi]\setminus[s_x-\d,s_x+\d]}|x-X(s')|^{-2} |X'(s')|^2\,ds'\\
\leq &\; C\int_{[s_x-\pi,s_x+\pi]\setminus[s_x-\d,s_x+\d]}\lam^{-2}|s'-s_x|^{-2} \|X'\|_{L^\infty_s}^2\,ds'
\leq C\lam^{-2} \|X'\|_{L^\infty_s}^2\d^{-1}.
\end{align*}
and
\begin{align*}
|V_5|
\leq
&\; C\int_{s_x-\pi}^{s_x+\pi} |x-X(s')|^{-1} |X'(s')| \cdot |\chi_\d'(s'-s_x)|\,ds'\\
\leq
&\; C\int_{[s_x-2\d,s_x+2\d]\setminus [s_x-\d,s_x+\d]} \lam^{-1} |s'-s_x|^{-1} \|X'\|_{L^\infty_s}\cdot \d^{-1}\,ds'
\leq C \lam^{-1} \|X'\|_{L^\infty_s} \d^{-1}.
\end{align*}

We shall handle $|V_3|$ in a more explicit way.
Without loss of generality, up to suitable translation, rotation, reflection, and a change of variable, we assume that $x = (0,-d)$ with $d>0$, $s_x= 0$, $X(s_x) = (0,0)$, and $X'(s_x) = (c,0)$ with $c\geq \lam>0$.
In order to bound $|V_3|$, it suffices to estimate
\begin{align*}
&\; \left|X''_j(0)\int_{-\pi}^{\pi}
\pa_k G_{ij}(cs',d) \cdot \chi_\d(s') \,ds'\right|\\
= &\; \left|X''_j(0)\int_0^{1} \int_{-\pi}^{\pi}
\pa_k G_{ij}(cs',d) \mathds{1}_{\{\chi_\d(s')\geq z\}} \,ds' \,dz \right|.
\end{align*}
By the assumptions on $\chi_\d$, for any $z\in [0,1]$, the set $\{s'\in \BT:\,\chi_\d(s')\geq z\}$ is an interval of the form $[-a,a]$ for some $a\in [0,\pi]$.
We calculate that, with $x:=(cs',d)$, for an arbitrary $a\in [0,\pi]$,
\[
\int_{-a}^{a}
\pa_k G_{ij}(cs',d)\,ds'
= \f{1}{4\pi} \int_{-a}^{a}
-\f{x_k}{|x|^2}\d_{ij} + \f{\d_{ik}x_j + x_i \d_{jk}}{|x|^2} -\f{2x_ix_j x_k}{|x|^4} \,ds'.
\]
Thanks to the oddness of the integral,
\[
\int_{-a}^{a} \f{x_1}{|x|^2}\,ds'
= \int_{-a}^{a} \f{cs'}{(cs')^2+d^2}\,ds' = 0,
\quad\mbox{and}\quad
\int_{-a}^a \f{x_1^3}{|x|^4}\,ds' = 0.
\]
Hence, what remains in the integral must contain a factor of $x_2$ in the numerator, which gives
\[
\left|\int_{-a}^{a}
\pa_k G_{ij}(cs',d)\,ds'\right|
\leq C\int_{-a}^{a} \f{x_2}{|x|^2}\,ds'
= C\int_{-a}^{a} \f{d}{(cs')^2+d^2}\,ds'
\leq Cc^{-1}\leq C\lam^{-1}.
\]
Therefore, $|V_3|\leq C\lam^{-1}\|X''\|_{L_s^\infty}$.

Summarizing all these estimates yields that
\[
|\na u(x)|
\leq C\lam^{-2} \|X'\|_{L_s^\infty}\|X''\|_{\dot{C}_s^\al} \d^{\al} + C\lam^{-2} \|X'\|_{L^\infty_s}^2\d^{-1} + C\lam^{-1}\|X''\|_{L_s^\infty},
\]
where $C$ depends on $\al$.
Observe that $\|X'\|_{L_s^\infty}\leq C_\al \|X''\|_{\dot{C}_s^\al}$ for some $C_\al>0$ that only depends on $\al$, so we take
\[
\d: = \left(\f{\|X'\|_{L_s^\infty}}{C_\al\|X''\|_{\dot{C}_s^\al}}\right)^{\f{1}{1+\al}}\in (0,1].
\]
This gives
\[
|\na u(x)|
\leq C\lam^{-2} \|X'\|_{L_s^\infty}^{1+\f{\al}{1+\al}}\|X''\|_{\dot{C}_s^\al}^{\f{1}{1+\al}} + C\lam^{-1}\|X''\|_{L_s^\infty}
\leq C\lam^{-2} \|X'\|_{L_s^\infty}^{1+\f{\al}{1+\al}}\|X''\|_{\dot{C}_s^\al}^{\f{1}{1+\al}},
\]
where $C$ depends on $\al$.
In the last inequality, we used the interpolation inequality.
Therefore, the desired Lipschitz estimate for $u_X$ follows.
\end{proof}
\end{lem}

\begin{prop}
\label{prop: higher regularity}

Let $u_0 = u_0(x)\in L^p(\BR^2)$ for some $p\in (2,\infty)$ and $X_0 = X_0(s)\in C^1(\BT)$, which satisfy $\di u_0 = 0$ and $|X_0|_* >0$.
Let $(u,X)$ be the unique local mild solution to \eqref{eqn: NS equation}-\eqref{eqn: initial data} on $[0,T]$.
Let $M$, $\lam$, and $Q$ be defined as in \eqref{eqn: def of M}-\eqref{eqn: def of Q}, respectively.

\begin{enumerate}
\item For any $\b\in [\f2p,1]$,
\beq
\sup_{\eta\in (0,T]} \eta^\b \|X'(\eta)\|_{C^\b_s(\BT)}
+ \sup_{\eta\in (0,T)} \eta^\b \|\pa_t X(\eta)\|_{C^\b_s(\BT)}
\leq C(T, p,\nu,\lam,M,Q).
\label{eqn: C beta estimate for X' and X_t}
\eeq
Note that the constant $C$ depends on $T,\, p,\,\nu,\,\lam,\,M,\,Q$, but not on $\b$.
\item
For any $\b\in (0,1)$,
\beq
\sup_{\eta\in (0,T]} \eta^{1+\b} \|X''(\eta)\|_{C^{\b}_s(\BT)}
+\sup_{\eta\in (0,T)} \eta^{1+\b} \|\pa_t X'(\eta)\|_{C^{\b}_s(\BT)}
\leq C(T,\b,p,\nu,\lam,M,Q).
\label{eqn: C 1+beta estimate for X' and X_t}
\eeq

\item For any $\b\in [0,1)$, $\lim_{t\to 0^+} t^{1+\b}\|X''(t)\|_{\dot{C}_s^\b(\BT)} = 0$.

\item
For any $\b\in [\f2p,1)$,
\[
\sup_{\eta\in (0,T]}\eta^{\b}\|u(\eta)\|_{\dot{C}^\b_x(\BR^2)}
\leq C(T, \b, p,\nu,\lam,M,Q).
\]
Moreover, for any $t\in (0,T]$,
\beq
\|\na u(t)\|_{L^\infty_x(\BR^2)}
\leq C(T,p,\nu,\lam,M,Q) t^{-1}. 
\label{eqn: Lipschitz estimate for u}
\eeq
\end{enumerate}
\begin{proof}

The first estimate follows by interpolating Lemma \ref{lem: C gamma estimate for mild solutions} (with $\g = \f2p$ there) and Lemma \ref{lem: smoothing lemma for X' in 2gamma} (with $\g=\f12+\f1p$ there).
The second estimate follows from Lemma \ref{lem: smoothing lemma for X' in 2gamma}.
The third claim follows by interpolating this with Lemma \ref{lem: intermediate norms vanish as t goes to zero}.
The $C^\b_x$-estimate for $u$ was proved in Lemma \ref{lem: C gamma estimate for mild solutions}.
It remains to show \eqref{eqn: Lipschitz estimate for u}.

Fix $t_0\in (0,T]$.
To bound $\|\na u(t_0)\|_{L^\infty_x(\BR^2)}$, we follow the argument in the proof of Lemma \ref{lem: smoothing lemma for X' in 2gamma} by resetting the initial time to be $t_0/2$.
Then \eqref{eqn: formula for u new initial time} holds, which gives
\begin{align*}
\|\na u(t_0)\|_{L^\infty_x(\BR^2)}
\leq &\; \left\|\left(\mathrm{Id}- e^{\f{\nu t_0}2\Delta}\right)\na u_{X(t_0)}\right\|_{L^\infty_x}
+\big\|\na \tilde{h}_{X,t_0/2}^{\nu}(t_0)\big\|_{L^\infty_x}
\\
&\; +\Big\|e^{\f{\nu t_0}2\Delta} u(t_0/2) \Big\|_{\dot{C}^1_x}
+\big\| \tilde{B}_{\nu,t_0/2}[u](t_0)\big\|_{\dot{C}^1_x}.
\end{align*}

By Lemma \ref{lem: estimate for h_X^nu with time singularity} with e.g.\;$\g = 1-\f1p$,
\begin{align*}
&\;\big\|\na \tilde{h}_{X,t_0/2}^{\nu}(t_0)\big\|_{L^\infty_x}\\
\leq
&\; C(\nu t_0)^{\g-\f12}\left[\nu^{-\g} \lam^{-1} \|X'\|_{L^\infty_t L_s^\infty} \sup_{\eta\in [t_0/2,t_0]}\|X'(\eta)\|_{C_s^\g }^{1-\g} \sup_{\eta\in [t_0/2,t_0]}\|\pa_t X(\eta)\|_{\dot{C}_s^{\g}}^\g \right.\\
&\;\qquad +C (\nu t_0)^{-\f{1+\g}{2}} \lam^{-1-\g} t_0\cdot \|X'\|_{L^\infty_t L_s^\infty}
\\
&\;\left.\quad \qquad \cdot
\left(\sup_{\eta\in [t_0/2,t_0]}\|X'(\eta)\|_{C_s^\g} \sup_{\eta\in [t_0/2,t_0]}\|\pa_t X(\eta)\|_{L_s^\infty}
+ \|X'\|_{L^\infty_t L_s^\infty} \sup_{\eta\in [t_0/2,t_0]}\|\pa_t X(\eta)\|_{\dot{C}_s^{\g}} \right)\right]
\\
\leq
&\; C(T, \g, p,\nu,\lam,M,Q)\left( t_0^{-\g} +t_0^{-\f{1+\g}{2}}\cdot t_0\cdot t_0^{-\g-\f1p}\right) t_0^{\g-\f12}\\
\leq
&\; C(T, p,\nu,\lam,M,Q) t_0^{-\f12-\f{1}{2p}}.
\end{align*}

Now by applying Lemma \ref{lem: estimate for B_nu u with time singularity} (with $\b = \f2p$ there), Lemma \ref{lem: Lipschitz continuity of flow field} (with $\al = \f1p$ there), and Lemma \ref{lem: parabolic estimates}, and also using the updated bounds in \eqref{eqn: C beta estimate for X' and X_t} and \eqref{eqn: C 1+beta estimate for X' and X_t} to derive that (cf.\;the estimates in the proof of Lemma \ref{lem: smoothing lemma for X' in 2gamma})
\begin{align*}
&\;\|\na u(t_0)\|_{L^\infty_x(\BR^2)}\\
\leq &\; C\lam^{-2} \|X'(t_0)\|_{L_s^\infty}^{\f{p+2}{p+1}} \|X''(t_0)\|_{\dot{C}_s^{1/p}}^{\f{p}{p+1}}
+ C(T, p,\nu,\lam,M,Q)  t_0^{-\f12-\f1{2p}}\\
&\; + C(\nu t_0)^{-\f12-\f1p} Q
+ C (\nu t_0)^{\f1p}  \nu^{-1}
\sup_{\eta\in [t_0/2,t_0]}
\|u(\eta)\|_{L_x^\infty}  \|u(\eta)\|_{\dot{C}_x^{2/p}}
\\
\leq &\; C(T, p,\nu,\lam,M,Q)  t_0^{-1} 
\\
&\; +
C (\nu t_0)^{\f1p}  \nu^{-1}\cdot t_0^{-\f1p} C(T, p,\nu,\lam,M,Q)\cdot t_0^{-\f2p}C(T, p,\nu,\lam,M,Q)\\
\leq &\; C(T, p,\nu,\lam,M,Q)  t_0^{-1}. 
\end{align*}
Since $t_0\in (0,T]$ is arbitrary, this completes the proof.
\end{proof}

\begin{rmk}
\label{rmk: optimality of higher-order estimates}
The results in Proposition \ref{prop: higher regularity} are optimal in two ways.
\begin{itemize}
\item The quantitative bounds for the higher-order norms of $X$ and $u$ have time-singularities at $t=0$ with the optimal powers of $t$.
    Indeed, the corresponding norms of the linear term $e^{-\f{t}{4}\Lam}X_0$ in \eqref{eqn: fixed pt equation for X} satisfies estimates with the very powers of $t$, and that further makes the estimates for $u$ admit singularity in $t$ as is claimed (see \eqref{eqn: fixed pt equation for u} and Lemma \ref{lem: estimate for u_X^nu}).

\item 
    As is mentioned in Remark \ref{rmk: optimality of higher-order estimates main thm}, in general, the flow field $u$ can at most be Lipschitz in space.
    Then in view of Lemma \ref{lem: estimate for B_nu u} and its proof, the spatial regularity of $B_\nu[u]$ cannot go beyond $C^{1,1}_x(\BR^2)$.
    By Lemma \ref{lem: parabolic estimates for Duhamel term of fractional Laplace}, we can only expect $\CI[B_\nu[u]\circ X]$ to have spatial regularity no higher than $C^{2,1}_s(\BT)$, and by \eqref{eqn: fixed pt equation for X}, that sets a limit for the regularity of $X$ within this framework of analysis.
\end{itemize}
\end{rmk}
\end{prop}

Before ending this subsection, let us prove a well-expected result on the volume conservation.
\begin{lem}
\label{lem: volume conservation}
For all $t\in [0,T]$,
\[
\f12\int_\BT X(s,t)\times X'(s,t)\,ds \equiv \f12\int_\BT X_0(s)\times X_0'(s)\,ds,
\]
i.e., the area of the region enclosed by the curve $X(\BT,t)$ is constant in time.
\begin{proof}
Let $\Omega_t$ denote the region enclosed by the curve $\G_t:=X(\BT,t)$.
Thanks to Proposition \ref{prop: higher regularity}, for any $t\in(0, T)$, $\G_t$ is a simple $C^{2,\b}$-curve with $\b\in (0,1)$.
Whenever $x\in \G_t$, we shall use $n(x,t)$ to denote the outer unit normal vector along $\G_t$ with respect to $\Omega_t$.

Thanks to the regularity of $X$ and $\pa_t X$ established in Proposition \ref{prop: higher regularity}, for $t\in (0,T)$,
\begin{align*}
\f{d}{dt}\f12\int_\BT X(s,t)\times X'(s,t)\,ds
= &\; \f12\int_\BT \pa_t X(s,t)\times X'(s,t)
+ X(s,t)\times \pa_t X'(s,t)\,ds\\
= &\; \int_\BT \pa_t X(s,t)\times X'(s,t)\,ds.
\end{align*}
We used the integration by parts in the last line.
Recall that $X$ satisfies \eqref{eqn: equation for X_t} in $(0,T)$ (see Lemma \ref{lem: time derivative and equation}).
Hence,
\begin{align*}
\f{d}{dt}\f12\int_\BT X(s,t)\times X'(s,t)\,ds
= &\; -\int_\BT u(X(s,t),t)\cdot X'(s,t)^\perp\,ds\\
= &\; \int_\BT u(X(s,t),t)\cdot n(X(s,t),t) |X'(s,t)|\,ds\\
= &\; \int_{\G_t} u(x,t)\cdot n(x,t) \,d\s
= \int_{\Omega_t} \di u(x,t)\,dx = 0.
\end{align*}
In the last line, we used the divergence theorem, the divergence-free property of $u$, and the regularity of $u$ proved in Proposition \ref{prop: higher regularity}.
Finally, using the continuity of $t\mapsto X(s,t)$ in $C^1(\BT)$ at $t = 0$ and $t = T$, we can conclude the desired claim.
\end{proof}
\end{lem}

\section{The Zero-Reynolds-number Limit}
\label{sec: zero Reynolds number limit}

We have constructed the mild solutions to the 2-D immersed boundary problem with the Navier-Stokes equation \eqref{eqn: NS equation}-\eqref{eqn: initial data}.
It is heuristically conceivable that, as $\nu \to +\infty$, or equivalently as $Re\to 0^+$, the solution would converge in a suitable sense to the one to the Stokes case \eqref{eqn: Stokes equation}-\eqref{eqn: initial data Stokes case}, which formally corresponds to $Re = 0$.
In this section, we shall rigorously justify this convergence and quantify the difference between these two cases.
The main result of this section is Proposition \ref{prop: zero Reynolds number limit}.

As is reviewed in Section \ref{sec: Stokes case revisited}, there have been quite a few results in the literature on the existence and uniqueness of solutions to the Stokes case \eqref{eqn: Stokes equation}-\eqref{eqn: initial data Stokes case}, but here we choose to state the following one that is more compatible with our analysis.
Its proof is essentially the same as (and yet much simpler than) that for the Navier-Stokes case in Section \ref{sec: local well-posedness} and Section \ref{sec: properties of mild solutions}, so we shall omit it completely.
Readers are also referred to \cite[Theorem 1.12 and \S\,3.7]{Chen2024wellposedness_quasilinear} for a similar result on the case of general nonlinear elasticity laws.

\begin{prop}
\label{prop: well-posedness Stokes case}
Suppose that $X_0 = X_0(s)\in C^1(\BT)$ satisfies $|X_0|_*>0$.
Then for some $T>0$ that only depends on $X_0$, there exists a unique $X_\dag = X_\dag(s,t)$ satisfying the following conditions:
\begin{enumerate}[label=(\roman*)]
\item \label{item: regularity and well-stretched condition}
$X_\dag \in C([0,T];C^1(\BT))$, and $\inf_{t\in [0,T]}|X_\dag(t)|_* >0$;
\item \label{item: bounds}
for some $\g\in (0,1)$,
\[
\|X_\dag'\|_{L^\infty_T L^\infty_s(\BT)}
+ \sup_{\eta\in (0,T]}\eta^\g \|X_\dag'(\eta)\|_{\dot{C}^\g_s(\BT)}
< +\infty;
\]
\item \label{item: the equation Duhamel form}
and, for any $t\in [0,T]$,
\[
X_\dag(t) = e^{-\f{t}4 \Lam}X_0+\CI [g_{X_\dag}](t),
\]
where $g_{X_\dag}$ was defined as in \eqref{eqn: def of g_X} in terms of $X_\dag$, and where $\CI[\cdot]$ was defined in \eqref{eqn: def of the operator I}.
\end{enumerate}

Let $M_0:=\|X_0'\|_{L_s^\infty(\BT)}$ and $\lam_0 := |X_0|_*$.
The solution $X_\dag$ additionally enjoys the following properties.
\begin{enumerate}
\item $|X_\dag(t)|_* \geq \lam_0/2$ and $\|X'_\dag(t)\|_{L_s^\infty}\leq CM_0$ for $t\in [0,T]$, where $C$ is a universal constant.

\item $X_\dag \in C^1_{loc}((0,T]\times \BT)$, and the equation
\beqo
\pa_t X_\dag = -\f14\Lam X_\dag + g_{X_\dag}
\eeqo
holds pointwise in $(0,T]\times \BT$.

\item
For any $\b\in(0,1]$,
\[
\sup_{\eta\in (0,T]} \eta^\b \|X_\dag'(\eta)\|_{C^\b_s(\BT)}
+ \sup_{\eta\in (0,T]} \eta^\b \|\pa_t X_\dag(\eta)\|_{C^\b_s(\BT)}
\leq C(\lam_0,M_0). 
\]

\item
For any $\b\in (0,1)$,
\[
\sup_{\eta\in (0,T]} \eta^{1+\b} \|X_\dag''(\eta)\|_{C^{\b}_s(\BT)}
+\sup_{\eta\in (0,T]} \eta^{1+\b} \|\pa_t X_\dag'(\eta)\|_{C^{\b}_s(\BT)}
\leq C(\b,\lam_0, M_0).
\]

\item
Given $\b\in (0,1)$, let $\r=\r_{X_0,\b}(t)$ be defined in terms of $X_0$ as in \eqref{eqn: def of rho}.
Then there exists $T'\in (0,T]$ depending on $\b$ and $X_0$, such that for all $t\in [0,T']$,
\beq
t^\b\|X_\dag'(t)\|_{\dot{C}^\g_s(\BT)} \leq 2\big(\r(t)+t^\b\big).
\label{eqn: modulus of X_dag at initial time}
\eeq

\end{enumerate}

\begin{rmk}
\label{rmk: def of mild solution Stokes case}
We shall call any function $X = X(s,t)$ that satisfies the conditions \ref{item: regularity and well-stretched condition}-\ref{item: the equation Duhamel form} to be a (mild) solution on $[0,T]$ to the Stokes case of the 2-D immersed boundary problem \eqref{eqn: Stokes equation}-\eqref{eqn: initial data Stokes case}.
Note that in the Stokes case, the flow field $u$ can be recovered by $u(x,t) = u_{X(t)}(x)$.
\end{rmk}
\end{prop}

Throughout this section, we assume $u_0 = u_0(x)\in L^p(\BR^2)$ with some $p\in (2,\infty)$ and $X_0 = X_0(s)\in C^1(\BT)$, and they satisfy that $\di u_0 = 0$ and $|X_0|_* >0$.
We still define $M_0$, $\lam_0$, and $ Q_0 $ as in \eqref{eqn: def of initial constants}.

Let $X_\dag = X_\dag(s,t)$ be the unique mild solution on $[0,T_\dag]$ to the Stokes case \eqref{eqn: Stokes equation}-\eqref{eqn: initial data Stokes case} constructed in Proposition \ref{prop: well-posedness Stokes case}, which satisfies all the characterizations there.
In particular, it holds for $t\in [0,T_\dag]$ that
\beq
X_\dag(t) = e^{-\f{t}4 \Lam}X_0+\CI [g_{X_\dag}](t).
\label{eqn: representation of X_dag}
\eeq
Fix an arbitrary $\g\in [\f12+\f1p,1) \cap (1-\f2p,1)$.
Note that this new range of $\g$ may be more restrictive than the previously used one (cf.\;\eqref{eqn: range of gamma}); we will explain this in Remark \ref{rmk: reason for choosing new range of gamma}.
With this $\g$ and with $\epsilon\leq M_0$ to be chosen later (see \eqref{eqn: choice of epsilon}), by the time-continuity of $X_\dag$ and \eqref{eqn: modulus of X_dag at initial time}, we may further take $T_\dag\leq 1$ to be smaller if needed, so that
\beq
\sup_{t\in [0,T_\dag]}\|X_\dag'(t)\|_{L^\infty_s(\BT)}\leq 2M_0,\quad
\sup_{t\in (0,T_\dag]} t^{\g}\|X_\dag'(t)\|_{\dot{C}^{\g}_s(\BT)} \leq \epsilon,\quad
\inf_{t\in [0,T_\dag]} |X_\dag(t)|_* \geq \f{3\lam_0}{4}.
\label{eqn: a priori bound for X dag}
\eeq
and for any $t_1,t_2\in [0,T_\dag]$,
\beq
\|X_\dag'(\cdot,t_1)-X_\dag'(\cdot,t_2)\|_{L^\infty_s(\T)} \leq \f{\lam_0}{8}.
\label{eqn: X dag at different times are close to each other}
\eeq
Such $T_\dag$ will depend on $X_0$, $\g$, and $\epsilon$.

Next we introduce the setup in the Navier-Stokes case.
Without loss of generality, we assume $\nu \geq 1$ in the rest of this section.
Let $(u_\nu,X_\nu)$ denote the unique maximal mild solution to the Navier-Stokes case \eqref{eqn: NS equation}-\eqref{eqn: initial data}, with its maximal lifespan being $T_{\nu,*}\in (0,+\infty]$ (see Definition \ref{def: maximal solution} and Theorem \ref{thm: maximal solution}).
It holds on $[0,T_{\nu,*})$ that (cf.\;\eqref{eqn: fixed pt equation for u} and \eqref{eqn: fixed pt equation for X})
\begin{align}
u_\nu(t) = &\; u_{X_\nu(t)}+\tilde{w}[u_\nu,X_\nu](t)
= u_{X_\nu}^\nu(t) + e^{\nu t\Delta}u_0 +B_{\nu}[u_\nu](t),
\label{eqn: representation of u^nu}
\\
X_\nu(t) = &\; e^{-\f{t}4 \Lam}X_0+\CI \big[g_{X_\nu} + \tilde{w}[u_\nu,X_\nu]\circ X_\nu \big](t),
\label{eqn: representation of X^nu}
\end{align}
where (cf.\;\eqref{eqn: def of w})
\[
\tilde{w}[u_\nu,X_\nu](t)
= \tilde{w}[u_\nu,X_\nu; \nu,u_0](t)
:= h_{X_\nu}^{\nu}+e^{\nu t\Delta}(u_0-u_{X_\nu(t)}) +B_{\nu}[u_\nu].
\]
Let $\g$ be taken as above.
By uniqueness, in short time, $(u_\nu,X_\nu)$ satisfies all the characterizations in Proposition \ref{prop: fixed-point solution} (with the chosen $\g$).
Besides, on any compact sub-interval of $[0,T_{\nu,*})$,  $(u_\nu,X_\nu)$ satisfies all the properties established in Section \ref{sec: properties of mild solutions}.
We thus define $T_\nu \in (0,T_{\nu,*})$ to be the maximum possible time such that the following holds:
\begin{itemize}
\item $T_\nu \leq T_\dag\leq 1$;
\item with the same $\epsilon$ as in \eqref{eqn: a priori bound for X dag},
\beq
\sup_{t\in [0,T_\nu]}\|X_\nu'(t)\|_{L^\infty_s(\BT)}\leq 3M_0,\quad
\sup_{t\in (0,T_\nu]} t^{\g}\|X_\nu'(t)\|_{\dot{C}^{\g}_s(\BT)} \leq 2\epsilon,\quad
\inf_{t\in [0,T_\nu]} |X_\nu(t)|_* \geq \f{\lam_0}{2},
\label{eqn: a priori bound for X nu}
\eeq
\item for any $t_1,t_2\in [0,T_\nu]$,
\beq
\|X_\nu'(\cdot,t_1)-X_\nu'(\cdot,t_2)\|_{L^\infty_s(\T)} \leq \f{\lam_0}{4}.
\label{eqn: X nu at different times are close to each other}
\eeq
\item
and, with $\epsilon'>0$ to be chosen below,
\beq
\sup_{t\in [0,T_\nu]}\|u_\nu(t)\|_{L^p_x(\BR^2)}\leq 2 Q_0 +\epsilon'.
\label{eqn: a priori bound for L^p norm of u nu}
\eeq
\end{itemize}
Here the maximality of $T_\nu$ means that \eqref{eqn: a priori bound for X nu}-\eqref{eqn: a priori bound for L^p norm of u nu} hold on $[0,T_\nu]$, but if $T_\nu$ gets replaced by any $T\in (T_\nu,T_\dag]$, at least one of the properties would fail.
Such $T_\nu$ is well-defined due to the time-continuity of $(u_\nu,X_\nu)$ (see Proposition \ref{prop: fixed-point solution}, Lemma \ref{lem: time continuity of u in L^p}, and Lemma \ref{lem: C gamma estimate for mild solutions}) as well as Corollary \ref{cor: maximal solution}; we only point out that, by Corollary \ref{cor: maximal solution}, as $T\to T_{\nu,*}^-$,
\[
\sup_{t\in (0,T]} t^{\g}\|X_\nu'(t)\|_{\dot{C}^{\g}_s(\BT)}+
\left(\inf_{t\in [0,T]} |X_\nu(t)|_*\right)^{-1}
+
\sup_{t\in [0,T]}\|u_\nu(t)\|_{L^p_x(\BR^2)}
\]
should diverge.

We choose the parameter $\epsilon'$ in \eqref{eqn: a priori bound for L^p norm of u nu} as follows.
By Lemma \ref{lem: estimate for u_X^nu},
\begin{align*}
\|u_{X_\nu}^\nu(t)\|_{L_x^p}
\leq &\; C_0'\lam_0^{-1+\f{2}{p}} M_0^2,
\end{align*}
where $C_0'$ is universal constant depending only on $p$.
Then we let
\beq
\epsilon' := 2C_0' \lam_0^{-1+\f{2}{p}} M_0^2.
\label{eqn: def of epsilon prime}
\eeq
It will be clear in the proof of Lemma \ref{lem: improved closeness to the Stokes solution} below why this choice of $\epsilon'$ suffices.

The main result to prove in this section is as follows.
\begin{prop}
\label{prop: zero Reynolds number limit}
Let $p$, $(u_0,X_0)$, $(u_\nu,X_\nu)$, and $X_\dag$ be given as above.
Fix $\g\in [\f12+\f1p,1)\cap (1-\f2p,1)$.
Define $M_0$, $\lam_0$, and $Q_0$ as in \eqref{eqn: def of initial constants}, let $T_\dag$ and $T_\nu$ be defined as above (also see Remark \ref{rmk: dependence of T_dag} below).
Then there exists $\nu_*\geq 1$, which depends on $\g$, $p$, $\lam_0$, $M_0$, and $Q_0$, such that $T_\nu = T_\dag$ for any $\nu \geq \nu_*$.
Hence, for $\nu \geq \nu_*$, \eqref{eqn: a priori bound for X nu}-\eqref{eqn: a priori bound for L^p norm of u nu} hold on $[0,T_\dag]$.
Moreover, for any $t\in (0,T_\dag]$,
\beq
\|(X_\nu-X_\dag)(t)\|_{L^\infty_s(\BT)}
\leq C t\cdot (\nu t)^{-\f1p},
\label{eqn: L inf different between X_dag and X_nu}
\eeq
\beq
\|(X_\nu-X_\dag)(t)\|_{\dot{C}_s^1(\BT)}
+ t^\g \|(X_\nu-X_\dag)'(t)\|_{\dot{C}^\g_s(\BT)}
\leq Ct \cdot (\nu t)^{-\f{1}{2}-\f1p},
\label{eqn: C 1 and C 1 gamma different between X_dag and X_nu}
\eeq
and
\beqo
\|u_\nu(t)-u_{X_\dag(t)}\|_{L^\infty_x(\BR^2)}
\leq C  \ln \nu \cdot t
\|(X_\nu'(t),X_\dag'(t))\|_{\dot{C}^1_s}\cdot (\nu t)^{-\f1p} + C(\nu t)^{-\f1p},
\eeqo
where the constants $C$ depend on $\g$, $p$, $\lam_0$, $M_0$, and $ Q_0 $.

\begin{rmk}
\label{rmk: extend the range of gamma in the zero Reynolds number limit}
For $\g\not \in [\f12+\f1p,1)\cap (1-\f2p,1)$, similar estimates can be derived from Proposition \ref{prop: zero Reynolds number limit} by the interpolation inequality.
\end{rmk}
\end{prop}

To prove this proposition, let us first justify more uniform-in-$\nu$ estimates for $(u_\nu,X_\nu)$ based on \eqref{eqn: a priori bound for X nu} and \eqref{eqn: a priori bound for L^p norm of u nu}.

\begin{lem}
\label{lem: uniform in nu bounds}
Assume \eqref{eqn: a priori bound for X nu} and \eqref{eqn: a priori bound for L^p norm of u nu}.
When $\nu\geq 1$ is taken to be suitably large, which depends on $p$, $\lam_0$, $M_0$, and $ Q_0 $, it holds that
\beq
\sup_{\eta\in (0,T_\nu]}(\nu \eta)^{\f1p}\|u_\nu(\eta)\|_{L_x^\infty(\BR^2)}
\leq C\nu^{\f1p}\lam_0^{-1}M_0^2 + C(p,\lam_0,M_0, Q_0 ),
\label{eqn: uniform in nu L inf bound for u nu}
\eeq
where $C$ is a constant depending only on $p$, and
\begin{align*}
\sup_{\eta\in (0,T_\nu]}\eta^{\f1p}\|\pa_t X_\nu(\eta)\|_{L_s^\infty(\BT)}
\leq &\;
C(p,\lam_0,M_0, Q_0 ),
\sup_{\eta\in (0,T_\nu]}\eta^{\g}\|\pa_t X_\nu(\eta)\|_{\dot{C}^\g_s(\BT)}
\leq &\; C(\g, p,\lam_0,M_0, Q_0).
\end{align*}

\begin{proof}
By \eqref{eqn: representation of u^nu}, Lemma \ref{lem: estimate for u_X^nu}, and \eqref{eqn: L inf estimate for B} in Lemma \ref{lem: estimate for B_nu u} (cf.\;\eqref{eqn: spatial L inf estimate for v}),
\begin{align*}
&\; (\nu t)^{\f1p}\|u_\nu(t)-u_{X_\nu}^\nu(t)\|_{L_x^\infty}
\\
\leq &\;
C  Q_0  +  C(\nu t)^{\f1p} \nu^{-1} (\nu t)^{-\f12}\sup_{\eta\in(0,t]}(\nu \eta)^{\f12-\f1p}\|u_\nu(\eta)\|_{L_x^p}\cdot
\sup_{\eta\in(0,t]}(\nu \eta)^{\f12}\|u_\nu(\eta)\|_{L_x^\infty}
\\
\leq &\; C  Q_0  \\
&\; +  C_1 \nu^{-1} (\nu t)^{\f12-\f1p} \sup_{\eta\in(0,t]}\|u_\nu(\eta)\|_{L_x^p}\\
&\;\quad \cdot
\left(\nu^{\f1p} \lam_0^{-1}\|X_\nu'\|_{L^\infty_{T_\nu} L_s^{\infty}}
\sup_{\eta\in(0,T_\nu]} \eta^{1/p} \|X_\nu'(\eta)\|_{\dot{C}_s^{1/p}}
+
\sup_{\eta\in(0,T_\nu]}(\nu \eta)^{\f1p}\|u_\nu(\eta)-u_{X_\nu}^\nu(\eta)\|_{L_x^\infty}
\right),
\end{align*}
where $C$ and $C_1$ only depend on $p$.
Note that, given the bounds for $(u_\nu,X_\nu)$ from Proposition \ref{prop: fixed-point solution}, $\sup_{\eta\in(0,t]}(\nu \eta)^{\f1p}\|u_\nu(\eta)-u_{X_\nu}^\nu(\eta)\|_{L_x^\infty}$ is known to be finite.
In view of \eqref{eqn: a priori bound for L^p norm of u nu} and $T_\nu \leq T_\dag \leq 1$, we may assume $\nu\geq 1$ to be sufficiently large, which depends on $p$, $\lam_0$, $M_0$, and $ Q_0 $, so that
\[
C_1 \nu^{-\f12} (\nu T_\nu)^{\f12-\f1p} \sup_{\eta\in(0,T_\nu]}\|u_\nu(\eta)\|_{L_x^p}\leq \f12.
\]
Then with \eqref{eqn: a priori bound for X nu} and the interpolation inequality, we can obtain that
\[
\sup_{\eta\in (0,T_\nu]}(\nu \eta)^{\f1p}\|u_\nu(\eta)-u_{X_\nu}^\nu(\eta)\|_{L_x^\infty(\BR^2)}
\leq C  Q_0  + C\nu^{-\f12+\f1p} \lam_0^{-1}M_0^2
\leq C(p,\lam_0,M_0, Q_0 ).
\]
Here we used the fact that $\epsilon\leq M_0$.
As a result, by Lemma \ref{lem: estimate for u_X^nu},
\[
\sup_{\eta\in (0,T_\nu]}(\nu \eta)^{\f1p}\|u_\nu(\eta)\|_{L_x^\infty(\BR^2)}
\leq C\nu^{\f1p}\lam_0^{-1}M_0^2 + C(p,\lam_0,M_0, Q_0 ),
\]
where $C$ depends on $p$, and by \eqref{eqn: differential equation of X in the mild solution},
\beqo
\sup_{\eta\in (0,T_\nu]}\eta^{\f1p}\|\pa_t X_\nu(\eta)\|_{L_s^\infty(\BT)}
\leq \sup_{\eta\in (0,T_\nu]}\eta^{\f1p}\|u_\nu(\eta)\|_{L_x^\infty}
\leq C(p,\lam_0,M_0, Q_0 ).
\eeqo

Similarly, by \eqref{eqn: Holder estimate for B u prelim} in Lemma \ref{lem: estimate for B_nu u}, for any $t\in (0,T_\nu]$,
\beq
\begin{split}
&\; \|B_\nu [u_\nu](t)\|_{\dot{C}_x^\g}\\
\leq &\; C \nu^{-1} (\nu t)^{-(1+\g)/2} \sup_{\eta\in (0,t]} (\nu \eta)^{(1+\g)/2}\|(u_\nu\otimes u_\nu)(\eta)\|_{L_x^{2/(1-\g)}}
\\
\leq &\; C\nu^{-1} (\nu t)^{-\f{2}{p} + \f{1-\g}{2}} \left(\sup_{\eta\in (0,t]}\|u_\nu(\eta)\|_{L_x^p}\right)^{\f{p(1-\g)}{2}}
\left(\sup_{\eta\in (0,t]} (\nu \eta)^{\f1p}\|u_\nu(\eta)\|_{L_x^\infty}\right)^{2-\f{p(1-\g)}{2}},
\end{split}
\label{eqn: Holder estimate for B u nu}
\eeq
where $C$ depends on $\g$.
Here we used the interpolation inequality and the fact $\f{2}{1-\g}>p$.
Using \eqref{eqn: a priori bound for L^p norm of u nu} and \eqref{eqn: uniform in nu L inf bound for u nu},
\begin{align*}
(\nu t)^{\g}\|u_\nu(t)-u_{X_\nu}^\nu(t)\|_{\dot{C}_x^\g}
\leq &\; C (\nu t)^{\f{\g}2-\f1p} Q_0 \\
&\; +  C(\nu t)^{\g} \nu^{-1} (\nu t)^{-\f{2}{p} + \f{1-\g}{2}} C(p,\lam_0,M_0,Q_0)^{\f{p(1-\g)}{2}}\\
&\;\quad \cdot
\left(C\nu^{\f1p}\lam_0^{-1}M_0^2 + C(p,\lam_0,M_0, Q_0 )\right)^{2-\f{p(1-\g)}{2}}
\\
\leq &\;
C (\nu t)^{\f{\g}2-\f1p} Q_0 + C(\g,p,\lam_0,M_0, Q_0 ) \nu^{-1+\g} t^{\f{1+\g}{2}-\f{2}{p}}.
\end{align*}
Since $\g\geq \f2p$, $\nu \geq 1$, and $t\leq T_\nu \leq 1$,
\begin{align*}
\sup_{\eta\in (0,T_\nu]}(\nu \eta)^{\g}\|u_\nu(\eta)\|_{\dot{C}_x^\g(\BR^2)}
\leq &\;
\sup_{\eta\in (0,T_\nu]}(\nu \eta)^{\g} \|u_{X_\nu}^\nu(\eta)\|_{\dot{C}_x^\g} + C \nu^{\f{\g}2-\f1p} Q_0
+ C(\g,p,\lam_0,M_0, Q_0 )\\
\leq &\; C\nu^\g \lam_0^{-1-\g}M_0^2 + C \nu^{\f{\g}2-\f1p} Q_0 + C(\g,p,\lam_0,M_0, Q_0 ),
\end{align*}
where $C$ depends on $\g$.
Hence, by \eqref{eqn: differential equation of X in the mild solution},
\beqo
\sup_{\eta\in (0,T_\nu]}\eta^{\g}\|\pa_t X_\nu(\eta)\|_{\dot{C}^\g_s(\BT)}
\leq \sup_{\eta\in (0,T_\nu]}\eta^{\g}\|u_\nu(\eta)\|_{\dot{C}^\g_x} \|X_\nu'(\eta)\|_{L^\infty_s}^\g
\leq C(\g, p,\lam_0,M_0, Q_0 ).
\eeqo
\end{proof}
\end{lem}

We will use a continuity method to prove Proposition \ref{prop: zero Reynolds number limit}.
The following lemma plays a crucial role.

\begin{lem}
\label{lem: improved closeness to the Stokes solution}

Let $T_\nu \leq T_\dag \leq 1$ be defined as above, and assume \eqref{eqn: a priori bound for X dag}, \eqref{eqn: X dag at different times are close to each other}, and
\eqref{eqn: a priori bound for X nu}-\eqref{eqn: a priori bound for L^p norm of u nu}, where $\epsilon$ and $\epsilon'$ are chosen in \eqref{eqn: choice of epsilon} and \eqref{eqn: def of epsilon prime}, respectively.

When $\nu\geq 1$ is sufficiently large, which depends on $\g$, $p$, $\lam_0$, $M_0$, and $ Q_0 $, it holds that
\begin{itemize}
\item
\beq
\sup_{t\in [0,T_\nu]}\|X_\nu'(t)\|_{C_s(\BT)}\leq \f52 M_0,\quad
\sup_{t\in (0,T_\nu]} t^{\g}\|X_\nu'(t)\|_{\dot{C}^{\g}_s(\BT)} \leq \f32\epsilon,\quad
\inf_{t\in [0,T_\nu]} |X_\nu(t)|_* \geq \f{5\lam_0}{8},
\label{eqn: improved bound for X nu}
\eeq
\item
for any $t_1,t_2\in [0,T_\nu]$,
\beq
\|X_\nu'(\cdot,t_1)-X_\nu'(\cdot,t_2)\|_{L^\infty_s(\T)} \leq \f{3\lam_0}{16}.
\label{eqn: improved bound for X nu at different times are close to each other}
\eeq
\item
and
\beq
\sup_{t\in [0,T_\nu]}\|u_\nu(t)\|_{L^p_x(\BR^2)}\leq \f34\big(2 Q_0  + \epsilon'\big).
\label{eqn: improved bound for L^p norm of u nu}
\eeq
\end{itemize}

\begin{proof}
By Lemma \ref{lem: estimate for u_X^nu}, \eqref{eqn: L p estimate for B} in Lemma \ref{lem: estimate for B_nu u}, \eqref{eqn: representation of u^nu}, and \eqref{eqn: a priori bound for X nu}, for any $t\in [0,T_\nu]$ (cf.\;\eqref{eqn: L p estimate for v}),
\begin{align*}
\|u_\nu(t)\|_{L_x^p}
\leq &\; \|u_{X_\nu}^\nu(t)\|_{L_x^p} + \|e^{\nu t\Delta}u_0\|_{L_x^p} + \|B_{\nu}[u_\nu](t)\|_{L_x^p}
\\
\leq &\; C_0'\lam_0^{-1+\f{2}{p}} M_0^2 +  Q_0
+ C\nu^{-1} (\nu t)^{\f12-\f1p} \left(\sup_{\eta\in(0,t]}\|u_\nu(\eta)\|_{L_x^p}\right)^2,
\end{align*}
where $C_0'$ and $C$ are universal constants depending only on $p$.
By \eqref{eqn: a priori bound for L^p norm of u nu} and \eqref{eqn: def of epsilon prime},
\[
\|u_\nu(t)\|_{L_x^p}
\leq C_0'\lam_0^{-1+\f{2}{p}} M_0^2 +  Q_0
+ C\nu^{-1} (\nu t)^{\f12-\f1p} \left(Q_0 + C_0' \lam_0^{-1+\f{2}{p}} M_0^2\right)^2.
\]
Since $T_\nu\leq T_\dag \leq 1$, by taking $\nu$ to be sufficiently large, which depends on $p$, $\lam_0$, $M_0$, and $ Q_0 $,
we can achieve that
\beqo
\sup_{t\in [0,T_\nu]}\|u_\nu(t)\|_{L^p_x(\BR^2)}\leq \f32 \left( Q_0 + C_0' \lam_0^{-1+\f{2}{p}} M_0^2\right)
= \f34\big(2 Q_0  + \epsilon'\big),
\eeqo
which is \eqref{eqn: improved bound for L^p norm of u nu}.

By \eqref{eqn: representation of X_dag} and \eqref{eqn: representation of X^nu},
\beq
(X_\nu - X_\dag)(t) = \CI \big[g_{X_\nu}-g_{X_\dag} + \tilde{w}[u_\nu,X_\nu]\circ X_\nu \big](t).
\label{eqn: representation of X_dag-X_nu}
\eeq
To bound $(X_\nu - X_\dag)$, we apply Lemma \ref{lem: parabolic estimates for Duhamel term of fractional Laplace} as in \eqref{eqn: estimate for Y-semigroup solution prelim} and find that, for any $t\in [0,T_\nu]$ with $T_\nu\leq 1$,
\beq
\begin{split}
&\; \|(X_\nu-X_\dag)(t)\|_{\dot{C}_s^1(\BT)}
+ t^\g \|(X_\nu-X_\dag)'(t)\|_{\dot{C}^\g_s(\BT)}\\
\leq &\; C \sup_{\eta\in (0,t]} \eta^\g \big\|(g_{X_\nu}-g_{X_\dag})(\eta)\big\|_{\dot{C}_s^\g(\BT)}\\
&\; + Ct^{\f12-\f1p} \sup_{\eta\in (0,t]} \eta^{\f12+\f1p} \big\|\big(h_{X_\nu}^{\nu} + e^{\nu \eta\Delta}(u_0-u_{X_\nu(\eta)})\big)\circ X_\nu(\eta)\big\|_{\dot{C}_s^1(\BT)}\\
&\;+ C \sup_{\eta\in (0,t]} \eta^\g \big\|B_{\nu}[u_\nu]\circ X_\nu(\eta)\big\|_{\dot{C}_s^{\g}(\BT)},
\end{split}
\label{eqn: bound for X nu - X_dag prelim}
\eeq
where $C$ depends on $p$ and $\g$.

By Lemma \ref{lem: improved estimates for g_X-g_Y} and the interpolation inequality,
\begin{align*}
&\;\|g_{X_\nu}(t)-g_{X_\dag}(t)\|_{\dot{C}^\g_s(\BT)}
\\
\leq
&\; C \lam_0^{-2}
\big(\|X_\nu'(t)\|_{L_s^\infty}+\|X_\dag'(t)\|_{L_s^\infty}\big) \big(\|X_\nu'(t)\|_{\dot{C}_s^{\g/2}}+\|X_\dag'(t)\|_{\dot{C}_s^{\g/2}}\big)\\
&\;\cdot \Big[\|X_\nu'(t)-X_\dag'(t)\|_{\dot{C}_s^{\g/2}} + \lam_0^{-1} \big(\|X_\nu'(t)\|_{\dot{C}_s^{\g/2}}+\|X_\dag'(t)\|_{\dot{C}_s^{\g/2}}\big) \|X_\nu'(t)-X_\dag'(t)\|_{L_s^\infty}\Big]
\\
\leq
&\; C t^{-\g} \lam_0^{-2} M_0^{3/2}
\left[\sup_{\eta\in (0,t]} \eta^{\g}\left(\|X_\nu'(\eta)\|_{\dot{C}_s^{\g}} +\|X_\dag'(\eta)\|_{\dot{C}_s^{\g}}\right)\right]^{1/2}\\
&\;\cdot
\left[\sup_{\eta\in (0,t]}\eta^{\g/2}\|(X_\nu'-X_\dag')(\eta)\|_{\dot{C}_s^{\g/2}} + \lam_0^{-1} M_0 \sup_{\eta\in (0,t]}\|({X_\nu}-X_\dag)'(\eta)\|_{L_s^\infty}
\right].
\end{align*}
By the interpolation inequality, \eqref{eqn: a priori bound for X dag}, and \eqref{eqn: a priori bound for X nu}, it simplifies to
\begin{align*}
&\;\sup_{\eta\in (0,t]} \eta^\g \|(g_{X_\nu}-g_{X_\dag})(\eta)\|_{\dot{C}^\g_s(\BT)} \\
\leq
&\; C \lam_0^{-3} M_0^{5/2} \epsilon^{1/2}
\left[\sup_{\eta\in (0,t]}\|({X_\nu}-X_\dag)'(\eta)\|_{L_s^\infty}
+ \sup_{\eta\in (0,t]}\eta^{\g}\|(X_\nu'-X_\dag')(\eta)\|_{\dot{C}_s^{\g}}
\right],
\end{align*}
where the constant $C$ depends only on $\g$.

The last two lines of \eqref{eqn: bound for X nu - X_dag prelim} can be bounded as in \eqref{eqn: C 1 gamma estimate for Y minus semigroup solution} and \eqref{eqn: Holder estimate for B u nu}, respectively.
Indeed, with $\g\geq \f12+\f1p$ and
\[
N_\nu:= \sup_{\eta\in (0,T_\nu)}\eta^{\f1p}\|\pa_t X_\nu(\eta)\|_{L_s^\infty(\BT)} +
\sup_{\eta\in (0,T_\nu)}\eta^{\g}\|\pa_t X_\nu(\eta)\|_{\dot{C}^\g_s(\BT)},
\]
we proceed as in \eqref{eqn: C 1 gamma estimate for Y minus semigroup solution} to find that
\beq
\begin{split}
&\; t^{\f12-\f1p} \sup_{\eta\in (0,t]} \eta^{\f12+\f1p} \big\|\big(h_{X_\nu}^{\nu} + e^{\nu \eta\Delta}(u_0-u_{X_\nu(\eta)})\big)\circ X_\nu(\eta)\big\|_{\dot{C}_s^1(\BT)}
\\
\leq &\; Ct^{\f12-\f1p} \nu^{-\f12-\f1p} \left(  Q_0  + \lam_0^{-1+\f2p} M_0^2\right) M_0 \\
&\; + C t^{\f12-\f1p} \nu^{-\f{1}{2}-\f1p} \lam_0^{-1+\f2p}
\left( M_0^{2-\g} N_\nu^{\g} +\lam_0^{-1} M_0^2 N_\nu \cdot t^{\f1p}\right) M_0
\\
\leq
&\; C(\g, p,\lam_0,M_0, Q_0) \cdot t \cdot (\nu t)^{-\f{1}{2}-\f1p},
\end{split}
\label{eqn: Lipschitz estimate for h and semigroup term}
\eeq
Here we used the bound $N_\nu\leq C(\g, p,\lam_0,M_0, Q_0)$ due to Lemma \ref{lem: uniform in nu bounds}.
On the other hand, following \eqref{eqn: Holder estimate for B u nu} and applying \eqref{eqn: a priori bound for L^p norm of u nu}, \eqref{eqn: def of epsilon prime}, and Lemma \ref{lem: uniform in nu bounds},
\begin{align*}
&\;\sup_{\eta\in (0,t]} \eta^\g \big\|B_{\nu}[u_\nu]\circ X_\nu(\eta)\big\|_{\dot{C}_s^{\g}(\BT)}
\\
\leq &\;  C
t^\g \cdot \nu^{-1} (\nu t)^{-\f{2}{p} + \f{1-\g}{2}}
\left(\sup_{\eta\in (0,t]}\|u_\nu(\eta)\|_{L_x^p}\right)^{\f{p(1-\g)}{2}}
\left(\sup_{\eta\in (0,t]} (\nu \eta)^{\f1p}\|u_\nu(\eta)\|_{L_x^\infty}\right)^{2-\f{p(1-\g)}{2}}
M_0^\g
\\
\leq
&\; C(\g, p,\lam_0,M_0, Q_0 ) \cdot t \cdot (\nu t)^{-\f{1}{2}-\f1p}. 
\end{align*}
In the last line, we used the assumptions that $p>2$, $\g > \f2p$, $\nu\geq 1$, and $t\leq 1$.

Combining all these estimates into \eqref{eqn: bound for X nu - X_dag prelim}, we obtain that
\beq
\begin{split}
&\; \|(X_\nu-X_\dag)(t)\|_{\dot{C}_s^1(\BT)}
+ t^\g \|(X_\nu-X_\dag)'(t)\|_{\dot{C}^\g_s(\BT)}\\
\leq &\; C_2 \lam_0^{-3} M_0^{5/2} \epsilon^{1/2}
\left[\sup_{\eta\in (0,t]}\|({X_\nu}-X_\dag)'(\eta)\|_{L_s^\infty}
+ \sup_{\eta\in (0,t]}\eta^{\g}\|(X_\nu'-X_\dag')(\eta)\|_{\dot{C}_s^{\g}}
\right]
\\
&\; + C(\g, p,\lam_0,M_0, Q_0 ) \cdot t \cdot (\nu t)^{-\f{1}{2}-\f1p},
\end{split}
\label{eqn: C 1 gamma bound for X_dag-X_nu prelim}
\eeq
where $C_2$ depends only on $\g$.

In view of this, we shall choose $\epsilon\leq M_0$ suitably small, so that
\beq
C_2 \lam_0^{-3}M_0^{5/2}\epsilon^{1/2}\leq \f12.
\label{eqn: choice of epsilon}
\eeq
Then with $T_\nu \leq T_\dag \leq 1$, for any $t\in (0,T_\nu]$,
\beq
\|(X_\nu-X_\dag)(t)\|_{\dot{C}_s^1(\BT)}
+ t^\g \|(X_\nu-X_\dag)'(t)\|_{\dot{C}^\g_s(\BT)}\\
\leq C_3(\g, p,\lam_0,M_0, Q_0 ) \cdot  t \cdot (\nu t)^{-\f{1}{2}-\f1p},
\label{eqn: C 1 gamma bound for X_dag-X_nu}
\eeq
We additionally require $\nu$ to be sufficiently small, which depends on $\g$, $p$, $\lam_0$, $M_0$, and $ Q_0 $, such that
\[
C_3(\g, p,\lam_0,M_0, Q_0 ) \cdot \nu^{-\f{1}{2}-\f1p} \leq \min\left(\f{M_0}2,\,\f{\epsilon}{2},\,\f{\lam_0}{32}\right).
\]
Then from \eqref{eqn: a priori bound for X dag} and \eqref{eqn: X dag at different times are close to each other}, one can readily show that $X_\nu$ satisfies \eqref{eqn: improved bound for X nu} and \eqref{eqn: improved bound for X nu at different times are close to each other}.

This completes the proof.
\end{proof}

\begin{rmk}
\label{rmk: reason for choosing new range of gamma}
As is explained in Remark \ref{rmk: optimality of the convergence}, the estimates \eqref{eqn: C 1 gamma bound for X_dag-X_nu prelim} and \eqref{eqn: C 1 gamma bound for X_dag-X_nu} are optimal in terms of the $\nu$- and $t$-dependence given the presence of the term $e^{\nu t\D}u_0$ in the dynamics of $X_\nu$.
The assumption $\g \geq \f12+\f1p$ was made to achieve such optimal form of the bounds.
We point out that it was used when estimating $\|h_{X_\nu}^{\nu}\circ X_\nu\|_{\dot{C}_s^1(\BT)}$ in \eqref{eqn: Lipschitz estimate for h and semigroup term}; also see \eqref{eqn: W 1 inf estimate for h} in Lemma \ref{lem: L inf and W 1 inf estimate for h_X^nu with optimal power of nu}.
\end{rmk}
\begin{rmk}
\label{rmk: dependence of T_dag}
Combining \eqref{eqn: a priori bound for X dag}, \eqref{eqn: X dag at different times are close to each other}, and \eqref{eqn: choice of epsilon}, we can conclude that $T_\dag$ depends on $\g$, $\lam_0$, $M_0$, and $X_0$.
\end{rmk}
\end{lem}

Then the convergence as $\nu \to +\infty$ follows immediately.

\begin{proof}[Proof of Proposition \ref{prop: zero Reynolds number limit}]

We may define $\nu_*$ to be sufficiently large so that, as long as $\nu \geq \nu_*$, Lemma \ref{lem: uniform in nu bounds} and Lemma \ref{lem: improved closeness to the Stokes solution} hold.
Such $\nu_*$ depends on $\g$, $p$, $\lam_0$, $M_0$, and $ Q_0 $.
In the rest of the proof, we always assume $\nu \geq \nu_*$.

The claim that $T_\nu = T_\dag$ for all $\nu\geq \nu_*$ follows from Lemma \ref{lem: improved closeness to the Stokes solution} and the time-continuity of $(u_\nu,X_\nu)$.
Indeed, if $T_\nu<T_\dag$ for some $\nu\geq \nu_*$, then the time-continuity of $(u_\nu,X_\nu)$ as well as \eqref{eqn: improved bound for X nu}-\eqref{eqn: improved bound for L^p norm of u nu} implies that there exists $T_\nu'\in (T_\nu,T_\dag]$, such that \eqref{eqn: a priori bound for X nu}-\eqref{eqn: a priori bound for L^p norm of u nu} still hold with $T_\nu$ there replaced by $T_\nu'$.
This contradicts with the maximality of $T_\nu$.

The estimates \eqref{eqn: a priori bound for X nu}-\eqref{eqn: a priori bound for L^p norm of u nu} for $(u_\nu,X_\nu)$ on $[0,T_\dag]$ then follow.
The bound \eqref{eqn: C 1 and C 1 gamma different between X_dag and X_nu} has been proved in \eqref{eqn: C 1 gamma bound for X_dag-X_nu}.

To show \eqref{eqn: L inf different between X_dag and X_nu}, we derive by Lemma \ref{lem: estimate for u_X}, Lemma \ref{lem: estimate for h_X^nu} (also see the arguments in Lemma \ref{lem: L inf and W 1 inf estimate for h_X^nu with optimal power of nu}), \eqref{eqn: L inf estimate for B} in Lemma \ref{lem: estimate for B_nu u}, Lemma \ref{lem: parabolic estimates} that (also see \eqref{eqn: L inf estimate for w}), for all $t\in (0,T_\nu]$,
\begin{align*}
&\; \|\tilde{w}[u_\nu,X_\nu](t)\|_{L^\infty_x(\BR^2)}\\
\leq
&\;\|h_{X_\nu}^{\nu}(t)\|_{L^\infty_x(\BR^2)}
+ \big\|e^{\nu t\D}(u_0-u_{X_\nu(t)})\big\|_{L^\infty_x(\BR^2)}
+ \|B_{\nu}[u_\nu](t)\|_{L^\infty_x(\BR^2)}\\
\leq
&\; C(\nu t)^{-\f1p} \min\left(\lam_0^{-1+\f2p},\, (\nu t)^{-\f12+\f1p}\right)\\
&\;\cdot
\|X_\nu'\|_{L^\infty_t L_s^\infty}  \left(\sup_{\eta\in (0,t]}\eta^\g \|X_\nu'(\eta)\|_{C_s^\g}\right)^{1-\g} \left(\sup_{\eta\in (0,t]}\eta^\g \|\pa_t X_\nu(\eta)\|_{\dot{C}_s^{\g}}\right)^{\g} \\
&\; + C (\nu t)^{-\f1p}  \min\left(\lam_0^{-1+\f2p},\, (\nu t)^{-\f12+\f1p}\right) t^{\f1p}\\
&\;\quad \cdot \lam_0^{-1} \left( \|X_\nu'\|_{L^\infty_t L_s^\infty}
\sup_{\eta\in (0,t]}\eta^{1-\f2p} \|X_\nu'(\eta)\|_{C_s^{1-\f2p}} \sup_{\eta\in (0,t]}\eta^{\f1p}\|\pa_t X_\nu(\eta)\|_{L_s^\infty} \right.\\
&\;\qquad \qquad \left.+ \|X_\nu'\|_{L^\infty_t L_s^\infty}^2 \sup_{\eta\in (0,t]}\eta^{1-\f1p}\|\pa_t X_\nu(\eta)\|_{\dot{C}_s^{1-\f2p}} \right)
\\
&\; + C (\nu t)^{-\f1p} \|u_0-u_{X_\nu(t)}\|_{L^p_x} \\
&\; + C\nu^{-1} (\nu t)^{\f12-\f2p}
\sup_{\eta\in(0,t]}\|u_\nu(\eta)\|_{L_x^p}
\sup_{\eta\in(0,t]}(\nu \eta)^{\f1p}\|u_\nu(\eta)\|_{L_x^\infty}.
\end{align*}
By Lemma \ref{lem: uniform in nu bounds} as well as \eqref{eqn: a priori bound for X nu} and \eqref{eqn: a priori bound for L^p norm of u nu}, for any $0<t\leq T_\dag \leq 1$,
\beq
\begin{split}
&\; \|\tilde{w}[u_\nu,X_\nu](t)\|_{L^\infty_x(\BR^2)}\\
\leq
&\; C(\g, p,\lam_0,M_0, Q_0 ) (\nu t)^{-\f{1}{p}}
+ C (\nu t)^{-\f1p} \big( Q_0  + \lam_0^{-1+\f2p} M_0^2\big) \\
&\; + C\nu^{-1} (\nu t)^{\f12-\f2p} C(p,\lam_0,M_0, Q_0 )
\left[C(p)\nu^{\f1p}\lam_0^{-1}M_0^2 + C(p,\lam_0,M_0, Q_0 )\right]
\\
\leq &\; C(\g, p,\lam_0,M_0, Q_0 ) (\nu t)^{-\f{1}{p}}.
\end{split}
\label{eqn: L inf estimate for tilde w}
\eeq
Moreover, by Lemma \ref{lem: improved estimates for g_X-g_Y}, \eqref{eqn: a priori bound for X dag}, \eqref{eqn: a priori bound for X nu}, \eqref{eqn: C 1 and C 1 gamma different between X_dag and X_nu}, and the interpolation inequality,
\begin{align*}
&\; \|g_{X_\nu}(t)-g_{X_\dag}(t)\|_{L^\infty_s(\BT)}\\
\leq &\; C\lam_0^{-2} \big(\|X_\nu'(t)\|_{L_s^\infty}+\|X_\dag'(t)\|_{L_s^\infty}\big)^2\\
&\;\cdot \Big[\|X_\nu'(t)-X_\dag'(t)\|_{\dot{C}_s^{1/2}}
+ \lam_0^{-1} \big(\|X_\nu'(t)\|_{\dot{C}_s^{1/2}}+\|X_\dag'(t)\|_{\dot{C}_s^{1/2}}\big) \|X_\nu'(t)-X_\dag'(t)\|_{L_s^\infty} \Big]
\\
\leq &\; C\lam_0^{-2}M_0^2 \Big[t^{-\f12}\cdot t^{\f12}\|X_\nu'(t)-X_\dag'(t)\|_{\dot{C}_s^{1/2}}
+ t^{-\f12} \lam_0^{-1}  M_0 \|X_\nu'(t)-X_\dag'(t)\|_{L^\infty} \Big]
\\
\leq &\; C\lam_0^{-3}M_0^3  t^{-\f12} \cdot C(\g, p, \lam_0, M_0,  Q_0 )\cdot  t\cdot (\nu t)^{-\f{1}{2}-\f1p}
\\
\leq &\; C(\g, p, \lam_0, M_0,  Q_0 ) (\nu t)^{-\f1p}.
\end{align*}
By \eqref{eqn: representation of X_dag-X_nu},
\begin{align*}
\|X_\nu(t)-X_\dag(t)\|_{L^\infty_s(\BT)}
\leq &\; \int_0^t \|g_{X_\nu}(\tau)-g_{X_\dag}(\tau)\|_{L^\infty_s}
+ \|\tilde{w}[u_\nu,X_\nu](\tau)\|_{L^\infty_x}\,d\tau\\
\leq &\; C(\g, p, \lam_0, M_0,  Q_0 ) \cdot t\cdot (\nu t)^{-\f1p}.
\end{align*}

Finally, by Lemma \ref{lem: estimate for u_X}, Lemma \ref{lem: estimate for Laplace of heat operator applied to u_X-u_Y}, and \eqref{eqn: representation of u_X using time integral}, with $\d\in (0,t]$ to be chosen,
\begin{align*}
&\; \|u_{X_\nu(t)}-u_{X_\dag(t)}\|_{L^\infty_x(\BR^2)}\\
\leq &\; \int_0^\infty
\big\|\D e^{\tau \D}\big(u_{X_\nu(t)}-u_{X_\dag(t)}\big)\big\|_{L^\infty_x(\BR^2)}\,d\tau\\
\leq &\; C \int_0^\d
\lam_0^{-1-\f2p}\tau^{-1+\f1p} \|(X_\nu',X_\dag')\|_{\dot{C}_s^{2/p}} \|(X_\nu',X_\dag')\|_{L_s^\infty} \,d\tau\\
&\;+ C \int_\d^t \lam_0^{-2}\tau^{-1} \\
&\;\qquad \quad \cdot \left(\|(X_\nu',X_\dag')\|_{L_s^\infty}
\|(X_\nu',X_\dag')\|_{\dot{C}^1_s}\|X_\nu-X_\dag\|_{L_s^\infty}
+ \|(X_\nu',X_\dag')\|_{L_s^\infty}^2 \|X_\nu-X_\dag\|_{\dot{C}_s^1}\right) d\tau
\\
&\;+ C \int_t^\infty \tau^{-3/2} \|(X_\nu',X_\dag')\|_{L_s^\infty} \|(X_\nu-X_\dag)'\|_{L_s^\infty}
+ \lam_0^{-1}\tau^{-3/2} \|(X_\nu',X_\dag')\|_{L_s^\infty}^2 \|X_\nu-X_\dag\|_{L_s^\infty} \,d\tau
\\
\leq &\; C \d^{1/p}\cdot
\lam_0^{-1-\f2p} \|(X_\nu',X_\dag')\|_{\dot{C}_s^{2/p}} \|(X_\nu',X_\dag')\|_{L_s^\infty}\\
&\;+ C\lam_0^{-2} \ln \big(\d^{-1} t\big) \\
&\;\quad \cdot
\left(\|(X_\nu',X_\dag')\|_{L_s^\infty}\|(X_\nu',X_\dag')\|_{\dot{C}^1_s}\|X_\nu-X_\dag\|_{L_s^\infty}
+ \|(X_\nu',X_\dag')\|_{L_s^\infty}^2 \|(X_\nu-X_\dag)'\|_{L_s^\infty}\right)
\\
&\;+ Ct^{-\f12}\left(\|(X_\nu',X_\dag')\|_{L_s^\infty} \|(X_\nu-X_\dag)'\|_{L_s^\infty}
+ \lam_0^{-1} \|(X_\nu',X_\dag')\|_{L_s^\infty}^2 \|X_\nu-X_\dag\|_{L_s^\infty}\right).
\end{align*}
Using \eqref{eqn: a priori bound for X dag}, \eqref{eqn: a priori bound for X nu}, \eqref{eqn: L inf different between X_dag and X_nu} and \eqref{eqn: C 1 and C 1 gamma different between X_dag and X_nu}, we simplify it as
\begin{align*}
&\; \|u_{X_\nu(t)}-u_{X_\dag(t)}\|_{L^\infty_x(\BR^2)}
\\
\leq &\; C \d^{\f1p}\cdot  t^{-\f2p} \lam_0^{-1-\f2p} M_0^2\\
&\;+ C(\g, p, \lam_0, M_0,  Q_0 ) \ln \big(\d^{-1} t\big)
\left( \|(X_\nu',X_\dag')\|_{\dot{C}^1_s}\cdot t\cdot (\nu t)^{-\f1p}
+ t\cdot (\nu t)^{-\f12-\f1p}\right)
\\
&\;+ C(\g, p, \lam_0, M_0,  Q_0 ) \cdot t^{-\f12} \left( t\cdot (\nu t)^{-\f12-\f1p} + t\cdot (\nu t)^{-\f1p}\right).
\end{align*}
Take $\d := t^2(\nu t)^{-1}\in (0,t]$.
Since $\nu \geq 1$ and $t\leq T_\dag \leq 1$, we find that
\begin{align*}
\|u_{X_\nu(t)}-u_{X_\dag(t)}\|_{L^\infty_x(\BR^2)}
\leq &\; C(\g, p, \lam_0, M_0,  Q_0 )  \ln \nu \cdot t
\|(X_\nu'(t),X_\dag'(t))\|_{\dot{C}^1_s}\cdot (\nu t)^{-\f1p}
\\
&\; + C(\g, p, \lam_0, M_0,  Q_0 ) (\nu t)^{-\f1p}.
\end{align*}
Therefore, by \eqref{eqn: representation of u^nu} and \eqref{eqn: L inf estimate for tilde w},
\begin{align*}
\|u_\nu(t)-u_{X_\dag(t)}\|_{L^\infty_x(\BR^2)}
\leq &\; \|u_{X_\nu(t)}-u_{X_\dag(t)}\|_{L^\infty_x(\BR^2)}
+ \|\tilde{w}[u_\nu,X_\nu](t)\|_{L^\infty_x(\BR^2)}\\
\leq &\; C(\g, p, \lam_0, M_0,  Q_0)  \ln \nu \cdot t
\|(X_\nu'(t),X_\dag'(t))\|_{\dot{C}^1_s}\cdot (\nu t)^{-\f1p}
\\
&\; + C(\g, p, \lam_0, M_0,  Q_0) (\nu t)^{-\f1p}.
\end{align*}
\end{proof}

\section{Local Solutions with $u_0\in L^2\cap L^p(\BR^2)$}
\label{sec: local well-posedness for u_0 in L^2 and L^p}

Let $(u_0,X_0)$ be assumed as in Section \ref{sec: properties of mild solutions}.
From now on, let us additionally assume that $u_0\in L^2(\BR^2)$.
This is motivated by the need of establishing the energy law of the system (see Proposition \ref{prop: energy law}), which will be used in the study of the global solutions in Section \ref{sec: global well-posedness}.

We still use $(u,X)$ to denote the unique mild solution to \eqref{eqn: NS equation}-\eqref{eqn: initial data} on $[0,T]$, and let $M$, $\lam$, $Q$ be defined as in \eqref{eqn: def of M}-\eqref{eqn: def of Q}, respectively.
Besides, we denote $K_0:= \|u_0\|_{L_x^2(\BR^2)} < +\infty$.

\begin{lem}
\label{lem: L 2 bound for the mild solution u}
$u\in C([0,T];L^2(\BR^2))$, and
\[
\|u\|_{L^\infty_T L^2_x(\BR^2)}\leq C(T, p,\nu,\lam,M,Q,K_0).
\]

\begin{proof}
We first derive an upper bound for $\|u(t)\|_{L^2_x(\BR^2)}$.
This can be achieved by repeatedly using \eqref{eqn: fixed pt equation for u} and Lemma \ref{lem: estimate for B_nu u}.
We only sketch it.

For any $q\in [2,p]$, by \eqref{eqn: fixed pt equation for u}, Lemma \ref{lem: estimate for u_X^nu}, and \eqref{eqn: L q estimate for B general} in Lemma \ref{lem: estimate for B_nu u}, with $r\in [1,q]$ satisfying that $0< \f1r <\f12 + \f1q$,
\begin{align*}
&\; \|u(t)\|_{L_x^q(\BR^2)}\\
\leq &\; \|u_0\|_{L^q_x}
+ C(\nu t)^{\f1{2q}} \lam^{-1+\f1q}
\|X'\|_{L^\infty_t L_s^\infty}^2
+ C\nu^{-1} (\nu t)^{\f12+\f1q-\f1r}
\sup_{\eta\in(0,t]}\|(u\otimes u)(\eta)\|_{L_x^r}.
\end{align*}
If $p \in (2,4]$, we simply take $(q,r) = (2,\f{p}2)$ to obtain that
\[
\|u(t)\|_{L_x^2(\BR^2)}
\leq \|u_0\|_{L^2_x}
+ C(\nu t)^{\f14} \lam^{-\f12} M^2
+ C\nu^{-1} (\nu t)^{1-\f2p}\|u\|_{L^\infty_t L_x^{p}}^2
\leq C(T, p,\nu,\lam,M,Q,K_0).
\]
If $p> 4$, we first take $q = r = p/2$, and find that
\[
\|u(t)\|_{L_x^{p/2}(\BR^2)}
\leq \|u_0\|_{L^{p/2}_x}
+ C(\nu t)^{\f1{p}} \lam^{-1+\f2p} M^2
+ C\nu^{-1} (\nu t)^{\f12}
\|u\|_{L^\infty_t L_x^{p}}^2
\leq C(T, p,\nu,\lam,M,Q,K_0).
\]
Note that $\|u_0\|_{L^{p/2}_x}\leq C(K_0,Q)$ by interpolation.
Then we bootstrap with this updated estimate.
After finite steps of iteration, we can conclude that
\[
\|u\|_{L^\infty_T L_x^2(\BR^2)}
\leq
C(T, p,\nu,\lam,M,Q,K_0).
\]

Next we prove $u\in C([0,T];L^2(\BR^2))$. 
With $r:=(\f12+\f1p)^{-1}\in (1,2)$, we apply Lemma \ref{lem: estimate for u_X^nu} and \eqref{eqn: L q estimate for B general} in Lemma \ref{lem: estimate for B_nu u} to find that, for $t\in [0,T]$,
\[
\big\|u_X^\nu(t)+B_\nu[u](t)\big\|_{L^r_x(\BR^2)}
\leq C(\nu t)^{\f1{2r}} \lam^{-1+\f1r}
\|X'\|_{L^\infty_t L_s^\infty}^2
+  C\nu^{-1} (\nu t)^{\f12}\|u\|_{L^\infty_t L_x^2}
\|u\|_{L^\infty_t L_x^p},
\]
which is uniformly bounded for $t\in [0,T]$.
On the other hand, $t\mapsto u_X^\nu(t)+B_\nu[u](t) = u(t)-e^{\nu t\D}u_0$ is continuous in the $L^p(\BR^2)$-topology due to Lemma \ref{lem: time continuity of u in L^p} and the fact $u_0\in L^p(\BR^2)$.
By interpolation, $t\mapsto u_X^\nu(t)+B_\nu[u](t)$ is also continuous in the $L^2(\BR^2)$-topology.
Since $u_0\in L^2(\BR^2)$ implies that $t\mapsto e^{\nu t \D}u_0$ is continuous in $L^2(\BR^2)$, we use \eqref{eqn: fixed pt equation for u} to conclude that $u\in C([0,T];L^2(\BR^2))$.
\end{proof}
\end{lem}

Let us also show that $\na u(t)\in L^2(\BR^2)$ for all $t\in (0,T]$.
\begin{lem}
\label{lem: u is H^1 at any positive time}
\[
\sup_{\eta\in (0,T]}\eta^{1/2}\|\na u(\eta)\|_{L_x^2(\BR^2)}
\leq C(T, p,\nu,\lam,M,Q,K_0).
\]
In particular, $u(t)\in H^1(\BR^2)$ for any $t\in (0,T]$.

\begin{proof}
Recall \eqref{eqn: fixed pt equation for u}.
By Lemma \ref{lem: parabolic estimates},
\[
\|\na e^{\nu t\Delta}u_0\|_{L_x^2(\BR^2)}
\leq C(\nu t)^{-\f12}\|u_0\|_{L^2_x(\BR^2)}
= C(\nu t)^{-\f12} K_0,
\]
where $C$ is a universal constant.
To bound the $L^2$-norm of $\na u_X^\nu$, we interpolate the results of Lemma \ref{lem: estimate for u_X} to find that, with $\g\in (0,1)$ to be chosen,
\begin{align*}
&\;\|\na \Delta e^{t\Delta}u_{X}\|_{L_x^2(\BR^2)}\\
\leq &\; C\min\left(\lam^{-\f12-\g} t^{-\f54+\f{\g}{2}}\|X'\|_{L_s^{\infty}}\|X'\|_{\dot{C}_s^{\g}},\,
\lam^{-\f12}t^{-\f54}\|X'\|_{L_s^{\infty}}\|X'\|_{L_s^2},\,
t^{-\f32}\|X'\|_{L^2_s}^2
\right).
\end{align*}
By \eqref{eqn: u_X^nu in terms of u_X}, for any $\g\in(\f12,1)$,  with $\d\in (0,\f{t}{2}]$ to be determined,
\begin{align*}
\|\na u_X^{\nu}(t)\|_{L_x^2(\BR^2)}
\leq &\; C\nu\int_{t-\d}^t \lam^{-\f12-\g} (\nu(t-\tau))^{-\f54+\f{\g}{2}}\tau^{-\g}\cdot \|X'\|_{L^\infty_t L_s^{\infty}}\cdot \tau^{\g} \|X'(\tau)\|_{\dot{C}_s^{\g}} \, d\tau
\\
&\; + C\nu\int_{t/2}^{t-\d}
\lam^{-\f12}(\nu(t-\tau))^{-\f54}\|X'\|_{L^\infty_t L_s^{\infty}} \|X'\|_{L^\infty_t L_s^2}
\, d\tau
\\
&\; + C\nu\int_0^{t/2}
(\nu(t-\tau))^{-\f32}\|X'\|_{L^\infty_t L^2_s}^2
\, d\tau
\\
\leq &\; C(T, \g, p,\nu,\lam,M,Q)\left[ \d^{-\f14+\f{\g}{2}}t^{-\g}
+ \mathds{1}_{\{\d< \f{t}2\}}\d^{-\f14}
+ t^{-\f12}\right].
\end{align*}
In the last inequality, we used Proposition \ref{prop: higher regularity}.
Letting $\g := \f34$ and $\d := \f12\min(t^2,t)$, we find that
\[
\|\na u_X^{\nu}(t)\|_{L_x^2(\BR^2)}
\leq C(T, p,\nu,\lam,M,Q)t^{-1/2}.
\]

Finally, by \eqref{eqn: u_2 tilde}, again with $\d\in[0,\f{t}{2}]$ to be determined,
\begin{align*}
B_{\nu}[u](t) = e^{\nu\d\Delta} B_{\nu}[u](t-\d) -\int_{t-\d}^t e^{\nu (t-\tau)\Delta}\mathbb{P}(u\cdot\nabla u)(\tau)\, d\tau.
\end{align*}
Hence, by \eqref{eqn: L 2 estimate for B} in Lemma \ref{lem: estimate for B_nu u}, Proposition \ref{prop: higher regularity}, and Lemma \ref{lem: parabolic estimates},
\begin{align*}
&\; \|\na B_{\nu}[u](t)\|_{L_x^2(\BR^2)} \\
\leq &\; \big\|\na e^{ \nu \d\Delta} B_{\nu}[u](t-\d)\big\|_{L_x^2} + \left\|\int_{t-\d}^t e^{\nu (t-\tau)\Delta}\na \mathbb{P}(u\cdot\nabla u)(\tau)\, d\tau\right\|_{L_x^2}
\\
\leq &\; C(\nu \d)^{-\f12}\|B_{\nu}[u](t-\d)\|_{L_x^2} + C\int_{t-\d}^t (\nu (t-\tau))^{-\f12} \|(u\cdot\nabla u)(\tau)\|_{L_x^2}\, d\tau
\\
\leq &\; C(\nu \d)^{-\f12} \nu^{-1} \|u\|_{L^\infty_t L_x^2} \sup_{\eta\in(0,t]}(\nu \eta)^{\f12-\f1p}\|u(\eta)\|_{L_x^p}
+ C\nu^{-\f12} \d^{\f12} \|u\|_{L^\infty_t L^2_x} \sup_{\eta\in [t/2,t]}\|\nabla u(\eta)\|_{L_x^\infty}
\\
\leq &\; C(T, p,\nu,\lam,M,Q,K_0)\left[\d^{-\f12} t^{\f12-\f1p}
+ \d^{\f12}t^{-1}\right].
\end{align*}
In the last line, we used Proposition \ref{prop: higher regularity} and Lemma \ref{lem: L 2 bound for the mild solution u}.
Choosing $\d := \f12\min(t^{2-\f2p},t)$ yields that
\[
\|\na B_{\nu}[u](t)\|_{L_x^2(\BR^2)}
\leq C(T, p,\nu,\lam,M,Q,K_0)t^{-\f12}.
\]
Here in the case $t\geq 1$, we absorbed some powers of $t$ into the constant $C$ which may depend on $T$.

Combining all the estimates, we conclude the desired estimate.
\end{proof}
\end{lem}

Now we can prove the energy law of the system.

\begin{prop}
\label{prop: energy law}
Let $u_0 = u_0(x)\in L^2\cap L^p(\BR^2)$ for some $p\in (2,\infty)$ and $X_0 = X_0(s)\in C^1(\BT)$, which satisfy $\di u_0 = 0$ and $|X_0|_* >0$.
Let $(u,X)$ be the unique local mild solution to \eqref{eqn: NS equation}-\eqref{eqn: initial data} on $[0,T]$.
Then for any $t\in [0,T]$,
\beqo
\|X'(t)\|_{L_s^2(\BT)}^2 + \nu^{-1}\|u(t)\|_{L_x^2(\BR^2)}^2
+ 2\int_0^t\|\na u(\tau)\|_{L^2_x(\BR^2)}^2\,d\tau
= \|X'_0\|_{L_s^2(\BT)}^2 + \nu^{-1}\|u_0\|_{L_x^2(\BR^2)}^2.
\eeqo

\begin{proof}
Take $\va = \va(x)\in C_0^\infty(\BR^2)$ to be a non-negative radial mollifier in $\BR^2$, satisfying that $\va$ is supported in the unit disc $B_1(0)$ centered at the origin, and $\int_{\BR^2} \va(x)\,dx = 1$.
For $\e>0$, we define $\va_\e(x) := \e^{-2} \va(x/\e)$.

Fix $0<\e\ll 1$.
Take $[t_1,t_2]\subset (0,T)$ and consider an arbitrary $t\in [t_1,t_2]$.
By Proposition \ref{prop: higher regularity}, $X = X(\cdot,t)\in C^{2,\al}(\BT)$ for any fixed $\al\in (0,1)$, and $|X(t)|_*\geq \lam$.
By \eqref{eqn: expression of u_11}, we claim that $\va_\e*u_{X(t)}$ solves
\beq
-\D (\va_\e* u_{X(t)}) = \mathbb{P} f_{X,\e},\quad f_{X,\e}(x,t): = \int_\BT \va_\e(x-X(s,t))X''(s,t)\,ds,
\label{eqn: def of f_X eps}
\eeq
where $\BP$ denotes the Leray projection operator in $\BR^2$.
Indeed, one can readily show that $f_{X,\e}(\cdot,t)\in C_0^\infty(\BR^2)$ and has integral $0$ in $\BR^2$, so
\begin{align*}
(-\D)^{-1} \BP f_{X,\e}
= &\; G*f_{X,\e}(x) = \int_{\BR^2} G(x-y)\int_{\BT} \va_\e(y-X(s,t))X''(s,t)\,ds\,dy\\
= &\; \int_{\BR^2} \va_\e(x-y)\int_{\BT} G(y-X(s,t))X''(s,t)\,ds\,dy
= \va_\e * u_{X(t)}.
\end{align*}
Hence, by \eqref{eqn: fixed pt equation for u} as well as \eqref{eqn: u_X^nu in terms of u_X} and \eqref{eqn: u_2 tilde},
\begin{align*}
\va_\e* u(t)
= &\; \va_\e* e^{\nu t\Delta}u_0 - \nu\int_0^t e^{\nu (t-\tau)\Delta} \Delta \va_\e* u_{X(\tau)}\, d\tau -\int_0^te^{\nu (t-\tau)\Delta}\va_\e*  \mathbb{P}(u\cdot\nabla u)(\tau)\, d\tau
\\
= &\; e^{\nu t\Delta} (\va_\e* u_0) + \int_0^t e^{\nu (t-\tau)\Delta} \BP \big[ \nu f_{X,\e} -\va_\e*
(u\cdot \na u)\big](\tau) \, d\tau.
\end{align*}
By virtue of Proposition \ref{prop: higher regularity}, Lemma \ref{lem: L 2 bound for the mild solution u}, and \eqref{eqn: def of f_X eps}, one can show that, on $[t_1,t_2]\subset (0,T)$, the mappings $t\mapsto f_{X,\e}(t)$ and $t\mapsto \va_\e*(u\cdot \na u)(t)$ are continuous in the $C^\b_x(\BR^2)$-topology for some $\b\in(0,1)$; we omit the details.
This allows us to deduce that
\begin{align*}
&\; (\pa_t -\nu\D) (\va_\e* u)(t)\\
= &\; (\pa_t -\nu\D) e^{\nu t\Delta} (\va_\e* u_0)
+ (\pa_t -\nu\D) \int_0^t e^{\nu (t-\tau)\Delta} \BP \big[ \nu f_{X,\e} -\va_\e*
(u\cdot\na u)\big](\tau) \, d\tau
\\
= &\; \BP \big[ \nu f_{X,\e} -\va_\e*
(u\cdot\na u)\big](t).
\end{align*}
We note that $\va_\e*u$ is smooth in space.
Besides, also by Proposition \ref{prop: higher regularity} and Lemma \ref{lem: L 2 bound for the mild solution u}, the last line of the above equation is smooth in space and uniformly bounded in $L^2(\BR^2)$ for all $t\in [t_1,t_2]$.
It is thus legitimate to take inner product of the above equation with $\va_\e*u$, which gives
\begin{align*}
&\; \f12 \f{d}{dt} \|\va_\e* u\|_{L_x^2(\BR^2)}^2 + \nu \|\na (\va_\e* u)\|_{L^2_x(\BR^2)}^2\\
= &\; \int_{\BR^2}(\pa_t -\nu\D) (\va_\e* u)\cdot (\va_\e* u) \,dx
\\
= &\; \int_{\BR^2} \big[ \nu f_{X,\e} -\va_\e*
(u\cdot\na u)\big] \cdot (\va_\e* u) \,dx
\\
= &\; \nu \int_{\BR^2} (\va_\e* u)(x) \int_\BT \va_\e(x-X(s,t))X''(s,t)\,ds\, \,dx\\
&\; -\int_{\BR^2} \big[\va_\e*
(u\cdot\na u)-(\va_\e* u) \cdot\na (\va_\e* u)\big] \cdot (\va_\e* u) \,dx.
\end{align*}
Indeed, the first equality follows from \cite[Chapter III, Section 1.4, Lemma 1.2]{temam1984navier} as well as the smoothness of $(\va_\e* u)$ and its time derivative.
In the last equality, we used the fact that, since $\va_\e*u\in L^2\cap C^1(\BR^2)$ and it is divergence-free, it holds that
\[
\int_{\BR^2} (\va_\e* u) \cdot\na (\va_\e* u) \cdot (\va_\e* u) \,dx = 0.
\]
Hence,
\begin{align*}
&\; \f12 \f{d}{dt} \|\va_\e* u\|_{L_x^2(\BR^2)}^2 + \nu \|\na (\va_\e* u)\|_{L^2_x(\BR^2)}^2\\
= &\; \nu \int_\BT \big[(\va_\e*\va_\e)* u\big](X(s,t))\cdot X''(s,t)\,ds\\
&\; +\int_{\BR^2} \big[\va_\e*
(u\otimes u)-(\va_\e* u)\otimes(\va_\e* u)\big] : \na (\va_\e* u) \,dx.
\end{align*}
Integrate this over $t\in [t_1,t_2]$ yields that
\begin{align*}
&\; \f12\left[\|(\va_\e* u)(t_2)\|_{L_x^2(\BR^2)}^2-\|(\va_\e* u)(t_1)\|_{L_x^2(\BR^2)}^2\right]
+ \nu \int_{t_1}^{t_2}\|\na (\va_\e* u)(t)\|_{L^2_x(\BR^2)}^2\,dt\\
= &\; \nu \int_{t_1}^{t_2}\int_\BT \big[(\va_\e*\va_\e)* u\big](X(s,t))\cdot X''(s,t)\,ds\,dt\\
&\; + \int_{t_1}^{t_2}\int_{\BR^2} \big[\va_\e*
(u\otimes u)-(\va_\e* u)\otimes(\va_\e* u)\big] : \na (\va_\e* u) \,dx\,dt.
\end{align*}

Now send $\e\to 0^+$.
Thanks to Proposition \ref{prop: higher regularity}, Lemma \ref{lem: L 2 bound for the mild solution u}, and Lemma \ref{lem: u is H^1 at any positive time}, $u(t)$ is uniformly bounded in $H^1\cap W^{1,\infty}(\BR^2)$ for $t\in [t_1,t_2]$, so by the dominated convergence theorem,
\begin{align*}
&\lim_{\e\to 0^+}\|(\va_\e* u)(t_i)\|_{L_x^2(\BR^2)}^2
= \|u(t_i)\|_{L_x^2(\BR^2)}^2\quad (i = 1,2),\\
&\lim_{\e\to 0^+} \int_{t_1}^{t_2}\|\na (\va_\e* u)(t)\|_{L^2_x(\BR^2)}^2\,dt
=\int_{t_1}^{t_2}\|\na u(t)\|_{L^2_x(\BR^2)}^2\,dt,
\end{align*}
and
\[
\lim_{\e\to 0^+}\int_{t_1}^{t_2}\int_{\BR^2} \big[\va_\e*
(u\otimes u)-(\va_\e* u)\otimes(\va_\e* u)\big] : \na (\va_\e* u) \,dx\,dt = 0.
\]
Moreover, by Lemma \ref{lem: time derivative and equation} and Proposition \ref{prop: higher regularity},
\begin{align*}
&\; \lim_{\e\to 0^+}\nu \int_{t_1}^{t_2}\int_\BT \big[(\va_\e*\va_\e)* u\big](X(s,t))\cdot X''(s,t)\,ds\,dt\\
= &\; \nu \int_{t_1}^{t_2}\int_\BT u(X(s,t))\cdot X''(s,t)\,ds\,dt
= \nu \int_{t_1}^{t_2}\int_\BT \pa_t X(s,t)\cdot X''(s,t)\,ds\,dt\\
= &\; -\nu \int_{t_1}^{t_2}\int_\BT \pa_t X'(s,t)\cdot X'(s,t)\,ds\,dt
= -\nu \int_{t_1}^{t_2} \f12 \f{d}{dt}\| X'(\cdot,t)\|_{L_s^2}^2\,dt\\
= &\; -\f{\nu}{2} \left[\| X'(\cdot,t_2)\|_{L_s^2}^2-\| X'(\cdot,t_1)\|_{L_s^2}^2\right].
\end{align*}
Combining these limits with the identity above, we obtain that
\begin{align*}
&\; \f12\left[\|u(t_2)\|_{L_x^2(\BR^2)}^2 -\|u(t_1)\|_{L_x^2(\BR^2)}^2\right]
+ \nu \int_{t_1}^{t_2}\|\na u(t)\|_{L^2_x(\BR^2)}^2\,dt\\
= &\; - \f{\nu}{2} \left[\|X'(t_2)\|_{L_s^2}^2- \|X'(t_1)\|_{L_s^2}^2\right].
\end{align*}
Since $0<t_1<t_2\leq T$ are arbitrary, we send $t_1\to 0^+$ and use the time-continuity of $(u,X)$ in $L_x^2(\BR^2)\times C_s^1(\BT)$ to conclude the desired energy estimate.
\end{proof}
\end{prop}


\section{Global Solutions for Small Initial Datum}
\label{sec: global well-posedness}

In the section, we shall prove that starting from small initial datum, we can obtain global solutions (see Definition \ref{def: maximal solution}).
By saying the initial data is small, we mean that it is close to an equilibrium (i.e., a steady-state solution) in a suitable sense, which will be specified later.
The main result of this section is Proposition \ref{prop: global solution}.

\subsection{Preliminaries}

Recall that, for $X=X(s)$, we defined its Fourier expansion in \eqref{eqn: Fourier series} and \eqref{eqn: Fourier coefficients}.
It is not difficult to show that
\beq
\|X'\|_{L_s^2(\BT)}^2 = 2\pi\sum_{n\in\BZ}n^2|a_n|^2.
\label{eqn: Fourier representation of H1 seminorm}
\eeq
Also recall that, if $X = X(s)\in C^1(\BT)$ is a Jordan curve parameterized in the counterclockwise direction, its effective radius $R>0$ was defined in \eqref{eqn: def of effective radius}.
Using the Fourier expansion \eqref{eqn: Fourier series} and \eqref{eqn: Fourier coefficients} of $X = X(s)$, we have that
\beq
\pi R^2 = \f12\int_{\BT} X(s)\times X'(s)\, ds
= \pi\sum_{n\in\BZ}n|a_n|^2.
\label{eqn: Fourier representation of area}
\eeq

We claim that the only equilibria (i.e., steady-states) of the system \eqref{eqn: NS equation}-\eqref{eqn: kinematic equation} are
\[
u(x)\equiv 0,\quad
X(s) = b_0 + \begin{pmatrix}
\cos s & -\sin s \\
\sin s & \cos s
\end{pmatrix} b_1
\]
for some $b_0,b_1\in \BR^2$ with $|b_1|>0$. 
Indeed, in order for $(u,X)$ to be a steady-state solution, the energy law in Proposition \ref{prop: energy law} and the decay condition of $u$ at spatial infinity imply that $u\equiv 0$.
Then one can show that $p$ is constant in the interior and exterior of the domain enclosed by the curve $X(\BT)$.
The rest of the argument is the same as that in \cite[\S\,5.1]{mori2019well}.

Motivated by this, for $X=X(s)$, we defined $X^*$ and $\Pi X$ in \eqref{eqn: def of projection operator Pi}, where $X^*(s)$ is closest equilibrium to $X(s)$ in the $L^2$-sense, while $\Pi X$ characterizes deviation of $X$ from an equilibrium.
We list some simple facts regarding $\Pi X$ without proofs.

\begin{itemize}
\item
The operator $\Pi$ commutes with the differentiation, and it is bounded in $\dot{H}^1(\BT)$ and $\dot{C}^{\g}(\BT)$ for every $\g\geq 0$, namely,
\beq
\|\Pi X\|_{\dot{H}^1_s(\BT)}\leq \|X\|_{\dot{H}^1_s(\BT)},\quad
\|\Pi X\|_{\dot{C}^\g_s(\BT)}\leq C\| X\|_{\dot{C}^\g_s(\BT)}.
\label{eqn: boundedness of projection Pi}
\eeq

\item With $R$ defined in terms of $X$ by \eqref{eqn: Fourier representation of area}, it holds that
\beq
\f12 \|\Pi X'\|_{L_s^2(\BT)}^2
\leq \|X'\|_{L_s^2(\BT)}^2- 2\pi R^2
\leq 2\|\Pi X'\|_{L_s^2(\BT)}^2.
\label{eqn: projection is equivalent to excess energy}
\eeq
Indeed, $n^2-n\in [\f12n^2, 2n^2]$ for all $n\in\BZ\setminus\{1\}$ while it vanishes for $n = 1$, so \eqref{eqn: projection is equivalent to excess energy} follows from \eqref{eqn: Fourier representation of H1 seminorm} and \eqref{eqn: Fourier representation of area}.
This in particular implies $\|X'\|_{L_s^2}^2\geq 2\pi R^2$, where $2\pi R^2$ is the total elastic energy of a string configuration in equilibrium with the effective radius being $R$.

\item
If $\|\Pi X'\|_{L^{\infty}_s}\leq \f1{16} R$, then
\beq
\|{X^*}'\|_{L^\infty_s(\BT)}
\leq R+\|\Pi X'\|_{L^{\infty}_s},
\quad
\|X'\|_{L^\infty_s(\BT)}
\leq R+2\|\Pi X'\|_{L^{\infty}_s},
\label{eqn: a priori estimates for X under smallness conditions}
\eeq
and
\beq
|X|_* \geq \f{2}{\pi} R- 2\|\Pi X'\|_{L^{\infty}_s}.
\label{eqn: a priori estimates for X well stretched constant under smallness conditions}
\eeq
Indeed, assuming the expansion \eqref{eqn: Fourier series}, we find by  \eqref{eqn: Fourier representation of H1 seminorm} and \eqref{eqn: Fourier representation of area} that
\[
|R^2-|a_1|^2|
\leq \f{1}{2\pi} \|\Pi X'\|_{L^{2}_s}^2
\leq \|\Pi X'\|_{L^{\infty}_s}^2,
\]
so under the assumption,
\[
R- \|\Pi X'\|_{L^{\infty}_s}
\leq |a_1|
\leq R+\|\Pi X'\|_{L^{\infty}_s}.
\]
Then the desired estimates follow from the facts that
\[
\|{X^*}'\|_{L^\infty_s}=|a_1|,\quad
\|X'\|_{L^\infty_s}\leq |a_1|+\|\Pi X'\|_{L^\infty_s},\quad
|X|_*\geq \f{2}{\pi}|a_1|-\|\Pi X'\|_{L^{\infty}_s}.
\]
\end{itemize}

\subsection{Proof of the solution being global}
\label{sec: proof of global well-posedness}
Following the conditions of Theorem \ref{thm: global solution for small data general case}, we assume that $u_0 = u_0(x)\in L^2\cap L^p(\BR^2)$ for some $p\in (2,\infty)$ and $X_0 = X_0(s)\in C^1(\BT)$, which satisfies $\di u_0 = 0$, $|X_0|_* > 0$, and $X_0$ is parameterized in the counterclockwise direction.
In order to show that we can obtain a global solution when $(u_0,X_0)$ is small in a suitable sense, we make some preparations.

\begin{enumerate}[label = (\roman*)]
\item
Let $(u,X)$ be the maximal mild solution to \eqref{eqn: NS equation}-\eqref{eqn: initial data}, with its maximal lifespan being $T_*\in (0,+\infty]$ (see Definition \ref{def: maximal solution}).

\item
Let $R$ be defined by \eqref{eqn: Fourier representation of area}, with $X$ there replaced by $X_0 = X_0(s)$.
For the mild solution $(u,X)$, by Lemma \ref{lem: volume conservation}, the effective radius of $X(\cdot,t)$ is time-invariant, so it is always $R$.

\item
Denote
\beq
\e_0^2 := 2\|\Pi X_0'\|_{L_s^2(\BT)}^2 +\nu^{-1}\|u_0\|_{L_x^2(\BR^2)}^2.
\label{eqn: def of epsilon_0}
\eeq
By Theorem \ref{thm: energy law thm} and \eqref{eqn: projection is equivalent to excess energy}, for all $t\in [0,T_*)$,
\begin{align*}
\f12 \|\Pi X'(t)\|_{L^2_s(\BT)}^2 + \nu^{-1}\|u(t)\|_{L^2_x(\BR^2)}^2
\leq &\; \|X'(t)\|_{L^2_s(\BT)}^2 - 2\pi R^2 + \nu^{-1}\|u(t)\|_{L^2_x(\BR^2)}^2\\
\leq &\; \|X_0'\|_{L^2_s(\BT)}^2 - 2\pi R^2 + \nu^{-1}\|u_0\|_{L^2_x(\BR^2)}^2\\
\leq &\; 2\|\Pi X_0'\|_{L_s^2(\BT)}^2 +\nu^{-1}\|u_0\|_{L_x^2(\BR^2)}^2
= \e_0^2,
\end{align*}
so for any $t\in [0,T_*)$,
\beq
\|\Pi X'\|_{L^\infty_t L^2_s(\BT)} \leq \sqrt{2}\e_0,\quad
\nu^{-1}\|u\|_{L^\infty_t L^2_x(\BR^2)}
\leq \e_0 \nu^{-1/2}.
\label{eqn: smallness from energy estimate}
\eeq

\end{enumerate}

\begin{prop}
\label{prop: global solution}
Let $p$, $(u_0,X_0)$, $(u,X)$, and $R$ be given as above.
There exists $\e_*>0$, which depends on $p$, $\nu$, and $R$, such that as long as
\beq
e_0:=
\|\Pi X_0'\|_{L^\infty_s(\BT)}+ \|u_0\|_{L^2_x(\BR^2)} + \|u_0\|_{L^p_x(\BR^2)} \leq \e_*,
\label{eqn: def of e_0}
\eeq
the solution $(u,X)$ is global.

Moreover, it holds for all $t>0$ that
\[
\|u(t)\|_{L^p_x(\BR^2)} \leq CR^{1+\f2p},
\]
where $C$ depends on $p$,
\[
\|X'(t)\|_{L^\infty_s(\BT)}\leq CR,\quad
|X(t)|_* \geq cR,
\]
where $C>c>0$ are universal constants, and
\[
\|X'(t)\|_{\dot{C}^{1/p}_s(\BT)}
\leq C(p,\nu,R)\big(1+t^{-1/p}\big).
\]
Lastly,
\[
\|\Pi X'(t)\|_{L^2_s(\BT)} + \nu^{-1/2}\|u(t)\|_{L^2_x(\BR^2)}
\leq C\big(1+\nu^{-1/2}\big)e_0,
\]
where $C$ is universal.

\begin{proof}
Without loss of generality, we assume that
\beq
\|u_0\|_{L^p_x(\BR^2)}\leq R^{1+\f2p},
\label{eqn: smallness of u_0}
\eeq
and, for some $\d_0\in (0,\f1{16})$ to be determined,
\beq
\|\Pi X_0'\|_{L^\infty_s(\BT)}\leq\d_0 R.
\label{eqn: smallness of X_0}
\eeq
By \eqref{eqn: a priori estimates for X under smallness conditions} and \eqref{eqn: a priori estimates for X well stretched constant under smallness conditions},
\[
\|X_0'\|_{L^\infty_s(\BT)}\leq (1+2\d_0)R,\quad
|X_0|_*\geq \left(\f{2}{\pi} - 2\d_0\right) R.
\]
By Theorem \ref{thm: maximal solution}, there exists $T>0$ such that, the maximal lifespan of $(u,X)$ satisfies $T_*>T$, and on $[0,T]$, 
\beq
\|X'\|_{L^\infty_T L^\infty_s(\BT)}\leq 2(1+2\d_0)R,\quad
\sup_{\eta\in (0,T]} \eta^{\f1p} \|X'(\cdot, \eta)\|_{\dot{C}_s^{1/p}(\BT)}
\leq 4(1+2\d_0)R,
\label{eqn: local C 1/p bound for X}
\eeq
\beq
|X(t)|_* \geq \left(\f{1}{\pi} - \d_0\right) R,
\label{eqn: local bound for X well stretched constant}
\eeq
and
\beq
\|u\|_{L^\infty_T L^p_x(\BR^2)} +
\sup_{\eta\in (0,T]} (\nu\eta)^{\f1p}\|u(\cdot,\eta)\|_{L^\infty_x(\BR^2)}
\leq CR^{1+\f2p},
\label{eqn: local L inf bound for u}
\eeq
where $C$ only depends on $p$.

We claim that, if $\d_0$ in \eqref{eqn: smallness of X_0} is taken to be sufficiently small, which depends on $p$, $\nu$, and $R$, the above $T$ can be made to only depend on $p$, $\nu$, and $R$ as well (but not implicitly on $X_0$).
Indeed, with $\g:= \max(\f12+\f1p,1-\f1p)$, we define as in \eqref{eqn: def of rho} and \eqref{eqn: def of phi} that
\[
\r_{X_0,\g}(t): = \sup_{\tau\in (0,t]}\tau^{\g}\big\|e^{-\f{\tau}{4} \Lam}X_0'\big\|_{\dot{C}^\g_s(\BT)},
\quad
\phi_{X_0}(t):= \sup_{\tau\in [0,t]}\big\|X_0-e^{-\f{\tau}{4}\Lam}X_0\big\|_{\dot{C}_s^1(\BT)}.
\]
Using \eqref{eqn: def of projection operator Pi}, \eqref{eqn: a priori estimates for X under smallness conditions}, \eqref{eqn: smallness of X_0}, and Lemma \ref{lem: parabolic estimates for fractional Laplace}, we derive that
\beq
\r_{X_0,\g}(t)
\leq
\sup_{\tau\in (0,t]}\tau^{\g}\big\|e^{-\f{\tau}{4} \Lam}{X_0^*}'\big\|_{\dot{C}^\g_s}
+ \sup_{\tau\in (0,t]}\tau^{\g}\big\|e^{-\f{\tau}{4} \Lam}\Pi X_0'\big\|_{\dot{C}^\g_s}
\leq C\big(t^\g R + \d_0 R\big),
\label{eqn: bound for rho}
\eeq
where $C$ only depends on $p$, and
\beq
\phi_{X_0}(t)
\leq
\sup_{\tau\in [0,t]}\big\|X_0^*-e^{-\f{\tau}{4}\Lam}X_0^*\big\|_{\dot{C}_s^1}
+ \big\|\Pi X_0-e^{-\f{\tau}{4}\Lam}\Pi X_0\big\|_{\dot{C}_s^1}
\leq C\big(tR + \d_0 R\big),
\label{eqn: bound for phi}
\eeq
where $C$ is a universal constant.
By virtue of Remark \ref{rmk: characterization of lifespan T}, \eqref{eqn: a priori estimates for X under smallness conditions}, \eqref{eqn: a priori estimates for X well stretched constant under smallness conditions}, and \eqref{eqn: smallness of X_0}, $T$ only needs to be sufficiently small, so that
\beq
\begin{split}
&\mbox{$(\r_{X_0,\g}(T)+T^\s)$ and $\phi_{X_0}(T)$ are smaller than}\\
&\mbox{some constants that depend on $p$, $\nu$, and $R$,}
\end{split}
\label{eqn: smallness condition on T for X_0 close to equilibrium}
\eeq
where $\s:=\f12(\f12-\f1p)$.
In view of \eqref{eqn: bound for rho} and \eqref{eqn: bound for phi}, we can find a sufficiently small $\d_0\in (0,\f1{16})$, and then choose $T\in(0,1)$ to be small as well, both of which only depend on $p$, $\nu$, and $R$, so that under the assumption \eqref{eqn: smallness of X_0}, the smallness condition \eqref{eqn: smallness condition on T for X_0 close to equilibrium} is satisfied.
We shall fix this choice of $\d_0$ and $T$ in the rest of the proof.

By the interpolation inequality, \eqref{eqn: boundedness of projection Pi}, \eqref{eqn: smallness from energy estimate}, and \eqref{eqn: local C 1/p bound for X},
\[
\|\Pi X'(T/2)\|_{L^\infty_s(\BT)}
\leq C \|\Pi X'(T/2)\|_{L_s^2}^{\f{2}{2+p}} \|X'(T/2)\|_{\dot{C}_s^{1/p}}^{\f{p}{2+p}}
\leq C \e_0^{\f{2}{2+p}}\big(T^{-\f1p}R\big)^{\f{p}{2+p}},
\]
where the constant $C$ only depends on $p$.
Similarly, by the interpolation inequality, \eqref{eqn: smallness from energy estimate}, and \eqref{eqn: local L inf bound for u},
\[
\|u(T/2)\|_{L^p_x(\BR^2)}
\leq \|u(T/2)\|_{L^2_x}^{\f2p}\|u(T/2)\|_{L^\infty_x}^{1-\f2p}
\leq
C\big(\e_0 \nu^{\f12}\big)^{\f2p} \left[ (\nu T)^{-\f1p} R^{1+\f2p}\right]^{1-\f2p} ,
\]
where $C$ only depends on $p$ as well.
Since $\d_0$ and $T$ only depend on $p$, $\nu$, and $R$, by assuming $\e_0$ to be sufficiently small, which only depends on $p$, $\nu$, and $R$ as well, we can make
\[
\|u(T/2)\|_{L^p_x(\BR^2)}
\leq R^{1+\f2p},
\quad
\|\Pi X'(T/2)\|_{L^\infty_s(\BT)}
\leq \d_0 R.
\]

Now we view $(u(T/2),X(T/2))$ as the new initial data, and apply Theorem \ref{thm: maximal solution} to obtain its maximal local solution, which we denote by $(\tilde{u},\tilde{X})$.
Since $(u(T/2),X(T/2))$ enjoys the same bound as $(u_0,X_0)$, by the same argument as above, we find that for the same $T>0$, $(\tilde{u},\tilde{X})$ is well-defined on $[0,T]$ and enjoys the same estimates as in \eqref{eqn: local C 1/p bound for X}-\eqref{eqn: local L inf bound for u}.
By uniqueness, $(\tilde{u},\tilde{X})$ must coincide with the restriction of $(u,X)$ on $[T/2,T_*)$ (up to a time shift by $T$).
Hence, $T_*>\f32 T$.
One may repeat this infinitely to obtain $T_*=+\infty$, i.e., $(u,X)$ is a global solution.

We note that the smallness conditions \eqref{eqn: smallness of u_0} and \eqref{eqn: smallness of X_0} as well as the smallness of $\e_0$ (see \eqref{eqn: def of epsilon_0}) can be achieved by requiring $\e_*$ in \eqref{eqn: def of e_0} to be sufficiently small, which only depends on $p$, $\nu$, and $R$.

Finally, the claimed estimates follow from the fact that \eqref{eqn: smallness from energy estimate} and \eqref{eqn: local C 1/p bound for X}-\eqref{eqn: local L inf bound for u} hold for $(u,X)$ on each time interval $[\f12(k-1)T,\f12(k+1)T]$ for any $k\in \BZ_+$,
\end{proof}
\end{prop}

\appendix

\section{Parabolic Estimates and Calculus Results}
\label{sec: parabolic estimates}

We first recall some standard parabolic estimates for the heat equation in $\BR^2$, whose proof will be omitted.

\begin{lem}
\label{lem: parabolic estimates}
Recall that, with $\mu\in (0,1)$, the Morrey norm $M^{1,\mu}(\BR^2)$ was defined in \eqref{eqn: def of Morrey norm}.
Let $\{e^{t\D}\}_{t\geq 0}$ be the standard heat semigroup on $\BR^2$.
Let $f = f(x)$ be defined on $\BR^2$.
Assume $\g\in (0,1)$.
\begin{enumerate}
\item
For any $t>0$,
\begin{align*}
\|e^{t\Delta}f\|_{L^{\infty}(\R^2)}\leq &\; Ct^{-\mu/2}\|f\|_{M^{1,\mu}(\R^2)},
\\
\|\na^{3}(-\Delta)^{-1}e^{t\Delta}f\|_{M^{1,\mu}(\R^2)}\leq &\; Ct^{-1/2}\|f\|_{M^{1,\mu}(\R^2)},
\end{align*}
where $C$ depends on $\mu$.
\item
For any $\b>0$ and $t>0$,
\[
\|\na^{2}(-\Delta)^{-1}e^{t\Delta}f\|_{\dot{C}^{\beta}(\R^2)}\leq Ct^{-(\mu+\beta)/2}\|f\|_{M^{1,\mu}(\R^2)},
\]
where $C$ depends on $\beta$ and $\mu$.

\item
For any $\g\in(0,1)$, $\b>\g$, and any $t>0$,
\begin{align*}
\|\na^{2}(-\Delta)^{-1}e^{t\Delta}f\|_{\dot{C}^{\beta}(\R^2)}
\leq &\; Ct^{-\beta/2}\|f\|_{L^\infty(\R^2)},
\\
\|\na^{2}(-\Delta)^{-1}e^{t\Delta}f\|_{\dot{C}^{\beta}(\R^2)}
\leq &\; Ct^{-(\beta-\g)/2}\|f\|_{\dot{C}^\g(\R^2)},
\\
\|e^{t\Delta}f-f\|_{C(\R^2)}\leq &\; Ct^{\g/2}\|f\|_{\dot{C}^\g(\R^2)},
\end{align*}
where $C$ depends on $\beta$ and $\gamma$.

\item
For any $\beta\geq 0$, $p\in[1,\infty]$, and any $t>0$,
\[
\|e^{t\D}f\|_{\dot{C}^{\beta}(\R^2)}
\leq Ct^{-\b/2-1/p}\|f\|_{L^p(\R^2)},
\]
where $C$ depends on $\beta$ and $p$.
\end{enumerate}
\end{lem}

There are similar estimates regarding the fractional Laplacian $\Lam$ on $\BT$ and the associated fractional heat semigroup.
\begin{lem}
\label{lem: parabolic estimates for fractional Laplace}
Recall $\Lam := (-\D)^{1/2}_\BT$.
Let $\{e^{-t\Lam}\}_{t\geq 0}$ be the semigroup associated with the fractional heat operator $\pa_t + \Lam$ on $\BT$.
Let $f = f(s)$ be defined on $\BT$.
Assume $\g\in (0,1)$.
Then
\begin{enumerate}
\item $\|\Lam f\|_{{C}^{\gamma}(\BT)}
\leq C\|f\|_{\dot{C}^{1,\gamma}(\BT)}$,
where $C$ depends on $\g$.
\item For any $t>0$,
\[
\|e^{-t\Lam}f\|_{L^{\infty}(\BT)}
\leq \|f\|_{L^\infty(\BT)}.
\]
For any $\beta\geq 0$ and $t>0$,
\[
\|e^{-t\Lam}f\|_{\dot{C}^{\beta}(\BT)}
\leq Ct^{-\beta}\|f\|_{L^\infty(\BT)},
\]
where $C$ depends on $\b$.

\item
For $0\leq\alpha\leq\beta$, and $t>0$,
\[
\| e^{-t\Lam}f\|_{\dot{C}^{\beta}(\BT)}
\leq Ct^{\alpha-\beta}\|f\|_{\dot{C}^{\alpha}(\BT)}.
\]
For $0\leq \alpha<\beta<\alpha+1$, and $t>0$,
\[
\| e^{-t\Lam}f-f\|_{\dot{C}^{\alpha}(\BT)}
\leq Ct^{\beta-\alpha}\|f\|_{\dot{C}^{\beta}(\BT)}.
\]
Here the constant $C$ depends on $\al$ and $\b$.
\end{enumerate}
\end{lem}

We also have the following improved estimates. 
\begin{lem}
\label{lem: parabolic estimates improved for fractional Laplace}
Let $f =f (s)$ be defined on $\BT$.
For any $\al\in [0,1]$ and $t>0$,
\[
\big\|e^{-t\Lam }f\big\|_{\dot{C}^{1,\al}(\BT)}
\leq \f{2}{e^t+1} \|f\|_{\dot{C}^{1,\al}(\BT)}.
\]

\begin{proof}
Note that
\[
e^{-t\Lam }f(s): = \int_{\BT}P_t(s-s') f(s')\,ds',
\]
where
\[
P_t(s): = \f{1}{2\pi} \cdot \f{\sinh t}{\cosh t-\cos s}
\geq \f{1}{2\pi} \cdot \f{\sinh t}{1+\cosh t}.
\]
Hence,
\begin{align*}
\big\|e^{-t\Lam}f\big\|_{\dot{C}^1(\BT)}
= &\; \left\| \int_{\BT} \left[P_t(s')-\f{1}{2\pi} \cdot \f{\sinh t}{1+\cosh t}\right] f'(s-s')\,ds'\right\|_{L^\infty(\BT)}\\
\leq &\; \f{1}{2\pi} \int_{\BT}\f{\sinh t}{\cosh t-\cos s'}-\f{\sinh t}{1+\cosh t} \,ds' \cdot \|f'\|_{L^\infty}\\
= &\; \left[1-\f{\sinh t}{1+\cosh t}\right] \|f'\|_{L^\infty}
= \f{2}{e^t+1} \|f'\|_{L^\infty}.
\end{align*}
Similarly, with $\al\in (0,1]$, for distinct $s_1,s_2\in \BT$,
\begin{align*}
&\;\left|e^{-t\Lam} f'(s_1)-e^{-t\Lam} f'(s_2)\right|\\
\leq &\; \left| \int_{\BT} \left[P_t(s')-\f{1}{2\pi} \cdot \f{\sinh t}{1+\cosh t}\right] \big(f'(s_1-s')-f'(s_2-s')\big) \,ds'\right|\\
\leq &\; \f{1}{2\pi} \int_{\BT}\f{\sinh t}{\cosh t-\cos s'}-\f{\sinh t}{1+\cosh t} \,ds' \cdot \|f'\|_{\dot{C}^\al}|s_1-s_2|_\BT^\al \\
\leq &\; \f{2}{e^t+1}\|f'\|_{\dot{C}^\al}|s_1-s_2|_\BT^\al,
\end{align*}
which implies the desired inequality.
\end{proof}
\end{lem}

The next lemma bounds the time-integral given by the operator $\CI$.

\begin{lem}
\label{lem: parabolic estimates for Duhamel term of fractional Laplace}
Recall the operator $\CI$ was defined in \eqref{eqn: def of the operator I}, i.e., given $W = W(s,t)$ on $\BT\times [0,T]$, define $\CI[W]$ on $\BT\times [0,T]$ by
\[
\CI[W](t):=\int_0^t e^{-\f14(t-\tau)\Lam}W(\tau)\,d\tau.
\]
\begin{enumerate}
\item
Assume that $\va:(0,+\infty)\to (0,+\infty)$ is decreasing on $(0,+\infty)$ and it is integrable on any bounded interval.
Then for $\g\in (0,1)$ and any $t\in (0,T]$,
\[
\|\CI[W](t)\|_{\dot{C}_s^{1,\g}(\BT)}
\leq \f{C}{t}\int_0^t \va(\tau)\,d\tau \cdot
\sup_{\tau\in(0,T]} \va(\tau)^{-1}\|W(\tau)\|_{\dot{C}_s^{\g}(\T)},
\]
where $C$ only depends on $\g$.

\item
For $\al\in [0,1)$, $\b\in (\al,1+\al]$, $\mu\in [0,1)$, and $t\in (0,T]$,
\[
\|\CI[W](t)\|_{\dot{C}_s^{1,\al}(\BT)}
\leq Ct^{\b-\al-\mu} \min\big(1,t^{\al-\b}\big) \sup_{\tau\in (0,t]}\tau^{\mu} \|W(\tau)\|_{\dot{C}_s^\b(\BT)},
\]
where $C$ depends on $\al$, $\b$ and $\mu$.

As a corollary of the above two estimates, for $0<\g\leq \b<1$, it holds for any $t\in (0,T]$ that
\[
\|\CI[W](t)\|_{\dot{C}_s^{1}(\BT)}
+t^\g\|\CI[W](t)\|_{\dot{C}_s^{1,\g}(\BT)}
\leq C\sup_{\tau\in (0,t]}\tau^\b \|W(\tau)\|_{\dot{C}_s^\b(\BT)},
\]
where $C$ depends on $\b$ and $\g$.

\item Assume $t_1\in (0,T)$.
Suppose that, for some $\d>0$ and $\g\in (0,1)$,
\[
t\mapsto \|\CI[W](t)\|_{C^{1,\g}(\BT)}\mbox{ and }t\mapsto \|W(t)\|_{C^\g(\BT)}\mbox{ are bounded in }[t_1-\d,t_1+\d],
\]
and in addition,
\[
\lim_{t\to t_1}\|W(t)-W(t_1)\|_{C(\BT)} = 0.
\]
Then $\pa_t \CI [W](t_1)$ is well-defined pointwise, and it holds pointwise that
\[
\pa_t \CI[W](t_1) = -\f14\Lam \CI[W](t_1) + W(t_1).
\]
\end{enumerate}

\begin{proof}
By Lemma \ref{lem: interpolation} and Lemma \ref{lem: parabolic estimates for fractional Laplace},
\begin{align*}
\|\CI[W](t)\|_{\dot{C}_s^{1,\g}(\BT)}
\leq &\; \f{C}{t}\int_0^t \va(\tau)\,d\tau \left|\sup_{\tau\in (0,T)}(t-\tau)^{1-\g} \va(\tau)^{-1} \|\pa_s e^{-\f14(t-\tau)\Lam} W(\tau)\|_{L_s^{\infty}}\right|^{1-\g}\\
&\;\cdot \left|\sup_{\tau\in (0,T)}(t-\tau)^{2-\g} \va(\tau)^{-1} \|\pa_s^2 e^{-\f14(t-\tau)\Lam}W(\tau)\|_{L_s^{\infty}}\right|^\g
\\
\leq &\; \f{C}{t}\int_0^t \va(\tau)\,d\tau
\left|\sup_{\tau\in (0,T)} (t-\tau)^{1-\g} \cdot (t-\tau)^{-(1-\g)}\va(\tau)^{-1}
\|W(\tau)\|_{\dot{C}_s^\g}\right|^{1-\g}\\
&\; \cdot \left|\sup_{\tau\in (0,T)} (t-\tau)^{2-\g}\cdot (t-\tau)^{-(2-\g)}\va(\tau)^{-1}
\|W(\tau)\|_{\dot{C}_s^\g}\right|^\g\\
\leq &\; \f{C}{t}\int_0^t \va(\tau)\,d\tau \cdot
\sup_{\tau\in(0,T]} \va(\tau)^{-1}\|W(\tau)\|_{\dot{C}_s^{\g}(\T)}.
\end{align*}
This proves the first inequality.

Regarding the second claim, we apply Lemma \ref{lem: parabolic estimates for fractional Laplace} and Lemma \ref{lem: parabolic estimates improved for fractional Laplace} to find that, for $\al \in [0,1)$, $\b\in (\al,1+\al]$, and $\mu\in [0,1)$,
\begin{align*}
\|\CI[W](t)\|_{\dot{C}_s^{1,\al}(\BT)}
\leq &\; \int_0^t \big\|e^{-\f14(t-\tau)\Lam}W(\tau)\big\|_{\dot{C}_s^{1,\al}}\,d\tau\\
\leq
&\; C\int_0^t e^{-\f18(t-\tau)} (t-\tau)^{-(1+\al-\b)} \tau^{-\mu}\cdot \tau^{\mu} \|W(\tau)\|_{\dot{C}_s^\b}\,d\tau\\
\leq &\; C\left[e^{-\f{t}{16}}t^{-(1+\al-\b)+1-\mu} + \min\big(1,t^{-\al+\b}\big) t^{-\mu}\right] \sup_{\tau\in (0,t]}\tau^{\mu} \|W(\tau)\|_{\dot{C}_s^\b(\BT)}.
\end{align*}
This proves the desired estimate.

Finally, to prove the last claim, we need to justify that
\[
\lim_{t_2\to t_1}\f{\CI[W](t_2)-\CI[W](t_1)}{t_2-t_1} = -\f14\Lam \CI[W](t_1) + W(t_1)\mbox{ in }C(\BT).
\]

By definition, for any $t_1<t_2$,
\begin{align*}
&\; \CI[W](t_2)-\CI[W](t_1)\\
= &\; \int_{t_1}^{t_2} e^{-\f14(t_2-\tau)\Lam}W(\tau)\,d\tau + \int_0^{t_1} \big(e^{-\f14(t_2-t_1)\Lam} -\mathrm{Id}\big) e^{-\f14(t_1-\tau)\Lam}W(\tau)\,d\tau
\\
= &\; \int_{t_1}^{t_2} e^{-\f14(t_2-\tau)\Lam}W(t_1)\,d\tau
+ \int_{t_1}^{t_2} e^{-\f14(t_2-\tau)\Lam} \big[W(\tau)-W(t_1)\big]\,d\tau
+ \big(e^{-\f14(t_2-t_1)\Lam} -\mathrm{Id}\big) \CI[W](t_1).
\end{align*}
Since $W(t_1),\Lam\CI[W](t_1)\in C(\BT)$, one can readily show that, in the $C(\BT)$-topology,
\begin{align*}
&\lim_{t_2\to t_1^+}\f{1}{t_2-t_1}\int_{t_1}^{t_2} e^{-\f14(t_2-\tau)\Lam}W(t_1)\,d\tau = W(t_1),
\end{align*}
and
\[
\lim_{t_2\to t_1^+}\f{1}{t_2-t_1}\big(e^{-\f14(t_2-t_1)\Lam} -\mathrm{Id}\big) \CI[W](t_1)
=
\lim_{\d\to 0^+}- \f{1}{4\d}\int_{0}^{\d} e^{-\f{\tau}{4}\Lam} \Lam \CI[W](t_1)\,d\tau = -\f14\Lam \CI[W](t_1).
\]
Moreover,
\begin{align*}
&\; \left\|\f{1}{t_2-t_1}\int_{t_1}^{t_2} e^{-\f14(t_2-\tau)\Lam}\big[W(\tau)-W(t_1)\big]\,d\tau\right\|_{C(\BT)}\\
\leq &\; \f{1}{t_2-t_1}\int_{t_1}^{t_2} \|W(\tau)-W(t_1)\|_{C(\BT)} \,d\tau
\leq \sup_{\tau\in [t_1,t_2]} \|W(\tau)-W(t_1)\|_{C(\BT)}.
\end{align*}
Thanks to the time-continuity of $W$,
\begin{align*}
&\; \lim_{t_2\to t_1^+}\f{1}{t_2-t_1}\int_{t_1}^{t_2} e^{-\f14(t_2-\tau)\Lam}\big[W(\tau)-W(t_1)\big]\,d\tau = 0\mbox{ in }C(\BT).
\end{align*}
Therefore, it holds in the $C(\BT)$-topology that
\[
\lim_{t_2\to t_1^+}\f{\CI[W](t_2)-\CI[W](t_1)}{t_2-t_1} = -\f14\Lam \CI[W](t_1) + W(t_1).
\]

On the other hand, for $t_1>t_2$, one similarly derives that
\begin{align*}
&\; \CI[W](t_1)-\CI[W](t_2)\\
= &\; \int_{t_2}^{t_1} e^{-\f14(t_1-\tau)\Lam}W(\tau)\,d\tau + \int_0^{t_2} \big(e^{-\f14(t_1-t_2)\Lam} -\mathrm{Id}\big) e^{-\f14(t_2-\tau)\Lam}W(\tau)\,d\tau
\\
= &\; \int_{t_2}^{t_1} e^{-\f14(t_1-\tau)\Lam}W(t_1)\,d\tau
+ \big(e^{-\f14(t_1-t_2)\Lam} -\mathrm{Id}\big) \CI[W](t_1)
\\
&\;+ \int_{t_2}^{t_1} e^{-\f14(t_1-\tau)\Lam}\big[W(\tau)-W(t_1)\big]\,d\tau
+ \big(e^{-\f14(t_1-t_2)\Lam} -\mathrm{Id}\big)\big[ \CI[W](t_2)-\CI[W](t_1)\big].
\end{align*}
The first three terms can be handled as before; we omit the details.
It remains to show that
\[
\lim_{t_2\to t_1^-}\f{1}{t_1-t_2}\big(e^{-\f14(t_1-t_2)\Lam} -\mathrm{Id}\big)\big[ \CI[W](t_2)-\CI[W](t_1)\big] = 0\mbox{ in }C(\BT).
\]
To that end, we first claim that $t\mapsto \Lam\CI[W](t)$ is continuous in $C(\BT)$ at $t = t_1$.
Indeed, for $\b\in (0,\g)$, and $t_1>t_2$ with $|t_1-t_2|<\d\ll 1$,
\begin{align*}
&\;\big\| \Lam\CI[W](t_1)-\Lam\CI[W](t_2)\big\|_{C(\BT)}\\
\leq &\; C\| \CI[W](t_1)-\CI[W](t_2)\|_{\dot{C}^{1,\b}(\BT)}\\
\leq &\; C\int_{t_2}^{t_1} \|e^{-\f14(t_1-\tau)\Lam}W(\tau)\|_{\dot{C}^{1,\b}(\BT)}\,d\tau
+ C\big\|\big(e^{-\f14(t_1-t_2)\Lam} -\mathrm{Id}\big) \CI[W](t_2)\big\|_{\dot{C}^{1,\b}(\BT)}
\\
\leq &\; C\int_{t_2}^{t_1} (t_1-\tau)^{-(1+\b-\g)} \|W(\tau)\|_{\dot{C}^\g(\BT)}\,d\tau
+ C|t_1-t_2|^{\g-\b} \|\CI[W](t_2)\|_{\dot{C}^{1,\g}(\BT)}
\\
\leq &\;C|t_1-t_2|^{\g-\b} \left[ \sup_{\tau\in [t_1-\d,t_1]}\|W(\tau)\|_{\dot{C}^\g(\BT)}
+ \|\CI[W](t_2)\|_{\dot{C}^{1,\g}(\BT)}\right],
\end{align*}
so the desired time-continuity follows from the assumptions.
Then we derive that
\begin{align*}
&\;\left\|\f{1}{t_1-t_2}\big(e^{-\f14(t_1-t_2)\Lam} -\mathrm{Id}\big)\big[ \CI[W](t_2)-\CI[W](t_1)\big]\right\|_{C(\BT)}\\
= &\;
\left\|\f{1}{4(t_1-t_2)}\int_0^{t_1-t_2} \Lam e^{-\f{\tau}{4} \Lam} \big[ \CI[W](t_2)-\CI[W](t_1)\big]\,d\tau \right\|_{C(\BT)}\\
\leq &\;
\f{1}{4(t_1-t_2)}\int_0^{t_1-t_2} \big\|\Lam \big[ \CI[W](t_2)-\CI[W](t_1)\big]\big\|_{C(\BT)} \,d\tau\\
= &\;
\f14 \big\|\Lam \CI[W](t_2)- \Lam \CI[W](t_1) \big\|_{C(\BT)},
\end{align*}
which converges to $0$ as $t_2\to t_1^-$.
Therefore, we can conclude that
\[
\lim_{t_2\to t_1^-}\f{\CI[W](t_1)-\CI[W](t_2)}{t_1-t_2} = -\f14\Lam \CI[W](t_1) + W(t_1)\mbox{ in }C(\BT),
\]

This completes the proof.
\end{proof}
\end{lem}


Lastly, we collect some useful auxiliary calculus results in the following lemma.

\begin{lem}
\label{lem: Holder estimate for composition of functions}
Suppose  $\g\in (0,1)$.
Let $f = f(x)$ be defined on $\BR^2$, and $X,X_1,X_2:\BT\to \BR^2$.
Then
\begin{align*}
\|f\circ X\|_{\dot{C}^\g_s(\BT)}\leq &\; \|f\|_{\dot{C}^\g_x(\BR^2)}\|X'\|_{L^\infty_s(\BT)}^\g,\\
\|f\circ X\|_{\dot{C}^\g_s(\BT)}\leq &\;
\|\na f\|_{L^\infty_x(\BR^2)}\|X\|_{\dot{C}^\g_s(\BT)},\\
\|f\circ X_1 - f\circ X_2\|_{L^\infty_s(\BT)}
\leq &\; \|\na f\|_{L^\infty_x(\BR^2)}\|X_1-X_2\|_{L^\infty_s(\BT)},
\end{align*}
and
\begin{align*}
&\;\|f\circ X_1 - f\circ X_2\|_{\dot{C}^\g_s(\BT)}\\
\leq &\;
\|f\|_{\dot{C}^1_x(\BR^2)} \|X_1-X_2\|_{\dot{C}_s^\g(\BT)}
+ C\|f\|_{\dot{C}_x^{1,\g}(\BR^2)} \|(X_1',X_2')\|_{L^\infty_s(\BT)}^\g \|X_1-X_2\|_{L^\infty_s(\BT)},
\end{align*}
where $C$ only depends on $\g$.

\begin{proof}
We only show the last estimate since the others are straightforward to verify.

For arbitrary $s_1,s_2\in \BT$,
\begin{align*}
&\;|f\circ X_1(s_1) - f\circ X_2(s_1) - f\circ X_1(s_2) + f\circ X_2(s_2)|\\
= &\; \bigg|\int_0^1 (X_1(s_1)-X_2(s_1))\cdot \na f\big(\th X_1(s_1) + (1-\th)X_2(s_1)\big) \\
&\; \quad
- (X_1(s_2)-X_2(s_2))\cdot \na f\big(\th X_1(s_2) + (1-\th)X_2(s_2)\big)
\,d\th \bigg|\\
\leq &\; \int_0^1 |X_1(s_1)-X_2(s_1)
- X_1(s_2)+ X_2(s_2)|\cdot \big|\na f\big(\th X_1(s_1) + (1-\th)X_2(s_1)\big)\big|\\
&\; \quad
+ |X_1(s_2)-X_2(s_2)|\cdot \big|\na f\big(\th X_1(s_1) + (1-\th)X_2(s_1)\big)- \na f\big(\th X_1(s_2) + (1-\th)X_2(s_2)\big)\big|
\,d\th\\
\leq &\; \int_0^1 \|X_1-X_2\|_{\dot{C}_s^\g} |s_1-s_2|^\g \|\na f\|_{L^\infty_x}\\
&\; \quad
+ \|X_1-X_2\|_{L^\infty_x} \|\na f\|_{\dot{C}_x^\g} \big|\th X_1(s_1)  - \th X_1(s_2) + (1-\th)X_2(s_1) - (1-\th)X_2(s_2)\big|^\g \,d\th\\
\leq &\;  \|X_1-X_2\|_{\dot{C}_s^\g} |s_1-s_2|^\g \|\na f\|_{L^\infty_x}
+ C \|X_1-X_2\|_{L^\infty_s}  \|\na f\|_{\dot{C}_x^\g} \|(X_1',X_2')\|_{L^\infty_s}^\g |s_1-s_2|^\g.
\end{align*}
This implies the desired bound.
\end{proof}
\end{lem}

\section{
The Improved Estimates for 
$g_X$}
\label{sec: improved estimates for g_X}

In this section, we provide self-consistent justifications of Lemma \ref{lem: improved estimates for g_X} and Lemma \ref{lem: improved estimates for g_X-g_Y}, which are improved estimates for the nonlinear terms $g_X$ and $g_X-g_Y$ arising in the Stokes case (see \eqref{eqn: def of g_X}).
Since their representations are purely Lagrangian, in what follows, we shall omit the subscript $s$ in all the norm notations.

\subsection{Proof of Lemma \ref{lem: improved estimates for g_X}}
\label{sec: proof of improved estimates for g_X}

For $s \neq s'$, denote
\beq
L(s,s') := \f{X(s')-X(s)}{s'-s}.
\label{eqn: def of L}
\eeq
We understand $s'-s$ as the usual subtraction on $\BR$.
When $s'\in [s-\f32\pi,s+\f32\pi]$, it is not difficult to show by the regularity and the well-stretched property of $X$ that
\beq
C\lam \leq |L(s,s')|\leq \|X'\|_{L^\infty},\quad
\lam \leq \inf_{s'\in\BT} |X'(s')| \leq \|X'\|_{L^\infty},
\label{eqn: bounds for L(s,s')}
\eeq
and
\beq
|L(s,s')-X'(s')|
\leq \left|\f{1}{s'-s}\int_{s}^{s'} X'(\tau)-X'(s')\,d\tau\right|
\leq C\|X'\|_{\dot{C}^\g} |s-s'|^\g.
\label{eqn: bounds for L(s,s')-X'(s')}
\eeq
Note that here we choose not to define the subtraction $s'-s$ in the modulo $2\pi$ (i.e., making it range in $[-\pi,\pi)$) in order to allow some flexibility later in the H\"{o}lder estimate.

By the definition \eqref{eqn: def of g_X},
\beq
(g_X)_i(s)
= \int_\BT X_k'(s')\pa_k G_{ij}(X(s)-X(s')) X_j'(s')
+ \f{1}{4\pi}\cdot \f{X_i'(s')}{2\tan\f{s-s'}{2}}\,ds'.
\label{eqn: formula for g_X}
\eeq
From \eqref{eqn: Stokeslet in 2-D}, we can explicitly calculate that, for $x\neq 0$,
\[
\pa_k G_{ij}(x) = \f{1}{4\pi}\left(-\f{\d_{ij}x_k}{|x|^2} +\f{\d_{ik}x_j + x_i\d_{jk}}{|x|^2} -\f{2x_ix_jx_k}{|x|^4}\right),
\]
which is homogeneous of degree $-1$, and thus
\beq
\pa_k G_{ij}(x) x_k
= -\f{1}{4\pi}\d_{ij}.
\label{eqn: an identity for grad G}
\eeq
Hence,
\beq
\begin{split}
(g_X)_i(s)
= &\; \int_\BT \f{1}{s-s'}\cdot \pa_k G_{ij}\left(\f{X(s)-X(s')}{s-s'}\right) X_k'(s')X_j'(s') + \f{1}{4\pi}\cdot \f{X_i'(s')}{2\tan\f{s-s'}{2}}\,ds'\\
= &\; \int_\BT \f{1}{s-s'}\left[\pa_k G_{ij}\big(L(s,s')\big)- \pa_k G_{ij}(X'(s'))\right] X_k'(s')X_j'(s')\,ds' \\
&\; -\f{1}{4\pi} \int_\BT \left(\f{1}{s-s'}- \f{1}{2\tan\f{s-s'}{2}}\right) X_i'(s') \,ds'.
\end{split}
\label{eqn: new form of g_X}
\eeq
Here we may take the range of $s'$ to be any interval of length $2\pi$ that contains, e.g., $(s-\f{\pi}{2},s+\f{\pi}{2})$, and the value of the right-hand side does not depend on the choice of the interval.

One can readily show that, for $k\in \BZ_+$,
\beq
|\na^k G(x)|\leq C|x|^{-k},\quad
|\na^k G(x) - \na^k G(y)|\leq C\min(|x|,|y|)^{-(k+1)}|x-y|.
\label{eqn: bound for G}
\eeq
Using this as well as \eqref{eqn: bounds for L(s,s')} and \eqref{eqn: bounds for L(s,s')-X'(s')},
\begin{align*}
|g_X(s)|
\leq &\; C\int_\BT \f{1}{|s-s'|}\cdot \lam^{-2}|L(s,s')-X'(s')| \cdot \|X'\|_{L^\infty}^2 \,ds'
+ C\int_\BT |X'(s')| \,ds'\\
\leq &\; C \int_\BT \f{1}{|s-s'|}\cdot \lam^{-2} \|X'\|_{\dot{C}^\g}|s-s'|^\g \cdot \|X'\|_{L^\infty}^2 \,ds'
+ C\|X'\|_{L^\infty}\\
\leq &\; C\lam^{-2}\|X'\|_{L^\infty}^2 \|X'\|_{\dot{C}^\g}.
\end{align*}
In the last line, we used the fact that $\lam \leq \|X'\|_{L^\infty}$.
This proves the $L^\infty$-bound.

Next we turn to the H\"{o}lder estimate for $g_X$.
We first rewrite \eqref{eqn: formula for g_X} as (cf.\;\eqref{eqn: expression of u_11 alternative general c} and \eqref{eqn: def of g_X})
\[
(g_X)_i(s) = \int_\BT X_k'(s')\pa_k G_{ij}\big(X(s)-X(s')\big)
\big(X_j'(s')-X_j'(s)\big)
+ \f{1}{4\pi}\cdot \f{X_i'(s')-X_i'(s)}{2\tan\f{s-s'}{2}}\,ds'.
\]
Note that the added terms actually vanish.
Using \eqref{eqn: an identity for grad G} and deriving as before, we obtain that
\begin{align*}
(g_X)_i(s)
= &\; \int_\BT \f{1}{s-s'} \left[\pa_k G_{ij}\big(L(s,s')\big)-\pa_k G_{ij}(X'(s'))\right]
X_k'(s') \big(X_j'(s')-X_j'(s)\big)\,ds'\\
&\; -\f{1}{4\pi} \int_\BT \left(\f{1}{s-s'}- \f{1}{2\tan\f{s-s'}{2}}\right) \big(X_i'(s')-X_i'(s)\big) \,ds'.
\end{align*}
Again, we may interpret $\BT$ here to be an arbitrary interval on $\BR$ of length $2\pi$ that contains, say, $(s-\f{\pi}{2},s+\f{\pi}{2})$.
In this way, $s-s'$ can be understood as the regular subtraction on $\BR$, while the value of the integral does not depend on the choice of the interval.

In the case $\g\in(0,\f12)$, we would like to bound $g_X(s_1)-g_X(s_2)$ for distinct $s_1,s_2\in \BT$.
Denote $r := |s_1-s_2|_\BT$.
Assume $s_2 = s_1+r$.
Without loss of generality, we also assume $r\in (0,\f{\pi}{4})$, as otherwise, the inequality $|g_X(s_1)-g_X(s_2)|\leq 2\|g_X\|_{L^\infty}$ is enough for deriving the H\"{o}lder estimate.

Let $I : = (s_1-r,s_2+r)$.
We split $g_X(s_1)-g_X(s_2)$ as
\beq
\begin{split}
&\;(g_X)_i(s_1)-(g_X)_i(s_2)\\
= &\; \int_I \f{1}{s_1-s'} \left[\pa_k G_{ij}\big(L(s_1,s')\big)-\pa_k G_{ij}(X'(s'))\right]
X_k'(s') \big(X_j'(s')-X_j'(s_1)\big)\,ds'\\
&\; -\int_I \f{1}{s_2-s'} \left[\pa_k G_{ij}\big(L(s_2,s')\big)-\pa_k G_{ij}(X'(s'))\right]
X_k'(s') \big(X_j'(s')-X_j'(s_2)\big)\,ds'\\
&\; + \int_{I^c} \left(\f{1}{s_1-s'}-\f{1}{s_2-s'}\right)\left[\pa_k G_{ij}\big(L(s_1,s')\big)- \pa_k G_{ij}(X'(s'))\right] \\ &\;\qquad \quad \cdot X_k'(s')\big(X_j'(s')-X_j'(s_1)\big)\,ds' \\
&\; +\int_{I^c} \f{1}{s_2-s'}\left[\pa_k G_{ij}\big(L(s_1,s')\big)- \pa_k G_{ij}\big(L(s_2,s')\big)\right] X_k'(s')\big(X_j'(s')-X_j'(s_1)\big)\,ds' \\
&\; + \int_{I^c} \f{1}{s_2-s'} \left[\pa_k G_{ij}\big(L(s_2,s')\big)- \pa_k G_{ij}(X'(s'))\right] X_k'(s')\big(X_j'(s_2)-X_j'(s_1)\big)\,ds' \\
&\; -\f{1}{4\pi} \int_\BT \left( \f{1}{s_1-s'}-\f{1}{2\tan\f{s_1-s'}{2}}
-\f{1}{s_2-s'}+\f{1}{2\tan\f{s_2-s'}{2}} \right) \big(X_i'(s') -X_i'(s_1) \big) \,ds'\\
&\; -\f{1}{4\pi} \int_\BT
\left(\f{1}{s_2-s'}-\f{1}{2\tan\f{s_2-s'}{2}}\right) \big(X_i'(s_2) -X_i'(s_1) \big) \,ds'\\
=:&\; \sum_{m = 1}^7 \big(g_m[X]\big)_i.
\end{split}
\label{eqn: splitting g_X(s_1)-g_X(s_2)}
\eeq
In all the integrals above, we choose to interpret $\BT:=[s_2-\pi,s_2+\pi)$, and $I^c : = \BT \setminus I$.
In what follows, we shall simply denote $g_m := g_m[X]$ whenever it incurs no confusion.

Arguing as before,
\begin{align*}
|g_1| \leq &\; C\int_{s_1-r}^{s_1+2r} \f{1}{|s_1-s'|}\cdot \lam^{-2}|L(s_1,s')-X'(s')| \cdot \|X'\|_{L^\infty}\|X'\|_{\dot{C}^\g}|s'-s_1|^\g \,ds'\\
\leq &\; C\lam^{-2}\|X'\|_{L^\infty}\|X'\|_{\dot{C}^\g}^2 \int_{s_1-r}^{s_1+2r} |s'-s_1|^{-1+2\g} \,ds'\\
\leq &\; C|s_1-s_2|^{2\g} \cdot \lam^{-2} \|X'\|_{L^\infty}\|X'\|_{\dot{C}^\g}^2.
\end{align*}
$|g_2|$ satisfies the same bound.

For any $s'\in I^c = [s_2-\pi,s_2+\pi)\setminus (s_2-2r,s_2+r)$, it holds that
\beq
\f12|s_2-s'| \leq |s_1-s'|\leq 2|s_2-s'|,
\quad
\left|\f{1}{s_1-s'}-\f{1}{s_2-s'}\right|\leq \f{C|s_1-s_2|}{|s_2-s'|^2}.
\label{eqn: estimate for (s_1-s')^-1 - (s_2-s')^-1}
\eeq
Note that this is true because we chose the ranges of $s'$ in both the integrals of $g_X(s_1)$ and $g_X(s_2)$ to be $[s_2-\pi,s_2+\pi)$, and also because we understood $s_1-s'$ and $s_2-s'$ as the usual subtraction on $\BR$.
Moreover, for any $s'\in I^c$,
\beq
\begin{split}
|L(s_1,s')-L(s_2,s')|
= &\; \left|\int_{s_1}^{s_2} \pa_{s}L(s,s')\,ds \right|
= \left|\int_{s_1}^{s_2} \f{-X'(s)(s'-s) + (X(s')-X(s))}{(s'-s)^2}\,ds \right|\\
\leq &\; C\int_{s_1}^{s_2} \|X'\|_{\dot{C}^\g} |s-s'|^{-1+\g} \,ds \leq C\|X'\|_{\dot{C}^\g}|s_1-s_2||s_2-s'|^{-1+\g}.
\end{split}
\label{eqn: estimate for L(s_1,s')-L(s_2,s')}
\eeq
Hence,
\begin{align*}
|g_3|+|g_4|
\leq
&\; C\int_{I^c} \f{|s_1-s_2|}{|s_2-s'|^2}\cdot \lam^{-2}|L(s_1,s')-X'(s')| \cdot \|X'\|_{L^\infty}\|X'\|_{\dot{C}^\g}|s'-s_1|^\g\,ds'\\
&\; + C\int_{I^c} \f{1}{|s_2-s'|}\cdot \lam^{-2}|L(s_1,s')-L(s_2,s')| \cdot \|X'\|_{L^\infty}\|X'\|_{\dot{C}^\g}|s'-s_1|^\g\,ds'\\
\leq
&\; C\int_{I^c} \f{|s_1-s_2|}{|s_2-s'|^2}\cdot \lam^{-2}|s_1-s'|^\g\|X'\|_{\dot{C}^\g}\cdot \|X'\|_{L^\infty}\|X'\|_{\dot{C}^\g}|s'-s_1|^{\g}\,ds'\\
&\; + C\int_{I^c} \f{1}{|s_2-s'|}\cdot \lam^{-2}\|X'\|_{\dot{C}^\g}|s_1-s_2||s_2-s'|^{-1+\g} \|X'\|_{L^\infty}\|X'\|_{\dot{C}^\g}|s'-s_1|^\g\,ds'\\
\leq
&\; C|s_1-s_2|\cdot \lam^{-2} \|X'\|_{L^\infty}\|X'\|_{\dot{C}^\g}^2 \int_{I^c} |s'-s_2|^{-2+2\g}\,ds'\\
\leq
&\; C|s_1-s_2|^{2\g}\cdot \lam^{-2} \|X'\|_{L^\infty}\|X'\|_{\dot{C}^\g}^2.
\end{align*}

Since
\begin{align*}
&\; \int_{s_2-\pi}^{s_2+\pi}  \f{1}{s_2-s'}X_k'(s')\pa_k G_{ij}\big(L(s_2,s')\big)
- \f{1}{s_2-s'} \pa_k G_{ij}(X'(s'))X_k'(s')\,ds'\\
= &\; \int_{s_2-\pi}^{s_2+\pi} -\pa_{s'}\big[G_{ij}\big(X(s_2)-X(s')\big)\big]
+ \f{1}{s_2-s'} \cdot \f{1}{4\pi}\d_{ij}\,ds' = 0,
\end{align*}
we find that
\beq
(g_5)_i
= -\big(X_j'(s_2)-X_j'(s_1)\big)
\int_{I}  \f{1}{s_2-s'}\big[\pa_k G_{ij}\big(L(s_2,s')\big)
- \pa_k G_{ij}(X'(s'))\big] X_k'(s')\,ds',
\label{eqn: g_5}
\eeq
which can be handled as $g_1$ and $g_2$:
\begin{align*}
|g_5|
\leq &\; C|s_1-s_2|^\g \|X'\|_{\dot{C}^\g}
\int_{I} \f{1}{|s_2-s'|}\cdot \lam^{-2}
\big|L(s_2,s')-X'(s')\big| \|X'\|_{L^\infty} \,ds'\\
\leq &\; C|s_1-s_2|^{2\g}\cdot \lam^{-2}\|X'\|_{L^\infty} \|X'\|_{\dot{C}^\g}^2.
\end{align*}

Since $s^{-1}-(2\tan(s/2))^{-1}$ is Lipschitz on $(-\f{3\pi}{2},\f{3\pi}{2})$,
\[
|g_6|\leq C|s_1-s_2|\|X'\|_{L^\infty}.
\]
Lastly, since $s'$ ranges in $[s_2-\pi,s_2+\pi)$, $g_7 = 0$ due to the oddness of the integrand.

Summarizing all the estimates yields that, for $|s_1-s_2|\leq \f{\pi}{4}$,
\[
|g_X(s_1)-g_X(s_2)| \leq C|s_1-s_2|^{2\g} \cdot \lam^{-2} \|X'\|_{L^\infty}\|X'\|_{\dot{C}^\g}^2,
\]
which implies that, for $\g\in (0,\f12)$,
\[
\|g_X\|_{\dot{C}^{2\g}(\BT)} \leq C \lam^{-2} \|X'\|_{L^\infty}\|X'\|_{\dot{C}^\g}^2.
\]

In the case $\g\in (\f12,1)$, the $C^{1,2\g-1}$-estimate for $g_X$ can be derived in a similar spirit.
Indeed, it has been proved in \cite[Proof of Proposition 3.8]{mori2019well} (also see \cite[Lemma 3.5]{lin2019solvability}) that
\begin{align*}
(g_X)'_i(s)
= &\;  \int_\BT \pa_{s}\left[X_k'(s')\pa_k G_{ij}\big(X(s)-X(s')\big)
+ \f{1}{4\pi}\f{\d_{ij}}{2\tan\f{s-s'}{2}}\right]
\big(X_j'(s')-X_j'(s)\big)\,ds'\\
= &\;  \int_\BT  \left[X_l'(s) X_k'(s')\pa_{kl} G_{ij}\big(X(s)-X(s')\big)
- \f{1}{4\pi}\f{\d_{ij}}{4\sin^2\f{s-s'}{2}}\right]
\big(X_j'(s')-X_j'(s)\big)\,ds'.
\end{align*}
Again, we may interpret $\BT$ to be an arbitrary interval of length $2\pi$ in $\BR$ which contains, say, $(s-\f{\pi}{2},s+\f{\pi}{2})$;
in this case, $|s'-s|\leq \f32\pi$.
We note that the integral does not depend on the choice of the interval.

Observe that $\pa_{kl}G_{ij}(x)$ is homogeneous of degree $-2$, and by differentiating \eqref{eqn: an identity for grad G}, for $x \neq 0$,
\[
\pa_{kl} G_{ij}(x) x_k + \pa_k G_{ij}(x)\d_{kl} = 0,
\]
which gives
\[
\pa_{kl} G_{ij}(x) x_k x_l = -x_l\pa_{l} G_{ij}(x) = \f{1}{4\pi}\d_{ij}.
\]
For $s\neq s'$, denote
\beq
\tilde{L}(s,s') := \f{X(s')-X(s)}{2\sin \f{s'-s}{2}},
\label{eqn: def of tilde L}
\eeq
where we understand $s'-s$ as the usual subtraction on $\BR$.
Then we derive as before to obtain
\beq
(g_X)'_i(s)
= \int_\BT \f{X_k'(s')}{4\sin^2 \f{s-s'}{2}}
\big[X_l'(s)\pa_{kl} G_{ij}\big(\tilde{L}(s,s')\big)
-X_l'(s') \pa_{kl} G_{ij}(X'(s'))\big]
\big(X_j'(s')-X_j'(s)\big)\,ds'.
\label{eqn: representation of g_X'}
\eeq

For distinct $s_1,s_2\in \BT$, define $r$ and $I$ as before.
We still assume $s_2 = s_1+r$ and $r\in (0,\f{\pi}{4})$.
Then
\beq
\begin{split}
&\;(g_X)'_i(s_1)-(g_X)'_i(s_2)\\
= &\; \int_I\f{X_k'(s')}{4\sin^2 \f{s_1-s'}{2}}
\big[X_l'(s_1)\pa_{kl} G_{ij}\big(\tilde{L}(s_1,s')\big)
-X_l'(s') \pa_{kl} G_{ij}(X'(s'))\big]
\big(X_j'(s')-X_j'(s_1)\big)\,ds'\\
&\; - \int_I \f{X_k'(s')}{4\sin^2 \f{s_2-s'}{2}}
\big[X_l'(s_2)\pa_{kl} G_{ij}\big(\tilde{L}(s_2,s')\big)
-X_l'(s') \pa_{kl} G_{ij}(X'(s'))\big]
\big(X_j'(s')-X_j'(s_2)\big)\,ds'\\
&\; + \int_{I^c} \left(\f{1}{4\sin^2 \f{s_1-s'}{2}}-\f{1}{4\sin^2 \f{s_2-s'}{2}}\right) X_k'(s')\\
&\;\qquad \qquad  \big[X_l'(s_1)\pa_{kl} G_{ij}\big(\tilde{L}(s_1,s')\big)
-X_l'(s') \pa_{kl} G_{ij}(X'(s'))\big]
\big(X_j'(s')-X_j'(s_1)\big)\,ds'\\
&\; + \int_{I^c} \f{X_k'(s')}{4\sin^2 \f{s_2-s'}{2}}
\big(X_l'(s_1)-X_l'(s_2)\big)
\pa_{kl} G_{ij}\big(\tilde{L}(s_1,s')\big)
\big(X_j'(s')-X_j'(s_1)\big)\,ds'\\
&\; + \int_{I^c} \f{X_k'(s')}{4\sin^2 \f{s_2-s'}{2}}
X_l'(s_2)\big[\pa_{kl} G_{ij}\big(\tilde{L}(s_1,s')\big)
-\pa_{kl} G_{ij}\big(\tilde{L}(s_2,s')\big)\big]
\big(X_j'(s')-X_j'(s_1)\big)\,ds'\\
&\; + \int_{I^c} \f{X_k'(s')}{4\sin^2 \f{s_2-s'}{2}}
\big[X_l'(s_2)\pa_{kl} G_{ij}\big(\tilde{L}(s_2,s')\big)
-X_l'(s') \pa_{kl} G_{ij}(X'(s'))\big]
\big(X_j'(s_2)-X_j'(s_1)\big)\,ds'
\\
=:&\; \sum_{m = 1}^6 \big(\tilde{g}_m[X]\big)_i.
\end{split}
\label{eqn: splitting g_X'(s_1)-g_X'(s_2)}
\eeq
Once again, in all the integrals above, we choose to interpret $\BT:=[s_2-\pi,s_2+\pi)$, and $I^c : = \BT \setminus I$.
For brevity, we shall simply denote $\tilde{g}_m := \tilde{g}_m[X]$ whenever it incurs no confusion.

Then we may estimate $\tilde{g}_m$ as before, with in mind that $\g\in (\f12,1)$.
We only sketch them as follows.
First we note that, for $s'\in [s-\f32\pi,s+\f32\pi]$ with $s'\neq s$,
\[
c|s'-s|\leq \left|2\sin\f{s'-s}{2}\right|\leq |s'-s|,\quad
c\lam \leq |\tilde{L}(s,s')|\leq C\|X'\|_{L^\infty}
\]
for some universal $c$ and $C$, and
\begin{align*}
|\tilde{L}(s,s')-X'(s')|
\leq &\; \left|\f{s'-s}{2\sin \f{s'-s}{2}}\right||L(s,s')-X'(s')|
+ \left|1-\f{s'-s}{2\sin \f{s'-s}{2}}\right||X'(s')|\\
\leq &\; C\|X'\|_{\dot{C}^\g} |s-s'|^\g + C\|X'\|_{L^\infty}|s-s'|
\leq C\|X'\|_{\dot{C}^\g} |s-s'|^\g.
\end{align*}
Then
\beq
\begin{split}
&\; \big|X_l'(s_1)\pa_{kl} G_{ij}\big(\tilde{L}(s_1,s')\big)
-X_l'(s') \pa_{kl} G_{ij}(X'(s'))\big|\\
\leq &\; |X_l'(s_1)-X_l'(s')|\big|\pa_{kl} G_{ij}\big(\tilde{L}(s_1,s')\big)\big|
+ |X_l'(s')| \big|\pa_{kl} G_{ij}\big(\tilde{L}(s_1,s')\big)
-\pa_{kl} G_{ij}(X'(s'))\big|\\
\leq &\; C|s_1-s'|^\g \|X'\|_{\dot{C}^\g}\cdot \lam^{-2}
+ C\|X'\|_{L^\infty}\cdot \lam^{-3}|\tilde{L}(s_1,s')-X'(s')|\\
\leq &\; C \lam^{-3} \|X'\|_{L^\infty}\|X'\|_{\dot{C}^\g}|s_1-s'|^\g.
\end{split}
\label{eqn: estimate for X' G'' difference}
\eeq
Hence,
\begin{align*}
|\tilde{g}_1|
\leq &\; C\int_I \f{\|X'\|_{L^\infty}}{|s_1-s'|^2}
\cdot \lam^{-3}\|X'\|_{L^\infty}\|X'\|_{\dot{C}^\g}|s_1-s'|^\g \cdot \|X'\|_{\dot{C}^\g}|s_1-s'|^\g \,ds'\\
\leq &\; C\lam^{-3}\|X'\|_{\dot{C}^\g}^2 \|X'\|_{L^\infty}^2 \int_I |s_1-s'|^{-2+2\g} \,ds'\\
\leq &\; C|s_1-s_2|^{2\g-1} \cdot \lam^{-3}\|X'\|_{\dot{C}^\g}^2 \|X'\|_{L^\infty}^2,
\end{align*}
and $|\tilde{g}_2|$ satisfies the same estimate.
Moreover, for $s'\in I^c$, by \eqref{eqn: estimate for L(s_1,s')-L(s_2,s')},
\beq
\begin{split}
&\;|\tilde{L}(s_1,s')-\tilde{L}(s_2,s')|\\
\leq &\; \left|\f{s'-s_1}{2\sin\f{s'-s_1}{2}} \right||L(s_1,s')-L(s_2,s')| + \left|\f{s'-s_1}{2\sin\f{s'-s_1}{2}} -\f{s'-s_2}{2\sin\f{s'-s_2}{2}} \right| |L(s_2,s')|\\
\leq &\; C\|X'\|_{\dot{C}^\g}|s_1-s_2||s_2-s'|^{-1+\g} + C|s_1-s_2| \|X'\|_{L^\infty}\\
\leq &\; C\|X'\|_{\dot{C}^\g}|s_1-s_2||s_2-s'|^{-1+\g},
\end{split}
\label{eqn: estimate for tilde L(s_1,s') - tilde L(s_2,s')}
\eeq
so we obtain that
\begin{align*}
&\;|\tilde{g}_3|+|\tilde{g}_4|+|\tilde{g}_5|+|\tilde{g}_6|\\
\leq
&\; C\int_{I^c} \f{|s_1-s_2|}{|s_2-s'|^3}\|X'\|_{L^\infty} \cdot \lam^{-3}
\|X'\|_{L^\infty}\|X'\|_{\dot{C}^\g}|s_1-s'|^\g
\cdot \|X'\|_{\dot{C}^\g}|s'-s_1|^\g \,ds'\\
&\; + C\int_{I^c} \f{1}{|s_2-s'|^2}\|X'\|_{L^\infty}
\cdot |s_1-s_2|^\g\|X'\|_{\dot{C}^\g}\cdot \lam^{-2}\cdot \|X'\|_{\dot{C}^\g}|s'-s_1|^\g\,ds'\\
&\; + C\int_{I^c} \f{1}{|s_2-s'|^2}\|X'\|_{L^\infty}^2
\cdot \lam^{-3} \|X'\|_{\dot{C}^\g} |s_1-s_2||s_2-s'|^{-1+\g} \cdot \|X'\|_{\dot{C}^\g}|s'-s_1|^\g \,ds'\\
&\; + C\int_{I^c} \f{1}{|s_2-s'|^2}\|X'\|_{L^\infty}
\cdot \lam^{-3} \|X'\|_{L^\infty}\|X'\|_{\dot{C}^\g}|s_2-s'|^\g\cdot
\|X'\|_{\dot{C}^\g}|s_1-s_2|^\g\,ds'\\
\leq
&\; C|s_1-s_2|^{2\g-1}\cdot \lam^{-3}\|X'\|_{L^\infty}^2\|X'\|_{\dot{C}^\g}^2.
\end{align*}
Therefore, when $\g\in (\f12,1)$,
\[
\|g_X\|_{\dot{C}^{1,2\g-1}(\BT)}
\leq
C\lam^{-3}\|X'\|_{L^\infty}^2\|X'\|_{\dot{C}^\g}^2.
\]
This completes the proof of Lemma \ref{lem: improved estimates for g_X}.

\subsection{Proof of Lemma \ref{lem: improved estimates for g_X-g_Y}}
\label{sec: proof of improved estimates for g_X-g_Y}

Similar to \eqref{eqn: def of L}, we denote for $s\neq s'$ that
\[
L_X(s,s') := \f{X(s')-X(s)}{s'-s},\quad
L_Y(s,s') := \f{Y(s')-Y(s)}{s'-s}.
\]
Thanks to \eqref{eqn: new form of g_X},
\beq
\begin{split}
&\; (g_X)_i(s)-(g_Y)_i(s)\\
= &\; \int_\BT \f{1}{s-s'}\left[\pa_k G_{ij}\big(L_X(s,s')\big)- \pa_k G_{ij}(X'(s'))\right.\\
&\;\qquad \qquad \quad \left.- \pa_k G_{ij}\big(L_Y(s,s')\big) + \pa_k G_{ij}(Y'(s'))\right] X_k'(s')X_j'(s')\,ds' \\
&\; +\int_\BT \f{1}{s-s'}\left[\pa_k G_{ij}\big(L_Y(s,s')\big)- \pa_k G_{ij}(Y'(s'))\right] \big(X_k'(s')X_j'(s')-Y_k'(s')Y_j'(s')\big)\,ds' \\
&\; -\f{1}{4\pi} \int_\BT \left(\f{1}{s-s'}- \f{1}{2\tan\f{s-s'}{2}}\right) \big(X_i'(s')-Y_i'(s')\big) \,ds'.
\end{split}
\label{eqn: splitting g_X-g_Y}
\eeq
Again, we understand $s'-s$ as the usual subtraction on $\BR$, and $\BT$ can be taken to be an arbitrary interval of length $2\pi$ in $\BR$ that contains, e.g., $(s-\f{\pi}{2},s+\f{\pi}{2})$.

We first claim that, for $k\in \BZ_+$,
\beq
\begin{split}
&\;\big| \na^k G\big(L_X(s,s')\big)- \na^k G(X'(s'))
- \na^k G\big(L_Y(s,s')\big) + \na^k G(Y'(s'))\big|\\
\leq &\; C |s-s'|^\g \Big[\lam^{-(k+1)}\|X'-Y'\|_{\dot{C}^\g}
+ \lam^{-(k+2)} \|X'-Y'\|_{L^\infty} \big(\|X'\|_{\dot{C}^\g}+\|Y'\|_{\dot{C}^\g}\big)\Big].
\end{split}
\label{eqn: estimate double difference}
\eeq
Indeed, if $\|X'-Y'\|_{L^\infty}\geq \lam/2$, we simply use \eqref{eqn: bounds for L(s,s')}, \eqref{eqn: bounds for L(s,s')-X'(s')}, and \eqref{eqn: bound for G} to derive that
\begin{align*}
&\;\big| \na^k G\big(L_X(s,s')\big)- \na^k G(X'(s'))
- \na^k G\big(L_Y(s,s')\big) + \na^k G(Y'(s'))\big|\\
\leq &\;\big| \na^k G\big(L_X(s,s')\big)- \na^k G(X'(s'))\big|+
\big|\na^k G\big(L_Y(s,s')\big) - \na^k G(Y'(s'))\big|\\
\leq &\;C \lam^{-(k+1)} |L_X(s,s')-X'(s')|
+ C\lam^{-(k+1)} |L_Y(s,s')-Y'(s')|\\
\leq &\;C \lam^{-(k+1)}\big(\|X'\|_{\dot{C}^\g}+\|Y'\|_{\dot{C}^\g}\big)|s-s'|^\g\\
\leq &\;C \lam^{-(k+2)} \|X'-Y'\|_{L^\infty} \big(\|X'\|_{\dot{C}^\g}+\|Y'\|_{\dot{C}^\g}\big)|s-s'|^\g.
\end{align*}
Otherwise, if $\|X'-Y'\|_{L^\infty}\leq \lam/2$, we let $Z(s;\th) := \th X(s)+(1-\th) Y(s)$, where $\th\in [0,1]$.
Then
\begin{align*}
&\;\big| \na^k G\big(L_X(s,s')\big)- \na^k G(X'(s'))
- \na^k G\big(L_Y(s,s')\big) + \na^k G(Y'(s'))\big|\\
= &\;\left|\int_0^1 \f{d}{d\th}\left[\na^k G\big(L_{Z(\cdot;\th)}(s,s')\big) - \na^k G(Z'(s';\th))\right] d\th\right|\\
\leq &\;\int_0^1 \left| \f{(X-Y)(s)-(X-Y)(s')}{s-s'}\cdot \na^{k+1} G\big(L_{Z(\cdot;\th)}(s,s')\big) - (X-Y)'(s')\cdot \na^{k+1} G(Z'(s';\th)) \right| d\th\\
\leq &\;\int_0^1 \left| \f{(X-Y)(s)-(X-Y)(s')}{s-s'}-(X-Y)'(s')\right|
\big|\na^{k+1} G\big(L_{Z(\cdot;\th)}(s,s')\big)\big| \\
&\;\qquad + |(X-Y)'(s')|\big|\na^{k+1} G\big(L_{Z(\cdot;\th)}(s,s')\big)-\na^{k+1} G(Z'(s';\th)) \big|\, d\th.
\end{align*}
Since
\[
\inf_{s\in \BT}|Z'(s)|
\geq |Z(\cdot;\th)|_*
\geq |Y|_*- \|X'-Y'\|_{L^\infty}
\geq \lam/2,
\]
by \eqref{eqn: bounds for L(s,s')}, \eqref{eqn: bounds for L(s,s')-X'(s')}, and \eqref{eqn: bound for G},
\begin{align*}
&\;\big| \na^k G\big(L_X(s,s')\big)- \na^k G(X'(s'))
- \na^k G\big(L_Y(s,s')\big) + \na^k G(Y'(s'))\big|\\
\leq &\; C\int_0^1 \|X'-Y'\|_{\dot{C}^\g}|s-s'|^\g\cdot
\big|L_{Z(\cdot;\th)}(s,s')\big|^{-(k+1)}\\
&\;\qquad + \|X'-Y'\|_{L^\infty}\cdot \lam^{-(k+2)} \big|L_{Z(\cdot;\th)}(s,s')-Z'(s';\th) \big|\, d\th\\
\leq &\; C \lam^{-(k+1)}\|X'-Y'\|_{\dot{C}^\g}|s-s'|^\g
+ C \lam^{-(k+2)} \|X'-Y'\|_{L^\infty} \big(\|X'\|_{\dot{C}^\g}+\|Y'\|_{\dot{C}^\g}\big)|s-s'|^\g.
\end{align*}
This proves \eqref{eqn: estimate double difference}.

Now combining \eqref{eqn: estimate double difference} with \eqref{eqn: splitting g_X-g_Y}, and also using \eqref{eqn: bounds for L(s,s')}, \eqref{eqn: bounds for L(s,s')-X'(s')}, and \eqref{eqn: bound for G}, we find that
\beq
\begin{split}
&\; |g_X(s)-g_Y(s)|\\
\leq &\; C\int_\BT \f{|s-s'|^\g }{|s-s'|} \Big[\lam^{-2}\|X'-Y'\|_{\dot{C}^\g}
+ \lam^{-3} \|X'-Y'\|_{L^\infty} \big(\|X'\|_{\dot{C}^\g}+\|Y'\|_{\dot{C}^\g}\big)\Big] \|X'\|_{L^\infty}^2 \,ds' \\
&\; +C \int_\BT \f{1}{|s-s'|}\cdot \lam^{-2} |L_Y(s,s') - Y'(s')| (\|X'\|_{L^\infty}+\|Y'\|_{L^\infty})
\|X'-Y'\|_{L^\infty}\,ds' \\
&\; +C \int_\BT \|X'-Y'\|_{L^\infty} \,ds'\\
\leq &\; C\big(\|X'\|_{L^\infty}+\|Y'\|_{L^\infty}\big)^2 \Big[\lam^{-2}\|X'-Y'\|_{\dot{C}^\g}
+ \lam^{-3} \|X'-Y'\|_{L^\infty} \big(\|X'\|_{\dot{C}^\g}+\|Y'\|_{\dot{C}^\g}\big)\Big],
\end{split}
\label{eqn: L inf bound for g_X-g_Y}
\eeq
which gives the $L^\infty$-bound.

The H\"{o}lder estimates for $g_X-g_Y$ can be justified as in the proof of Lemma \ref{lem: improved estimates for g_X}.
We first consider the case $\g\in (0,\f12)$.
Given distinct $s_1,s_2\in \BT$, we only consider the case $r:=|s_1-s_2|_\BT\in (0,\f{\pi}{4})$ and $s_2 = s_1 + r$.
Using the notations in \eqref{eqn: splitting g_X(s_1)-g_X(s_2)}, we write
\[
(g_X-g_Y)(s_1)-(g_X-g_Y)(s_2)
=  \sum_{m = 1}^7 g_m[X]-g_m[Y].
\]
We note that $g_7[X] = g_7[Y] = 0$.

By \eqref{eqn: bounds for L(s,s')-X'(s')} and \eqref{eqn: estimate double difference},
\begin{align*}
&\; |g_1[X]-g_1[Y]| \\
\leq
&\; \int_I \f{1}{|s_1-s'|} \left|\na G\big(L_X(s_1,s')\big)-\na G(X'(s'))-\na G\big(L_Y(s_1,s')\big)-\na G(Y'(s'))\right|\\
&\;\qquad \cdot |X'(s')| \big|X'(s')-X'(s_1)\big|\,ds'\\
&\; + \int_I \f{1}{|s_1-s'|} \left|\na G\big(L_Y(s_1,s')\big)-\na G(Y'(s'))\right| |(X'-Y')(s')| |X'(s')-X'(s_1)|\,ds'\\
&\; + \int_I \f{1}{|s_1-s'|} \left|\na G\big(L_Y(s_1,s')\big)-\na G(Y'(s'))\right|
|Y'(s')| \big|(X'-Y')(s')-(X'-Y')(s_1)\big|\,ds'\\
\leq
&\; C\int_{s_1-r}^{s_1+2r} \f{|s_1-s'|^\g}{|s_1-s'|} \Big[\lam^{-2}\|X'-Y'\|_{\dot{C}^\g}
+ \lam^{-3} \|X'-Y'\|_{L^\infty} \big(\|X'\|_{\dot{C}^\g}+\|Y'\|_{\dot{C}^\g}\big)\Big] \\ &\;\qquad \qquad \cdot \|X'\|_{L^\infty}\|X'\|_{\dot{C}^\g}|s'-s_1|^\g \,ds'\\
&\; + C\int_{s_1-r}^{s_1+2r} \f{1}{|s_1-s'|}\cdot \lam^{-2}|s_1-s'|^\g \|Y'\|_{\dot{C}^\g} \cdot \|X'-Y'\|_{L^\infty}\|X'\|_{\dot{C}^\g}|s'-s_1|^\g \,ds'\\
&\; + C\int_{s_1-r}^{s_1+2r} \f{1}{|s_1-s'|}\cdot \lam^{-2}|s_1-s'|^\g \|Y'\|_{\dot{C}^\g} \cdot \|Y'\|_{L^\infty}\|X'-Y'\|_{\dot{C}^\g}|s'-s_1|^\g \,ds'
\\
\leq
&\; C|s_1-s_2|^{2\g} \cdot \lam^{-2} \big(\|X'\|_{L^\infty}+\|Y'\|_{L^\infty}\big) \big(\|X'\|_{\dot{C}^\g}+\|Y'\|_{\dot{C}^\g}\big)\\
&\; \cdot \Big[\|X'-Y'\|_{\dot{C}^\g}+
\lam^{-1} \big(\|X'\|_{\dot{C}^\g}+\|Y'\|_{\dot{C}^\g}\big)\|X'-Y'\|_{L^\infty} \Big].
\end{align*}
Obviously, $|g_2[X]-g_2[Y]|$ satisfies the same bound.

Using \eqref{eqn: estimate for (s_1-s')^-1 - (s_2-s')^-1} and \eqref{eqn: estimate double difference}, we derive analogously that
\begin{align*}
&\; |g_3[X]-g_3[Y]|\\
\leq
&\; C\int_{I^c} \f{|s_1-s_2|}{|s_2-s'|^2}\cdot
|s_1-s'|^\g \Big[\lam^{-2}\|X'-Y'\|_{\dot{C}^\g}
+ \lam^{-3} \|X'-Y'\|_{L^\infty} \big(\|X'\|_{\dot{C}^\g}+\|Y'\|_{\dot{C}^\g}\big)\Big]\\
&\;\qquad \cdot \|X'\|_{L^\infty}\|X'\|_{\dot{C}^\g}|s'-s_1|^\g\,ds'\\
&\; + C\int_{I^c} \f{|s_1-s_2|}{|s_2-s'|^2}\cdot \lam^{-2}\|Y'\|_{\dot{C}^\g}|s_1-s'|^\g \cdot \|X'-Y'\|_{L^\infty}\|X'\|_{\dot{C}^\g}|s'-s_1|^\g
\,ds'\\
&\; + C\int_{I^c} \f{|s_1-s_2|}{|s_2-s'|^2}\cdot \lam^{-2}\|Y'\|_{\dot{C}^\g}|s_1-s'|^\g \cdot \|Y'\|_{L^\infty}\|X'-Y'\|_{\dot{C}^\g}|s'-s_1|^\g \,ds'\\
\leq
&\; C|s_1-s_2|^{2\g} \cdot \lam^{-2} \big(\|X'\|_{L^\infty}+\|Y'\|_{L^\infty}\big) \big(\|X'\|_{\dot{C}^\g}+\|Y'\|_{\dot{C}^\g}\big)\\
&\; \cdot \Big[\|X'-Y'\|_{\dot{C}^\g}+
\lam^{-1} \big(\|X'\|_{\dot{C}^\g}+\|Y'\|_{\dot{C}^\g}\big)\|X'-Y'\|_{L^\infty} \Big].
\end{align*}

To bound $|g_4[X]-g_4[Y]|$, we first follow the argument in the proof of \eqref{eqn: estimate double difference} to find that (also see \eqref{eqn: estimate for L(s_1,s')-L(s_2,s')}), for any $s'\in I^c$ and $k\in \BZ_+$,
\beq
\begin{split}
&\;\big| \na^k G\big(L_X(s_1,s')\big)- \na^k G\big(L_X(s_2,s')\big)
- \na^k G\big(L_Y(s_1,s')\big) + \na^k G\big(L_Y(s_2,s')\big)\big|\\
\leq &\; C |s_1-s_2||s_2-s'|^{-1+\g}\cdot \lam^{-(k+1)}\Big[\|X'-Y'\|_{\dot{C}^\g}
+ \lam^{-1}\|X'-Y'\|_{L^\infty}
\big(\|X'\| _{\dot{C}^\g}+\|Y'\| _{\dot{C}^\g}\big) \Big].
\end{split}
\label{eqn: estimate double difference s_1 and s_2}
\eeq
We omit its proof.
Using this and also \eqref{eqn: estimate for L(s_1,s')-L(s_2,s')}, we obtain that
\begin{align*}
&\; |g_4[X]-g_4[Y]|\\
\leq
&\; C\int_{I^c} \f{|s_1-s_2|}{|s_2-s'|}
|s_2-s'|^{-1+\g} \cdot \lam^{-2}\Big[\|X'-Y'\|_{\dot{C}^\g}
+ \lam^{-1}\|X'-Y'\|_{L^\infty}
\big(\|X'\| _{\dot{C}^\g}+\|Y'\| _{\dot{C}^\g}\big) \Big]\\
&\;\qquad \cdot \|X'\|_{L^\infty}\|X'\|_{\dot{C}^\g}|s'-s_1|^\g\,ds'\\
&\; + C\int_{I^c} \f{1}{|s_2-s'|}\cdot \lam^{-2} \|Y'\|_{\dot{C}^\g}|s_1-s_2||s_2-s'|^{-1+\g} \cdot \|X'-Y'\|_{L^\infty}\|X'\|_{\dot{C}^\g}|s'-s_1|^\g
\,ds'\\
&\; + C\int_{I^c}\f{1}{|s_2-s'|}\cdot \lam^{-2} \|Y'\|_{\dot{C}^\g}|s_1-s_2||s_2-s'|^{-1+\g} \cdot \|Y'\|_{L^\infty}\|X'-Y'\|_{\dot{C}^\g}|s'-s_1|^\g \,ds'\\
\leq
&\; C|s_1-s_2|^{2\g} \cdot \lam^{-2} \big(\|X'\|_{L^\infty}+\|Y'\|_{L^\infty}\big) \big(\|X'\|_{\dot{C}^\g}+\|Y'\|_{\dot{C}^\g}\big)\\
&\; \cdot \Big[\|X'-Y'\|_{\dot{C}^\g}+
\lam^{-1} \big(\|X'\|_{\dot{C}^\g}+\|Y'\|_{\dot{C}^\g}\big)\|X'-Y'\|_{L^\infty} \Big].
\end{align*}

By \eqref{eqn: g_5} as well as \eqref{eqn: bounds for L(s,s')-X'(s')} and \eqref{eqn: estimate double difference},
\begin{align*}
&\; |g_5[X] -g_5[Y]|\\
\leq &\; C|s_1-s_2|^\g \|X'-Y'\|_{\dot{C}^\g}
\int_{I} \f{1}{|s_2-s'|}\cdot \lam^{-2}
\|X'\|_{\dot{C}^\g}|s_2-s'|^\g\cdot \|X'\|_{L^\infty} \,ds'\\
&\; + C|s_1-s_2|^\g \|Y'\|_{\dot{C}^\g}\\
&\;\quad \cdot
\int_{I} \f{|s_2-s'|^\g }{|s_2-s'|} \Big[\lam^{-2}\|X'-Y'\|_{\dot{C}^\g}
+ \lam^{-3} \|X'-Y'\|_{L^\infty} \big(\|X'\|_{\dot{C}^\g}+\|Y'\|_{\dot{C}^\g}\big)\Big]
\|X'\|_{L^\infty} \,ds'\\
&\; + C|s_1-s_2|^\g \|Y'\|_{\dot{C}^\g}
\int_{I} \f{1}{|s_2-s'|}\cdot \lam^{-2}
\|Y'\|_{\dot{C}^\g}|s_2-s'|^\g\cdot \|X'-Y'\|_{L^\infty} \,ds'\\
\leq &\; C|s_1-s_2|^{2\g} \cdot \lam^{-2} \big(\|X'\|_{L^\infty}+\|Y'\|_{L^\infty}\big) \big(\|X'\|_{\dot{C}^\g}+\|Y'\|_{\dot{C}^\g}\big)\\
&\; \cdot \Big[\|X'-Y'\|_{\dot{C}^\g}+
\lam^{-1} \big(\|X'\|_{\dot{C}^\g}+\|Y'\|_{\dot{C}^\g}\big)\|X'-Y'\|_{L^\infty} \Big].
\end{align*}

Finally,
\[
|g_6[X] -g_6[Y]|\leq C|s_1-s_2|\|X'-Y'\|_{L^\infty}.
\]

Summarizing all the above estimates yields that, for $|s_1-s_2|\leq \f{\pi}{4}$,
\begin{align*}
&\; |(g_X-g_Y)(s_1)-(g_X-g_Y)(s_2)|\\
\leq &\; C|s_1-s_2|^{2\g} \cdot \lam^{-2} \big(\|X'\|_{L^\infty}+\|Y'\|_{L^\infty}\big) \big(\|X'\|_{\dot{C}^\g}+\|Y'\|_{\dot{C}^\g}\big)\\
&\; \cdot \Big[\|X'-Y'\|_{\dot{C}^\g}+
\lam^{-1} \big(\|X'\|_{\dot{C}^\g}+\|Y'\|_{\dot{C}^\g}\big)\|X'-Y'\|_{L^\infty} \Big].
\end{align*}
This, together with \eqref{eqn: L inf bound for g_X-g_Y}, implies the desired estimate in the case $\g\in (0,\f12)$.

Next we study the case $\g\in (\f12,1)$.
As in \eqref{eqn: def of tilde L}, for $s\neq s'$, we denote
\[
\tilde{L}_X(s,s') := \f{X(s')-X(s)}{2\sin \f{s'-s}{2}},\quad
\tilde{L}_Y(s,s') := \f{Y(s')-Y(s)}{2\sin \f{s'-s}{2}}.
\]
Here $s'-s$ is again understood as the usual subtraction on $\BR$.

For distinct $s_1,s_2\in \BT$, define $r$ and $I$ as before.
We still assume $s_2 = s_1+r$ and $r\in (0,\f{\pi}{4})$ without loss of generality.
Thanks to \eqref{eqn: representation of g_X'} and using the notations in \eqref{eqn: splitting g_X'(s_1)-g_X'(s_2)},
\[
(g_X-g_Y)'(s_1)-(g_X-g_Y)'(s_2)
=\sum_{m = 1}^6 \tilde{g}_m[X]-\tilde{g}_m[Y].
\]

As in \eqref{eqn: estimate double difference} and \eqref{eqn: estimate double difference s_1 and s_2}, we also have for $k\in \BZ_+$ that
\beq
\begin{split}
&\;\big| \na^k G\big(\tilde{L}_X(s,s')\big)- \na^k G(X'(s'))
- \na^k G\big(\tilde{L}_Y(s,s')\big) + \na^k G(Y'(s'))\big|\\
\leq &\; C |s-s'|^\g \cdot \lam^{-(k+1)}\Big[\|X'-Y'\|_{\dot{C}^\g}
+ \lam^{-1} \|X'-Y'\|_{L^\infty} \big(\|X'\|_{\dot{C}^\g}+\|Y'\|_{\dot{C}^\g}\big)\Big],
\end{split}
\label{eqn: estimate double difference tilde}
\eeq
and for any $k\in \BZ_+$ and $s'\in I^c$,
\beq
\begin{split}
&\;\big| \na^k G\big(\tilde{L}_X(s_1,s')\big)- \na^k G\big(\tilde{L}_X(s_2,s')\big)
- \na^k G\big(\tilde{L}_Y(s_1,s')\big) + \na^k G\big(\tilde{L}_Y(s_2,s')\big)\big|\\
\leq &\; C |s_1-s_2||s_2-s'|^{-1+\g}\cdot \lam^{-(k+1)}\Big[\|X'-Y'\|_{\dot{C}^\g}
+ \lam^{-1}\|X'-Y'\|_{L^\infty}
\big(\|X'\| _{\dot{C}^\g}+\|Y'\| _{\dot{C}^\g}\big) \Big].
\end{split}
\label{eqn: estimate double difference s_1 and s_2 tilde}
\eeq
whose proofs will be omitted as they are the same as those of \eqref{eqn: estimate double difference} and \eqref{eqn: estimate double difference s_1 and s_2}.

As a result of \eqref{eqn: estimate double difference tilde},
\begin{align*}
&\;\Big| X_l'(s_1)\pa_{kl} G\big(\tilde{L}_X(s_1,s')\big)
-X_l'(s') \pa_{kl} G(X'(s'))\\
&\; - Y_l'(s_1)\pa_{kl} G\big(\tilde{L}_Y(s_1,s')\big)
+ Y_l'(s') \pa_{kl} G(Y'(s'))\Big|\\
\leq &\;\Big| \big(X_l'(s_1)-Y_l'(s_1)\big)\pa_{kl} G\big(\tilde{L}_X(s_1,s')\big)
- \big(X_l'(s')-Y_l'(s')\big) \pa_{kl} G(X'(s'))\Big|\\
&\; + \Big| Y_l'(s_1)\big[\pa_{kl} G\big(\tilde{L}_X(s_1,s')\big)- \pa_{kl} G\big(\tilde{L}_Y(s_1,s')\big)\big]
- Y_l'(s') \big[\pa_{kl} G(X'(s'))-\pa_{kl} G(Y'(s'))\big]\Big|
\\
\leq &\;\big| (X-Y)'(s_1)-(X-Y)'(s')\big| \big|\na^2 G\big(\tilde{L}_X(s_1,s')\big)\big|\\
&\;+\big|(X-Y)'(s')\big| \big|\na^2 G\big(\tilde{L}_X(s_1,s')\big)- \na^2 G(X'(s'))\big|\\
&\; + \big| Y'(s_1)- Y'(s')\big|\big|\na^2 G\big(\tilde{L}_X(s_1,s')\big)- \na^2 G\big(\tilde{L}_Y(s_1,s')\big)\big|\\
&\; + |Y'(s')|\Big|\na^2 G \big(\tilde{L}_X(s_1,s')\big)- \na^2 G\big(\tilde{L}_Y(s_1,s')\big)
- \na^2 G(X'(s')) + \na^2 G(Y'(s'))\Big|
\\
\leq &\;C|s_1-s'|^\g\|X'-Y'\|_{\dot{C}^\g} \big|\tilde{L}_X(s_1,s')\big|^{-2}
+C \|X'-Y'\|_{L^\infty} \cdot
\lam^{-3} \big|\tilde{L}_X(s_1,s')-X'(s')\big|\\
&\; + C|s_1-s'|^\g\|Y'\|_{\dot{C}^\g}
\cdot \lam^{-3} \big|\tilde{L}_X(s_1,s')-\tilde{L}_Y(s_1,s')\big|\\
&\; + C \|Y'\|_{L^\infty} \cdot |s_1-s'|^\g \cdot \lam^{-3} \Big[\|X'-Y'\|_{\dot{C}^\g}
+ \lam^{-1} \|X'-Y'\|_{L^\infty} \big(\|X'\|_{\dot{C}^\g}+\|Y'\|_{\dot{C}^\g}\big)\Big]
\\
\leq &\; C |s_1-s'|^\g \big(\|X'\|_{L^\infty}+\|Y'\|_{L^\infty}\big) \\
&\; \cdot \lam^{-3}\Big[\|X'-Y'\|_{\dot{C}^\g}
+ \lam^{-1} \|X'-Y'\|_{L^\infty} \big(\|X'\|_{\dot{C}^\g}+\|Y'\|_{\dot{C}^\g}\big)\Big].
\end{align*}
Using this estimate as well as \eqref{eqn: estimate for X' G'' difference}, we derive that
\begin{align*}
&\;\big|\tilde{g}_1[X]-\tilde{g}_1[Y]\big|\\
\leq &\; \int_I\f{|X'(s')-Y'(s')|}{4\sin^2 \f{s_1-s'}{2}}
\big|X_l'(s_1)\pa_{kl} G\big(\tilde{L}_X(s_1,s')\big)
-X_l'(s') \pa_{kl} G(X'(s'))\big|
\big|X_j'(s')-X_j'(s_1)\big|\,ds'\\
&\; + \int_I\f{|Y'(s')|}{4\sin^2 \f{s_1-s'}{2}}
\big|X_l'(s_1)\pa_{kl} G\big(\tilde{L}_X(s_1,s')\big)
-X_l'(s') \pa_{kl} G(X'(s'))\\
&\;\qquad \qquad \qquad -Y_l'(s_1)\pa_{kl} G\big(\tilde{L}_Y(s_1,s')\big)
+Y_l'(s') \pa_{kl} G(Y'(s'))\big|
\big|X'(s')-X'(s_1)\big|\,ds'\\
&\; + \int_I\f{|Y'(s')|}{4\sin^2 \f{s_1-s'}{2}}
\big|Y_l'(s_1)\pa_{kl} G\big(\tilde{L}_Y(s_1,s')\big)
-Y_l'(s') \pa_{kl} G(Y'(s'))\big|
\big|(X-Y)'(s')-(X-Y)'(s_1)\big|\,ds'
\\
\leq &\; C\int_I\f{\|X'-Y'\|_{L^\infty}}{|s_1-s'|^2}
\cdot \lam^{-3} \|X'\|_{L^\infty}\|X'\|_{\dot{C}^\g}|s_1-s'|^\g
\cdot |s'-s_1|^\g\|X'\|_{\dot{C}^\g}\,ds'\\
&\; + C\int_I\f{\|Y'\|_{L^\infty}}{|s_1-s'|^2}
\cdot |s_1-s'|^\g \big(\|X'\|_{L^\infty}+\|Y'\|_{L^\infty}\big) \\
&\; \qquad \cdot \lam^{-3} \Big[\|X'-Y'\|_{\dot{C}^\g}
+ \lam^{-1} \|X'-Y'\|_{L^\infty} \big(\|X'\|_{\dot{C}^\g}+\|Y'\|_{\dot{C}^\g}\big)\Big]
\cdot |s'-s_1|^\g\|X'\|_{\dot{C}^\g}\,ds'\\
&\; + C\int_I\f{\|Y'\|_{L^\infty}}{|s_1-s'|^2}
\cdot \lam^{-3} \|Y'\|_{L^\infty}\|Y'\|_{\dot{C}^\g}|s_1-s'|^\g
\cdot |s'-s_1|^\g\|(X-Y)'\|_{\dot{C}^\g}\,ds'
\\
\leq &\; C|s_1-s_2|^{2\g-1}
\big(\|X'\|_{L^\infty}+\|Y'\|_{L^\infty}\big)^2 \big(\|X'\|_{\dot{C}^\g}+\|Y'\|_{\dot{C}^\g}\big)\\
&\; \cdot \lam^{-3} \Big[\|X'-Y'\|_{\dot{C}^\g}
+ \lam^{-1} \|X'-Y'\|_{L^\infty} \big(\|X'\|_{\dot{C}^\g}+\|Y'\|_{\dot{C}^\g}\big)\Big].
\end{align*}
$|\tilde{g}_2[X]-\tilde{g}_2[Y]|$ enjoys the same bound.
Similarly,
\begin{align*}
&\;\big|\tilde{g}_3[X]-\tilde{g}_3[Y]\big|
+ \big|\tilde{g}_4[X]-\tilde{g}_4[Y]\big|
+ \big|\tilde{g}_6[X]-\tilde{g}_6[Y]\big|
\\
\leq &\; C|s_1-s_2|^{2\g-1} \big(\|X'\|_{L^\infty}+\|Y'\|_{L^\infty}\big)^2 \big(\|X'\|_{\dot{C}^\g}+ \|Y'\|_{\dot{C}^\g}\big)\\
&\; \cdot \lam^{-3}\Big[\|X'-Y'\|_{\dot{C}^\g}
+ \lam^{-1} \|X'-Y'\|_{L^\infty} \big(\|X'\|_{\dot{C}^\g}+\|Y'\|_{\dot{C}^\g}\big)\Big].
\end{align*}
We omit the details.
For $\tilde{g}_5[X]-\tilde{g}_5[Y]$, we first derive that
\begin{align*}
&\;\big|\tilde{g}_5[X]-\tilde{g}_5[Y]\big|\\
\leq
&\; C\int_{I^c} \f{|X_k'(s') X_l'(s_2)-Y_k'(s') Y_l'(s_2)|}{|s_2-s'|^2}\\
&\;\qquad \cdot
\big|\na^2 G\big(\tilde{L}_X(s_1,s')\big)
-\na^2 G\big(\tilde{L}_X(s_2,s')\big)\big|
\big|X'(s')-X'(s_1)\big|\,ds'\\
&\; + C\int_{I^c} \f{|Y_k'(s') Y_l'(s_2)|}{|s_2-s'|^2}
\big|\na^2 G\big(\tilde{L}_X(s_1,s')\big)
-\na^2 G\big(\tilde{L}_X(s_2,s')\big)\\
&\;\qquad \qquad \qquad
-\na^2 G\big(\tilde{L}_Y(s_1,s')\big)
+\na^2 G\big(\tilde{L}_Y(s_2,s')\big)\big|
\big|X'(s')-X'(s_1)\big|\,ds'\\
&\; + C\int_{I^c} \f{|Y_k'(s') Y_l'(s_2)|}{|s_2-s'|^2}
\big|\na^2 G\big(\tilde{L}_Y(s_1,s')\big)
- \na^2 G\big(\tilde{L}_Y(s_2,s')\big)\big|\\
&\;\qquad \quad \cdot \big|(X-Y)'(s')-(X-Y)'(s_1)\big|\,ds'
\\
\leq &\; C\int_{I^c} \f{1}{|s_2-s'|^2}\big(\|X'\|_{L^\infty}+\|Y'\|_{L^\infty}\big) \|X'-Y'\|_{L^\infty}\\
&\;\qquad \cdot \lam^{-3}\big|\tilde{L}_X(s_1,s')-\tilde{L}_X(s_2,s')\big|
\cdot |s'-s_1|^\g \|X'\|_{\dot{C}^\g}\,ds'\\
&\; + C\int_{I^c} \f{\|Y'\|_{L^\infty}^2}{|s_2-s'|^2}
\big|\na^2 G\big(\tilde{L}_X(s_1,s')\big)
-\na^2 G\big(\tilde{L}_X(s_2,s')\big)\\
&\;\qquad \qquad \qquad
-\na^2 G\big(\tilde{L}_Y(s_1,s')\big)
+\na^2 G\big(\tilde{L}_Y(s_2,s')\big)\big|
\cdot |s'-s_1|^\g\|X'\|_{\dot{C}^\g}\,ds'\\
&\; + C\int_{I^c} \f{\|Y'\|_{L^\infty}^2}{|s_2-s'|^2}
\cdot \lam^{-3}\big|\tilde{L}_Y(s_1,s')-\tilde{L}_Y(s_2,s')\big|
\cdot |s'-s_1|^\g\|X'-Y'\|_{\dot{C}^\g}\,ds'.
\end{align*}
By virtue of \eqref{eqn: estimate for tilde L(s_1,s') - tilde L(s_2,s')} and \eqref{eqn: estimate double difference s_1 and s_2 tilde},
\begin{align*}
&\;\big|\tilde{g}_5[X]-\tilde{g}_5[Y]\big|
\\
\leq &\; C\int_{I^c} \f{1}{|s_2-s'|^2}\big(\|X'\|_{L^\infty}+\|Y'\|_{L^\infty}\big) \|X'-Y'\|_{L^\infty}\\
&\;\qquad \cdot \lam^{-3}\|X'\|_{\dot{C}^\g}|s_1-s_2||s_2-s'|^{-1+\g}
\cdot |s'-s_1|^\g \|X'\|_{\dot{C}^\g}\,ds'\\
&\; + C\int_{I^c} \f{\|Y'\|_{L^\infty}^2}{|s_2-s'|^2}
\cdot |s_1-s_2||s_2-s'|^{-1+\g}\\
&\;\qquad \quad \cdot \lam^{-3}\Big[\|X'-Y'\|_{\dot{C}^\g}
+ \lam^{-1}\|X'-Y'\|_{L^\infty}
\big(\|X'\| _{\dot{C}^\g}+\|Y'\| _{\dot{C}^\g}\big) \Big]
\cdot |s'-s_1|^\g\|X'\|_{\dot{C}^\g}\,ds'\\
&\; + C\int_{I^c} \f{\|Y'\|_{L^\infty}^2}{|s_2-s'|^2}
\cdot \lam^{-3}\|Y'\|_{\dot{C}^\g}|s_1-s_2||s_2-s'|^{-1+\g}
\cdot |s'-s_1|^\g\|X'-Y'\|_{\dot{C}^\g}\,ds'
\\
\leq &\; C|s_1-s_2|^{2\g-1} \big(\|X'\|_{L^\infty}+\|Y'\|_{L^\infty}\big)^2
\big(\|X'\|_{\dot{C}^\g}+\|Y'\|_{\dot{C}^\g}\big)\\
&\;\cdot \lam^{-3}\Big[\|X'-Y'\|_{\dot{C}^\g}
+ \lam^{-1}\|X'-Y'\|_{L^\infty}
\big(\|X'\| _{\dot{C}^\g}+\|Y'\| _{\dot{C}^\g}\big) \Big].
\end{align*}

Combining all the estimates, we obtain that
\begin{align*}
&\;\big|(g_X-g_Y)'(s_1)-(g_X-g_Y)'(s_2)\big|\\
\leq &\; C|s_1-s_2|^{2\g-1} \big(\|X'\|_{L^\infty}+\|Y'\|_{L^\infty}\big)^2
\big(\|X'\|_{\dot{C}^\g}+\|Y'\|_{\dot{C}^\g}\big)\\
&\;\cdot \lam^{-3}\Big[\|X'-Y'\|_{\dot{C}^\g}
+ \lam^{-1}\big(\|X'\| _{\dot{C}^\g}+\|Y'\| _{\dot{C}^\g}\big) \|X'-Y'\|_{L^\infty}\Big],
\end{align*}
which implies the desired H\"{o}lder estimate in the case $\g\in (\f12,1)$.


\end{document}